\documentclass[twoside,a4paper]{article}
\usepackage{amsthm,amsmath,amssymb,amscd,graphicx,amsfonts,stmaryrd}
\usepackage{mathrsfs}
\usepackage[all]{xy}
\usepackage[margin=25mm]{geometry}
\usepackage{bigints}
\usepackage{mathtools}
\usepackage{enumitem}
\usepackage{comment}
\usepackage{fancyhdr}
\usepackage{latexsym}
\usepackage{amsfonts}
\usepackage{bm}
\usepackage[d]{esvect}
\usepackage{array}
\usepackage{tikz}
\usepackage{hyphenat}

\thepage

\allowdisplaybreaks[1]
\mathtoolsset{showonlyrefs=true}
\usetikzlibrary{positioning, intersections, calc, arrows.meta,math}

\DeclareMathOperator{\Leb}{Leb}
\DeclareMathOperator{\id}{id}
\DeclareMathOperator{\closure}{cl}
\DeclareMathOperator{\dom}{dom}
\DeclareMathOperator{\At}{At} 
 
\DeclareMathOperator{\diam}{diam} 
 
\DeclareMathOperator{\tr}{tr} 

\numberwithin{equation}{section}

\theoremstyle{definition}
\newtheorem{exm} {Example}[section]
\newtheorem{dfn}[exm] {Definition}
\newtheorem{rem}[exm] {Remark}

\theoremstyle{plane}
\newtheorem{lem}[exm]{Lemma}
\newtheorem{prop}[exm]{Proposition}
\newtheorem{thm}[exm]{Theorem}
\newtheorem{cor}[exm]{Corollary}
\newtheorem{assum}[exm]{Assumption}

\newcommand{\NN}{\mathbb{N}}
\newcommand{\ZN}{\mathbb{Z}}
\newcommand{\ZNp}{\mathbb{Z}_{\geq 0}}
\newcommand{\QN}{\mathbb{Q}}
\newcommand{\RN}{\mathbb{R}} 
\newcommand{\RNp}{\mathbb{R}_{\geq 0}}
\newcommand{\RNpp}{\mathbb{R}_{>0}}

\newcommand{\Rspace}{\mathbb{F}}

\newcommand{\GHVspace}{\mathbb{G}}
\newcommand{\GHPspace}{\mathbb{G}_{c}}

\newcommand{\Borel}{\mathcal{B}}

\newcommand{\closed}{\mathcal{C}}
\newcommand{\compact}{\mathcal{C}_{c}}

\newcommand{\domain}{\mathcal{D}}
\newcommand{\form}{\mathcal{E}}
\newcommand{\rdomain}{\mathcal{F}}

\newcommand{\law}{\mathcal{L}}

\newcommand{\cS}{\mathcal{S}}
\newcommand{\cT}{\mathcal{T}}

\newcommand{\cX}{\mathcal{X}}

\newcommand{\Meas}{\mathcal{M}}
\newcommand{\finMeas}{\Meas_{\mathrm{fin}}}
\newcommand{\disMeasures}{\mathcal{M}^{\mathrm{dis}}}
\newcommand{\intMeasures}{\mathcal{N}}
\newcommand{\pointMeas}{\mathscr{P}}

\newcommand{\rbcM}{\mathfrak{M}}

\newcommand{\pointMap}{\mathfrak{p}}
\newcommand{\vertMap}{\mathfrak{m}}
\newcommand{\MeasureMap}{\mathfrak{v}}

\newcommand{\probFunct}{\tau_{\mathcal{P}}}
\newcommand{\PointFunct}{\tau^{\mathrm{pt}}}
\newcommand{\nPointFunct}[1]{\tau^{#1 \text{-}\mathrm{pts}}}
\newcommand{\finMeasFunct}{\tau^{\finMeas}}
\newcommand{\MeasFunct}{\tau^{\Meas}}
\newcommand{\markedMeasFunct}[1]{\tau^{\Meas(\cdot \times #1)}}
\newcommand{\markedfinMeasFunct}[1]{\tau^{\finMeas(\cdot \times #1)}}
\newcommand{\SkorohodFunct}{\tau^{J_{1}}}
\newcommand{\fixedFunct}[1]{\tau^{#1}}
\newcommand{\unifFunct}{\tau^{\mathrm{Unif}}}
\newcommand{\disMeasFunct}{\tau^{\Meas^{\mathrm{dis}}}}

\newcommand{\hatC}{\widehat{C}}

\newcommand{\HausdMet}[1]{d_{H}^{#1}}
\newcommand{\lHausdMet}[2]{d_{\bar{H}}^{#1, #2}}

\newcommand{\neighb}[2]{#1^{\langle #2 \rangle}}

\newcommand{\muh}{\mu^{\#}}
\newcommand{\dotmuh}{\dot{\mu}^{\#}}

\title{Aging and sub-aging for Bouchaud trap models on resistance metric spaces} 
\date{}
\author{Ryoichiro Noda\thanks{Research Institute for Mathematical Sciences, Kyoto University, Kyoto, 606-8502,
JAPAN. E-mail:sgrndr@kurims.kyoto-u.ac.jp}}

\begin{document}

\maketitle
\begin{abstract}
  In this paper,
  we prove that if a sequence of electrical networks converges in the local Gromov-Hausdorff topology 
  and satisfies a non-explosion condition,
  then the associated Bouchaud trap models (BTMs) also converge 
  and exhibit aging.
  Moreover, when local structures of electrical networks converge,
  we prove sub-aging.
  Our results are applicable to a wide class of low-dimensional graphs, 
  including the two-dimensional Sierpi\'{n}ski gasket, critical Galton-Watson trees, and the critical Erd\H{o}s-R\'{e}nyi random graph.
  The proof consists of two main steps:
  Polish metrization of the vague-and-point-process topology and showing the precompactness of transition densities of BTMs.
\end{abstract}

\tableofcontents


\section{Introduction} \label{sec: introduction}

Aging refers to the phenomenon in which a system never reaches equilibrium on laboratory time scales, 
and is typically observed in disordered media such as spin glasses at low temperatures.
This out-of-equilibrium physical behavior has been of great interest in condensed matter physics for over thirty years 
and has been much discussed in the literature; for the physical background see \cite{Bouchaud_Cugliandolo_etc_97_Out}, for example.

Research on aging in mathematics began about twenty years ago, 
and aging has been proven in several spin glass models:
see \cite{Arous_02_aging,Mathieu_Mourrat_15_Aging} and the references therein. 
To understand the mechanism of aging,
in \cite{Bouchaud_92_Weak},
Bouchaud proposed a toy model, now called the Bouchaud trap model (BTM).
Over several papers,
Ben Arous and \v{C}ern\'{y} studied aging in the Bouchaud trap model (BTM), 
and their results are summarized in \cite{Arous_Cerny_06_Dynamics}. 
The BTM treated in this article is set out in Definition \ref{1. dfn: BTM} below,
but, for discussion, here we introduce a simplified version of it.
The (symmetric and simplified) Bouchaud trap model refers to a Markov chain on a randomly weighted graph defined as follows:
fix a connected, simple, undirected graph $G = (V, E)$ with finite vertex set $V$ and edge set $E$;
let $\nu = (\xi_{x})_{x \in V}$ be a family of i.i.d.\ random variables
built on a probability space with probability measure $\mathsf{P}$
such that $u^{\alpha}\mathsf{P}(\xi_{x} > u) \to 1$
as $u \to \infty$ for some $\alpha \in (0,1)$;
conditional on $\nu$, the BTM is the continuous-time Markov chain $(X^{\nu}, \{P_{x}^{\nu}\}_{x \in V})$ on $V$ 
whose jump rate $w_{xy}$ is given by $w_{xy} \coloneqq \xi_{x}^{-1}$ if $\{x,y\} \in E$ and otherwise $w_{xy} \coloneqq 0$.
The idea to quantify aging is to consider a suitably chosen two-point function $C(t_{w}, t_{w} + t)$ of $X^{\nu}$ between times $t_{w}$ and $t_{w} + t$.
It has been observed that good two-point functions are given by, for example,
\begin{gather}
  C_{1}(t_{w}, t_{w} + t) 
  \coloneqq 
  \mathsf{E} 
  \left[
    P^{\nu}
    \bigl( X^{\nu}(t_{w}) = X^{\nu}(t_{w}+t) \bigr)
  \right],\\
  C_{2}(t_{w}, t_{w} + t) 
  \coloneqq 
  \mathsf{E} 
  \left[
    P^{\nu}
    \bigl(
      X^{\nu}(t_{w}) = X^{\nu}(t_{w}+s),\, \forall s \in [0, t]
    \bigr)
  \right],
\end{gather}
where $\mathsf{E}$ denotes the expectation with respect to $\mathsf{P}$.
Physical experiments that are represented by this model suggest that $C(t_{w}, t_{w}+t)$ only depends on $t/h(t_{w})$
(for sufficiently large $t_{w}$),
where $h$ is an increasing function.
Therefore,
aging is proven by showing the existence of the following non-trivial limit:
\begin{equation}
  R(\theta) 
  = 
  \lim_{t_{w} \to \infty} 
  C(t_{w}, t_{w}+\theta h(t_{w})).
\end{equation}
If $h(t_{w}) = t_{w}$, then it is called (full) aging,
and if $h(t_{w}) = o(t_{w})$, 
then it is called sub-aging.

In \cite{Arous_Cerny_05_Bouchaud,Arous_Cerny_08_The,Arous_Cerny_Mountford_06_Aging},
Ben Arous, \v{C}ern\'{y} and Mountford studied the BTM for $G = \ZN^{d}$ with nearest neighbor edges,
proving aging for $C_{1}$ and sub-aging for $C_{2}$, as defined above.
In particular, 
it was found that the descriptions of the limits of the two-point functions are significantly different between $d=1$ and $d \geq 2$. 
This is due to the difference in the scaling limits of the BTMs. 
When $d=1$, 
the scaling limit of the BTM is a Markov process called Fontes-Isopi-Newman (F.I.N.) diffusion,
firstly proven in \cite{Fontes_Isopi_Newman_02_Random}.
On the other hand, 
when $d \geq 2$, 
it converges to a non-Markovian process called the fractional-kinetics process \cite{Arous_Cerny_07_Scaling}. 
This suggests that the BTMs can be divided into low-dimensional and high-dimensional regimes. 
Ben Arous and \v{C}ern\'{y} generalized their discussions of the BTM on $\ZN^{d}$ with $d \geq 2$
and obtained a method to prove (sub-)aging for a class of high-dimensional graphs, including complete graphs. 
Example graphs in the low-dimensional regime other than $\ZN$ were found by Croydon, Hambly and Kumagai \cite{Croydon_Hambly_Kumagai_17_Time-changes}. 
Using the theory of resistance forms developed by Kigami \cite{Kigami_01_Analysis, Kigami_12_Resistance}, 
they obtained that if a sequence of graphs converges in the local Gromov-Hausdorff-vague topology
(introduced in Section \ref{sec: GH-type topologies}) 
as resistance metric spaces equipped with the counting measures
and it satisfies the uniform volume doubling (UVD) condition
(see \cite[Definition 1.1]{Croydon_Hambly_Kumagai_17_Time-changes}),
then the associated BTMs converge.
In particular, their results are applicable to the (two-dimensional) Sierpi\'{n}ski gasket.
However, (sub-)aging results were left open. 
In this paper, 
we improve their results by replacing the UVD condition with a weaker condition, the non-explosion condition introduced in \cite{Croydon_18_Scaling}.
Moreover,
we show (sub-)aging for the associated BTMs.
Our results are applicable to a wide class of low-dimensional graphs, 
including the Sierpi\'{n}ski gasket, the critical Galton-Watson tree, and the critical Erd\H{o}s-R\'{e}nyi random graph.

To present our main results,
we begin by introducing several pieces of notation.
We write $\RNp \coloneqq [0, \infty)$, $\RNpp \coloneqq (0, \infty)$ and $\ZNp \coloneqq \ZN \cap \RNp$.
For $a, b \in \mathbb{R}$,
we write $a \vee b \coloneqq \max\{a, b\}$ and $a \wedge b \coloneqq \min\{a, b\}$.
Given a metric space $(S, d)$,
we set, for $x \in S$ and $r>0$,
\begin{equation}
  B_{S}(x,r) 
  =
  B_{d}(x, r) 
  \coloneqq 
  \{
    y \in S \mid d(x, y) < r
  \}, 
  \quad 
  D_{S}(x,r)
  =
  D_{d}(x, r) 
  \coloneqq 
  \{
    y \in S \mid d(x, y) \leq r
  \}.
\end{equation}
We say that $(S, d)$ is \textit{boundedly compact} 
if and only if $D_{S}(x, r)$ is compact for all $x \in S$ and $r>0$.
Note that a boundedly-compact metric space is complete, separable and locally compact.
A tuple $(S, d, \rho, \mu)$ is said to be a \textit{rooted-and-measured boundedly-compact metric space} 
if and only if $(S, d)$ is a boundedly-compact metric space,
$\rho$ is a distinguished element of $S$ called the \textit{root},
and $\mu$ is a Radon measure on $S$, 
that is,
$\mu$ is a Borel measure on $S$ 
such that $\mu(K) < \infty$ for every compact subset $K$.
Given a rooted-and-measured boundedly-compact metric space $G = (S, d, \rho, \mu)$,
we define a rooted-and-measured compact metric space $G^{(r)} = (S^{(r)}, d^{(r)}, \rho^{(r)}, \mu^{(r)})$
by setting 
\begin{equation}  \label{1. eq: dfn of restriction operator}
  S^{(r)} \coloneqq \closure(B_{d}(\rho, r)), \quad 
  d^{(r)} \coloneqq d|_{S^{(r)} \times S^{(r)}}, \quad 
  \rho^{(r)} \coloneqq \rho, \quad 
  \mu^{(r)}(\cdot) \coloneqq \mu( \cdot \cap S^{(r)}),
\end{equation}
where $\closure(\cdot)$ denotes the closure of a set.
We write $\mathbb{G}$ for the collection of rooted-and-measured isometric equivalence classes 
of rooted-and-measured boundedly-compact metric spaces
and equip $\mathbb{G}$ with the local Gromov-Hausdorff-vague topology.
(See Section \ref{sec: GH-type topologies} for details).

Our argument relies on the theory of resistance forms. 
Here we will prepare the minimum necessary information on resistance forms 
in order to state our main results. 
See Section \ref{sec: resistance forms} for details.
Let $(F, R)$ be a regular resistance metric space and write $(\form, \rdomain)$ for the corresponding regular resistance form.
Given a Radon measure $\mu$ on $F$ of full support,
there exists a related regular Dirichlet form $(\form, \domain)$ on $L^{2}(F, \mu)$ 
and also an associated strong Markov process $(X^{\mu} = (X_{t}^{\mu})_{t \geq 0}, \{P_{x}^{\mu}\}_{x \in F})$,
which we call the \textit{process associated with} $(F, R, \mu)$
(it is known that the process $X^{\mu}$ can be chosen so that it is a Hunt process).
We note that if the resistance metric $R$ is recurrent in the sense of Definition \ref{4. dfn: recurrent resistance metric} below,
then it is automatically regular \cite[Corollary 3.22]{Noda_pre_Scaling}.

We next introduce electrical networks and associated resistance forms.

\begin{dfn} [{Electrical networks}]
  Let $(V, E)$ be a connected, simple, undirected graph with finite or countably many vertices,
  where $V$ denotes the vertex set and $E$ denotes the edge set.
  (NB.\ A graph being simple means that it has no loops and no multiple edges.)
  For $x, y \in V$,
  we write $x \sim y$ if and only if $\{ x, y \} \in E$.
  Let $\{\mu(x, y)\}_{x, y \in V}$ be a family of non-negative real numbers such that 
  $\mu(x,y) = \mu(y, x)$ for all $x, y \in V$,
  $\mu(x, y) > 0$ if and only if $x \sim y$,
  and
  \begin{equation}
    \mu(x) 
    \coloneqq 
    \sum_{y \in V} \mu(x, y) < \infty,
    \quad 
    \forall x \in V.
  \end{equation}
  We call $\mu(x,y)$ the \textit{conductance} on the edge $\{x,y\}$
  and $(V, E, \mu)$ an \textit{electrical network}.
  Note that the edge set $E$ is uniquely determined by conductances.
  We equip $V$ with the discrete topology 
  and define a Radon measure $\muh$ as the \textit{counting measure} on $V$, 
  which is a Radon measure on $V$ given by 
  \begin{equation}
    \muh(A) 
    \coloneqq 
    \# A
    = 
    \sum_{x \in V} \delta_{x}(A)
    \quad 
    A \subseteq V,
  \end{equation}
  where $\delta_{x}$ is the Dirac measure putting mass $1$ at $x$.
  Given an electrical network $G$,
  we write $V_{G}$, $E_{G}$, $\{\mu_{G}(x,y)\}_{x, y \in V_{G}}$, and $\muh_{G}$
  for the vertex set, the edge set, the conductances, and the counting measure, respectively.
  When we say that $G$ is a \textit{rooted electrical network},
  there exists a distinguished vertex,
  which we denote by $\rho_{G} \in V_{G}$.
\end{dfn}

\begin{dfn} [{Resistance forms associated with electrical networks}]  
  \label{1. dfn: resistance forms associated with electrical networks}
  Let $G$ be an electrical network.
  For functions $f,g: V_{G} \to \RN$,
  we set 
  \begin{equation}
    \form_{G}(f, g) 
    \coloneqq 
    \frac{1}{2} 
    \sum_{x, y \in V_{G}} 
    \mu_{G}(x,y) (f(x) - f(y)) (g(x) - g(y))
  \end{equation}
  (if the right-hand side is well-defined).
  We then set $ \rdomain_{G} \coloneqq \{ f \in \RN^{V_{G}} \mid \form(f, f) < \infty \}$
  and, for each $x, y \in F$,
  \begin{equation}
    R_{G}(x,y) 
    \coloneqq 
    \sup 
    \{
      \form(f, f)^{-1} \mid f \in \rdomain,\, f(x)=1,\, f(y) = 0
    \},
  \end{equation}
  where we define $\sup \emptyset \coloneqq 0$.
\end{dfn}

We note that, for any electrical network $G$,
the pair $(\form_{G}, \rdomain_{G})$ is a regular resistance form,
$R_{G}$ is the corresponding resistance metric,
and the topology on $V_{G}$ induced from $R_{G}$ is the discrete topology
(see \cite[Theorem 4.1]{Noda_pre_Scaling}). 
We say that $G$ is a \textit{recurrent electrical network} if $R_{G}$ is a recurrent resistance metric
in the sense of Definition \ref{4. dfn: recurrent resistance metric}.
Given an electrical network $G$ and a measure $\nu$ on $V_{G}$ of full support,
we write $X_{G}^{\nu} = (X_{G}^{\nu}(t))_{t \geq 0}$ for the process associated with $(V_{G}, R_{G}, \nu)$.
By \cite[Theorem 4.1]{Noda_pre_Scaling},
$X_{G}^{\nu}$ is the minimal continuous-time Markov chain on $V_{G}$
with generator 
\begin{equation}
  (\Delta_{G}^{\nu}f)(x) 
  \coloneqq 
  \sum_{y \in V_{G}} \frac{\mu_{G}(x, y)}{\nu(\{x\})} (f(y) - f(x)).
\end{equation} 

Now, we define the (symmetric) Bouchaud trap model on an electrical network.
Throughout this paper,
we fix a constant $\alpha \in (0, 1)$.

\begin{dfn} [{The random variable $\xi$}] \label{1. dfn: random variable xi}
  Let $\xi$ be a positive random variable
  built on a probability space with probability measure $P_{\xi}$ 
  such that there exists a slowly varying function $\ell$ satisfying 
  $P_{\xi}(\xi \geq u) = u^{-\alpha} \ell(u)$.
\end{dfn}

\begin{dfn} [The Bouchaud trap model] \label{1. dfn: BTM}
  Fix an electrical network $G$.
  Let $\{\xi_{x}^{G}\}_{x \in V_{G}}$ be i.i.d.\ random variables with $\xi_{x}^{G} \stackrel{\mathrm{d}}{=} \xi$.
  Set $\nu_{G} \coloneqq \sum_{x \in V_{G}} \xi_{x}^{G} \delta_{x}$,
  which is a random measure on $V_{G}$.
  Conditional on $\nu_{G}$,
  the (symmetric) Bouchaud trap model (BTM) is defined as the continuous-time Markov chain $X_{G}^{\nu_{G}}$. 
  The random measure $\nu_{G}$ is called a \textit{trap}.
\end{dfn}

Our first result concerns BTMs for a convergent sequence of deterministic (scaled) electrical networks.
For each $n \in \mathbb{N}$,
let $G_{n}$ be a rooted recurrent electrical network such that $(V_{G_{n}}, R_{G_{n}})$ is boundedly compact.
We simply write 
\begin{equation}
  V_{n} \coloneqq V_{G_{n}}, \quad 
  \mu_{n} \coloneqq \mu_{G_{n}}, \quad 
  \muh_{n} \coloneqq \muh_{G_{n}}, \quad 
  \rho_{n} \coloneqq \rho_{G_{n}}, \quad 
  R_{n} \coloneqq R_{G_{n}}, \quad 
  \nu_{n} \coloneqq \nu_{G_{n}}, \quad 
  X_{n}^{\nu_{n}} \coloneqq X_{G_{n}}^{\nu_{G_{n}}}
\end{equation}
We write $\mathsf{P}_{n}$ for the underlying probability measure of the random measure $\nu_{n}$.
Let $(a_{n})_{n \geq 1}$ and $(b_{n})_{n \geq 1}$ be two sequences of positive real numbers 
with $a_{n} \wedge b_{n} \to \infty$.
We then define 
\begin{equation}
  c_{n} 
  \coloneqq
  \inf\{u >0 \mid P_{\xi}(\xi > u) < b_{n}^{-1}\}.
\end{equation}

\begin{assum} \label{1. assum: aging, deterministic version} \leavevmode
  \begin{enumerate} [label = (\roman*)]
    \item \label{1. assum item: deterministic, convergence of spaces}
      It holds that 
      \begin{equation}
        (V_{n}, a_{n}^{-1}R_{n}, \rho_{n}, b_{n}^{-1} \muh_{n}) 
        \to 
        (F, R, \rho, \mu)
      \end{equation}
      in the local Gromov-Hausdorff-vague topology 
      for some $(F, R, \rho, \mu) \in \mathbb{G}$,
      where $\mu$ is of full support and non-atomic, that is,
      $\mu(\{x\})  = 0$ for all $x \in F$.
    \item \label{1. assum item: deterministic, the non-explosion condition} 
      It holds that 
      \begin{equation}
        \lim_{r \to \infty} 
        \liminf_{n \to \infty}
        a_{n}^{-1} R_{n}( \rho_{n}, B_{R_{n}}( \rho_{n}, a_{n}r)^{c}) 
        = 
        \infty.
      \end{equation}
  \end{enumerate}
\end{assum}

Under Assumption \ref{1. assum: aging, deterministic version},
the limiting metric space $(F, R)$ is a recurrent resistance metric space
(see \cite[Theorem 5.1]{Noda_pre_Scaling}).
Given a Radon measure $\nu$ on $F$ of full support,
we write $X^{\nu} = (X^{\nu}(t))_{t \geq 0}$ for the process associated with $(F, R, \nu)$.
For the description of the limit of Bouchaud trap models $X_{n}^{\nu_{n}}$,
we define a random measure $\nu$ on $F$ as follows.
Let $\pi$ be a Poisson random measure on $F \times \RNpp$ 
with intensity measure $\mu(dx) \alpha v^{-1-\alpha} dv$
defined on a probability space equipped with probability measure $\mathsf{P}$.
We then define a random measure $\nu$ on $F$ by setting 
\begin{equation}
  \nu(A) \coloneqq \int 1_{A}(x) v\, \pi(dx dv),
  \quad 
  \forall A \in \Borel(F),
\end{equation}
where $\Borel(F)$ denotes the collection of Borel subsets of $F$.
The random measure $\nu$ is a fully-supported Radon measure almost surely (see Lemma \ref{6. lem: aging, properties of limiting trap} below)
and is an analogue of the F.I.N. measure, the speed measure of the F.I.N.\ diffusion.
The random measures $\pi$ and $\nu$ are the limits of traps.
More precisely,
in Section \ref{sec: Proof of main results},
it will be proven that, under Assumption \ref{1. assum: aging, deterministic version},
$(\pi_{n}, c_{n}^{-1} \nu_{n}) \xrightarrow{\mathrm{d}} (\pi, \nu)$,
where $\pi_{n}$ is given by 
\begin{gather}
  \pi_{n} \coloneqq \sum_{x \in V_{n}} \delta_{(x, c_{n}^{-1} \nu_{n}(\{x\}))}.
\end{gather}
Note that the convergence of $\pi_{n}$ to $\pi$ contains information of atoms of traps,
which the convergence of $c_{n}^{-1} \nu_{n}$ to $\nu$ does not guarantee.

Write $\tilde{X}_{n}^{\nu_{n}}(t) \coloneqq X_{n}^{\nu_{n}}(a_{n}c_{n}t)$
and define $\tilde{\law}_{n}^{\nu_{n}}$ to be the law of $\tilde{X}_{n}^{\nu_{n}}$
under $P_{\rho_{n}}^{\nu_{n}}$, i.e.,
\begin{equation}
  \tilde{\law}_{n}^{\nu_{n}} (\cdot)
  \coloneqq 
  P_{\rho_{n}}^{\nu_{n}} \bigl( (\tilde{X}_{n}^{\nu_{n}}(t))_{t \geq 0} \in \cdot \bigr).
\end{equation}
Note that conditional on $\nu_{n}$,
$\tilde{\law}_{n}^{\nu_{n}}$ is a probability measure on $D(\RNp, V_{n})$,
which denotes the space of cadlag functions with values in $V_{n}$ equipped with the usual $J_{1}$-Skorohod topology.
Since $P_{\rho_{n}}^{\nu_{n}}(X_{n}^{\nu_{n}} \in \cdot )$ is measurable with respect to $\nu_{n}$ 
(see \cite[Theorem 6.1]{Noda_pre_Convergence}),
$\tilde{\law}_{n}^{\nu_{n}}$ is a random element of $\mathcal{P}(D(\RNp, V_{n}))$, 
which is defined to be the space of the probability measures on $D(\RNp, V_{n})$ equipped with the weak topology.
Similarly, we define 
\begin{equation}
  \law^{\nu} (\cdot)
  \coloneqq 
  P_{\rho}^{\nu} ( X^{\nu} \in \cdot),
\end{equation}
which is a random element of $\mathcal{P}(D(\RNp, F))$.

For aging,
we consider the following two-point functions:
\begin{equation}
  \tilde{\Phi}_{n}^{\nu_{n}}(s, t)
  \coloneqq 
  P_{\rho_{n}}^{\nu_{n}} \bigl( \tilde{X}_{n}^{\nu_{n}}(s) = \tilde{X}_{n}^{\nu_{n}} ( t ) \bigr),
  \quad 
  \Phi^{\nu}(s, t)
  \coloneqq 
  P_{\rho}^{\nu} \bigl( X^{\nu}(s) = X^{\nu}(t) \bigr)
\end{equation}
for $s, t > 0$. 
We note that $\tilde{\Phi}_{n}^{\nu_{n}}$ and $\Phi^{\nu}$ are continuous (see Section \ref{sec: (sub-)aging for deterministic traps}).
In the first result below,
we obtain that the BTMs $\tilde{X}_{n}^{\nu_{n}}$ and the associated aging functions $\tilde{\Phi}_{n}^{\nu_{n}}$ converge 
to $X^{\nu}$ and $\Phi^{\nu}$.

\begin{thm} \label{1. thm: aging for deterministic models}
  Under Assumption \ref{1. assum: aging, deterministic version}, 
  it holds that
  \begin{equation} \label{1. thm eq: convergence for aging result}
    \Bigl( V_{n}, a_{n}^{-1}R_{n}, \rho_{n}, b_{n}^{-1} \muh_{n}, 
      \mathsf{P}_{n}
      \bigl(
        (c_{n}^{-1} \nu_{n}, \tilde{\law}_{n}^{\nu_{n}}, \tilde{\Phi}_{n}^{\nu_{n}}) \in \cdot
      \bigr) 
    \Bigr)
    \to 
    \Bigl( F, R, \rho, \mu, 
      \mathsf{P} 
      \bigl( 
        (\nu, \law^{\nu}, \Phi^{\nu}) \in \cdot 
      \bigr) 
    \Bigr)
  \end{equation}
  in the space $\rbcM(\MeasFunct \times \probFunct(\disMeasFunct \times \probFunct(\SkorohodFunct) \times \fixedFunct{C(\RNpp^{2}, \RNp)}))$ 
  (defined in Section \ref{sec: GH-type topologies} below).
  In particular, $\tilde{\Phi}_{n}^{\nu_{n}} \xrightarrow{\mathrm{d}} \Phi^{\nu}$ in $C(\RNpp^{2}, \RNp)$
  with respect to the compact-convergence topology,
  where the limit is positive with probability $1$. 
\end{thm}

\begin{rem} \label{1. rem: convergence of annealed aging functions}
  Fix $s, t > 0$.
  Since $\tilde{\Phi}_{n}^{\nu_{n}}(s,t) \leq 1$,
  the family $(\tilde{\Phi}_{n}^{\nu_{n}}(s,t))_{n \geq 1}$ of random variables is uniformly integrable.
  Combining this with Theorem \ref{1. thm: aging for deterministic models},
  we obtain that 
  \begin{equation}
    \lim_{n \to \infty}
    \mathsf{E}_{n}
    \bigl[ \tilde{\Phi}_{n}^{\nu_{n}}(s,t) \bigr]
    = 
    \mathsf{E} 
    \bigl[ \Phi^{\nu}(s,t) \bigr],
  \end{equation}
  where $\mathsf{E}_{n}$ and $\mathsf{E}$ denote the expectations with respect to $\mathsf{P}_{n}$ and $\mathsf{P}$, respectively.
  This recovers the aging result of \cite{Fontes_Isopi_Newman_02_Random}.
\end{rem}

For sub-aging,
we additionally assume convergence of local structures of electrical networks.
Roughly speaking,
we assume that a uniformly chosen vertex and the total conductance at the vertex converge jointly;
this assumption is precisely stated below.
Given an electrical network $G$,
we define a map $\psi_{G}: V_{G} \to V_{G} \times \RNp$ by setting 
\begin{equation} \label{1. eq: mark map}
  \psi_{G}(x) \coloneqq (x, \mu_{G}(x)).
\end{equation}
Write $\dotmuh_{G}$ for the pushforward measure of $\muh_{G}$ by $\psi_{G}$,
which is a Radon measure on $V_{G} \times \RNp$.
We emphasize that $\dotmuh_{G}$ is a natural object in terms of graph theory.
Indeed,
if $G$ is a simple electrical network, that is, $\mu_{G}(x,y) = 1$ if $\mu_{G}(x,y) > 0$,
then $\dotmuh_{G}(A \times \{k\})$ is the number of vertices in $A \subseteq V_{G}$ whose degree is $k$.
We simply write $\dotmuh_{n} \coloneqq \dotmuh_{G_{n}}$.

\begin{assum} \label{1. assum: sub-aging, deterministic version}
  Assumption \ref{1. assum: aging, deterministic version}\ref{1. assum item: deterministic, the non-explosion condition} is satisfied,
  and 
  \begin{equation} \label{1. assum item: sub-aging, convergence of spaces for random spaces}
    (V_{n}, a_{n}^{-1}R_{n}, \rho_{n}, b_{n}^{-1} \dotmuh_{n} ) 
    \to 
    (F, R, \rho, \dot{\mu})
  \end{equation}
  for some  $(F, R, \rho, \dot{\mu}) \in \rbcM(\markedMeasFunct{\RNp})$ in the space $\rbcM(\markedMeasFunct{\RNp})$ 
  (defined in Section \ref{sec: GH-type topologies} below).
  Moreover,
  the measure $\mu$ on $F$ defined by 
  \begin{equation} \label{1. eq: dfn of mu for sub-aging}
    \mu(A) 
    \coloneqq 
    \dot{\mu}(A \times \RNp),
    \quad 
    \forall A \in \Borel(F),
  \end{equation}
  is of full support and non-atomic.
\end{assum}

Under Assumption \ref{1. assum: sub-aging, deterministic version},
we let $\dot{\pi}(dx dw dv)$ be the Poisson random measure on $F \times \RNp \times \RNpp$ 
with intensity measure $\dot{\mu}(dx dw) \alpha v^{-1-\alpha} dv$ defined on a probability space with probability measure $\mathsf{P}$.
Define a random measure $\nu$ on $F$ by setting 
\begin{equation}
  \nu(A)
  \coloneqq 
  \int 1_{A}(x) v\, \dot{\pi}(dx dw dv),
  \quad 
  \forall A \in \Borel(F).
\end{equation}
Then, $\nu$ is a fully-support Radon measure almost surely (see Lemma \ref{6. lem: sub-aging, properties of limiting trap} below).
In Section \ref{sec: Proof of main results},
it will be proven that, under Assumption \ref{1. assum: sub-aging, deterministic version},
$(\dot{\pi}_{n}, c_{n}^{-1} \nu_{n}) \to (\dot{\pi}, \nu)$,
where $\dot{\pi}_{n}$ is given by 
\begin{equation}
  \dot{\pi}_{n} 
  \coloneqq 
  \sum_{x \in V_{n}} \delta_{(x, \mu_{n}(x), c_{n}^{-1} \nu_{n}(\{x\}))}.
\end{equation}
Note that $\dot{\pi}_{n}$ contains information of local structure:
the total conductance at every vertex.

For sub-aging, we consider the following two-point function:
\begin{equation}
  \tilde{\Psi}_{n}^{\nu_{n}}(s, t)
  \coloneqq
  P_{\rho_{n}}^{\nu_{n}}
  \bigl(
    X_{n}^{\nu_{n}}(a_{n}c_{n}t) = X_{n}^{\nu_{n}}(a_{n}c_{n}t + t'),\ \forall t' \in [0, c_{n}s]
  \bigr)
\end{equation}
for $s \geq 0, t >0$.
By the Markov property and the fact that the waiting time of $X_{n}^{\nu_{n}}$ at $x$
has the exponential distribution with mean $\nu_{n}(\{x\})/ \mu_{n}(x)$,
we deduce that 
\begin{align}
  \tilde{\Psi}_{n}^{\nu_{n}}(s, t)
  &=
  E_{\rho_{n}}^{\nu_{n}}
  \left[
    \exp(- \mu_{n}(\tilde{X}_{n}^{\nu_{n}}(t)) s/ c_{n}^{-1} \nu_{n}(\tilde{X}_{n}^{\nu_{n}}(t)))
  \right] \\
  &=
  \int e^{-ws/v} P_{\rho_{n}}^{\nu_{n}}(\tilde{X}_{n}^{\nu_{n}}(t) = x)\, \dot{\pi}_{n}(dx dw dv).
\end{align}
Following this expression,
we define 
\begin{equation}
  \Psi^{\nu}(s, t)
  \coloneqq 
  \int
    e^{- ws/ v} P_{\rho}^{\nu}(X^{\nu}(t) = x)\,
  \dot{\pi}(dx dw dv)
\end{equation}
for $s \geq 0, t > 0$.
We note that the functions $\tilde{\Psi}_{n}^{\nu_{n}}$ and $\Psi^{\nu}$ are continuous
on $\RNp \times \RNpp$ (see Section \ref{sec: (sub-)aging for deterministic traps}).
Under Assumption \ref{1. assum: sub-aging, deterministic version},
we obtain not only the same aging result as Theorem \ref{1. thm: aging for deterministic models},
but also a sub-aging result,
that is,
$\tilde{\Psi}_{n}^{\nu_{n}}$ converges to $\Psi^{\nu}$. 

\begin{thm} \label{1. thm: sub-aging for deterministic models}
  Under Assumption \ref{1. assum: sub-aging, deterministic version},
  it holds that
  \begin{align} \label{1. thm eq: convergence for sub-aging result}
    &
    \Bigl( V_{n}, a_{n}^{-1}R_{n}, \rho_{n}, b_{n}^{-1} \dotmuh_{n}, 
      \mathsf{P}_{n}
      \bigl(
        (c_{n}^{-1} \nu_{n}, \tilde{\law}_{n}^{\nu_{n}}, \tilde{\Phi}_{n}^{\nu_{n}}, \tilde{\Psi}_{n}^{\nu_{n}}) \in \cdot
      \bigr) 
    \Bigr)\\
    \to &
    \Bigl( F, R, \rho, \dot{\mu}, 
      \mathsf{P} 
      \bigl( 
        (\nu, \law^{\nu}, \Phi^{\nu}, \Psi^{\nu}) \in \cdot 
      \bigr) 
    \Bigr)
  \end{align}
  in the space $\rbcM(\markedMeasFunct{\RNp} \times \probFunct(\disMeasFunct \times \probFunct(\SkorohodFunct) \times \fixedFunct{C(\RNpp^{2}, \RNp)} \times \tau^{C(\RNp \times \RNpp, \RNp)}))$.
  In particular, $\tilde{\Psi}_{n}^{\nu_{n}} \xrightarrow{\mathrm{d}} \Psi^{\nu}$ in $C(\RNp \times \RNpp, \RNp)$,
  where the limit is positive with probability $1$.
\end{thm}

\begin{rem} \label{1. rem: convergence of annealed sub-aging functions}
  Similarly to Remark \ref{1. rem: convergence of annealed aging functions},
  from Theorem \ref{1. thm: sub-aging for deterministic models},
  we obtain the convergence of the annealed sub-aging functions, i.e., for any $s \geq 0$ and $t > 0$,
  \begin{equation}
    \lim_{n \to \infty} 
    \mathsf{E}_{n} 
    \bigl[ \tilde{\Psi}_{n}^{\nu_{n}}(s,t) \bigr]
    = 
    \mathsf{E} 
    \bigl[ \Psi^{\nu}(s,t) \bigr].
  \end{equation}
\end{rem}

\begin{rem}
  By definition,
  the waiting time of the BTM $X_{G}^{\nu_{G}}$ at a vertex $x$ is determined by the trap $\nu_{G}(\{x\})$ and a $1$-neighbor local structure of $G$,
  that is, the total conductance $\mu_{G}(x) = \sum_{y \sim x} \mu_{G}(x,y)$.
  This is the reason why we assume the convergence of $\dotmuh_{n}$ in Assumption \ref{1. assum: sub-aging, deterministic version}.
  Even when one considers other trap models,
  it seems possible to derive similar results by assuming convergence of corresponding local structures.
  (For a generalized trap model, see \cite{BenArous_Cabezas_Cerny_Royfman_15_Randomly}.)
\end{rem}

Next, we consider random electrical networks.
Namely,
we assume that $(V_{n}, R_{n}, \rho_{n})$ is a random element of $\GHVspace$
and we denote its underlying probability measure by $\mathbf{P}_{n}$.
For aging, we consider a random version of Assumption \ref{1. assum: aging, deterministic version}, given below.

\begin{assum} \label{1. assum: aging, random version} \leavevmode
  \begin{enumerate} [label = (\roman*)]
    \item \label{1. assum item: random, convergence of spaces}
      It holds that 
      \begin{equation}
        (V_{n}, a_{n}^{-1}R_{n}, \rho_{n}, b_{n}^{-1} \muh_{n}) 
        \xrightarrow{\mathrm{d}}
        (F, R, \rho, \mu)
      \end{equation}
      in the local Gromov-Hausdorff-vague topology 
      for some random element $(F, R, \rho, \mu)$ of $\mathbb{G}$,
      where $\mu$ is of full support and non-atomic with probability $1$.
    \item \label{1. assum item: random, the non-explosion condition} 
      It holds that 
      \begin{equation}
        \lim_{r \to \infty} 
        \liminf_{n \to \infty}
        \mathbf{P}_{n}
        \bigl(
          a_{n}^{-1} R_{n}( \rho_{n}, B_{R_{n}}( \rho_{n}, a_{n}r)^{c}) > \lambda
        \bigr)
        = 1,
        \quad 
        \forall \lambda > 0.
      \end{equation}
  \end{enumerate}
\end{assum}

\begin{thm} \label{1. thm: aging for random models}
  Under Assumption \ref{1. assum: aging, random version}, 
  the convergence \eqref{1. thm eq: convergence for aging result} holds in distribution.
\end{thm}

\begin{rem} \label{1. rem: convergence of annealed aging functions for random electrical networks}
  Similarly to Remark \ref{1. rem: convergence of annealed aging functions},
  under Assumption \ref{1. assum: aging, random version},
  we deduce from Theorem \ref{1. thm: aging for random models} that,
  for any $s, t> 0$,
  $\mathsf{E}_{n} \bigl[\tilde{\Phi}_{n}^{\nu_{n}}(s,t)\bigr] \xrightarrow{\mathrm{d}} \mathsf{E} \bigl[\Phi^{\nu}(s,t)\bigr]$
  and 
  \begin{equation}
    \lim_{n \to \infty} 
    \mathbf{E}_{n} 
    \bigl[
      \mathsf{E}_{n} \bigl[\tilde{\Phi}_{n}^{\nu_{n}}(s,t)\bigr]
    \bigr]
    = 
    \mathbf{E}
    \bigl[
      \mathsf{E} \bigl[\Phi^{\nu}(s,t)\bigr]
    \bigr].
  \end{equation}
  where $\mathbf{E}_{n}$ and $\mathbf{E}$ denote the expectations with respect to $\mathbf{P}_{n}$ and $\mathbf{P}$, respectively.
\end{rem}

For sub-aging, we consider a random version of Assumption \ref{1. assum: sub-aging, deterministic version}, given below.

\begin{assum} \label{1. assum: sub-aging, random version}
  Assumption \ref{1. assum: aging, random version}\ref{1. assum item: random, the non-explosion condition} is satisfied,
  and there exists a random element $(F, R, \dot{\mu})$ of $\rbcM(\markedMeasFunct{\RNp})$ 
  such that 
  \begin{equation} \label{1. assum eq: sub-aging, convergence for random version}
    (V_{n}, a_{n}^{-1}R_{n}, \rho_{n}, b_{n}^{-1} \dotmuh_{n} ) 
    \xrightarrow{\mathrm{d}} 
    (F, R, \rho, \dot{\mu})
  \end{equation}
  in the space $\rbcM(\markedMeasFunct{\RNp})$.
  Moreover,
  a measure $\mu$ on $F$ defined by 
  \begin{equation} \label{1. eq: dfn of mu for sub-aging, random version}
    \mu(A) 
    \coloneqq 
    \dot{\mu}(A \times \RNp),
    \quad 
    \forall A \in \Borel(F)
  \end{equation}
  is of full support and non-atomic with probability $1$.
\end{assum}

\begin{thm} \label{1. thm: sub-aging for random models}
  Under Assumption \ref{1. assum: sub-aging, random version},
  the convergence \eqref{1. thm: sub-aging for deterministic models} holds in distribution.
\end{thm}

\begin{rem} \label{1. rem: convergence of annealed sub-aging functions for random electrical networks}
  Similarly to Remark \ref{1. rem: convergence of annealed aging functions},
  under Assumption \ref{1. assum: sub-aging, random version},
  we deduce from Theorem \ref{1. thm: sub-aging for random models} that,
  for any $s \geq 0$ and $t > 0$,
  $\mathsf{E}_{n} \bigl[\tilde{\Psi}_{n}^{\nu_{n}}(s,t)\bigr] \xrightarrow{\mathrm{d}} \mathsf{E} \bigl[\Psi^{\nu}(s,t)\bigr]$ 
  and 
  \begin{equation}
    \lim_{n \to \infty} 
    \mathbf{E}_{n} 
    \bigl[
      \mathsf{E}_{n} \bigl[\tilde{\Psi}_{n}^{\nu_{n}}(s,t)\bigr]
    \bigr]
    = 
    \mathbf{E}
    \bigl[
      \mathsf{E} \bigl[\Psi^{\nu}(s,t)\bigr]
    \bigr].
  \end{equation}
\end{rem}

\begin{rem}
  More generally than our BTMs,
  one can consider non-symmetric BTMs (cf.\ \cite[Definition 2.1]{Arous_Cerny_06_Dynamics}).
  Fix $a \in (0, 1]$ and an electrical network $G$.
  Conditional on $\nu_{G}$,
  the non-symmetric BTM with parameter $a$ on $G$ refers to the continuous-time Markov chain $X^{\nu_{G}}$
  with generator 
  \begin{equation}
    (\Delta^{\nu_{G}}f)(x) 
    \coloneqq 
    \sum_{y \in V_{G}} \frac{\mu_{G}(x, y) (\nu_{G}(\{x\})\nu_{G}(\{y\}))^{a}}{\nu_{G}(\{x\})} (f(y) - f(x)).
  \end{equation} 
  If we define an electrical network $G'$ by setting $V_{G'} \coloneqq V_{G}$ 
  and $\mu_{G'}(x,y) \coloneqq \mu_{G}(x, y) (\nu_{G}(\{x\})\nu_{G}(\{y\}))^{a}$,
  then $X^{\nu_{G}}$ is the process associated with $(V_{G'}, R_{G'}, \nu_{G})$.
  Thus,
  when one considers a non-symmetric BTM,
  the associated resistance metric becomes random and does not coincide with the resistance metric on $G$.
  Our arguments work even for these non-symmetric BTMs,
  and similar results hold once the corresponding assumptions with respect to associated resistance metrics are verified.
  However, it is not easy to check in general,
  and so we only consider symmetric BTMs in this article.
\end{rem}

In the proof of our main results,
it is crucial to find a coupling of traps $\nu_{n}$ and $\nu$ 
so that $c_{n}^{-1} \nu_{n} \to \nu$ almost surely 
vaguely and in the point process sense,
where we recall that the convergence in the point process sense is notion of convergence of discrete measures 
introduced in \cite[Definition 2.2]{Fontes_Isopi_Newman_02_Random}
and means convergence of atoms.
When discrete measures converge both vaguely and in the point process sense,
we say that they converge in the vague-and-point-process topology. 
In that paper, 
where they prove the convergence of the scaled BTM on $\ZN$ to the F.I.N.\ diffusion,
they constructed such a coupling by using L\'{e}vy processes and the usual $J_{1}$-Skorohod topology.
However,
their argument cannot be applied to general graphs.
In this paper,
we construct the coupling by the Skorohod representation theorem (cf.\ \cite[Theorem 5.31]{Kallenberg_21_Foundations}).
Specifically,
we first show that it is possible to define a complete, separable metric on the set of discrete measures 
that induces the vague-and-point-process topology,
which can be seen as an extension of the usual $J_{1}$-Skorohod topology.
Then, it is an immediate consequence of the Skorohod representation theorem 
that there exists a desired coupling once the convergence $c_{n}^{-1} \nu_{n} \to \nu$ in distribution is verified.

The remainder of the article is organized as follows.
In Section \ref{sec: The vague-and-point-process topology},
we prove that the vague-and-point-process topology introduced above is a Polish topology.
In Section \ref{sec: GH-type topologies},
we introduce the Gromov-Hausdorff-type topologies which are used to discuss convergence of objects on different metric spaces.
In Section \ref{sec: resistance forms},
we recall some fundamental results 
about the theory of resistance forms and study transition densities of processes on measured resistance metric spaces.
In particular,
it is proven that if a family of measured resistance metric spaces is precompact in the local Gromov-Hausdorff-vague topology,
then the family of the transition densities of associated processes is precompact.
This result is used to prove the precompactness of two-point functions.
In Section \ref{sec: (sub-)aging for deterministic traps},
we prove that if deterministic traps converge in the vague-and-point-process topology,
then the (sub-)aging functions converge.
Combining this result with the above-mentioned coupling of traps,
we establish the main results in Section \ref{sec: Proof of main results}.
Finally, in Section \ref{sec: applications},
we present some examples to which our main results are applicable.


\section{The vague-and-point-process topology} \label{sec: The vague-and-point-process topology}

Convergence of discrete measures in the vague-and-point-process topology 
means the convergence both in the vague topology and in the point process sense \cite[Definition 2.2]{Fontes_Isopi_Newman_02_Random}.
In Section \ref{sec: The vague metric},
we recall fundamental results on the vague topology,
and then we introduce and study the vague-and-point-process topology in Sections \ref{sec: the space of discrete measures} and \ref{sec: a space including the space of discrete measures}.
We note that, given a topological space $S$, $\Borel(S)$ denotes the totality of Borel subsets in $S$,
$\id_{S}$ denotes the identity map from $S$ to itself,
and, for a subset $A$ of $S$, $\partial A = \partial_{S} A$ denotes the boundary of $A$ in $S$.
For a function $f: S \to \RN$, we write $\| f\|_{\infty} \coloneqq \sup\{|f(x)| \mid x \in S\}$.


\subsection{The vague metric} \label{sec: The vague metric}

In this subsection,
we introduce the vague metric, which induces the vague topology on the set of measures.
Let $(S, d^{S}, \rho_{S})$ be a rooted boundedly-compact metric space.
Write $\finMeas(S)$
for the set of finite Borel measures on $S$,
which we equip with the weak topology.
Recall that the weak topology is induced from the \textit{Prohorov metric} $d_{P}^{S}$ 
given by 
\begin{equation} \label{2. eq: dfn of Prohorov metric}
  d_{P}^{S}(\mu, \nu) 
  \coloneqq
  \inf\{
    \varepsilon > 0 \mid \mu(A) \leq \nu(\neighb{A}{\varepsilon}) + \varepsilon,\, 
    \nu(A) \leq \mu(\neighb{A}{\varepsilon}) + \varepsilon,\, \forall A \subseteq \Borel(S)
  \},
\end{equation}
where we set 
\begin{equation} \label{2. eq: dfn of closed neighborhood}
  \neighb{A}{\varepsilon} 
  \coloneqq
  \{
    x \in S \mid \exists y \in A\ \text{such that}\ d^{S}(x,y) \leq \varepsilon
  \}.
\end{equation} 

\begin{dfn} [{The vague metric $d_{V}^{S, \rho_{S}}$}]
  We denote the set of Radon measures on $S$ by $\mathcal{M}(S)$.
  For $\mu \in \mathcal{M}(S)$,
  we write $\mu^{(r)}$ for the restriction of $\mu$ to $S^{(r)} \coloneqq \closure (B_{S}(\rho, r))$, 
  that is,
  $\mu^{(r)}$ is a finite Borel measure given by
  \begin{equation}  \label{2. dfn eq: restriction of a measure to a closed ball}
    \mu^{(r)}(\cdot) 
    \coloneqq
    \mu|_{S^{(r)}}(\cdot)
    =
    \mu \left( \cdot \cap S^{(r)} \right).
  \end{equation} 
  We then define, for each $\mu, \nu \in \mathcal{M}(S)$,
  \begin{equation}
    d_{V}^{S, \rho_{S}}(\mu, \nu)
    \coloneqq
    \int_{0}^{\infty} e^{-r} \left( 1 \wedge d_{P}^{S}(\mu^{(r)}, \nu^{(r)}) \right) dr.
  \end{equation}
\end{dfn}

\begin{thm} [{\cite[Theorems 2.37 and 2.39]{Noda_pre_Metrization}}] \label{2. thm: convergence in the vague topology}
  The function $d_{V}^{S, \rho_{S}}$ is a complete, separable metric on $\mathcal{M}(S)$.
  Let $\mu, \mu_{1}, \mu_{2}, \ldots$ be 
  Radon measures on $S$.
  Then these conditions are equivalent:
  \begin{enumerate}
    \item $\mu_{n}$ converges to $\mu$ with respect to $d_{V}^{S, \rho_{S}}$;
    \item $\mu_{n}^{(r)}$ converges weakly to $\mu^{(r)}$ for all but countably many $r>0$;
    \item $\mu_{n}$ converges vaguely to $\mu$, that is, for all continuous functions $f : S \to \RN$ with compact support,
      it holds that 
      \begin{equation}
        \lim_{n \to \infty} \int_{S} f(x)\, \mu_{n}(dx) = \int_{S} f(x)\, \mu(dx).
      \end{equation}
  \end{enumerate}
\end{thm}

We call $d_{V}^{S, \rho_{S}}$ the \textit{vague metric (with the root $\rho_{S}$)}. 
The vague metric is preserved by root-and-distance-preserving maps. 
This is important for metrization of Gromov-Hausdorff-type topologies in Section \ref{sec: GH-type topologies}.
Here, we note that, for two rooted boundedly-compact metric spaces $(S_{i}, d^{S_{i}}, \rho_{S_{i}})$,
a map $f: S_{1} \to S_{2}$ is said to be \textit{root-preserving} if $f(\rho_{S_{1}}) = \rho_{S_{2}}$,
and \textit{distance-preserving} if $d^{S_{2}}(f(x), f(y)) = d^{S_{1}}(x,y)$ for all $x,y \in S_{1}$. 

\begin{prop} [{\cite[Propositions 2.32 and 2.41]{Noda_pre_Metrization}}] \label{2. prop: hausdorff and vague metrics are preserved}
  Let $(S_{i}, d^{S_{i}}, \rho_{i}),\, i=1,2$ be rooted boundedly-compact metric spaces 
  and $f: S_{1} \to S_{2}$ be a root-and-distance-preserving map.
  Then, the map from $(\Meas(S_{1}),\allowbreak d_{V}^{S_{1}, \rho_{S_{1}}})$ to $(\Meas(S_{2}),\allowbreak d_{V}^{S_{2}, \rho_{S_{2}}})$ 
  given by $\mu \mapsto \mu \circ f^{-1}$ is distance-preserving.
\end{prop}


\subsection{The space \texorpdfstring{$\disMeasures(S)$}{of discrete measures}} \label{sec: the space of discrete measures}

In this subsection,
we define the vague-and-point-process topology,
which yields jointly vague convergence and convergence in the point process sense 
introduced in \cite[Definition 2.2]{Fontes_Isopi_Newman_02_Random}.
In particular, 
various characterizations of this topology in terms of convergence are given in Theorem \ref{2. thm: convergence in Mdis}.

Let $(S, d^{S}, \rho_{S})$ be a rooted boundedly-compact metric space.
Recall that a Radon measure $\nu$ is called a \textit{discrete measure} 
if it is written in the following form:
\begin{equation}  \label{2. eq: dfn of discrete measure}
  \nu 
  =
  \sum_{i \in I} w_{i} \delta_{x_{i}},
\end{equation}
where $I$ is a countable set,
$w_{i}$ is a positive number,
$x_{i}$ is an element of $S$ such that $x_{i} \neq x_{j}$ if $i \neq j$,
and $\delta_{x_{i}}$ denotes the Dirac measure putting mass $1$ at $x_{i}$.
The representation of \eqref{2. eq: dfn of discrete measure} is called an \textit{atomic decomposition} of $\nu$
and is unique up to the order of terms 
(cf.\ \cite[Lemma 1.6]{Kallenberg_17_Random}).
We say that $\nu$ is \textit{simple} if $w_{i} = 1$ for all $i \in I$.
We write $\At(\nu) = \{x_{i}\}_{i \in I}$ for the set of the atoms of $\nu$.

\begin{prop} \label{2. prop: atom convergence from simple measure convergence}
  Let $\nu, \nu_{1}, \nu_{2}, \cdots$ be simple measures on $S$.
  For each compact subset $K$ of $S$,
  consider the following condition.
  \begin{enumerate} [label = (VC), leftmargin = *]
    \item \label{2. prop: characterization of vague convergence of simple measures}
      Write $\{x_{i}\}_{i \in I_{K}} = \At(\nu) \cap K$ for the atoms lying in $K$.
      (NB. $I_{K}$ is a finite set.)
      Then, for all sufficiently large $n$,
      we can write $\At(\nu_{n}) \cap K = \{x_{i}^{(n)} \mid i \in I_{K}\}$ 
      in such a way that $x_{i}^{(n)} \to x_{i}$ for each $i \in I_{K}$.
  \end{enumerate}
  Then, the statements below are equivalent with each other.
  \begin{enumerate} [label = (\roman*)]
    \item \label{2. prop item: vague convergence of simple measures}
      The measures $\nu_{n}$ converge to $\nu$ vaguely.
    \item \label{2. prop item: atoms pairing for any compact subsets}
      Any compact subset $K$ of $S$ with $\nu(\partial K) = 0$ 
      satisfies \ref{2. prop: characterization of vague convergence of simple measures}.
    \item \label{2. prop item: atoms pairing for increasing compact subsets}
      There exists an increasing sequence $(D_{k})_{k \geq 1}$ of relatively compact open subsets of $S$ 
      such that $\bigcup_{k \geq 1} D_{k} = S$, $\mu(\partial D_{K}) = 0$ for each $k$, 
      and each $\closure(D_{k})$ satisfies \ref{2. prop: characterization of vague convergence of simple measures}.
  \end{enumerate}
\end{prop}

\begin{proof}
  Assume that \ref{2. prop item: vague convergence of simple measures} is satisfied.
  Fix a compact subset of $F$ with $\nu(\partial K) = 0$.
  Since $\nu_{n}|_{K} \to \nu|_{K}$ weakly,
  we have that $\nu_{n}(K) \to \nu(K) = \# I_{K}$.
  Hence,
  for all sufficiently large $n$,
  $\#(\At(\nu_{n}) \cap K) = \# I_{K}$.
  Fix a small $\varepsilon \in (0,1)$ such that 
  $D_{S}(x_{i}, 2\varepsilon) \subseteq K$ for each $i \in I_{K}$
  and 
  $D_{S}(x_{i}, 2\varepsilon) \cap D_{S}(x_{j}, 2\varepsilon) = \emptyset$ if $i \neq j$. 
  Since we have that $d_{P}^{S}(\nu_{n}|_{K}, \nu|_{K}) < \varepsilon$ for all sufficiently large $n$,
  we deduce from the definition of the Prohorov metric that 
  \begin{equation}
    1-\varepsilon
    \leq 
    \nu_{n}(\{x_{i}\}^{\langle \varepsilon \rangle})
    \leq 
    1+\varepsilon,
  \end{equation}
  which implies that there is exactly one atom of $\nu_{n}$ in $D_{S}(x_{i}, \varepsilon)$,
  which is denoted by $x_{i}^{(n)}$.
  We then have that $\At(\nu_{n}) \cap K = \{x_{i}^{(n)} \mid i \in I_{K}\}$ for all sufficiently large $n$.
  Fix $i \in I_{K}$.
  Given $\eta \in (0, \varepsilon)$,
  by the same argument,
  one can check that 
  for all sufficiently large $n$
  there is exactly one atom of $\nu_{n}$ in $D_{S}(x_{i}, \eta)$.
  Since $D_{S}(x_{i}, \eta) \subseteq D_{S}(x_{i}, \varepsilon)$,
  the atom must be $x_{i}^{(n)}$.
  Therefore,
  we obtain that $x_{i}^{(n)} \to x_{i}$.

  The implication \ref{2. prop item: atoms pairing for any compact subsets} $\Rightarrow$ \ref{2. prop item: atoms pairing for increasing compact subsets} is straightforward.
  Assume that \ref{2. prop item: atoms pairing for increasing compact subsets} is satisfied.
  Fix a compactly supported continuous function $f$ on $S$.
  Let $D_{k}$ be such that $D_{k}$ contains the support of $f$.
  Write $\At(\nu) \cap D_{k} = \{x_{i}\}_{i \in I_{D_{k}}}$ and $\At(\nu_{n}) = \{x_{i}^{(n)}\}_{i \in I_{D_{k}}}$
  in such a way that $x_{i}^{(n)} \to x_{i}$ for each $i \in I_{D_{k}}$.
  Then, since $I_{K}$ is finite, we deduce that 
  \begin{equation}
    \lim_{n \to \infty} \int f(x)\, \nu_{n}(dx)
    = 
    \lim_{n \to \infty} \sum_{i \in I_{D_{k}}} f(x_{i}^{(n)})
    = 
    \sum_{i \in I_{D_{k}}} f(x_{i})
    = 
    \int f(x)\, \nu(dx),
  \end{equation}
  which establishes \ref{2. prop item: vague convergence of simple measures}.
\end{proof}

\begin{dfn} [{The space $\disMeasures(S)$}]
  We define $\disMeasures(S)$ to be the collection of discrete Radon measures $\mu$ on $S$.
\end{dfn}

Given $\nu \in \disMeasures(S)$ written in the form \eqref{2. eq: dfn of discrete measure},
we define 
\begin{equation} \label{2. eq: dfn of point operator}
  \pointMap(\nu) 
  \coloneqq 
  \sum_{i \in I} \delta_{(x_{i}, w_{i})}.
\end{equation}
It is easy to check that $\pointMap(\nu)$ is a discrete Radon measure on $S \times \RNpp$
by observing that there are only finitely many atoms of $\pointMap(\nu)$ 
in any compact subset of $S \times \RNpp$.

Before proceeding with the discussion of discrete measures, 
we make some technical remarks regarding the space $\RNpp$.
The space $\RNpp$ equipped with the usual Euclidean metric $d^{\RN}(v, w) = |v-w|$
is not a boundedly-compact metric space.
Indeed, $(0, 1]$ is bounded and closed in $\RNpp$ but not compact.
Throughout this paper,
we equip $\RNpp$ with another metric $d^{\RNpp}$ given by 
\begin{equation}
  d^{\RNpp}(v,w) 
  \coloneqq 
  |\log v - \log w|,
  \quad 
  v, w \in \RNpp.
\end{equation} 
It is elementary to check 
that the topology on $\RNpp$ induced from $d^{\RNpp}$ coincides with the Euclidean topology 
and $(\RNpp, d^{\RNpp})$ is boundedly compact.
We set $1 \in \RNpp$ to be the root of $\RNpp$.
Note that the closed ball with radius $r$ centered at $1$ is $[e^{-r}, e^{r}]$.
We then think the product space $S \times \RNpp$ as a rooted boundedly-compact metric space 
by equipping it with the root $\tilde{\rho}_{S} \coloneqq (\rho_{S}, 1)$ 
and the max product metric $d^{S \times \RNpp}$ given by
\begin{equation} \label{2. eq: max product metric}
  d^{S \times \RNpp}( (x, v), (y, w) )
  \coloneqq 
  d^{S}(x, y) \vee d^{\RNpp}(v, w),
  \quad 
  (x,v), (y, w) \in S \times \RNpp.
\end{equation}

For $\nu_{1}, \nu_{2} \in \disMeasures(S)$,
we define 
\begin{equation}
  d_{\disMeasures}^{S, \rho_{S}}(\nu_{1}, \nu_{2})
  \coloneqq 
  d_{V}^{S, \rho_{S}}(\nu_{1}, \nu_{2})
  \vee 
  d_{V}^{S \times \RNpp, \tilde{\rho}_{S}}(\pointMap(\nu_{1}), \pointMap(\nu_{2})).
\end{equation}

\begin{prop} \label{2. prop: checking metric on Mdis}
  The function $d_{\disMeasures}^{S, \rho_{S}}$ is a metric on $\disMeasures(S)$.
\end{prop}

\begin{proof}
  By Theorem \ref{2. thm: convergence in the vague topology},
  $d_{V}^{S, \rho_{S}}$ is a metric on $\disMeasures(S)$,
  which implies that $d_{\disMeasures}^{S, \rho_{S}}$ is positive definite.
  Symmetry and the triangle inequality are obvious.
\end{proof}

\begin{dfn} [{The vague-and-point-process topology}]
  We call the topology on $\disMeasures(S)$ induced by $d_{\disMeasures}^{S, \rho_{S}}$ 
  the \textit{vague-and-point-process topology}.
\end{dfn}

Below, we study the vague-and-point-process topology in terms of convergence.

\begin{prop} \label{2. prop: equivalence of convergence of weights}
  Let $\nu, \nu_{1}, \nu_{2}, \ldots$ be elements of $\disMeasures(S)$
  such that $\nu_{n} \to \nu$ vaguely.
  Fix $x_{n}, x \in S$ with $x_{n} \to x$.
  Then, $\nu_{n}(\{x_{n}\}) \to \nu(\{x\})$ if and only if 
  $\lim_{\delta \downarrow 0} \limsup_{n \to \infty} \nu_{n}(B_{S}(x_{n}, \delta) \setminus \{x_{n}\}) = 0$.
\end{prop}

\begin{proof}
  Since $\nu_{n} \to \nu$ vaguely and $x_{n} \to x$,
  we have that, for each $\delta > 0$,
  \begin{equation}
    \nu(D_{S}(x, \delta)) 
    \geq 
    \limsup_{n \to \infty} \nu_{n}(D_{S}(x, \delta))
    \geq 
    \limsup_{n \to \infty} \nu_{n}(\{x_{n}\}).
  \end{equation}
  Letting $\delta \downarrow 0$ yields that 
  $\nu(\{x\}) \geq \limsup_{n \to \infty}  \nu_{n}(\{x_{n}\})$.
  Similarly, we deduce that 
  \begin{align}  \label{2. eq: equivalence of convergence of weights 1}
    \nu(\{x\})
    &= 
    \lim_{\delta \downarrow 0} \nu(B_{S}(x, \delta)) \\
    &\leq
    \lim_{\delta \downarrow 0} \liminf_{n \to \infty} \nu_{n}(B_{S}(x, \delta)) \\  
    &\leq    
    \lim_{\delta \downarrow 0} \liminf_{n \to \infty} \nu_{n}(B_{S}(x_{n}, 2\delta)).
  \end{align}
  Moreover, we have that, for each $\delta > 0$,
  \begin{equation} \label{2. eq: equivalence of convergence of weights 2}
    \liminf_{n \to \infty} \nu_{n}(B_{S}(x_{n}, 2\delta)) 
    \leq 
    \limsup_{n \to \infty} \nu_{n}(B_{S}(x_{n}, 2\delta) \setminus \{x_{n} \}) 
    + 
    \liminf_{n \to \infty} \nu_{n}(\{x_{n}\}).
  \end{equation}
  Thus, from \eqref{2. eq: equivalence of convergence of weights 1} and \eqref{2. eq: equivalence of convergence of weights 2},
  we obtain the ``if'' part of the assertion.
  For the other direction,
  suppose that $\nu_{n}(\{x_{n}\}) \to \nu(\{x\})$.
  We then have that, for each $\delta > 0$,
  \begin{align}
    \limsup_{n \to \infty} \nu_{n}(B_{S}(x_{n}, \delta) \setminus \{x_{n}\})
    &=
    \limsup_{n \to \infty} \bigl\{ \nu_{n}(B_{S}(x_{n}, \delta)) - \nu_{n} (\{x_{n}\}) \bigr\} \\
    &\leq 
    \limsup_{n \to \infty} \nu_{n}(D_{S}(x_{n}, \delta)) - \liminf_{n \to \infty} \nu_{n} (\{x_{n}\}) \\
    &\leq    
    \limsup_{n \to \infty} \nu_{n}(D_{S}(x, 2\delta)) - \liminf_{n \to \infty} \nu_{n} (\{x_{n}\}) \\
    &\leq 
    \nu(D_{S}(x, 2\delta)) - \nu(\{x\}),
  \end{align}
  where we use the vague convergence $\nu_{n} \to \nu$ to establish the last inequality.
  Letting $\delta \downarrow 0$ yields the result.
\end{proof}

\begin{thm} \label{2. thm: convergence in Mdis}
  Let $\nu, \nu_{1}, \nu_{2}, \ldots$ be elements of $\disMeasures(S)$.
  The following statements are equivalent:
  \begin{enumerate} [label = (\roman*)]
    \item \label{2. thm item: VP convergence, wrt metric} 
      $\nu_{n} \to \nu$ in the vague-and-point-process topology;
    \item \label{2. thm item: VP convergence, wrt two topologies}
      $\nu_{n} \to \nu$ vaguely
      and $\pointMap(\nu_{n}) \to \pointMap(\nu)$ vaguely;
    \item \label{2. thm item: VP convergence, weak and FIN's condition}
      $\nu_{n} \to \nu$ vaguely and,
      for any $x \in \At(\nu)$,
      there exist atoms $x_{n} \in \At(\nu_{n})$ such that $x_{n} \to x$ and $\nu_{n}(\{x_{n}\}) \to \nu(\{x\})$;
    \item \label{2. thm item: VP convergence, weak and liminf FIN's condition}
      $\nu_{n} \to \nu$ vaguely and,
      for any $x \in \At(\nu)$,
      there exist atoms $x_{n} \in \At(\nu_{n})$ such that $x_{n} \to x$ 
      and $\lim_{\delta \to 0}\limsup_{n \to \infty} \nu_{n}(B_{S}(x_{n}, \delta) \setminus \{x_{n}\}) = 0$.
  \end{enumerate}
\end{thm}

\begin{proof}
  The equivalence of \ref{2. thm item: VP convergence, wrt metric} and \ref{2. thm item: VP convergence, wrt two topologies} 
  follows from the definition of $d_{\disMeasures}^{S, \rho_{S}}$.
  The implication \ref{2. thm item: VP convergence, wrt two topologies} $\Rightarrow$ \ref{2. thm item: VP convergence, weak and FIN's condition} 
  follows by applying Proposition \ref{2. prop: atom convergence from simple measure convergence} to the convergence of $\pointMap(\nu_{n})$ to $\pointMap(\nu)$.
  Assume that \ref{2. thm item: VP convergence, weak and FIN's condition} is satisfied.
  Let $r>0$ be such that the boundary of $K \coloneqq D_{S}(\rho_{S}, r) \times [e^{-r}, e^{r}]$ does not contain any atoms of $\pointMap(\nu)$.
  Write $\{(x_{i}, w_{i})\}_{i \in I_{K}}$ for the atoms of $\pointMap(\nu)$ lying in $K$.
  For each $i \in I_{K}$,
  we let $x_{i}^{(n)} \in \At(\nu_{n})$ be such that 
  $x_{i}^{(n)} \to x_{i}$ and $w_{i}^{(n)} \coloneqq \nu_{n}(\{x_{i}^{(n)}\}) \to w_{i}$.
  Since $(x_{i}, w_{i})$ is in the interior of $K$,
  we have that $(x_{i}^{(n)}, w_{i}^{(n)}) \in K$ for all $i \in I_{K}$
  (at least for all sufficiently large $n$).
  By Proposition \ref{2. prop: atom convergence from simple measure convergence},
  it remains to prove that $\At(\pointMap(\nu_{n})) \cap K = \{(x_{i}^{(n)}, w_{i}^{(n)})\}_{i \in I_{K}}$ for all sufficiently large $n$. 
  Suppose that it is not the case.
  Then, there exist a subsequence $(n_{k})_{k \geq 1}$ and $(x^{(n_{k})}, w^{(n_{k})}) \in \At(\pointMap(\nu_{n_{k}})) \cap K$ 
  such that $(x^{(n_{k})}, w^{(n_{k})}) \notin \{(x_{i}^{(n_{k})}, w_{i}^{(n_{k})})\}_{i \in I_{K}}$.
  Since $D_{S}(\rho_{S}, r)$ is compact and $\nu_{n} \to \nu$ vaguely,
  we have that $\limsup_{k \to \infty} \nu_{n_{k}}(D_{S}(\rho_{S}, r)) \leq \nu(D_{S}(\rho_{S}, r))$.
  This, combined with $e^{-r} \leq w^{(n_{k})}$ and $w_{i}^{(n_{k})} \to w_{i}$, 
  yields that 
  \begin{align}
    e^{-r} + \sum_{i \in I_{K}} w_{i}
    &\leq 
    \limsup_{k \to \infty} 
    \left(
      w^{(n_{k})} + \sum_{i \in I_{K}} w_{i}^{(n_{k})}
    \right)\\
    &\leq
    \limsup_{k \to \infty}  
    \nu_{n_{k}}(D_{S}(\rho_{S}, r))\\
    &\leq
    \nu(D_{S}(\rho_{S}, r))\\
    &=
    \sum_{i \in I_{K}} w_{i},
  \end{align}
  which is a contradiction.
  Therefore, we obtain \ref{2. thm item: VP convergence, wrt two topologies}.
  The equivalence of \ref{2. thm item: VP convergence, weak and FIN's condition} 
  and \ref{2. thm item: VP convergence, weak and liminf FIN's condition} 
  follows from Proposition \ref{2. prop: equivalence of convergence of weights}.
\end{proof}

The metric $d_{\disMeasures}^{S, \rho_{S}}$ is preserved by root-and-distance-preserving maps.
This fact is important for metrization of Gromov-Hausdorff-type topologies in Section \ref{sec: GH-type topologies}.

\begin{prop}  \label{2. prop: dcMdis is preserved}
  Let $(S_{i}, d^{S_{i}}, \rho_{i}),\, i=1,2$ be rooted boundedly-compact metric spaces 
  and $f: S_{1} \to S_{2}$ be a root-and-distance-preserving map.
  Then, the map from $(\disMeasures(S_{1}),\allowbreak d_{\disMeasures}^{S_{1}, \rho_{S_{1}}})$ 
  to $(\disMeasures(S_{2}),\allowbreak d_{\disMeasures}^{S_{2}, \rho_{S_{2}}})$ 
  given by $\nu \mapsto \nu \circ f^{-1}$ is distance-preserving.
\end{prop}

\begin{proof}
  It is easy to check that $\pointMap(\nu \circ f^{-1}) = \pointMap(\nu) \circ (f \times \id^{\RNpp})^{-1}$
  and $f \times \id^{\RNpp}$ is a root-and-distance preserving map 
  from $S_{1} \times \RNpp$ to $S_{2} \times \RNpp$.
  Since the Prohorov metrics and the vague metrics are preserved by root-and-distance-preserving maps
  (see \cite[Lemma 2.40 and Proposition 2.41]{Noda_pre_Metrization}),
  we obtain the desired result.
\end{proof}

Although the vague metrics $d_{V}^{S, \rho_{S}}$ and $d_{V}^{S \times \RNpp, \tilde{\rho}_{S}}$ are complete,
$d_{\disMeasures}^{S, \rho_{S}}$ is not complete in general.
To see this,
consider a sequence of measures whose atoms collide.
For example,
let $S = \RN$ and $\nu _{n} \coloneqq \delta_{0} + \delta_{n^{-1}}$.
It is elementary to check that 
$\nu _{n}$ converges vaguely to the measure $2 \delta_{0}$ on $\RN$,
and 
$\pointMap(\nu _{n}) = \delta_{(0, 1)} + \delta_{(n^{-1}, 1)}$ 
converges vaguely to the measure $2 \delta_{(0,1)}$ on $\RN \times \RNpp$.
Thus, $(\nu _{n})_{n \geq 1}$ is Cauchy with respect to $d_{\disMeasures}^{\RN, 0}$,
but it does not converge in $\disMeasures(\RN)$.
However, the vague-and-point-process topology is Polish,
that is, there exists another metric that is complete, separable and induces the same topology
(see Theorem \ref{2. thm: the VPP topology is Polish} below).
For further study of the vague-and-point-process topology, such as Polishness and a precompactness criterion,
it is convenient to introduce a larger space $\pointMeas(S)$, into which $\disMeasures(S)$ is topologically embedded.
This is the main aim of the following subsection.


\subsection{The space \texorpdfstring{$\pointMeas(S)$}{P(S)}} \label{sec: a space including the space of discrete measures}

As already explained above,
in this subsection,
we introduce a space $\pointMeas(S)$, into which $\disMeasures(S)$ is topologically embedded,
and study its topological properties.
In particular,
we prove that the vague-and-point-process topology is Polish (Theorem \ref{2. thm: the VPP topology is Polish}) 
and provide precompact and tightness criteria (Theorems \ref{2. thm: precompactness in sP} and \ref{2. thm: tightness in sP}).

We denote by $\intMeasures(S \times \RNpp)$
the collection of integer-valued Radon measures $\pi$ on $S \times \RNpp$,
i.e.,
$\pi(E) \in \ZNp \cup \{\infty\}$ for any Borel subset $E$ of $S \times \RNpp$.
Note that  
any $\pi \in \intMeasures(S \times \RNpp)$ is a discrete measure 
and if an atomic decomposition of $\pi$ is given by $\pi = \sum_{i \in I} \beta_{i} \delta_{(x_{i}, w_{i})}$,
then $\beta_{i}$ is a positive integer
(cf. \cite[Theorem 2.18]{Kallenberg_21_Foundations}).
We associate $\pi$ 
with a Borel measure $\MeasureMap(\pi)$ on $S$ 
by setting 
\begin{equation}
  \MeasureMap(\pi) (A)
  \coloneqq
  \int_{A \times \RNpp} 
  w\, \pi(dx dw) 
  =
  \sum_{i \in I} \beta_{i} w_{i} \delta_{x_{i}}(A),
  \quad 
  \forall A \in \Borel(S).
\end{equation}

\begin{dfn} [{The space $\pointMeas(S)$}]
  We define 
  \begin{equation}
    \pointMeas(S)
    \coloneqq 
    \left\{
      \pi \in \intMeasures(S \times \RNpp)
      \mid  
      \MeasureMap(\pi)\ \text{is a Radon measure on}\ S
    \right\}.
  \end{equation}
\end{dfn}

For each $\pi = \sum_{i \in I} \beta_{i} \delta_{(x_{i}, w_{i})} \in \pointMeas(S)$ and $r > 0$,
we define a Borel measure $\vertMap^{(r)}(\pi)$ on $\RNpp$ by setting 
\begin{equation}
  \vertMap^{(r)}(\pi)(A)
  \coloneqq 
  \int_{S^{(r)} \times A} w\, \pi(dx dw)
  = 
  \sum_{\substack{x_{i} \in S^{(r)}\\ w_{i} \in A}} \beta_{i} w_{i},
  \quad 
  \forall A \in \Borel(\RNpp).
\end{equation}
Note that $\vertMap^{(r)}(\pi)$ is a finite measure.
Indeed, we have that $\vertMap^{(r)}(\RNpp) = \MeasureMap(\pi)(S^{(r)})$, 
which is finite since $\MeasureMap(\pi)$ is a Radon measure.
We write
\begin{gather}
  M_{\varepsilon}^{(r)}(\pi) 
  \coloneqq 
  \vertMap^{(r)}(\pi)((0, \varepsilon])
  =
  \int_{S^{(r)} \times (0, \varepsilon]} w\, \pi(dx dw) 
  =
  \sum_{ \substack{x_{i} \in S^{(r)} \\ w_{i} \leq \varepsilon,}} \beta_{i} w_{i}, 
  \quad \varepsilon > 0,
  \label{2. eq: dfn of M_epsilon^r}\\
  W^{(r)}(\pi) 
  \coloneqq 
  \sup\{ w_{i} \mid x_{i} \in S^{(r)} \} 
  = 
  \inf 
  \bigl\{
    l > 0 \mid \vertMap^{(r)}(\pi)([l, \infty)) = 0
  \bigr\}.
  \label{2. eq: dfn of W^r}
\end{gather}
Below, we discuss some basic properties of $\vertMap^{(r)}(\pi)$, $M_{\varepsilon}^{(r)}$ and $W^{(r)}(\pi)$.

\begin{lem} \label{2. lem: properties of edge-weight M and heavy weigh W}
  Fix $\pi \in \pointMeas(S)$.
  Then the following statements hold.
  \begin{enumerate} [label = (\roman*)]
    \item \label{2. lem item: estimate on frakm}
      For any $r > s > 0$, $d_{P}^{\RNpp}(\vertMap^{(r)}(\pi), \vertMap^{(s)}(\pi)) \leq \MeasureMap(\pi)(S^{(r)} \setminus S^{(s)})$.
    \item \label{2. lem item: path property of frakm}
      The map $r \mapsto \vertMap^{(r)}(\pi) \in (\finMeas(\RNpp), d_{P}^{\RNpp})$ is left-continuous with right-hand limits.
    \item \label{2. lem item: monotonicity of small mass function}
      For each $r>0$, the function $\varepsilon \mapsto M_{\varepsilon}^{(r)}(\pi)$ is increasing
      and $\lim_{\varepsilon \to 0} M_{\varepsilon}^{(r)}(\pi) = 0$.
    \item \label{2. lem item: max weight is finite} 
      For each $r>0$, $W^{(r)}(\pi) \leq \MeasureMap(\pi)(S^{(r)}) < \infty$.
  \end{enumerate}
\end{lem}

\begin{proof}
  \ref{2. lem item: estimate on frakm}.
  We have that, for any Borel subset $A \subseteq \RNpp$,
  \begin{align}
    \vertMap^{(r)}(\pi)(A) 
    &= 
    \int_{S^{(r)} \times A} w\, \pi(dx dw)\\
    &\leq 
    \int_{S^{(s)} \times A} w\, \pi(dx dw)
    + 
    \int_{(S^{(r)} \setminus S^{(s)}) \times \RNpp} w\, \pi(dx dw)\\
    &=
    \vertMap^{(s)}(\pi)(A) 
    + 
    \MeasureMap(\pi)(S^{(r)} \setminus S^{(s)}),
  \end{align}
  and $\vertMap^{(s)}(\pi)(A) \leq \vertMap^{(r)}(\pi)(A)$.
  Thus, we obtain the desired result.

  \ref{2. lem item: path property of frakm}.
  The left-continuity follows from \ref{2. lem item: estimate on frakm} 
  and that $\lim_{s \uparrow r} \MeasureMap(\pi)(S^{(r)} \setminus S^{(s)}) = 0$.
  For the right-hand limits,
  fix $r > 0$ and 
  define a finite measure $\vertMap^{(r+)}(\pi)$ on $\RNpp$ by setting 
  \begin{equation}
    \vertMap^{(r+)}(\pi)(A) 
    = 
    \int_{D_{S}(\rho_{S}, r) \times A} w\, \pi(dx dw),
    \quad 
    \forall A \in \Borel(\RNpp).
  \end{equation}
  By the same argument as before,
  we deduce that, for any $s > r$,
  \begin{equation}
    d_{P}^{\RNpp}(\vertMap^{(r+)}(\pi), \vertMap^{(s)}(\pi)) 
    \leq 
    \MeasureMap(\pi)(S^{(s)} \setminus D_{S}(\rho_{S}, r)).
  \end{equation}
  Letting $s \downarrow r$ in the above inequality,
  we obtain that $\vertMap^{(s)}(\pi) \to \vertMap^{(r+)}(\pi)$.

  \ref{2. lem item: monotonicity of small mass function} and \ref{2. lem item: max weight is finite}.
  These are obvious by their definitions.
\end{proof}

We now introduce a metric on $\pointMeas(S)$.
For $\pi_{1}, \pi_{2} \in \pointMeas(S)$, 
we set 
\begin{gather}
  \mathsf{M}^{(r)} (\pi_{1}, \pi_{2}) 
  \coloneqq 
  d_{P}^{\RNpp}(\vertMap^{(r)}(\pi_{1}), \vertMap^{(r)}(\pi_{2})),
  \quad 
  \forall r > 0,
  \\ 
  d_{\pointMeas}^{S, \rho_{S}}(\pi_{1}, \pi_{2})
  = 
  \int_{0}^{\infty} 
    e^{-r} 
    \bigl(
      1 \wedge \mathsf{M}^{(r)}(\pi_{1}, \pi_{2}) 
    \bigr)\,
  dr 
  \vee 
  d_{V}^{S \times \RNpp, \tilde{\rho}_{S}} (\pi_{1}, \pi_{2})
  \vee 
  d_{V}^{S, \rho_{S}} (\MeasureMap(\pi_{1}), \MeasureMap(\pi_{2})).
\end{gather}
Note that the integrals are well-defined by Lemma \ref{2. lem: properties of edge-weight M and heavy weigh W}\ref{2. lem item: path property of frakm}
(cf.\ \cite[Proposition 2.5]{Noda_pre_Metrization}).

\begin{prop}
  The function $d_{\pointMeas}^{S, \rho_{S}}$ is a metric on $\pointMeas(S)$.
\end{prop}

\begin{proof}
  By Theorem \ref{2. thm: convergence in the vague topology},
  $d_{V}^{S \times \RNpp, \tilde{\rho}_{S}}$ is a metric on $\pointMeas(S)$,
  which implies that $d_{\pointMeas}^{S, \rho_{S}}$ is positive definite.
  Symmetry and the triangle inequality are obvious.
\end{proof}

Below, we provide some characterizations of the topology on $\pointMeas(S)$ in terms of convergence.

\begin{thm} [{Convergence in $\pointMeas(S)$}]  \label{2. thm: convergence in sP}
  Let $\pi, \pi_{1}, \pi_{2}, \ldots$ be elements of $\pointMeas(S)$.
  The following statements are equivalent:
  \begin{enumerate} [label = (\roman*)]
    \item \label{2. thm item: sP, wrt metric}
      $\pi_{n} \to \pi$ with respect to $d_{\pointMeas}^{S, \rho_{S}}$;
    \item \label{2. thm item: sP, in two topologies}
      $\pi_{n}  \to \pi$ vaguely and $\MeasureMap(\pi_{n} ) \to \MeasureMap(\pi)$ vaguely;
    \item \label{2. thm item: sP, vaguely, heavy weight and low mass}
      $\pi_{n}  \to \pi$ vaguely
      and, for each $r>0$,
      $\displaystyle  \limsup_{n \to \infty} W^{(r)}(\pi_{n} ) < \infty$
      and
      $\displaystyle \lim_{\varepsilon \downarrow 0} 
        \limsup_{n \to \infty} M_{\varepsilon}^{(r)}(\pi_{n} ) = 0$.
  \end{enumerate}
\end{thm}

\begin{proof}
  The implication \ref{2. thm item: sP, wrt metric} $\Rightarrow$ \ref{2. thm item: sP, in two topologies} is obvious.
  Suppose that \ref{2. thm item: sP, in two topologies} is satisfied.
  The vague convergence $\MeasureMap(\pi_{n} ) \to \MeasureMap(\pi)$ 
  implies that $\sup_{n \geq 1} \MeasureMap(\pi_{n} )(S^{(r)}) < \infty$ for each $r>0$.
  From Lemma \ref{2. lem: properties of edge-weight M and heavy weigh W}\ref{2. lem item: max weight is finite},
  we obtain that $\sup_{n \geq 1} W^{(r)}(\pi_{n} ) < \infty$ for each $r>0$.
  Let $\varepsilon, r> 0$ be such that 
  $\pi$ has no atoms on the boundary of $S^{(r)} \times (\varepsilon, \infty)$.
  We then choose $l>0$ so that 
  $\pi$ has no atoms on the boundary of $S^{(r)} \times (\varepsilon, l)$ and $\sup_{n \geq 1} W^{(r)}(\pi_{n} ) < l$.
  We have that 
  \begin{align}
    |M_{\varepsilon}^{(r)}(\pi_{n} ) - M_{\varepsilon}^{(r)}(\pi)| 
    &=
    \left|
      \int_{S^{(r)} \times (0, \varepsilon]} w\, \pi_{n} (dx dw) 
      - 
      \int_{S^{(r)} \times (0, \varepsilon]} w\, \pi(dx dw) 
    \right| \\
    &\leq 
    \bigl| \MeasureMap(\pi_{n} )(S^{(r)}) - \MeasureMap(\pi)(S^{(r)}) \bigr| \\
    &\quad
    + 
    \left|
      \int_{S^{(r)} \times (\varepsilon, l)} w\, \pi_{n} (dx dw) 
      - 
      \int_{S^{(r)} \times (\varepsilon, l)} w\, \pi(dx dw)
    \right|.
  \end{align}
  Since $\MeasureMap(\pi_{n})|_{S^{(r)}} \to \MeasureMap(\pi)|_{S^{(r)}}$ weakly 
  and $\pi_{n}|_{S^{(r)} \times (\varepsilon, l)} \to \pi|_{S^{(r)} \times (\varepsilon, l)}$ weakly,
  we deduce that $M_{\varepsilon}^{(r)}(\pi_{n}) \to M_{\varepsilon}^{(r)}(\pi)$.
  From the dominated convergence theorem 
  and Lemma \ref{2. lem: properties of edge-weight M and heavy weigh W}\ref{2. lem item: monotonicity of small mass function},
  we obtain \ref{2. thm item: sP, vaguely, heavy weight and low mass}.

  Assume that \ref{2. thm item: sP, vaguely, heavy weight and low mass} is satisfied.
  We will prove \ref{2. thm item: sP, wrt metric}.
  We note that, by Lemma \ref{2. lem: properties of edge-weight M and heavy weigh W}\ref{2. lem item: max weight is finite}, 
  $\sup_{n \geq 1} W^{(r)}(\pi_{n})$ is finite for each $r>0$.
  Fix a continuous function $f$ on $S$ whose support is contained in $S^{(r)}$,
  where we assume that $\MeasureMap(\pi)$ has no atoms on the boundary of $S^{(r)}$.
  Let $\varepsilon, l > 0$ be such that 
  $\pi$ has no atoms on the boundary of $S^{(r)} \times (\varepsilon, l)$ 
  and $\sup_{n \geq 1} W^{(r)}(\pi_{n}) < l$. 
  We then have that 
  \begin{align}
    &\left|
      \int f(x)\, \MeasureMap(\pi_{n})(dx) 
      - 
      \int f(x)\, \MeasureMap(\pi)(dx)
    \right| \\
    = 
    &\left|
      \int_{S^{(r)} \times (0, \infty)} w f(x)\, \pi_{n}(dx dw) 
      - 
      \int_{S^{(r)} \times (0, \infty)} w f(x)\, \pi(dx dw)
    \right| \\
    \leq 
    &\left|
      \int_{S^{(r)} \times (\varepsilon, l)} w f(x)\, \pi_{n}(dx dw) 
      - 
      \int_{S^{(r)} \times (\varepsilon, l)} w f(x)\, \pi(dx dw)
    \right| 
    +
    \| f \|_{\infty} 
    \bigl(
      M_{\varepsilon}^{(r)}(\pi_{n}) + M_{\varepsilon}^{(r)}(\pi)
    \bigr).
  \end{align}
  By Lemma \ref{2. lem: properties of edge-weight M and heavy weigh W}\ref{2. lem item: monotonicity of small mass function},
  \ref{2. thm item: sP, vaguely, heavy weight and low mass},
  and the weak convergence $\pi_{n}|_{S^{(r)} \times (\varepsilon, l)} \to \pi|_{S^{(r)} \times (\varepsilon, l)}$,
  we obtain that $\MeasureMap(\pi_{n}) \to \MeasureMap(\pi)$ vaguely.
  It remains to show that $\vertMap^{(r)}(\pi_{n}) \to \vertMap^{(r)}(\pi)$ weakly for all but countably many $r>0$.
  Fix $r > 0$ satisfying $\MeasureMap(\pi)(\partial S^{(r)}) = 0$.
  Let $\varepsilon, l>0$ be such that 
  $\pi$ has no atoms on the boundary of $S^{(r)} \times (\varepsilon, l)$ 
  and $\sup_{n \geq 1} W^{(r)}(\pi_{n}) < l$. 
  For a bounded continuous function $g$ on $\RNpp$,
  we have that 
  \begin{align}
    &\left|
      \int g(w)\, \vertMap^{(r)}(\pi_{n})(dw) 
      -
      \int g(w)\, \vertMap^{(r)}(\pi)(dw)
    \right| \\
    &= 
    \left|
      \int_{S^{(r)} \times \RNpp} g(w)w\, \pi_{n}(dx dw) 
      - 
      \int_{S^{(r)} \times \RNpp} g(w)w\, \pi(dx dw)
    \right| \\
    &\leq
    \left|
      \int_{S^{(r)} \times (\varepsilon, l)} g(w)w\, \pi_{n}(dx dw) 
      - 
      \int_{S^{(r)} \times (\varepsilon, l)} g(w)w\, \pi(dx dw)
    \right|
    + 
    \| g \|_{\infty} 
    \bigl(
      M_{\varepsilon}^{(r)}(\pi_{n}) + M_{\varepsilon}^{(r)}(\pi)
    \bigr)
  \end{align}
  Hence, similarly as before,
  we obtain that $\vertMap^{(r)}(\pi_{n}) \to \vertMap^{(r)}(\pi)$ weakly.
\end{proof}

For the next result,
recall the map $\pointMap$ from \eqref{2. eq: dfn of point operator}.
Via this map,
the space $\disMeasures(S)$ equipped with the vague-and-point-process topology is topologically embedded into $\pointMeas(S)$.

\begin{cor}  \label{2. cor: embedding of Mdis to sP}
  The map $\pointMap: \disMeasures(S) \to \pointMeas(S)$ is a topological embedding,
  i.e.,
  a homeomorphism onto its image.
  If we write $\pointMeas^{*}(S)$ for the image,
  then 
  \begin{equation}  \label{2. cor eq: image of Mdis by p}
    \pointMeas^{*}(S)
    =
    \Bigl\{
      \pi \in \pointMeas(S) 
      \mid
      \pi = \sum_{i \in I} \delta_{(x_{i}, w_{i})}\ \text{with}\ x_{i} \neq x_{j}\ \text{if}\ i \neq j
    \Bigr\},
  \end{equation}
  and the inverse map is $\MeasureMap |_{\pointMeas^{*}(S)}$, i.e., the restriction of the map $\MeasureMap$ to $\pointMeas^{*}(S)$.
\end{cor}

\begin{proof}
  The assertion \eqref{2. cor eq: image of Mdis by p} follows from definition.
  Using \eqref{2. cor eq: image of Mdis by p},
  one can check that $\MeasureMap |_{\pointMeas^{*}(S)}$ is the inverse map.
  The continuity of $\pointMap$ and $\MeasureMap |_{\pointMeas^{*}(S)}$ is obvious 
  by Theorem \ref{2. thm: convergence in Mdis} and Theorem \ref{2. thm: convergence in sP}.
\end{proof}

The following result is an immediate consequence of the above result.
From this result, 
convergence in the vague-and-point-process topology 
is proven by convergence in $\pointMeas(S)$.

\begin{cor} \label{2. cor: convergence in VPP from that in sP}
  Let $\pi, \pi_{1}, \pi_{2}, \ldots$ be elements of $\pointMeas^{*}(S)$.
  If $\pi_{n} \to \pi$ in $\pointMeas(S)$,
  then $\MeasureMap(\pi_{n}) \to \MeasureMap(\pi)$ in the vague-and-point-process topology
  (as elements in $\disMeasures(S)$).
\end{cor}

For the next result,
recall from \eqref{2. dfn eq: restriction of a measure to a closed ball} 
that $\pi^{(r)}$ denotes the restriction of $\pi$ to $S^{(r)} \times [e^{-r}, e^{r}]$.

\begin{lem} \label{2. lem: convergence of pi^r to pi}
  For any $\pi \in \pointMeas(S)$,
  we have that $\pi^{(s)} \to \pi$ in $\pointMeas(S)$ as $s \to \infty$.
\end{lem}

\begin{proof}
  By Theorem \ref{2. thm: convergence in the vague topology},
  we have that $\pi^{(s)} \to \pi$ vaguely. 
  It is easy to check that $W^{(r)}(\pi^{(s)}) \leq W^{(r)}(\pi)$ 
  and $M_{\varepsilon}^{(r)}(\pi^{(s)}) \leq M_{\varepsilon}^{(r)}(\pi)$
  for each $r>0$.
  It then follows 
  from Lemma \ref{2. lem: properties of edge-weight M and heavy weigh W}\ref{2. lem item: monotonicity of small mass function} and \ref{2. lem item: max weight is finite}
  and Theorem \ref{2. thm: convergence in sP} 
  that $\pi^{(s)} \to \pi$ in $\pointMeas(S)$.
\end{proof}

\begin{thm} [{Polishness of $\pointMeas(S)$}]\label{2. thm: sP is Polish}
  The metric $d_{\pointMeas}^{S, \rho_{S}}$ is complete and separable.
\end{thm}

\begin{proof}
  We first show the separability.
  Let $S_{d}$ be a countable dense subset of $S$.
  Write $\mathscr{D}$ for the collection of $\pi \in \pointMeas(S)$ such that
  \begin{equation}
    \pi 
    = 
    \sum_{i=1}^{n} 
    \beta_{i} \delta_{(x_{i}, w_{i})},
  \end{equation}
  where $n \in \NN$, $\beta_{i} \in \NN$, $x_{i} \in S_{d}$, and $w_{i} \in \QN \cap \RNpp$.
  Note that $\mathscr{D}$ is countable.
  It is easy to check that any $\pi \in \pointMeas(S)$ with finitely many atoms is approximated by a sequence in $\mathscr{D}$.
  For each $\pi \in \pointMeas(S)$,
  $\pi^{(r)}$ has only finitely many atoms.
  Hence, by Lemma \ref{2. lem: convergence of pi^r to pi},
  we deduce that $\mathscr{D}$ is dense in $\pointMeas(S)$.

  Next, we show the completeness.
  Let $(\pi_{n})_{n \geq 1}$ be a Cauchy sequence with respect to $d_{\pointMeas}^{S, \rho_{S}}$.
  Since $d_{V}^{S \times \RNpp, \tilde{\rho}_{S}}$ is complete,
  there exists a Radon measure $\pi$ on $S \times \RNpp$
  such that $\pi_{n} \to \pi$ vaguely.
  Then,
  for any bounded measurable set $E$ with $\pi(\partial E) = 0$,
  we have that $\pi_{n}(E) \to \pi(E)$,
  which implies that $\pi(E) \in \ZNp$.
  Hence $\pi \in \intMeasures(S \times \RNpp)$.
  For each $r>0$,
  let $g_{r}: S \to [0, 1]$ be a continuous function 
  such that $g_{r}|_{S^{(r)}} \equiv 1$ and $g_{r}|_{S \setminus S^{(r+1)}} \equiv 0$
  and 
  let $(f_{k})_{k \geq 1}$ be non-negative continuous functions on $S \times \RNpp$ with compact support 
  increasing to the constant function $1_{S \times \RNpp}$.
  Using the vague convergence $\pi_{n} \to \pi$ and the monotone convergence theorem,
  we deduce that 
  \begin{align}
    \MeasureMap(\pi)(S^{(r)})
    &\leq 
    \int_{S \times \RNpp} g_{r}(x) w\, \pi(dx dw)\\
    &= 
    \lim_{k \to \infty} 
    \int f_{k}(x, w) g_{r}(x) w\, \pi(dx dw) \\
    &= 
    \lim_{k \to \infty} 
    \lim_{n \to \infty}
    \int f_{k}(x, w) g_{r}(x) w\, \pi_{n}(dx dw) \\
    &\leq 
    \limsup_{n \to \infty} 
    \int g_{r}(x) w\, \pi_{n}(dx dw) \\
    &\leq
    \limsup_{n \to \infty}
    \MeasureMap(\pi_{n})(S^{(r+1)}),
  \end{align}
  Since $(\MeasureMap(\pi_{n}))_{n \geq 1}$ is tight in the vague topology,
  we have that $\sup_{n \geq 1} \MeasureMap(\pi_{n})(S^{(r+1)}) < \infty$.
  Hence, $\pi \in \pointMeas(S)$.
  Moreover,
  from Lemma \ref{2. lem: properties of edge-weight M and heavy weigh W}\ref{2. lem item: max weight is finite},
  it holds that $\sup_{n \geq 1} W^{(r)}(\pi_{n}) < \infty$ 
  for each $r>0$.
  Thus, Theorem \ref{2. thm: convergence in sP} implies that 
  it is enough to show that, for all but countably many $r>0$,
  \begin{equation}  \label{2. eq: sP, proof of completeness}
    \lim_{\varepsilon \downarrow 0}
    \limsup_{n \to \infty} 
    M_{\varepsilon}^{(r)}(\pi_{n})
    = 0.
  \end{equation}
  Since $d_{V}^{S, \rho_{S}}$ is complete,
  there exists a Radon measure $\mu$ on $S$ such that 
  $\MeasureMap(\pi_{n}) \to \mu$ vaguely on $S$.
  Fix $r>0$ such that $\mu( D_{S}(\rho_{S}, r) \setminus B_{S}(\rho_{S}, r)) = 0$.
  We then have that 
  \begin{align} 
    \limsup_{\delta \to 0} 
    \limsup_{n \to \infty} 
    \MeasureMap(\pi_{n})(S^{(r+\delta)} \setminus S^{(r-\delta)})
    &\leq
    \limsup_{\delta \to 0} 
    \limsup_{n \to \infty} 
    \MeasureMap(\pi_{n})(S^{(r+\delta)} \setminus B_{S}(\rho_{S}, r-\delta)) \\
    &\leq
    \limsup_{\delta \to 0} 
    \mu(S^{(r+\delta)} \setminus B_{S}(\rho_{S}, r-\delta)) \\
    &= 
    0.
    \label{2. eq: completeness of Ps, boundary mass disappear}
  \end{align}
  Fix $\delta > 0$.
  For any $s \in (r-\delta, r+\delta)$,
  we have from Lemma \ref{2. lem: properties of edge-weight M and heavy weigh W}\ref{2. lem item: estimate on frakm} that
  \begin{equation}
    d_{P}^{\RNpp}(\vertMap^{(s)}(\pi_{n}), \vertMap^{(r)}(\pi_{n}))
    \leq
    \MeasureMap(\pi_{n})(S^{(r+\delta)} \setminus S^{(r-\delta)}),
    \quad 
    \forall n \geq 1.
  \end{equation}
  We thus deduce that 
  \begin{align}
    1 \wedge d_{P}^{\RNpp}(\vertMap^{(r)}(\pi_{m}), \vertMap^{(r)}(\pi_{n}))
    &=
    (2\delta)^{-1} 
    \int_{r-\delta}^{r+\delta}
    (1 \wedge d_{P}^{\RNpp}(\vertMap^{(r)}(\pi_{m}), \vertMap^{(r)}(\pi_{n})))\,
    ds \\ 
    &\leq
    \MeasureMap(\pi_{m})(S^{(r+\delta)} \setminus S^{(r-\delta)})
    +
    \MeasureMap(\pi_{n})(S^{(r+\delta)} \setminus S^{(r-\delta)}) \\
    &\quad
    + 
    (2\delta)^{-1} 
    \int_{r-\delta}^{r+\delta}
    (1 \wedge d_{P}^{\RNpp}(\vertMap^{(s)}(\pi_{m}), \vertMap^{(s)}(\pi_{n})))\,
    ds \\
    &\leq
    \MeasureMap(\pi_{m})(S^{(r+\delta)} \setminus S^{(r-\delta)})
    +
    \MeasureMap(\pi_{n})(S^{(r+\delta)} \setminus S^{(r-\delta)}) \\
    &\quad
    + 
    (2\delta)^{-1} e^{r+\delta}\,
    d_{\pointMeas}^{S, \rho_{S}}(\pi_{m}, \pi_{n}),
  \end{align}
  which, combined with \eqref{2. eq: completeness of Ps, boundary mass disappear}, implies that 
  $(\vertMap^{(r)}(\pi_{n}))_{n \geq 1}$ is a Cauchy sequence with respect to $d_{P}^{\RNpp}$.
  In particular, it is tight 
  and so we have that 
  \begin{equation}
    \lim_{\varepsilon \to 0} 
    \limsup_{n \to \infty} 
    \vertMap^{(r)}(\pi_{n})((0, \varepsilon)) 
    = 0.
  \end{equation}
  Since $M_{\varepsilon/2}^{(r)}(\pi_{n}) \leq \vertMap^{(r)}(\pi_{n})((0, \varepsilon))$ by the definition of $M_{\varepsilon/2}^{(r)}(\pi_{n})$,
  we obtain \eqref{2. eq: sP, proof of completeness}.
\end{proof}

Now it is possible to prove that the vague-and-point-process topology is Polish.
By Propositions \ref{2. cor: embedding of Mdis to sP} and \ref{2. thm: sP is Polish},
it suffices to show that the set $\pointMeas^{*}(S)$ is an intersection of countably many open subsets of $\pointMeas(S)$,
and this follows from two lemmas below.
For each $n \in \NN$,
we set 
\begin{equation} \label{2. eq: dfn of sP^n}
  \pointMeas^{(n)}(S)
  \coloneqq 
  \Bigl\{
    \pi \in \pointMeas(S) 
    \mid
    \text{for some}\ r \in (n-n^{-1}, n),\ 
    \pi^{(r)} = \sum_{i \in I} \delta_{(x_{i}, w_{i})}\ \text{with}\ x_{i} \neq x_{j}\ \text{if}\ i \neq j
  \Bigr\}.
\end{equation}

\begin{lem} \label{2. lem: sP^n is open}
  The set $\pointMeas^{(n)}(S)$ is open in $\pointMeas(S)$.
\end{lem}

\begin{proof}
  Fix $\pi \in \pointMeas^{(n)}(S)$.
  Let $r \in (n-n^{-1}, n)$ be such that 
  $\pi^{(r)} = \sum_{i \in I} \delta_{(x_{i}, w_{i})}$
  with $x_{i} \neq x_{j}$ if $i \neq j$.
  Since the index set $I$ is finite,
  we may choose $\varepsilon \in (0, 1)$ 
  so that $n < \varepsilon^{-1}$, $n-n^{-1}+2\varepsilon < r$, and 
  \begin{equation}  \label{2. eq: sP^n is open}
    D_{S}(x_{i}, 2\varepsilon) \cap D_{S}(x_{j}, 2\varepsilon) = \emptyset \quad
    \text{if} \quad  
    i \neq j.
  \end{equation}
  Fix $\tilde{\pi} \in \pointMeas(S)$ with $d_{\pointMeas}^{S, \rho_{S}}(\pi, \tilde{\pi}) < \varepsilon e^{-1/\varepsilon}$.
  It suffices to show that $\tilde{\pi} \in \pointMeas^{(n)}(S)$.
  We can find $r' > \varepsilon^{-1}$ such that $d_{P}^{S \times \RNpp}(\pi^{(r')}, \tilde{\pi}^{(r')}) < \varepsilon$.
  Let $\{(y_{j}, v_{j})\}_{j \in J}$ be the atoms of $\tilde{\pi}^{(r-2\varepsilon)}$.
  By the definition of the Prohorov metric,
  we have that 
  \begin{equation}
    1 
    \leq 
    \tilde{\pi}(\{(y_{j}, v_{j})\}) 
    \leq 
    \pi \bigl( D_{S \times \RNpp}( (y_{j}, v_{j}), \varepsilon) \bigr) + \varepsilon.
  \end{equation}
  This, combined with \eqref{2. eq: sP^n is open}, 
  implies that there is exactly one atom of $\pi$ lying in $D_{S \times \RNpp}( (y_{j}, v_{j}), \varepsilon)$.
  Thus, $\tilde{\pi}(\{(y_{j}, v_{j})\}) = 1$.
  Assume that there exist $j \neq j'$ such that $y_{j} = y_{j'} = y$.
  It is then the case that 
  \begin{equation}
    2 
    = 
    \tilde{\pi}(\{(y, v_{j}), (y, v_{j'})\})
    \leq 
    \pi \bigl( D_{S \times \RNpp}((y, v_{j}), \varepsilon) \cup  D_{S \times \RNpp}((y, v_{j'}), \varepsilon) \bigr) + \varepsilon
  \end{equation}
  Since $C \coloneqq D_{S \times \RNpp}((y, v_{j}), \varepsilon) \cup  D_{S \times \RNpp}((y, v_{j'}), \varepsilon)$ is contained in $S^{(r)} \times [e^{-r}, e^{r}]$,
  there are at least two atoms of $\pi$ lying in $C$,
  which contradicts \eqref{2. eq: sP^n is open}.
  Therefore, we deduce that $\tilde{\pi} \in \pointMeas^{(n)}(S)$.
\end{proof}

\begin{lem} \label{2. lem: sP^n decreases to sP}
  It holds that $\pointMeas^{*}(S) = \bigcap_{n=1}^{\infty} \pointMeas^{(n)}(S)$.
\end{lem}

\begin{proof}
  This is immediate from the definition of $\pointMeas^{(n)}(S)$ and \eqref{2. cor eq: image of Mdis by p}.
\end{proof}

\begin{thm} [{Polishness of $\disMeasures(S)$}]\label{2. thm: the VPP topology is Polish}
  In general, 
  $d_{\disMeasures}^{S, \rho_{S}}$ is not complete.
  However,
  the vague-and-point-process topology is Polish.
\end{thm}

\begin{proof}
  At the end of Section \ref{sec: The vague-and-point-process topology},
  we checked that $d_{\disMeasures}^{S, \rho_{S}}$ is not complete in general.
  By Corollary \ref{2. cor: embedding of Mdis to sP}, Lemmas \ref{2. lem: sP^n is open} and \ref{2. lem: sP^n decreases to sP},
  we obtain that the vague-and-point-process topology on $\disMeasures(S)$ is Polish.
\end{proof}

We further study topological properties of $\pointMeas(S)$:
precompactness, the Borel $\sigma$-algebra, and tightness.

\begin{thm} [{Precompactness in $\pointMeas(S)$}] \label{2. thm: precompactness in sP}
  A non-empty subset $\{\pi_{j}\}_{j \in J}$ of $\pointMeas(S)$ is precompact in $\pointMeas(S)$ 
  if and only if the following conditions are satisfied.
  \begin{enumerate} [label = (\roman*)]
    \item \label{2. thm item: cpt in sP, vaguely compact}
      The subset $\{\pi_{j}\}_{j \in J}$ is precompact in the vague topology.
    \item \label{2. thm item: cpt in sP, W and M}
      For each $r>0$, 
      $\displaystyle \sup_{j \in J} W^{(r)}(\pi_{j}) < \infty$ 
      and $\displaystyle \lim_{\varepsilon \downarrow 0} \sup_{j \in J} M_{\varepsilon}^{(r)}(\pi_{j}) = 0$.
  \end{enumerate}
\end{thm}

\begin{proof}
  Suppose that $\{\pi_{j}\}_{j \in J}$ is precompact in $\pointMeas(S)$.
  Since the topology on $\pointMeas(S)$ is finer than the vague topology,
  we have \ref{2. thm item: cpt in sP, vaguely compact}.
  If \ref{2. thm item: cpt in sP, W and M} is not satisfied,
  then 
  for some $r>0$
  there exists a sequence $(j_{n})_{n \geq 1}$ in $J$ such that $W^{(r)}(\pi_{j_{n}}) \to \infty$
  or $M_{\varepsilon_{n}}^{(r)}(\pi_{j_{n}}) > \delta$ for some $\delta>0$ and $\varepsilon_{n} \downarrow 0$.
  However, since $(\pi_{j_{n}})_{n \geq 1}$ has a convergent subsequence,
  by Theorem \ref{2. thm: convergence in sP},
  we obtain a contradiction, which yields \ref{2. thm item: cpt in sP, W and M}.
  The converse assertion immediately follows from Theorem \ref{2. thm: convergence in sP}.
\end{proof}

To consider random element of $\pointMeas(S)$,
we first identify the Borel $\sigma$-algebra on $\pointMeas(S)$. 
In particular, we show that it coincides with the one generated from the vague topology on $\pointMeas(S)$
in Proposition \ref{2. prop: sigma-algebra on sP} below.

\begin{lem} \label{2. lem: vague implies the sP convergence for restrictions}
  Let $\pi, \pi_{1}, \pi_{2}, \ldots$ be elements in $\pointMeas(S)$ such that $\pi_{n} \to \pi$ vaguely.
  Fix $r>0$ such that $\pi$ has no atoms on the boundary of $S^{(r)} \times [e^{-r}, e^{r}]$.
  Then, $\pi_{n}^{(r)} \to \pi^{(r)}$ in $\pointMeas(S)$.
\end{lem}

\begin{proof}
  It is elementary to check that $\pi_{n}^{(r)} \to \pi^{(r)}$ vaguely 
  (and even weakly). 
  If $(x, w)$ is an atom of $\pi_{n}^{(r)}$,
  then we have that $w \in [e^{-r}, e^{r}]$. 
  Thus, by Theorem \ref{2. thm: convergence in sP}, we obtain the desired result.
\end{proof}

\begin{prop}  \label{2. prop: sP is vaguely measurable}
  The set $\pointMeas(S)$ is a Borel subset of $\Meas(S \times \RNpp)$ equipped with the vague topology.
\end{prop}

\begin{proof} 
  By \cite[Theorem 2.19]{Kallenberg_21_Foundations},
  the set $\intMeasures(S \times \RNpp)$ of integer valued Radon measures is a Borel subset of $\Meas(S)$.
  Since $\intMeasures(S \times \RNpp) \ni \pi \mapsto \MeasureMap(\pi)(S^{(r)})$ is measurable with respect to the vague topology for each $r>0$,
  we obtain the desired result.
\end{proof}

\begin{prop}  [{The Borel $\sigma$-algebra on $\pointMeas(S)$}] \label{2. prop: sigma-algebra on sP}
  The Borel $\sigma$-algebra on $\pointMeas(S)$ coincides with 
  the Borel $\sigma$-algebra generated from the vague topology.
\end{prop}

\begin{proof}
  To distinguish between the two topological spaces,
  we rewrite $\pointMeas(S)'$ for the topological space $\pointMeas(S)$ equipped with the vague topology,
  and let $\pointMeas(S)$ represent the topological space $\pointMeas(S)$ with the topology induced from $d_{\pointMeas}^{S, \rho_{S}}$
  (which we have considered so far).
  Write $\Sigma$ and $\Sigma'$ 
  for the Borel $\sigma$-algebras on $\pointMeas(S)$ and on $\pointMeas(S)'$, respectively.
  Since the topology on $\pointMeas(S)$ is finer than the vague topology,
  we have that $\Sigma \supseteq \Sigma'$.
  To show the converse relation,
  we let $f$ be a bounded continuous function on $\pointMeas(S)$.
  It suffices to show that $f$ is vaguely measurable, i.e., $\Sigma'$-measurable.
  For $\pi \in \pointMeas(S)'$ and $r>0$,
  if $\pi$ has no atoms on the boundary of $S^{(r)} \times [e^{-r}, e^{r}]$,
  then $\pi^{(r)} = \pi^{(s)}$ for all $s > 0$ sufficiently close to $r$.
  Hence,
  for each $\pi \in \pointMeas(S)'$,
  $f(\pi^{(r)})$ is continuous for all but countably many $r>0$ 
  and so we can define a map $f_{r}: \pointMeas(S)' \to \RN$ by setting 
  \begin{equation}
    f_{r}(\pi)
    \coloneqq 
    \int_{0}^{1} f(\pi^{(r+s)})\, ds.
  \end{equation}
  Suppose that $\pi_{n} \to \pi$ in the vague topology.
  Lemma \ref{2. lem: vague implies the sP convergence for restrictions} and the continuity of $f$ 
  imply that $f(\pi_{n}^{(s)}) \to f(\pi^{(s)})$ for all but countably many $s>0$. 
  Thus, the dominated convergence theorem yields that $f_{r}$ is continuous on $\pointMeas(S)'$.
  In particular, $f_{r}$ is $\Sigma'$-measurable.
  By Lemma \ref{2. lem: convergence of pi^r to pi}, the continuity of $f$ and the dominated convergence theorem,
  we have that $f_{r}(\pi) \to f(\pi)$ as $r \to \infty$ for each $\pi \in \pointMeas(S)'$.
  Hence, $f$ is $\Sigma'$-measurable.
\end{proof}

Now, we provide a tightness criterion for random elements of $\pointMeas(S)$.

\begin{thm} [{Tightness in $\pointMeas(S)$}] \label{2. thm: tightness in sP}
  Let $(\pi_{n})_{n \geq 1}$ be a sequence of random elements of $\pointMeas(S)$.
  Write $P_{n}$ for the underlying probability measure of $\pi_{n}$.
  Then, the sequence $(\pi_{n})_{n \geq 1}$ is tight if and only if the following conditions are satisfied.
  \begin{enumerate} [label = (\roman*)]
    \item \label{2. thm item: tightness in sP, tight wrt vague topology}
      The sequence $(\pi_{n})_{n \geq 1}$ is tight with respect to the vague topology.
    \item \label{2. thm item: tightness, tightness of W}
      For each $r>0$, 
      $\displaystyle \lim_{l \to \infty} \limsup_{n \to \infty} P_{n}(W^{(r)}(\pi_{n}) > l) = 0$.
    \item \label{2. thm item: tightness, tightness of M}
      For each $r, \delta>0$, 
      $\displaystyle \lim_{\varepsilon \downarrow 0} \limsup_{n \to \infty} P_{n}(M_{\varepsilon}^{(r)}(\pi_{n}) > \delta) = 0$.
  \end{enumerate}
\end{thm}

\begin{proof}
  Assume that $(\pi_{n})_{n \geq 1}$ is tight.
  The condition \ref{2. thm item: tightness in sP, tight wrt vague topology} is obvious.
  Fix $r, \delta, \eta >0$.
  By tightness, 
  there exists a compact subset $\mathcal{A}$ of $\pointMeas(S)$ 
  such that $P_{n}(\pi_{n} \notin \mathcal{A}) < \eta$ for all $n \geq 1$.
  Theorem \ref{2. thm: precompactness in sP} yields that 
  $l \coloneqq \sup_{\pi \in \mathcal{A}} W^{(r)}(\pi) < \infty$ 
  and 
  $\sup_{\pi \in \mathcal{A}} M_{\varepsilon_{0}}^{(r)}(\pi) < \delta$ 
  for some $\varepsilon_{0} > 0$.
  We then deduce that $\sup_{n \geq 1} P_{n}(W^{(r)}(\pi_{n}) > l) < \eta$
  and $\sup_{n \geq 1} P_{n}(M_{\varepsilon_{0}}^{(r)}(\pi_{n}) > \delta) < \eta$.
  Using the monotonicity of $M_{\cdot}^{(r)}(\pi_{n})$ stated in Lemma \ref{2. lem: properties of edge-weight M and heavy weigh W}\ref{2. lem item: monotonicity of small mass function},
  we obtain \ref{2. thm item: tightness, tightness of W} and \ref{2. thm item: tightness, tightness of M}.
  Conversely, assume that \ref{2. thm item: tightness in sP, tight wrt vague topology},
  \ref{2. thm item: tightness, tightness of W} and \ref{2. thm item: tightness, tightness of M} are satisfied.
  Fix $\varepsilon>0$.
  By \ref{2. thm item: tightness in sP, tight wrt vague topology},
  there exists a subset $\mathcal{A}$ of $\pointMeas(S)$ such that $\mathcal{A}$ is vaguely compact and 
  $P_{n}(\pi_{n} \notin \mathcal{A}) < \varepsilon$ for all $n \geq 1$.
  By Lemma \ref{2. lem: properties of edge-weight M and heavy weigh W}\ref{2. lem item: monotonicity of small mass function} and \ref{2. lem item: max weight is finite},
  we note that ``$\limsup_{n \to \infty}$'' in the statement of \ref{2. thm item: tightness, tightness of W} and \ref{2. thm item: tightness, tightness of M}
  can be replaced by ``$\sup_{n \geq 1}$''.
  Then,
  for each $k, m \in \NN$,
  we can find $l_{k} > 0$ and $\varepsilon_{k, m}>0$ such that 
  $\sup_{n \geq 1} P_{n}(W^{(k)}(\pi_{n}) > l_{k}) < 2^{-k} \varepsilon$ and 
  $\sup_{n \geq 1} P_{n}(M_{\varepsilon_{k,m}}^{(k)}(\pi_{n}) > m^{-1}) < 2^{-k-m} \varepsilon$.
  Define $\mathcal{A}'$ be a collection of $\pi \in \pointMeas(S)$ such that 
  $\pi \in \mathcal{A}$ and, for all $k, m \in \NN$,
  $W^{(k)}(\pi) \leq l_{k}$ and $ M_{\varepsilon_{k,m}}^{(k)}(\pi) \leq m^{-1}$.
  By Theorem \ref{2. thm: precompactness in sP},
  $\mathcal{A}'$ is precompact in $\pointMeas(S)$.
  Moreover, we have that 
  \begin{equation}
    P_{n}(\pi_{n} \notin \mathcal{A}')
    \leq 
    P_{n}(\pi_{n} \notin \mathcal{A}) 
    + 
    \sum_{k \in \NN} P_{n}(W^{(k)}(\pi_{n}) > l_{k})
    + 
    \sum_{k, m \in \NN} P_{n}(M_{\varepsilon_{k,m}^{(k)}}(\pi_{n}) > m^{-1})
    \leq 
    3\varepsilon.
  \end{equation}
  Therefore, $(\pi_{n})_{n \geq 1}$ is tight.
\end{proof}

Distributional convergence of random measures in the vague topology is well-studied in \cite[Section 4.2]{Kallenberg_17_Random}.
In the following result,
we provide a useful condition for strengthening distributional convergence of random measures in the vague topology 
to distributional convergence in $\pointMeas(S)$.

\begin{cor} \label{2. cor: convergence in distribution in sP}
  Let $\pi_{1}, \pi_{2}, \ldots$ be random elements of $\pointMeas(S)$.
  If $\pi_{n} \xrightarrow{\mathrm{d}} \pi$ in the vague topology 
  for some random element $\pi$ of $\Meas(S \times \RNpp)$
  and 
  the conditions \ref{2. thm item: tightness, tightness of W} and \ref{2. thm item: tightness, tightness of M} 
  of Theorem \ref{2. thm: tightness in sP} are satisfied,
  then $\pi$ is a random element of $\pointMeas(S)$ and $\pi_{n} \xrightarrow{\mathrm{d}} \pi$ in $\pointMeas(S)$.
\end{cor}

\begin{proof}
  By Theorem \ref{2. thm: tightness in sP},
  $(\pi_{n})_{n \geq 1}$ is tight in $\pointMeas(S)$.
  Let $(\pi_{n_{k}})_{k \geq 1}$ be a sequence such that $\pi_{n_{k}} \xrightarrow{\mathrm{d}} \pi'$ in $\pointMeas(S)$
  for some random element $\pi'$ of $\pointMeas(S)$.
  Then, by Theorem \ref{2. thm: convergence in sP},
  we have that $\pi_{n_{k}} \xrightarrow{\mathrm{d}} \pi$ in the vague topology.
  Hence,
  $\pi$ and $\pi'$ give the same probability distribution on $\Meas(S \times \RNpp)$ equipped with the vague topology.
  It then follows from Propositions \ref{2. prop: sP is vaguely measurable} and \ref{2. prop: sigma-algebra on sP} 
  that $\pi \stackrel{\mathrm{d}}{=} \pi'$ as random elements of $\pointMeas(S)$,
  which implies that $\pi_{n_{k}} \xrightarrow{\mathrm{d}} \pi$ in $\pointMeas(S)$.
  This completes the proof.
\end{proof}


\section{Gromov-Hausdorff-type topologies} \label{sec: GH-type topologies}

Gromov-Hausdorff-type topologies are topologies on the set of (equivalence classes) of metric spaces 
equipped with additional objects such as points, measures, and laws of stochastic processes.
Recently, the author established a general framework of metrization of these topologies \cite{Noda_pre_Metrization},
and we follow it in this article.
(NB.\ there is a related work by Khezeli \cite{Khezeli_23_A_unified}.)
In Section \ref{sec: The local Hausdorff topology},
we define the local Hausdorff topology, which is a modification of the Hausdorff topology for non-compact subsets.
We then introduce the framework presented in \cite{Noda_pre_Metrization} in Section \ref{sec: The GH-type topologies generated by functors}.
In Section \ref{sec: functors used in this paper},
we collect the Gromov-Hausdorff-type topologies used in this article.


\subsection{The local Hausdorff topology} \label{sec: The local Hausdorff topology}

We define the local Hausdorff topology,
which is a modification of the Hausdorff topology for non-compact subsets.

Fix a boundedly-compact metric space  $(S, d^{S}, \rho_{S})$.
Recall that, for a subset $A \subseteq S$,
the (closed) $\varepsilon$-neighborhood of $A$ in $(S, d^{S})$ is given by 
\begin{equation}
  \neighb{A}{\varepsilon} 
  \coloneqq
  \{
    x \in S \mid \exists y \in A\ \text{such that}\ d^{S}(x,y) \leq \varepsilon
  \}.
\end{equation} 
Let $\closed(S)$ be the set of closed subsets in $S$
and $\compact(S)$ be the set of compact subsets in $S$
(containing the empty set).
The \textit{Hausdorff metric} $\HausdMet{S}$ on $\compact(S)$ is defined by setting
\begin{equation}
  \HausdMet{S}(A, B)
  \coloneqq
  \inf\{
    \varepsilon \geq 0 \mid A \subseteq \neighb{B}{\varepsilon}, B \subseteq \neighb{A}{\varepsilon}
  \},
\end{equation}
where the infimum over the empty set is defined to be $\infty$.
It is known that $\HausdMet{S}$ is indeed a metric (allowed to take the value $\infty$ due to the empty set) on $\compact(S)$
(see \cite[Section 7.3.1]{Burago_Burago_Ivanov_01_A_course}),
and the induced topology is called the \textit{Hausdorff topology}.
To deal with non-compact sets,
we introduce a metric on $\closed(S)$.

\begin{dfn} [{The local Hausdorff metric $\lHausdMet{S}{\rho_{S}}$}]
  For $A \in \closed(S)$ and $r>0$,
  we write 
  \begin{equation}  \label{eq: restriction of a set to a closed ball}
    A^{(r)} 
    \coloneqq 
    \closure(A \cap B_{d^{S}}(\rho, r)), 
  \end{equation}
  where we recall that $\closure(\cdot)$ denotes the closure of a set.
  We then define, for $A,B \in \closed(S)$,
  \begin{equation}
    \lHausdMet{S}{\rho_{S}} (A, B) 
    \coloneqq 
    \int_{0}^{\infty} e^{-r} \bigl( 1 \wedge \HausdMet{S}(A^{(r)}, B^{(r)}) \bigr)\, dr.
  \end{equation}
\end{dfn}

The function $\lHausdMet{S}{\rho_{S}}$ is a metric on $\closed(S)$
and a natural extension of the Hausdorff metric for non-compact sets.
The following is a basic property of $\lHausdMet{S}{\rho_{S}}$.

\begin{thm} [{\cite[Theorems 2.28, 2.26, and 2.30]{Noda_pre_Metrization}}]
  The function $\lHausdMet{S}{\rho_{S}}$ is a metric 
  on $\closed(S)$
  and the metric space $(\closed(S), \lHausdMet{S}{\rho_{S}})$ is compact.
  A sequence $(A_{n})_{n \geq 1}$ converges to $A$ with respect to 
  the local Hausdorff metric $\lHausdMet{S}{\rho_{S}}$
  if and only if 
  $A_{n}^{(r)}$ converges to $A^{(r)}$
  in the Hausdorff topology
  for all but countably many $r>0$.
  Moreover,
  the topology on $\closed(S)$ induced from $\lHausdMet{S}{\rho_{S}}$ 
  is independent of the root $\rho_{S}$.
\end{thm}

\begin{dfn} [{The local Hausdorff topology}]
  We call $\lHausdMet{S}{\rho_{S}}$ the \textit{local Hausdorff metric (with the root $\rho_{S}$)}
  and the topology induced from the local Hausdorff metric the \textit{local Hausdorff topology}. 
\end{dfn}


\subsection{The Gromov-Hausdorff-type topologies generated by functors} \label{sec: The GH-type topologies generated by functors}
In this subsection, 
we recall the general theory on metrization of Gromov-Hausdorff-type topologies developed in \cite{Noda_pre_Metrization}.

We say that two rooted boundedly-compact metric spaces $(S_{i}, d^{S_{i}}, \rho_{S_{i}}),\, i=1,2$ are \textit{rooted isometric} 
if and only if there exists a root-preserving isometry $f: S_{1} \to S_{2}$.
Note that $f$ being an isometry means that $f$ is a surjective distance-preserving map (and hence $f$ is bijective).
We denote by $\rbcM$ the collection of rooted isometric equivalence classes of rooted boundedly-compact metric spaces 
and by $\rbcM_{c}$ the sub-collection of $(S, d^{S}, \rho_{S}) \in \rbcM$ such that 
$(S, d^{S})$ is compact.
We introduce functors that determine structures to be added to metric spaces.

\begin{dfn} [{Functor}] \label{3. dfn: functor}
  We call $\tau$ a \textit{functor} on $\rbcM_{c}$ if it satisfies the following.
  \begin{enumerate} [label = (\roman*)]
    \item 
      For every $ (S, d^{S}, \rho_{S}) \in \rbcM_{c}$,
      one has a metric space $(\tau(S, d^{S}, \rho_{S}), d_{\tau}^{S, \rho_{S}})$
      where $\tau(S, d^{S}, \rho_{S})$ is a set and $d_{\tau}^{S, \rho_{S}}$ is a metric on it.
      We simply write $\tau(S) \coloneqq \tau(S, d^{S}, \rho_{S})$.
    \item 
      For every $(S_{i}, d^{S_{i}}, \rho_{S_{i}}),\, i=1,2$ 
      and root-and-distance-preserving map $f: S_{1} \to S_{2}$, 
      one has a distance-preserving map $\tau_{f}: \tau(S_{1}) \to \tau(S_{2})$.
    \item 
      For any two root-and-distance-preserving maps $f :S_{1} \to S_{2},\, g : S_{2} \to S_{3}$,
      it holds that $\tau_{g \circ f} = \tau_{g} \circ \tau_{f}$.
    \item 
      For any $(S, d^{S}, \rho_{S}) \in \rbcM_{c}$, 
      it holds that $\tau_{\id_{S}} = \id_{\tau(S)}$.
  \end{enumerate}
  Similarly, we call $\tau$ a functor on $\rbcM$ if it satisfies the same conditions 
  not for $\rbcM_{c}$ but for $\rbcM$.
\end{dfn}

For a functor $\tau$ on $\rbcM$ and $(S_{i}, d^{S_{i}}, \rho_{S_{i}}, a_{S_{i}})$, $i=1,2$
such that $(S_{i}, d^{S_{i}}, \rho_{S_{i}})$ is a rooted boundedly-compact metric space and $a_{S_{i}} \in \tau(S_{i})$,
we say that $(S_{1}, d^{S_{1}}, \rho_{S_{1}}, a_{S_{1}})$ and $(S_{2}, d^{S_{2}}, \rho_{S_{2}}, a_{S_{2}})$ are $\tau$-\textit{equivalent} 
if and only if there exists a root-preserving isometry $f: S_{1} \to S_{2}$ 
such that $\tau_{f}(a_{S_{1}}) = a_{S_{2}}$.
We denote the collection of $\tau$-equivalence classes by $\rbcM(\tau)$.
We similarly define $\rbcM_{c}(\tau)$.

\begin{rem} \label{rem: how to regard G as a set}
  From the rigorous point of view of set theory,
  none of $\rbcM$, $\rbcM_{c}(\tau)$, or $\rbcM(\tau)$ is a set.
  However, it is possible to think of each as set.
  For example,
  regarding $\rbcM$,
  one can construct a legitimate set $\mathscr{M}$ of rooted boundedly-compact spaces 
  such that any rooted-and-measured boundedly-compact space is rooted isometric to an element of $\mathscr{M}$.
  For details,
  refer to \cite[Section 3.1]{Noda_pre_Metrization}.
  Therefore,
  in this article,
  we will proceed with the discussion by treating $\rbcM$, $\rbcM_{c}(\tau)$, and $\rbcM(\tau)$ as sets 
  to avoid complications.
\end{rem}

The metrics on $\rbcM_{c}(\tau)$ and $\rbcM(\tau)$ are given by generalizing the Gromov-Hausdorff metric \cite{Burago_Burago_Ivanov_01_A_course}.

\begin{dfn} [{The metrics $d_{\rbcM_{c}}^{\tau}$ and $d_{\rbcM}^{\tau}$}]  \label{3. dfn: metric on frakMtau}
  Let $\tau$ be a functor on $\rbcM$.
  For $\cS=(S,d^{S}, \rho_{S}, a_{S}), \cT=(T, d^{T}, \rho_{T}, a_{T}) \in \rbcM(\tau)$,
  we define
  \begin{equation} \label{3. eq: dfn of functor metric}
    d_{\rbcM}^{\tau} (\cS, \cT)
    \coloneqq
    \inf_{f, g, M}
    \left\{
      \lHausdMet{M}{\rho_{M}} \bigl( f(S), g(T) \bigr) 
      \vee  
      d_{\tau}^{M,\rho_{M}} \bigl( \tau_{f}(a_{S}), \tau_{g}(a_{T}) \bigr)  
    \right\},
  \end{equation}
  where the infimum is taken 
  over all $(M, d^{M}, \rho_{M}) \in \rbcM$ 
  and root-and-distance-preserving maps $f : S \to M$ and $g : T \to M$.
  For a functor $\tau$ on $\rbcM_{c}$,
  we define $d_{\rbcM_{c}}^{\tau}$ similarly.
\end{dfn}

\begin{rem} \label{3. rem: Gromov-Hausdorff topology}
  In the above definition,
  if one consider no functor and simply define 
  \begin{equation}
    d_{\rbcM} (\cS, \cT)
    \coloneqq
    \inf_{f, g, M} \lHausdMet{M}{\rho_{M}} \bigl( f(S), g(T) \bigr),
  \end{equation}
  then $d_{\rbcM}$ becomes a metric on $\rbcM$ and the induced topology is called the \textit{local Gromov-Hausdorff topology}
  (see \cite[text]{Noda_pre_Metrization}).
  Similarly, if one define $d_{\rbcM_{c}}$ as an analogue of $d_{\rbcM}$,
  then it is a metric on $\rbcM_{c}$ and the induced topology is the \textit{(pointed) Gromov-Hausdorff topology}
  (cf.\ \cite[Section 7.3]{Burago_Burago_Ivanov_01_A_course}).
\end{rem}

Henceforth,
we only consider a functor $\tau$ on $\rbcM$.
However,
we note that every result is modified for a functor on $\rbcM_{c}$ in the obvious way.
(Indeed, the proofs become much simpler.)

To ensure that $d_{\rbcM}^{\tau}$ is a metric,
we assume the following continuity condition on $\tau$.

\begin{assum} \label{3. assum: pointwise continuity}
  Fix $(S, d^{S}, \rho_{S}), (T, d^{T}, \rho_{T}) \in \rbcM$ arbitrarily.
  Let $f_{n} : S \to T,\, n \in \mathbb{N} \cup \{ \infty \}$ be root-and-distance-preserving maps.
  If $f_{n} \to f_{\infty}$ in the compact-convergence topology, i.e., uniformly on any compact subsets,
  then $\tau_{f_{n}}(a) \to \tau_{f_{\infty}}(a)$ in $\tau(T)$ for all $a \in \tau(S)$.
\end{assum}

\begin{dfn} [{Continuous functor}]
  We say that a functor $\tau$ is \textit{continuous}
  if $\tau$ satisfies Assumption \ref{3. assum: pointwise continuity}.
\end{dfn}

\begin{thm} [{\cite[Theorems 3.23 and 3.24]{Noda_pre_Metrization}}] \label{3. thm: convergence in GH topology}
  If $\tau$ is continuous,
  then the function $d_{\rbcM}^{\tau}$ is a metric on $\rbcM(\tau)$.
  A sequence $(S_{n}, d^{S_{n}}, \rho_{S_{n}}, a_{S_{n}}),\, n \in \NN$ in $\rbcM(\tau)$ converges to $(S, d^{S}, \rho_{S}, a_{S})$ with respect to $d_{\rbcM}^{\tau}$
  if and only if 
  there exist $(M, d^{M}, \rho_{M}) \in \rbcM$ 
  and root-and-distance-preserving maps $f_{n} : S_{n} \to M$ and $f: S \to M$
  such that
  \begin{equation}  \label{3. thm eq: embedding to a common space for convergent sequence}
    \lHausdMet{M}{\rho_{M}} (f_{n}(S_{n}), f(S)) \to 0,
    \quad
    d_{\tau}^{M, \rho_{M}} (\tau_{f_{n}}(a_{S_{n}}), \tau_{f}(a_{S})) \to 0.
  \end{equation}
\end{thm}

For completeness and separability,
we consider additional conditions.
For $\cS=(S, d^{S}, \rho_{S}),\, \cT=(T, d^{T}, \rho_{T}) \in \rbcM$,
we write $\cS \preceq \cT$ 
if and only if 
$S \subseteq T$, $d^{T}|_{S \times S} = d^{S}$, and $\rho_{S} = \rho_{T}$.

\begin{assum} \label{3. assum: completeness and separability}
  Let $\cS_{n} = (S_{n}, d^{S_{n}}, \rho_{S_{n}}),\, n \in \mathbb{N}\cup\{\infty\}$ 
  and $\Meas=(M, d^{M}, \rho_{M})$ be elements of $\rbcM$
  such that $\cS_{n} \preceq \Meas$ for all $n \in \mathbb{N} \cup \{\infty\}$
  and $S_{n}$ converges to $S_{\infty}$ in the local Hausdorff topology in $M$.
  We write $\iota_{n}: S_{n} \to M$ for the inclusion map.
  \begin{enumerate} [label=(\roman*), series=functor condition]
    \item \label{3. assum item: completeness condition}
      If $b \in \tau(M)$ and $a_{n} \in \tau(S_{n})$ are 
      such that $\tau_{\iota_{n}}(a_{S_{n}}) \to b$ in $\tau(M)$,
      then there exists $a \in \tau(S_{\infty})$ satisfying $b= \tau_{\iota_{\infty}}(a)$.
    \item  \label{3. assum item: separability condition}
      For every $a \in \tau(S_{\infty})$,
      there exists a sequence $a_{n} \in \tau(S_{n})$ 
      such that $\tau_{\iota_{n}}(a_{n}) \to \tau_{\iota_{\infty}}(a)$ in $\tau(M)$. 
  \end{enumerate}
\end{assum}

\begin{dfn} [{Complete, separable functor}]
  A functor $\tau$ is said to be \textit{complete (resp.\ separable)} 
  if and only if it satisfies 
  Assumption \ref{3. assum: completeness and separability}\ref{3. assum item: completeness condition}
  (resp.\ \ref{3. assum item: separability condition})
  and, for each $(S, d^{S}, \rho_{S}) \in \rbcM$,
  the metric space $(\tau(S), d_{\tau}^{S, \rho_{S}})$ is complete (resp.\ separable).
\end{dfn}

\begin{thm} [{\cite[Corollary 3.30]{Noda_pre_Metrization}}] \label{3. thm: complete and separable functor}
  If $\tau$ is complete, separable and continuous,
  then the function $d_{\rbcM}^{\tau}$ is a complete, separable metric on $\rbcM(\tau)$.
\end{thm}

An effective way to show the Polishness of a topological space 
for which a complete metric is difficult to find 
is to show that the space is a $G_{\delta}$-subset of a larger Polish space, 
as we did in Section \ref{sec: a space including the space of discrete measures}.
With this background,
we introduce the notion of topological subfunctors and Polish functors.

\begin{dfn} [{Topological subfunctor}] \label{3. dfn: topological subfunctor}
  Let $\tilde{\tau}$ and $\tau$ be functors.
  We say that $\tau$ is a \textit{topological subfunctor} of $\tilde{\tau}$ 
  if and only if the following conditions are satisfied.
  \begin{enumerate} [label = (T\arabic*)]
    \item \label{3. dfn item: topological subfunctor, topological embedding}
      For every $(S, d^{S}, \rho_{S}) \in \rbcM$,
      there exists a topological embedding of $\tau(S)$ into $\tilde{\tau}(S)$,
      that is,
      there exists a homeomorphism from $\tau(S)$ to a subset of $\tilde{\tau}(S)$.
      Using this map, we always regard $\tau(S)$ as a subspace of $\tilde{\tau}(S)$.
    \item  \label{3. dfn item: topological subfunctor, commutative relation}
      For every $(S_{i}, d^{S_{i}}, \rho_{S_{i}}),\, i=1,2$ 
      and root-and-distance-preserving map $f: S_{1} \to S_{2}$, 
      it holds that $\tau_{f} = \tilde{\tau}_{f}|_{\tau(S_{1})}$.
  \end{enumerate}
\end{dfn}

\begin{dfn} [{Polish functor}] \label{3. dfn: Polish functor}
  We say that a functor $\tau$ is \textit{Polish} if there exist a functor $\tilde{\tau}$
  and, for each $(S, d^{S}, \rho_{S}) \in \rbcM$,
  a sequence $(\tilde{\tau}_{k}(S))_{k=1}^{\infty}$ of open subsets in $\tilde{\tau}(S)$
  satisfying the following conditions.
  \begin{enumerate} [label = (P\arabic*), leftmargin = *]
    \item \label{3. dfn item: Polish functor, large functor is separable, complete and continuous}
      The functor $\tilde{\tau}$ is complete, separable and continuous. 
    \item \label{3. dfn item: Polish functor, the functor is a restriction}
      The functor $\tau$ is a topological subfunctor of $\tilde{\tau}$.
    \item \label{3. dfn item: Polish functor, sequence of open sets}
      For every $(S, d^{S}, \rho_{S}) \in \rbcM$,
      it holds that $\tau(S) = \bigcap_{k \geq 1} \tilde{\tau}_{k}(S)$.
    \item \label{3. dfn item: Polish functor, functorial subset property}
      Let $f: S_{1} \to S_{2}$ be a root-and-distance-preserving map
      between $(S_{i}, d^{S_{i}}, \rho_{S_{i}}) \in \rbcM,\, i=1,2$.
      Then, $\tilde{\tau}_{f}^{-1}(\tilde{\tau}_{k}(S_{2})) = \tilde{\tau}_{k}(S_{1})$
      for each $k \geq 1$. 
      In particular,
      $\tilde{\tau}_{f}^{-1}(\tau(S_{2})) = \tau(S_{1})$.
  \end{enumerate}
  We call $(\tilde{\tau}, (\tilde{\tau}_{k})_{k \geq 1})$ a \textit{Polish system} of $\tau$.
\end{dfn}

\begin{rem}
  It is easy to check that every complete, separable and continuous functor is a Polish functor.
\end{rem}

\begin{thm} [{\cite[Theorem 3.36]{Noda_pre_Metrization}}] \label{3. thm: Polish functor}
  If $\tau$ is a Polish functor,
  then the topology on $\rbcM(\tau)$ induced from $d_{\rbcM}^{\tau}$ is Polish.
  (NB.\ The metric $d_{\rbcM}^{\tau}$ is not necessarily a complete metric.)
\end{thm}

In this framework,
it is fairly easy to consider multiple objects.

\begin{dfn} [{The product functor}]
  Fix $N \in \mathbb{N}$.
  Let $(\tau^{(k)})_{k=1}^{N}$ be 
  a sequence of functors.
  The \textit{product functor} $\tau = \prod_{k=1}^{N} \tau^{(k)}$ is defined as follows:
  \begin{enumerate}
    \item
      For every $(S, d^{S}, \rho_{S}) \in \rbcM$, 
      we set $\tau(S) \coloneqq \prod_{k=1}^{N} \tau^{(k)}(S)$.
      We equip $\tau(S)$ with the max product metric (cf.\ \eqref{2. eq: max product metric}).
    \item 
      For every $(S_{i}, d^{S_{i}}, \rho_{S_{i}}),\, i=1,2$ 
      and root-and-distance-preserving map $f: S_{1} \to S_{2}$, 
      we set $\tau_{f} \coloneqq \prod_{k=1}^{N} \tau_{f}^{(k)}$, that is, 
      $\tau_{f}: \tau(S_{1}) \to \tau(S_{2})$ is a distance-preserving map given by 
      \begin{equation}
        \tau_{f}\bigl( (a_{k})_{k=1}^{N} \bigr) \coloneqq \bigl( \tau^{(k)}_{f}(a_{k}) \bigr)_{k=1}^{N}.
      \end{equation}
  \end{enumerate}
\end{dfn}

\begin{prop}  [{\cite[Proposition 3.38]{Noda_pre_Metrization}}]
  Fix $N \in \mathbb{N}$
  Let $(\tau^{(k)})_{k=1}^{N}$ be a sequence of functors.
  If each $\tau^{(k)}$ is continuous
  (resp.\ complete, separable, Polish),
  then so is the product functor $\prod_{k=1}^{N} \tau^{(k)}$.
\end{prop}

At the end of this section,
we introduce a functor that deals with the law of additional objects.

\begin{dfn} [{Probability functor}]
  Given a functor $\tau$,
  we define the probability functor $\sigma = \probFunct$ as follows.
  \begin{itemize}
    \item 
      For $(S, d^{S}, \rho_{S}) \in \rbcM$, 
      set $\sigma(S) \coloneqq \mathcal{P}(\tau(S))$ to be the collection of probability measures on $\tau(S)$
      and $d_{\sigma}^{S, \rho_{S}} \coloneqq d_{P}^{\tau(S)}$
      to be the Prohorov metric on $\mathcal{P}(\tau(S))$
      defined by the metric $d_{\tau}^{S, \rho_{S}}$ on $\tau(S)$. 
    \item 
      For each $(S_{i}, d^{S_{i}}, \rho_{S_{i}}) \in \rbcM,\, i=1,2$
      and root-and-distance-preserving map $f: S_{1} \to S_{2}$, 
      set $\sigma_{f}(P) \coloneqq P \circ \tau_{f}^{-1}$,
      that is, $\sigma_{f}(P)$ is the pushforward measure of $P$ by 
      the distance-preserving map $\tau_{f} : \tau(S_{1}) \to \tau(S_{2})$.
  \end{itemize}
\end{dfn}

\begin{thm} [{\cite[Theorem 4.32]{Noda_pre_Metrization}}]
  If $\tau$ is Polish,
  then so is $\probFunct(\tau)$.
\end{thm}


\subsection{Functors used in the present paper} \label{sec: functors used in this paper}

Thanks to the theory introduced in Section \ref{sec: The GH-type topologies generated by functors},
we can easily handle various Gromov-Hausdorff-type topologies by defining corresponding functors.
In this subsection,
we provide functors used for Gromov-Hausdorff-type topologies in our discussions.

\textbf{The functor for a fixed space $\fixedFunct{\Xi}$.}
Let $(\Xi, d^{\Xi})$ be a complete, separable metric space.
We then define a functor $\fixedFunct{\Xi}$ on $\rbcM$ as follows.
\begin{itemize}
  \item 
    For $(S, d^{S}, \rho_{S}) \in \rbcM$, 
    set $\fixedFunct{\Xi}(S) \coloneqq \Xi$ and $d_{\fixedFunct{\Xi}}^{S, \rho_{S}} \coloneqq d^{\Xi}$. 
  \item 
    For each $(S_{i}, d^{S_{i}}, \rho_{S_{i}}) \in \rbcM,\, i=1,2$
    and root-and-distance-preserving map $f: S_{1} \to S_{2}$, 
    set $\fixedFunct{\Xi}_{f} \coloneqq \id_{\Xi}$.
\end{itemize}
The functor $\fixedFunct{\Xi}$ is complete, separable and continuous (see \cite[Proposition 4.1]{Noda_pre_Metrization}).
Similarly, we define a functor on $\rbcM_{c}$,
which we denote by the same symbol $\fixedFunct{\Xi}$.

\textbf{The functor for points $\PointFunct$.}
We define a functor $\PointFunct$ on $\rbcM$ as follows.
\begin{itemize}
  \item 
    For $(S, d^{S}, \rho_{S}) \in \rbcM$, 
    set $\PointFunct(S) \coloneqq S$ and $d_{\PointFunct}^{S, \rho_{S}} \coloneqq d^{S}$. 
  \item 
    For each $(S_{i}, d^{S_{i}}, \rho_{S_{i}}) \in \rbcM,\, i=1,2$
    and root-and-distance-preserving map $f: S_{1} \to S_{2}$, 
    set $\PointFunct_{f}(x) \coloneqq f(x)$.
\end{itemize}
The functor $\PointFunct$ is complete, separable and continuous (see \cite[Proposition 4.4]{Noda_pre_Metrization}).
We write $\nPointFunct{n}$ for the $n$-product functor of $\PointFunct$.
Similarly, we define functors on $\rbcM_{c}$,
which we denote by the same symbol.

\textbf{The functor for measures $\finMeasFunct$ and $\MeasFunct$.} 
Define a functor $\MeasFunct$ on $\rbcM$ as follows.
\begin{itemize}
  \item 
    For $(S, d^{S}, \rho_{S}) \in \rbcM$, 
    set $\MeasFunct(S) \coloneqq \Meas(S)$ and $d_{\MeasFunct}^{S, \rho_{S}} \coloneqq d_{V}^{S, \rho_{S}}$. 
  \item 
    For each $(S_{i}, d^{S_{i}}, \rho_{S_{i}}) \in \rbcM,\, i=1,2$
    and root-and-distance-preserving map $f: S_{1} \to S_{2}$, 
    set $\MeasFunct_{f}(\mu) \coloneqq \mu \circ f^{-1}$.
\end{itemize}
Also, define a functor $\finMeasFunct$ on $\rbcM_{c}$ as follows.
\begin{itemize}
  \item 
    For $(S, d^{S}, \rho_{S}) \in \rbcM_{c}$, 
    set $\finMeasFunct(S) \coloneqq \finMeas(S)$ and $d_{\finMeasFunct}^{S, \rho_{S}} \coloneqq d_{P}^{S}$.
  \item 
    For each $(S_{i}, d^{S_{i}}, \rho_{S_{i}}) \in \rbcM_{c},\, i=1,2$
    and root-and-distance-preserving map $f: S_{1} \to S_{2}$, 
    set $\finMeasFunct_{f}(\mu) \coloneqq \mu \circ f^{-1}$.
\end{itemize}
Both functors are complete, separable and continuous (cf.\ \cite[Proposition 4.11]{Noda_pre_Metrization}).

We call the topology on $\GHPspace \coloneqq \rbcM_{c}(\finMeasFunct)$ 
the (pointed) Gromov-Hausdorff-Prohorov topology,
which was firstly introduced in \cite{Abraham_Delmas_Hoscheit_13_A_note}.
The topology on $\GHVspace \coloneqq \rbcM(\MeasFunct)$ is an extension of the Gromov-Hausdorff-Prohorov topology,
which we call the \textit{local Gromov-Hausdorff-vague topology}.
It is a consequence of Theorem \ref{3. thm: complete and separable functor} that 
the metrics on $\GHPspace$ and $\GHVspace$ are complete and separable.

\textbf{The functor for measures on marked spaces $\markedMeasFunct{\Xi}$ and $\markedfinMeasFunct{\Xi}$.} 
Let $(\Xi, d^{\Xi})$ be a boundedly-compact metric space.
Fix an element $\rho_{\Xi} \in \Xi$.
Define a functor $\markedMeasFunct{\Xi}$ on $\rbcM$ as follows.
\begin{itemize}
  \item 
    For $(S, d^{S}, \rho_{S}) \in \rbcM$, 
    set $\markedMeasFunct{\Xi}(S) \coloneqq \Meas(S \times \Xi)$ 
    and $d_{\markedMeasFunct{\Xi}}^{S, \rho_{S}} \coloneqq d_{V}^{S \times \Xi, (\rho_{S}, \rho_{\Xi})}$. 
  \item 
    For each $(S_{i}, d^{S_{i}}, \rho_{S_{i}}) \in \rbcM,\, i=1,2$
    and root-and-distance-preserving map $f: S_{1} \to S_{2}$, 
    set $\markedMeasFunct{\Xi}_{f}(\mu) \coloneqq \mu \circ (f \times \id_{\Xi})^{-1}$.
\end{itemize}
It is easy to check that 
the functor $\markedMeasFunct{\Xi}$ is complete, separable and continuous
by following the proof of \cite[Theorem 4.31]{Noda_pre_Metrization}.
Similarly, we define a functor $\markedfinMeasFunct{\Xi}$ on $\rbcM_{c}$ as follows.
\begin{itemize}
  \item 
    For $(S, d^{S}, \rho_{S}) \in \rbcM_{c}$, 
    set $\markedfinMeasFunct{\Xi}(S) \coloneqq \Meas_{\mathrm{fin}}(S \times \Xi)$ 
    and $d_{\markedfinMeasFunct{\Xi}}^{S, \rho_{S}} \coloneqq d_{P}^{S \times \Xi}$. 
  \item 
    For each $(S_{i}, d^{S_{i}}, \rho_{S_{i}}) \in \rbcM_{c},\, i=1,2$
    and root-and-distance-preserving map $f: S_{1} \to S_{2}$, 
    set $\markedfinMeasFunct{\Xi}_{f}(\mu) \coloneqq \mu \circ (f \times \id_{\Xi})^{-1}$.
\end{itemize}

\textbf{The functor for discrete measures $\disMeasFunct$.} 
We define $\disMeasFunct$ on $\rbcM$ as follows.
\begin{itemize}
  \item 
    For $(S, d^{S}, \rho_{S}) \in \rbcM$, 
    set $\disMeasFunct(S) \coloneqq \disMeasures(S)$ and $d_{\disMeasFunct}^{S, \rho_{S}} \coloneqq d_{\disMeasures}^{S, \rho_{S}}$. 
  \item 
  For each $(S_{i}, d^{S_{i}}, \rho_{S_{i}}) \in \rbcM,\, i=1,2$
  and root-and-distance-preserving map $f: S_{1} \to S_{2}$, 
  set $\disMeasFunct_{f}(\nu) \coloneqq \nu \circ f^{-1}$.
\end{itemize}

\begin{prop}
  The functor $\disMeasFunct$ is Polish.
\end{prop}

\begin{proof}
  Recall the space $\pointMeas(S)$ from Section \ref{sec: a space including the space of discrete measures}.
  Define a functor $\tau$ as follows.
  \begin{itemize}
    \item 
      For $(S, d^{S}, \rho_{S}) \in \rbcM$, 
      set $\tau(S) \coloneqq \pointMeas(S)$ and $d_{\tau}^{S, \rho_{S}} \coloneqq d_{\pointMeas}^{S, \rho_{S}}$. 
    \item 
    For each $(S_{i}, d^{S_{i}}, \rho_{S_{i}}) \in \rbcM,\, i=1,2$
    and root-and-distance-preserving map $f: S_{1} \to S_{2}$, 
    set $\tau_{f}(\pi) \coloneqq \pi \circ (f \times \id_{\RNpp})^{-1}$.
  \end{itemize}
  One can check that $\tau$ is complete, separable and continuous
  by the same argument as that of \cite[Proof of Proposition 4.11]{Noda_pre_Metrization}.
  For $k \in \NN$ and $(S, d^{S}, \rho_{S}) \in \rbcM$,
  we set $\tau_{k}(S) \coloneqq \pointMeas^{(k)}(S)$
  (recall this space from \eqref{2. eq: dfn of sP^n}).
  We then obtain a Polish system $(\tau, (\tau_{k})_{k \geq 1})$ of $\disMeasFunct$ 
  by Corollary \ref{2. cor: embedding of Mdis to sP} and Lemmas \ref{2. lem: sP^n is open} and \ref{2. lem: sP^n decreases to sP}.
  Therefore,
  the desired result follows from Theorem \ref{3. thm: Polish functor}.
\end{proof}

\textbf{The functor for cadlag curves $\SkorohodFunct$ and $\unifFunct$.} 
Given a boundedly-compact metric space $(S, d^{S})$ and an interval $I$ of $\RNp$,
we write $D(I, S)$ for the set of cadlag functions from $I$ to $S$.
For every $t > 0$,
we write $d_{J_{1}, t}^{S}$ for the complete, separable metric on $D([0, t], S)$
given by \cite[Equation (12.16)]{Billingsley_99_Convergence},
which induces the usual $J_{1}$-Skorohod topology.
Then, the Skorohod metric on $D(\RNp, S)$ is defined by setting, for $\xi, \eta \in D(\RNp, S)$,
\begin{equation} \label{3. eq: dfn of J_1 metric}
  d_{J_{1}}^{S}(\xi, \eta) 
  \coloneqq 
  \int_{0}^{\infty} e^{-t} \bigl( 1 \wedge d_{J_{1}, t}^{S}( \xi|_{[0,t]}, \eta|_{[0,t]} )\bigr)\, dt.
\end{equation}
The function $d_{J_{1}}^{S}$ is a complete, separable metric on $D(\RNp, S)$
inducing the usual $J_{1}$-Skorohod topology \cite[Theorem 2.6]{Whitt_80_Some}.
Define a functor $\SkorohodFunct$ as follows.
\begin{itemize}
  \item 
    For $(S, d^{S}, \rho_{S}) \in \rbcM$, 
    set $\SkorohodFunct(S) \coloneqq D(\RNp, S)$ and $d_{\SkorohodFunct}^{S, \rho_{S}} \coloneqq d_{J_{1}}^{S}$. 
  \item 
  For each $(S_{i}, d^{S_{i}}, \rho_{S_{i}}) \in \rbcM,\, i=1,2$
  and root-and-distance-preserving map $f: S_{1} \to S_{2}$, 
    set $\SkorohodFunct_{f}(\xi) \coloneqq f \circ \xi$.
\end{itemize}
The functor $\SkorohodFunct$ is complete, separable and continuous (see \cite[Proposition 4.15]{Noda_pre_Metrization}).
Similarly, we define a functor on $\rbcM_{c}$,
which we denote by the same symbol.

In Section \ref{sec: GW trees}, 
we will use another functor $\unifFunct$ on $\rbcM_{c}$ to deal with cadlag curves in the compact-convergence topology.
\begin{itemize}
  \item 
    For $(S, d^{S}, \rho_{S}) \in \rbcM_{c}$, 
    set $\unifFunct(S) \coloneqq D([0,1], S)$ 
    and $d_{\unifFunct}^{S}(f,g) \coloneqq \sup_{0 \leq t \leq 1} |f(t) - g(t)|$. 
  \item 
  For each $(S_{i}, d^{S_{i}}, \rho_{S_{i}}) \in \rbcM_{c},\, i=1,2$
  and root-and-distance-preserving map $f: S_{1} \to S_{2}$, 
    set $\unifFunct_{f}(\xi) \coloneqq f \circ \xi$.
\end{itemize}
The functor $\unifFunct$ is continuous but neither complete nor separable
as the uniform metric is not complete nor separable on the set of cadlag functions.


\section{Resistance forms and transition densities} \label{sec: resistance forms}

This section is divided into three subsections.
In Section \ref{sec: resistance preliminary},
we recall some basics of the theory of resistance forms and resistance metrics.
In Section \ref{sec: recurrent resistance metric and auxiliary results},
we introduce recurrent resistance metrics, which are assumed for electrical networks in the main results,
and presents some auxiliary results.
Then, in Section \ref{sec: transition densities},
we prove the precompactness of transition densities of stochastic processes on measured resistance metric spaces,
which plays a crucial role in the proof of our main results.


\subsection{Preliminary} \label{sec: resistance preliminary}

Following \cite{Croydon_18_Scaling},
in this subsection we recall some basic properties of resistance forms, starting with their definition.
The reader is referred to \cite{Kigami_12_Resistance} for further background.
Also,
for further study of resistance forms and their extended Dirichlet spaces,
see \cite[Section 3]{Noda_pre_Scaling}.

\begin{dfn} [{Resistance form and resistance metric, \cite[Definition 3.1]{Kigami_12_Resistance}}] 
  \label{4. dfn: resistance forms}
  Let $F$ be a non-empty set.
  A pair $(\form, \rdomain)$ is called a \textit{resistance form} on $F$ if it satisfies the following conditions.
  \begin{enumerate} [label=(RF\arabic*), leftmargin = *]
    \item \label{4. dfn cond: the domain of resistance form}
          The symbol $\rdomain$ is a linear subspace of the collection of functions $\{ f : F \to \RN \}$ 
          containing constants,
          and $\form$ is a non-negative symmetric bilinear form on $\rdomain$
          such that $\form(f,f)=0$ if and only if $f$ is constant on $F$.
    \item \label{4. dfn cond: the quotient space of resistance form is Hilbert}
          Let $\sim$ be the equivalence relation on $\rdomain$ defined by saying $f \sim g$ if and only if $f-g$ is constant on $F$.
          Then $(\rdomain/\sim, \form)$ is a Hilbert space.
    \item
          If $x \neq y$, then there exists a function $f \in \rdomain$ such that $f(x) \neq f(y)$.
    \item \label{4. dfn cond: resistance forms condition 4}
          For any $x, y \in F$,
          \begin{equation} \label{eq: definition of resistance metric}
            R_{(\form, \rdomain)}(x,y)
            \coloneqq
            \sup
            \left\{
            \frac{|f(x) - f(y)|^{2}}{\form(f,f)}
            :
            f \in \rdomain,\
            \form(f,f) > 0
            \right\}
            < \infty.
          \end{equation}     
    \item
          If $\bar{f}\coloneqq  (f \wedge 1) \vee 0$,
          then $\bar{f} \in \rdomain$ and $\form(\bar{f}, \bar{f}) \leq \form(f,f)$ for any $f \in \rdomain$.
  \end{enumerate}
\end{dfn}

For the following definition,
recall the effective resistance on an electrical network with a finite vertex set 
from \cite[Section 9.4]{Levin_Peres_17_Markov}
(see also \cite[Section 2.1]{Kigami_01_Analysis}).

\begin{dfn} [{Resistance metric, \cite[Definition 2.3.2]{Kigami_01_Analysis}}]
  \label{3. dfn: resistance metrics}
  A metric $R$ on a non-empty set $F$ is called a \textit{resistance metric}
  if and only if,
  for any non-empty finite subset $V \subseteq F$,
  there exists an electrical network $G$ with the vertex set $V$ 
  such that the effective resistance on $G$ coincides with $R|_{V \times V}$.
\end{dfn}

\begin{thm} [{\cite[Theorem 2.3.6]{Kigami_01_Analysis}}]  \label{3. thm: one-to-one correspondence of forms and metrics}
  Fix a non-empty subset $F$.
  There exists a one-to-one correspondence between resistance forms $(\form, \rdomain)$ on $F$ 
  and resistance metrics $R$ on $F$ via $R = R_{(\form, \rdomain)}$.
  In other words,
  a resistance form $(\form, \rdomain)$ is characterized by $R_{(\form, \rdomain)}$ 
  given in \ref{4. dfn cond: resistance forms condition 4}.
\end{thm}

In the assumptions for the main results of this article,
we consider effective resistance between sets.
This is precisely defined below.

\begin{dfn} [{Effective resistance between sets}]  \label{3. dfn: effective resistance between sets}
  Fix a resistance form $(\form, \rdomain)$ on $F$
  and write $R$ for the corresponding resistance metric.
  For sets $A, B \subseteq F$,
  we define
  \begin{equation}
    R(A,B)
    \coloneqq
    \left(
    \inf\{
    \form(f,f):
    f \in \rdomain,\
    f|_{A} \equiv 1,\
    f|_{B} \equiv 0
    \}
    \right)^{-1},
  \end{equation}
  which is defined to be zero if the infimum is taken over the empty set.
  Note that by \ref{4. dfn cond: resistance forms condition 4} 
  we clearly have $R( \{x\}, \{y\}) = R(x,y)$.
\end{dfn}

A simple lower bound on effective resistance between a point and a subset is given by metric entropy
as described below.
Note that,
for a metric space $(S, d^{S})$ and $\delta > 0$,
we write
\begin{equation}
  N_{d^{S}}(S, \delta) 
  \coloneqq 
  \inf
  \left\{\# A \mid A \subseteq S,\, S \subseteq \bigcup_{x \in A} D_{S}(x, \delta) \right\}.
\end{equation}
The family $(N_{d^{S}}(S, \delta))_{\delta > 0}$ is called the metric entropy of $S$
(cf.\ \cite{Marcus_Rosen_06_Markov}).

\begin{lem} [{\cite[Theorem 5.3]{Kigami_12_Resistance}}] \label{4. lem: lower bound on resistance between root and outside of ball}
  For any $x \in F$ and $r>0$,
  \begin{equation}
    R(x, B_{R}(x, r)^{c})
    \geq 
    \frac{r}{4 N_{R}(F, r/2)}
  \end{equation}
\end{lem}

We will henceforth assume that we have a non-empty set $F$ equipped with a resistance form $(\form,\rdomain)$,
and denote the corresponding resistance metric $R$.
Furthermore,
we assume that $(F, R)$ is locally compact and separable,
and the resistance form $(\form, \rdomain)$ is regular, as described by the following.

\begin{dfn} [{Regular resistance form, \cite[Definition 6.2]{Kigami_12_Resistance}}] \label{dfn: regular resistance forms}
  Let $C_{c}(F)$ be the collection of compactly supported, continuous functions on $(F,R)$
  equipped with the compact-convergence topology.
  A resistance form $(\form, \rdomain)$ on $F$ is called \textit{regular}
  if and only if $\rdomain \cap C_{c}(F)$ is dense in $C_{c}(F)$.
\end{dfn}

We next introduce related Dirichlet forms and stochastic processes.
First, suppose that we have a Radon measure $\mu$ of full support on $(F,R)$.
Let $\mathcal{B}(F)$ be the Borel $\sigma$-algebra on $(F,R)$
and $\mathcal{B}^{\mu}(F)$ be the completion of $\mathcal{B}(F)$ with respect to $\mu$.
Two extended real-valued functions are said to be $\mu$-equivalent
if they coincide outside a $\mu$-null set.
The space $L^{2}(F,\mu)$ consists of $\mu$-equivalence classes of square-integrable $\mathcal{B}^{\mu}(F)$-measurable extended real-valued functions on $F$.
Now, we define a bilinear form $\form_{1}$ on $\rdomain \cap L^{2}(F, \mu)$
by setting
\begin{equation}
  \form_{1}(f,\, g)
  \coloneqq
  \form(f,\, g)
  +
  \int_{F}fg\, d\mu.
\end{equation}
Then $(\rdomain \cap L^{2}(F, \mu), \form_{1})$ is a Hilbert space (see \cite[Theorem 2.4.1]{Kigami_01_Analysis}).
We write $\domain$ to be the closure of $\rdomain \cap C_{c}(F)$ with respect to $\form_{1}$.
Under the assumption that $(\form,\rdomain)$ is regular,
we then have from \cite[Theorem 9.4]{Kigami_12_Resistance}
that $(\form, \domain)$ is a regular Dirichlet form on $L^{2}(F,\mu)$
(see \cite{Fukushima_Oshima_Takeda_11_Dirichlet} for the definition of a regular Dirichlet form).
Moreover, standard theory gives us the existence of an associated Hunt process $((X_{t} )_{t \geq 0}, (P_{x} )_{x \in F} )$
(e.g. \cite[Theorem 7.2.1]{Fukushima_Oshima_Takeda_11_Dirichlet}).
We refer to this Hunt process as the \textit{(Hunt) process associated with $(F, R, \mu)$}.
Note that such a process is, in general, only specified uniquely for starting points outside a set of zero capacity.
However, in this setting,
every point has strictly positive capacity (see \cite[Theorem 9.9]{Kigami_12_Resistance}),
and so the process is defined uniquely everywhere.
From \cite[Theorem 10.4]{Kigami_12_Resistance},
$X$ admits a (unique) jointly continuous transition density with respect to $\mu$.


\subsection{Recurrent resistance metrics and auxiliary results} \label{sec: recurrent resistance metric and auxiliary results}

In this subsection,
we introduce recurrent resistance metrics, which we consider throughout this article.
We then present some auxiliary results that are used in the proofs of the main results.

\begin{dfn} [{Recurrent resistance metric}] \label{4. dfn: recurrent resistance metric}
  Let $(F, R)$ be a boundedly-compact resistance metric space.
  We say that $R$ is recurrent if and only if
  $\lim_{r \to \infty} R(\rho, B_{R}(\rho, r)^{c}) = \infty$ for some (or, equivalently, any) $\rho \in F$.
\end{dfn}

Henceforth, we write $\Rspace$ for the collection of $(F, R, \rho, \mu) \in \GHVspace$ 
such that $(F, R)$ is a recurrent resistance metric space and $\mu$ is of full support.
Fix $(F, R, \rho, \mu) \in \Rspace$.
We note that the resistance form $(\form, \rdomain)$ associated with $(F, R)$ is regular 
and the Dirichlet form $(\form, \domain)$ associated with $(F, R, \mu)$ is recurrent 
(see \cite[Corollary 3.22]{Noda_pre_Scaling} and \cite[Lemma 2.3]{Croydon_18_Scaling}).
Write $(X=(X_{t})_{t \geq 0}, \{P_{x}\}_{x \in F})$ for the process associated with $(F, R, \mu)$.

The first lemma regards traces of $X$ onto subsets.
For further details,
the reader is referred to \cite{Fukushima_Oshima_Takeda_11_Dirichlet,Noda_pre_Scaling}.
For a non-empty closed subset $B$ of $F$,
we define a PCAF $A$ of $X$ by setting $A(t) \coloneqq  \int_{0}^{t}1_{B}(X(s))\, ds$
and $\gamma$ to be the right-continuous inverse of $A$,
i.e., $\gamma(t) \coloneqq  \inf \{ s >0 : A(s)>t \}$.
Then, the \textit{trace} of $X$ onto $B$ is given by setting $\tr_{B} X \coloneqq X \circ \gamma$.

\begin{lem} \label{4. lem: trace result}
  Fix a non-empty open subset $U$ of $F$.
  Set $B \coloneqq \closure(U)$, i.e., the closure of $B$ in $F$.
  Let $(Y, \{Q_{x}\}_{x \in B})$ be the Hunt process associated with $(B, R|_{B \times B}, \mu|_{B})$.
  For any $x \in B$,
  it holds that $P_{x}(\tr_{B}X \in \cdot)=Q_{x}(Y \in \cdot)$
  as probability measures on $D(\RNp, B)$ 
  equipped with the usual $J_{1}$-Skorohod topology.
\end{lem}

\begin{proof}
  The metric $R_{B \times B}$ is a regular resistance metric on $B$ by \cite[Theorem 8.4]{Kigami_12_Resistance}.
  Since $\mu$ is of full support on $F$,
  one can check that $\mu|_{B}$ is of full support on $B$.
  Hence, the result follows immediately from \cite[Theorem 3.30]{Noda_pre_Scaling}.
\end{proof}

By combining the trace technique described above with the following estimates of exit times of $X$ from balls, 
various analyses of processes on recurrent resistance metric spaces 
essentially reduce to analyses of processes on compact resistance metric spaces.
For a subset $A \subseteq F$,
we denote by $\tau_{A}$ the first exit time of $X$ from $A$, i.e.,
\begin{equation}
  \tau_{A}
  \coloneqq 
  \inf\{
    t > 0 \mid X(t) \notin A
  \}.
\end{equation}

\begin{lem} [{\cite[Lemma 4.2]{Croydon_18_Scaling}}] \label{4. lem: exit time estimate}
  For any $x \in F$, $\delta \in (0, R(x, B_{R}(x,r)^{c}))$ and $T \geq 0$,
  it holds that
  \begin{equation}
    P_{x}
    (\tau_{B_{R}(x,r)} \leq T)
    \leq
    \frac{4\delta}{R(x,B_{R}(x,r)^{c})}
    +
    \frac{4T}
    {\mu(B_{R}(x,\delta))(R(x,B_{R}(x,r)^{c})-\delta)}.
  \end{equation}
\end{lem}

We now prove new results,
which provide lower bounds for a probability that $X$ is at a point at a fixed time.

\begin{lem} \label{4. lem: lower bound for point probability}
  Assume that $(F, R)$ is compact.
  Then,
  for any $x \in F$ and $t \geq 0$,
  \begin{equation}
    P_{x}(X(t) = x) 
    \geq 
    \frac{\mu(\{x\})}{\mu(F)}.
  \end{equation}
\end{lem}

\begin{proof}
  Fix $x \in F$ and $t \geq 0$.
  When $F$ is a finite set, 
  then $X$ is simply a Markov chain on an electrical network with vertex set $F$ (see \cite[Theorem 4.2]{Noda_pre_Scaling}).
  So, the result is proven by the same argument as the proof of \cite[Lemma 2.5]{Fontes_Isopi_Newman_02_Random}.
  For a general compact resistance metric space $(F, R)$,
  we first show that the following statement holds for each $n \in \NN$.
  \begin{enumerate} [label = (\Alph*)]
    \item \label{4. proof item: partition of space}
      There exists a finite collection of disjoint non-empty Borel subsets $\{F_{i, n}\}_{i=1}^{k_{n}}$ 
      such that 
      $F = \bigcup_{i=1}^{k_{n}} \neighb{F_{i, n}}{n^{-1}}$, 
      $F_{1, n} = B_{R}(x, n^{-1})$, $\mu(F_{i,n}) > 0$, $\mu(\bigcup_{i=1}^{k_{n}} F_{i,n}) = \mu(F)$, and $\diam F_{i, n} \leq 2/n$,
      where $\diam A$ denotes the diameter of a subset $A$.
      (Recall the closed neighborhood $\neighb{\cdot}{\varepsilon}$ from \eqref{2. eq: dfn of closed neighborhood}.)
  \end{enumerate}
  Let $\{x_{i}\}_{i=1}^{k}$ be a finite subset of $F$ such that 
  $F = \bigcup_{i=1}^{k} B_{R}(x_{i}, n^{-1})$ and $x_{1} = x$.
  We then set $F_{1} \coloneqq B_{R}(x_{1}, n^{-1})$ 
  and inductively, for $i \geq 2$, $F_{i} \coloneqq B_{R}(x_{i}, n^{-1}) \setminus \bigcup_{j=1}^{i-1} F_{i}$.
  Define $I_{0}$ to be the collection of $i$ such that $\mu(F_{i}) = 0$,
  and define $I_{+} \coloneqq \{1, \ldots, k\} \setminus I_{0}$.
  It remains to prove that $F = \bigcup_{i \in I_{+}} \neighb{F_{i}}{n^{-1}}$.
  Suppose that there exists an $x \in F \setminus \bigcup_{i \in I_{+}} \neighb{F_{i}}{n^{-1}}$.
  This implies that $B_{R}(x, n^{-1}) \cap \bigcup_{i \in I_{+}} F_{i} = \emptyset$.
  Hence, it is the case that $B_{R}(x, n^{-1}) \subseteq \bigcup_{i \in I_{0}} F_{i}$.
  Since $\mu$ is of full support,
  it follows that $\mu(\bigcup_{i \in I_{0}} F_{i}) > 0$, which is a contradiction.
  Thus, we obtain that $F = \bigcup_{i \in I_{+}} \neighb{F_{i}}{n^{-1}}$.

  For each $n \in \NN$,
  we let $\{F_{i, n}\}_{i=1}^{k_{n}}$ be a finite collection of disjoint Borel subsets satisfying \ref{4. proof item: partition of space}.
  Choose an element $x_{i,n} \in F_{i,n}$ for each $i$
  with $x_{1, n} = x$.
  We then write $F_{n} \coloneqq \{x_{i, n}\}_{i=1}^{k_{n}}$, $R_{n} \coloneqq R|_{F_{n} \times F_{n}}$, and $x_{n} \coloneqq x$.
  We define a fully-supported Radon measure $\mu_{n}$ on $F_{n}$ by setting $\mu_{n}(\{x_{i,n}\}) \coloneqq \mu(F_{i, n})$.
  Noting that $(F_{n}, R_{n})$ is a recurrent resistance metric space,
  we let $(X_{n}, \{P_{y}^{(n)}\}_{y \in F_{n}})$ be the process associated with $(F_{n}, R_{n}, \mu_{n})$.
  It is not difficult to check that 
  \begin{equation}
    (F_{n}, R_{n}, x_{n}, \mu_{n}) \to (F, R, x, \mu)
  \end{equation}
  in $\GHPspace$ with respect to the (pointed) Gromov-Hausdorff-Prohorov topology
  (recall this topology from Section \ref{sec: functors used in this paper}).
  Thus, by \cite[Theorem 1.2]{Croydon_18_Scaling} and Theorem \ref{3. thm: convergence in GH topology},
  it is possible to embed $(F_{n}, R_{n})$ and $(F, R)$ isometrically into a common rooted compact metric space $(K, d^{K}, x_{K})$
  in such a way that $x_{n} = x = x_{K}$ as elements of $K$,
  $F_{n} \to F$ in the Hausdorff topology as subsets of $K$,
  $\mu_{n} \to \mu$ weakly as measures on $K$,
  and $P_{x_{n}}^{(n)}(X_{n} \in \cdot) \xrightarrow{\mathrm{d}} P_{x}(X \in \cdot)$ as probability measures on $D(\RNp, K)$.
  By the quasi-left-continuity of $X$,
  $X$ is continuous at $t$ almost surely.
  Hence,
  we have that $X_{n}(t) \xrightarrow{\mathrm{d}} X(t)$.
  Noting that $\{x\} = \{x_{n}\} = \{x_{K}\} \subseteq K$ is closed,
  we deduce that 
  \begin{equation} \label{4. eq: point probability, approximation}
    P_{x}(X(t) = x) \geq \limsup_{n \to \infty} P_{x_{n}}^{(n)}(X_{n}(t) = x_{n}).
  \end{equation}
  As noted at the beginning,
  the desired result holds for any finite resistance metric space.
  Thus, it holds that
  \begin{equation}
    P_{x_{n}}^{(n)}(X_{n}(t) = x_{n}) 
    \geq 
    \mu_{n}(\{x_{n}\})/ \mu_{n}(F_{n}) 
    = 
    \mu(B_{R}(x, n^{-1}))/ \mu(F)
  \end{equation}
  Combining this with \eqref{4. eq: point probability, approximation},
  we obtain the desired result.
\end{proof}

Using the trace technique,
the lower bound given above is improved as follows.

\begin{prop} \label{4. prop: improved lower bound for point probability}
  For any $x \in F$, $t \geq 0$, and $\varepsilon > 0$,
  \begin{equation}
    P_{x}(X(t) = x) 
    \geq 
    \frac{\mu(\{x\})}{\mu(D_{R}(x, \varepsilon))}
    - 
    P_{x}(\tau_{B_{R}(x, \varepsilon)} \leq t).
  \end{equation}
\end{prop}

\begin{proof}
  Fix $x \in F$ and $\varepsilon > 0$, and write $B \coloneqq \closure(B_{R}(x, \varepsilon))$.
  Let $(Y, \{Q_{x}\}_{x \in B})$ be the Hunt process associated with $(B, R|_{B \times B}, \mu|_{B})$.
  If $\tau_{B} > t$,
  then we have that $\tr_{B} X(t) = X(t)$.
  Thus, by Lemma \ref{4. lem: trace result}, we deduce that 
  \begin{align}
    P_{x}(X(t) = x) 
    &\geq 
    P_{x}(X(t) = x,\, \tau_{B} > t) \\
    &=
    P_{x}(\tr_{B}X (t) = x,\, \tau_{B} > t)\\
    &\geq
    P_{x}(\tr_{B}X (t) = x) 
    - 
    P_{x}( \tau_{B} \leq t) \\
    &=
    Q_{x}( Y(t) = x) 
    - 
    P_{x}( \tau_{B} \leq t) \\
    &\geq
    \frac{\mu(\{x\})}{\mu(B)} 
    - 
    P_{x}( \tau_{B} \leq t),
  \end{align}
  where we apply Lemma \ref{4. lem: lower bound for point probability} to obtain the last inequality.
  Since $B_{R}(x, \varepsilon) \subseteq B \subseteq D_{R}(x, \varepsilon)$,
  we obtain the desired result.
\end{proof}

The following result is used later to show the non-triviality of the (sub-)aging function.

\begin{prop} \label{4. prop: positivity of diagonal heat kernel}
  For any $x \in F$ and $t > 0$, $p(t,x,x) > 0$.
\end{prop}

\begin{proof}
  Fix $t > 0$ and $x \in F$.
  Write $(\form, \domain)$ for the Dirichlet form associated with $(F, R, \mu)$.
  Suppose that $p(t, x, \cdot)$ is constant.
  Since $(\form, \domain)$ is recurrent,
  $X$ is conservative (see \cite[Lemma 1.6.5 and Excercise 4.5.1]{Fukushima_Oshima_Takeda_11_Dirichlet}).
  It follows that $p(t, x, y) > 0$ for some $y \in F$.
  Hence, we obtain that $p(t, x, x) > 0$.
  Suppose that $p(t, x, \cdot)$ is not constant.
  Since $p(t, x, \cdot) \in \domain$ (see \cite[Theorem 10.4]{Kigami_12_Resistance}),
  by \cite[Lemma 1.3.3]{Fukushima_Oshima_Takeda_11_Dirichlet} and the Chapman-Kolmogorov equation,
  we deduce that 
  \begin{equation}
    t\, \form(p(t, x, \cdot), p(t, x, \cdot))
    \leq 
    \int_{F} p \bigl( t/2, x, y \bigr)^{2}\, \mu(dy)
    =
    p(t, x, x).
  \end{equation}
  Since $p(t, x, \cdot)$ is not constant,
  it holds from \ref{4. dfn cond: the domain of resistance form} that $\form(p(t, x, \cdot), p(t, x, \cdot)) > 0$.
  Hence, we complete the proof.
\end{proof}


\subsection{Transition densities of processes on resistance metric spaces} \label{sec: transition densities}

In this subsection, 
we prove that when measured resistance metric spaces converge in the local Gromov-Hausdorff-vague topology, 
the family of the transition densities of the associated processes is precompact (Proposition \ref{4. prop: precompactness of transition densities})
in the sense that it is uniformly bounded and equicontinuous on every compact subset.

Fix $(F_{n}, R_{n}, \rho_{n}, \mu_{n}) \in \Rspace$ for each $n \in \NN$ and $(F, R, \rho, \mu) \in \Rspace$.
We define $(p_{n}(t, x, y))_{t > 0, x, y \in F_{n}}$ to be the jointly continuous transition density of 
the process associated with $(F_{n}, R_{n}, \mu_{n})$.
Similarly, we define $(p(t, x, y))_{t > 0, x, y \in F}$.
To state the result, we introduce new notation.
Given a rooted boundedly-compact metric space $(M, d^{M}, \rho_{M})$,
we set, for $T > 0$ and $r>0$,
$K_{d^{M}}(T, r) \coloneqq [T, \infty) \times D_{M}(\rho_{S}, r) \times D_{M}(\rho_{S}, r)$.
For a function $f: \RNpp \times  M \times M \to \RN$,
we define 
\begin{equation} \label{4. eq: dfn of notation S}
  \mathfrak{s}_{M}(f, T, r) \coloneqq \sup_{(t, x, y) \in K_{d^{M}}(T, r)} f(t, x, y),
\end{equation} 
and, for each $\delta > 0$,
\begin{equation} \label{4. eq: dfn of notation D}
  \mathfrak{d}_{M}(f, T, r, \delta) 
  \coloneqq 
  \sup
  \left\{
    |f(t, x, y) - f(t', x', y')|\,
    \middle|\,
    \begin{array}{l}
      (t,x,y), (t', x', y') \in K_{d^{M}}(T, r),\\
      |t-t'| \vee d^{M}(x,x') \vee d^{M}(y, y') \leq \delta
    \end{array}
  \right\}.
\end{equation}

\begin{prop} \label{4. prop: precompactness of transition densities}
  Assume that $(F_{n}, R_{n}, \rho_{n}, \mu_{n})$ converges to $(F, R, \rho, \mu)$ in the local Gromov-Hausdorff-vague topology.
  Then, for any $T > 0$ and $r>0$,
  \begin{equation}
    \sup_{n \geq 1} \mathfrak{s}_{F_{n}}(p_{n}, T, r) < \infty,
    \quad 
    \lim_{\delta \downarrow 0} \sup_{n \geq 1} \mathfrak{d}_{F_{n}}(p_{n}, T, r, \delta) = 0.
  \end{equation}
\end{prop}

\begin{proof}
  Fix $T > 0$ and $r>0$.
  Note that the local Gromov-Hausdorff-vague convergence 
  and $\mu$ being of full support imply that 
  \begin{equation}
    c_{1} 
    \coloneqq
    \inf_{n \geq 1} 
    \inf_{x \in D_{R_{n}}(\rho_{n}, r)} 
    \mu_{n}(D_{R_{n}}(x, r)) 
    > 0
  \end{equation}
  (cf.\ \cite[Corollary 5.7]{Athreya_Lohr_Winter_16_The_gap}).
  By \cite[Equation (10.4)]{Kigami_12_Resistance},
  it holds that, for any $t \geq T$ and $x \in D_{R_{n}}(\rho_{n}, r)$,
  \begin{equation}
    p_{n}(t, x, x) 
    \leq 
    2r t^{-1} + \sqrt{2} \mu_{n}(D_{R_{n}}(\rho_{n}, r))^{-1} 
    \leq 
    2rT^{-1} + \sqrt{2} c_{1}^{-1}.
  \end{equation}
  Using the Chapman-Kolmogorov equation and Cauchy-Schwarz inequality,
  we deduce that, for any $(t, x, y) \in K_{R_{n}}(T, r)$,
  \begin{align}
    p_{n}(t, x, y) 
    &=
    \int p_{n}(t/2, x, z) p_{n}(t/2, z, y)\, \mu_{n}(dz) \\
    &\leq
    \left(
      \int p_{n}(t/2, x, z)^{2}\, \mu(dz)
    \right)^{1/2} 
    \left(
      \int p_{n}(t/2, z, y)^{2}\, \mu(dz)
    \right)^{1/2} \\
    &=
    p_{n}(t, x, x)^{1/2} p_{n}(t, y, y)^{1/2} \\ 
    &\leq
    2rT^{-1} + \sqrt{2}c_{1}^{-1},
  \end{align}
  which shows the first result.
  By following the proof of \cite[Theorem 10.4]{Kigami_12_Resistance}
  (specifically, the top of page 44 of \cite{Kigami_12_Resistance}),
  we obtain that 
  \begin{align}
    |p_{n}(t,x,y) - p_{n}(t', x', y')| 
    &\leq
    \sqrt{\frac{p_{n}(t,x,x)R_{n}(y,y')}{t}}
    + 
    \sqrt{\frac{p_{n}(t,y',y')R_{n}(x,x')}{t}} \\
    &\quad 
    + 
    2|t-t'| 
    \frac{\sqrt{p_{n}(s/2,x',x') p_{n}(s/2, y',y')}}{s},
  \end{align}
  where $s$ is a value between $t$ and $t'$.
  Thus, for $(t, x, y), (t', x', y') \in K_{R_{n}}(T, r)$ with $|t-t'| \vee R_{n}(x,x') \vee R_{n}(y, y') \leq \delta$,
  we deduce that 
  \begin{align}
    |p_{n}(t,x,y) - p_{n}(t', x', y')|  
    \leq
    2 \sqrt{T^{-1} \mathfrak{s}_{F_{n}}(p_{n}, T, r) \delta}
    + 
    2 \delta T^{-1} \mathfrak{s}_{F_{n}}(p_{n}, T/2, r).
  \end{align}
  This, combined with the first result, yields the second result.
\end{proof}


\section{Aging and sub-aging for deterministic traps} \label{sec: (sub-)aging for deterministic traps}

In this section,
we prove aging and sub-aging results 
for processes on resistance metric spaces associated with deterministic traps.
Throughout this section,
we fix a sequence $(F_{n}, R_{n}, \rho_{n})_{n \geq 1}$ of rooted recurrent resistance metric spaces 
and a rooted recurrent resistance metric space $(F, R, \rho)$.


\subsection{Aging result}

We first prove an aging result, Theorem \ref{5. thm: aging}.
Due to the length of the proof, we divide this subsection into two smaller sections.

We let $\nu_{n} = \sum_{i \in I_{n}} v_{i}^{(n)} \delta_{x_{i}^{(n)}}$ be a fully-supported discrete measure on $F_{n}$
and $\nu = \sum_{i \in I} v_{i} \delta_{x_{i}}$ be a fully-supported discrete measure on $F$.
We write $(X_{n}^{\nu_{n}}, \{P_{x}^{\nu_{n}}\}_{x \in F_{n}})$ for the Hunt process associated with $(F_{n}, R_{n}, \nu_{n})$
and $p_{n}^{\nu_{n}}$ for the jointly continuous transition density of $X_{n}^{\nu_{n}}$ with respect to $\nu_{n}$.
Similarly,
we write $(X^{\nu}, \{P_{x}^{\nu}\}_{x \in F})$ for the Hunt process associated with $(F, R, \nu)$
and $p^{\nu}$ for the jointly continuous transition density of $X^{\nu}$ with respect to $\nu$.
For a subset $A \subseteq F$,
we denote by $\tau_{A}^{\nu}$ the first exit time of $X^{\nu}$ from $A$, i.e.,
\begin{equation}
  \tau_{A}^{\nu} 
  \coloneqq 
  \inf\{
    t > 0 \mid X^{\nu}(t) \notin A
  \}.
\end{equation}
We similarly define $\tau_{A}^{\nu_{n}}$ for the first exit time of $X_{n}^{\nu_{n}}$ from $A \subseteq F_{n}$.
We define an aging function $\Phi^{\nu}$ associated with $X^{\nu}$ and $\rho$ by setting 
\begin{equation} \label{5. eq: aging function}
  \Phi^{\nu}(s,t) \coloneqq P_{\rho}^{\nu}(X^{\nu}(s) = X^{\nu}(t)).
\end{equation}
Similarly, we define $\Phi_{n}^{\nu_{n}}$ to be the aging function associated with $X_{n}^{\nu_{n}}$ and $\rho_{n}$.

Throughout this subsection,
we suppose that the following condition is satisfied.
\begin{assum} \label{5. assum: aging result}
  It holds that 
  \begin{equation}
    \bigl( F_{n}, R_{n}, \rho_{n}, \nu_{n}, 
      P_{\rho_{n}}^{\nu_{n}}( X_{n}^{\nu_{n}} \in \cdot) \bigr) 
    \to 
    \bigl( F, R, \rho, \nu, P_{\rho}^{\nu}( X^{\nu} \in \cdot) \bigr).
  \end{equation}
  in the space $\rbcM(\disMeasFunct \times \probFunct(\SkorohodFunct))$
\end{assum}
Under Assumption \ref{5. assum: aging result},
by Theorem \ref{3. thm: convergence in GH topology},
we may assume that $(F_{n}, R_{n}, \rho_{n})$ and $(F, R, \rho)$ are embedded isometrically
into a common rooted boundedly-compact metric space $(M, d^{M}, \rho_{M})$ in such a way that 
$\rho_{n} = \rho = \rho_{M}$ as elements of $M$,
$F_{n} \to F$ in the local Hausdorff topology as closed subsets in $M$,
and $\nu_{n} \to \nu$ in the vaguely-and-point-process topology as discrete measures on $M$,
and $X_{n}^{\nu_{n}} \xrightarrow{\mathrm{d}} X^{\nu}$ in $D(\RNp, M)$.
In Sections \ref{sec: precompactness of the aging funcitons} and \ref{sec: convergence of the aging functions} below,
we assume this embedding.


\subsubsection{Precompactness of the aging functions} \label{sec: precompactness of the aging funcitons}

Here, we prove that the family $(\Phi_{n}^{\nu_{n}})_{n \geq 1}$ of the aging functions is precompact 
(Proposition \ref{5. prop: precompactness of aging-functions}).

\begin{lem} \label{5. lem: technical lemmas for pcpt of aging functions}
  The following statements hold.
  \begin{enumerate} [label = (\roman*)]
    \item \label{5. lem item: non-explosion}
      For each $l > 0$, 
      \begin{equation} \label{5. eq: exit time estimate for aging function}
        \lim_{r \to \infty} 
        \sup_{n \geq 1} 
        P_{\rho_{n}}^{\nu_{n}} \Bigl( \tau_{B_{R_{n}}(\rho_{n}, r)}^{\nu_{n}} \leq l \Bigr) 
        = 0.
      \end{equation}
    \item \label{5. lem item: does not exit ball quickly}
      For any $x_{n} \in F_{n}$ and $x \in F$ such that $x_{n} \to x$ in $M$,
      and any $\delta > 0$,
      \begin{equation}
        \lim_{\eta \to 0} 
        \limsup_{n \to \infty} 
        P_{x_{n}}^{\nu_{n}}(\tau_{B_{R_{n}}(x_{n}, \delta)}^{\nu_{n}} \leq \eta)  
        = 0.
      \end{equation} 
  \end{enumerate}
\end{lem}

\begin{proof}
  \ref{5. lem item: non-explosion}.
  This is a consequence of the weak convergence of $X_{n}^{\nu_{n}}$ to $X^{\nu}$ in the usual $J_{1}$-Skorohod topology
  (see \cite[Theorem 23.8]{Kallenberg_21_Foundations}).

  \ref{5. lem item: does not exit ball quickly}. 
  If $x_{n} \to x$ in $M$,
  then we have from \cite[Theorem 1.2]{Croydon_18_Scaling} that 
  $P_{x_{n}}^{\nu_{n}}(X_{n}^{\nu_{n}} \in \cdot) \to P_{x}^{\nu}(X^{\nu} \in \cdot)$
  as probability measures on $D(\RNp, M)$.
  This yields that 
  \begin{equation}
    \bigl( F_{n}, R_{n}, x_{n}, \nu_{n}, 
      P_{x_{n}}^{\nu_{n}}( X_{n}^{\nu_{n}} \in \cdot) \bigr) 
    \to 
    \bigl( F, R, x, \nu, P_{x}^{\nu}( X^{\nu} \in \cdot) \bigr).
  \end{equation}
  in the space $\rbcM(\disMeasFunct \times \probFunct(\SkorohodFunct))$.
  Hence, it is enough to show the result for $x_{n} = \rho_{n}$ and $x = \rho$.
  Fix $\delta > 0$.
  For $r>0$,
  we write $(X_{n}^{\nu_{n}^{(r)}}, \{P_{x}^{\nu_{n}^{(r)}}\}_{x \in F_{n}^{(r)}})$
  for the Hunt process associated with $(F_{n}^{(r)}, R_{n}^{(r)}, \nu_{n}^{(r)})$,
  where we recall the restriction operator $\cdot^{(r)}$ from \eqref{1. eq: dfn of restriction operator}.
  We write $\tau_{A}^{\nu_{n}^{(r)}}$ for the first exit time of $X_{n}^{\nu_{n}^{(r)}}$ from a set $A \subseteq F_{n}$.
  Using Lemma \ref{4. lem: trace result},
  we deduce that, for any $l > \eta$,
  \begin{equation} \label{5. eq: quick exit comparison}
    P_{\rho_{n}}^{\nu_{n}} \bigl( \tau_{B_{R_{n}}(\rho_{n}, \delta)}^{\nu_{n}} \leq \eta \bigr)
    \leq    
    P_{\rho_{n}}^{\nu_{n}} \bigl( \tau_{B_{R_{n}}(\rho_{n}, r)}^{\nu_{n}} \leq l \bigr)
    + 
    P_{\rho_{n}}^{\nu_{n}^{(r)}} \Bigl( \tau_{B_{R_{n}^{(r)}}(\rho_{n}, \delta)}^{\nu_{n}^{(r)}} \leq \eta \Bigr)
  \end{equation}
  Note that the convergence of $F_{n}$ to $F$ in the local Hausdorff topology implies that 
  \begin{equation}
    c_{r} \coloneqq \sup_{n \geq 1} N_{R_{n}^{(r)}}(F_{n}^{(r)}, \delta/2) < \infty
  \end{equation}
  (cf.\ \cite[Theorem 3.13]{Noda_pre_Metrization}).
  Lemma \ref{4. lem: lower bound on resistance between root and outside of ball} yields that 
  $R_{n}^{(r)}(\rho_{n}, B_{R_{n}^{(r)}}(\rho_{n}, \delta)^{c}) \geq \delta/(4c_{r}) \eqqcolon c'_{r}$ for all $n \geq 1$.
  By Lemma \ref{4. lem: exit time estimate}, we obtain that, for any $\varepsilon < c'_{r}$,
  \begin{equation} \label{5. eq: quick exit for traces}
    P_{\rho_{n}}^{\nu_{n}^{(r)}} \Bigl( \tau_{B_{R_{n}^{(r)}}(\rho_{n}, \delta)}^{\nu_{n}^{(r)}} \leq \eta \Bigr)
    \leq 
    \frac{4\varepsilon}{c'_{r}}
    +
    \frac{4 \eta}
    {\nu_{n}(B_{R_{n}}(\rho_{n},\varepsilon))(c'_{r}-\varepsilon)}.
  \end{equation}
  By the vague convergence of $\nu_{n}$ to $\nu$ and the assumption that $\nu$ is of full support,
  we have that $\inf_{n \geq 1} \nu_{n}(B_{R_{n}}(\rho_{n},\varepsilon)) > 0$.
  Therefore,
  using \ref{5. lem item: non-explosion}, \eqref{5. eq: quick exit comparison}, and \eqref{5. eq: quick exit for traces},
  we deduce the desired result.
\end{proof}

Note that $([l^{-1}, l]^{2})_{l \geq 1}$ is a sequence of compact subsets increasing to $\RNpp^{2}$.
To prove the equicontinuity of $(\Phi_{n}^{\nu_{n}})_{n \geq 1}$ on $[l^{-1}, l]^{2}$,
we divide $[l^{-1}, l]^{2}$ into a part near the diagonal and the other part as follows:
for $l \in \NN$ and $\eta>0$,
\begin{align}
  T_{1}(l, \eta) &\coloneqq \{(s,t) \in [l^{-1}, l]^{2} \mid |t-s| \leq \eta \},\\
  T_{2}(l, \eta) &\coloneqq \{(s,t) \in [l^{-1}, l]^{2} \mid \eta \leq |t-s|\}.
\end{align}
In Lemmas \ref{5. lem: pcpt of aging functions near diagonal} and \ref{5. lem: pcpt of aging functions off diagonal} below,
we prove the equicontinuity of $(\Phi_{n}^{\nu_{n}})_{n \geq 1}$ on $T_{1}(l, \eta)$ and $T_{2}(l, \eta)$, respectively.

\begin{lem} \label{5. lem: pcpt of aging functions near diagonal}
  For every $l \in \NN$ and $\eta > 0$,
  \begin{equation}
    \lim_{\eta \to 0} 
    \limsup_{n \to \infty} 
    \sup_{(s,t) \in T_{1}(l, \eta)} 
    |\Phi_{n}^{\nu_{n}}(s,t) - 1| 
    = 0.
  \end{equation}
\end{lem}

\begin{proof}
  Fix $(s, t) \in T_{1}(l, \eta)$ with $s \leq t$.
  Since $1 = P_{\rho_{n}}^{\nu_{n}}(X_{n}^{\nu_{n}}(s) \in F_{n})$,
  it holds that
  \begin{align}
    |\Phi_{n}^{\nu_{n}}(s,t) - 1| 
    &= 
    |P_{\rho_{n}}^{\nu_{n}}(X_{n}^{\nu_{n}}(s) = X_{n}^{\nu_{n}}(t)) - P_{\rho_{n}}^{\nu_{n}}(X_{n}^{\nu_{n}}(s) \in F_{n})|\\
    &\leq     
    2P_{\rho_{n}}^{\nu_{n}}(\tau_{B_{R_{n}}(\rho_{n}, r)}^{\nu_{n}} \leq l)
    + 
    \left| 
      P_{\rho_{n}}^{\nu_{n}}(X_{n}^{\nu_{n}}(s) = X_{n}^{\nu_{n}}(t) \in F_{n}^{(r)}) 
      - 
      P_{\rho_{n}}^{\nu_{n}}(X_{n}^{\nu_{n}}(s) \in F_{n}^{(r)})
    \right|.
  \end{align}
  By Lemma \ref{5. lem: technical lemmas for pcpt of aging functions}\ref{5. lem item: non-explosion},
  it is enough to show that, for all but countably many $r > 0$,
  \begin{equation}
    \lim_{\eta \to 0} 
    \limsup_{n \to \infty} 
    \sup_{(s,t) \in T_{1}(l, \eta)} 
    \left| 
      P_{\rho_{n}}^{\nu_{n}}(X_{n}^{\nu_{n}}(s) = X_{n}^{\nu_{n}}(t) \in F_{n}^{(r)}) 
      - 
      P_{\rho_{n}}^{\nu_{n}}(X_{n}^{\nu_{n}}(s) \in F_{n}^{(r)})
    \right|
    = 0.
  \end{equation}
  Fix $r>0$ such that the boundary of $F^{(r)}$ contains no atoms of $\nu$.
  We have from Proposition \ref{4. prop: precompactness of transition densities} that
  \begin{equation}
    A \coloneqq \sup_{n \geq 1} \sup\{p_{n}^{\nu_{n}}(t, x, y) \mid l^{-1} \leq t \leq l,\, x,y \in F_{n}^{(r)}\} < \infty.
  \end{equation}
  Using the transition density $p_{n}^{\nu_{n}}$,
  we deduce that, for any $\varepsilon > 0$,
  \begin{align}
    \lefteqn{
      \left| 
        P_{\rho_{n}}^{\nu_{n}}(X_{n}^{\nu_{n}}(s) = X_{n}^{\nu_{n}}(t) \in F_{n}^{(r)}) 
        - 
        P_{\rho_{n}}^{\nu_{n}}(X_{n}^{\nu_{n}}(s) \in F_{n}^{(r)})
      \right|
    }\\
    &= 
    \int_{F_{n}^{(r)}} p_{n}^{\nu_{n}}(s, \rho_{n}, x)
    \left( 1- P_{x}^{\nu_{n}}(X_{n}^{\nu_{n}}(t-s) = x) \right) \nu_{n}(dx)\\
    &\leq
    A\,
    M_{\varepsilon}^{(r)}(\pointMap(\nu_{n}))
    + 
    A
    \sum_{i \in I_{n}(\varepsilon)} 
    \left( 1- P_{x_{i}^{(n)}}^{\nu_{n}}(X_{n}^{\nu_{n}}(t-s) = x_{i}^{(n)}) \right) v_{i}^{(n)},
  \end{align}
  where we define $I_{n}(\varepsilon) \coloneqq \{i \in I_{n} \mid x_{i}^{(n)} \in F_{n}^{(r)},\, v_{i}^{(n)} \geq \varepsilon\}$
  and recall $M_{\varepsilon}^{(r)}$ from \eqref{2. eq: dfn of M_epsilon^r}.
  By Theorem \ref{2. thm: convergence in sP},
  it holds that 
  $\lim_{\varepsilon \to 0} \limsup_{n \to \infty} M_{\varepsilon}^{(r)}(\pointMap(\nu_{n})) = 0$.
  Thus, it suffices to show that, for all but countably many $\varepsilon> 0$,
  \begin{equation}
    \lim_{\eta \to 0} 
    \limsup_{n \to \infty} 
    \sup_{(s,t) \in T_{1}(l, \eta)} 
    \sum_{i \in I_{n}(\varepsilon)} 
    \left( 1- P_{x_{i}^{(n)}}^{\nu_{n}}(X_{n}^{\nu_{n}}(t-s) = x_{i}^{(n)}) \right) v_{i}^{(n)}
    = 0.
  \end{equation}
  Proposition \ref{4. prop: improved lower bound for point probability} yields that,
  for any $(s, t) \in T_{1}(l, \eta)$ and $\varepsilon, \varepsilon' > 0$,
  \begin{align}
    \lefteqn{
      \sum_{i \in I_{n}(\varepsilon)} 
      \left( 1- P_{x_{i}^{(n)}}^{\nu_{n}}(X_{n}^{\nu_{n}}(t-s) = x_{i}^{(n)}) \right) v_{i}^{(n)}}\\
    &\leq  
    \sum_{i \in I_{n}(\varepsilon)} 
    \left( 1- \frac{v_{i}^{(n)}}{\nu_{n}(D_{R_{n}}(x_{i}^{(n)}, \varepsilon'))} 
      + P_{x_{i}^{(n)}}^{\nu_{n}} \Bigl( \tau_{B_{R_{n}}(x_{i}^{(n)}, \varepsilon')}^{\nu_{n}} \leq \eta \Bigr) 
    \right) v_{i}^{(n)}\\
    &\leq    
    \sum_{i \in I_{n}(\varepsilon)} 
    \nu_{n}(D_{R_{n}}(x_{i}^{(n)}, \varepsilon') \setminus \{x_{i}^{(n)}\})
    +
    W^{(r)}(\pointMap(\nu_{n}))
    \sum_{i \in I_{n}(\varepsilon)}
    P_{x_{i}^{(n)}}^{\nu_{n}} \Bigl( \tau_{B_{R_{n}}(x_{i}^{(n)}, \varepsilon')}^{\nu_{n}} \leq \eta \Bigr),
  \end{align}
  where we recall $W^{(r)}$ from \eqref{2. eq: dfn of W^r}.
  Set $W \coloneqq \sup_{n \geq 1} W^{(r)}(\pointMap(\nu_{n})) + W^{(r)}(\pointMap(\nu)) + 1$,
  which is finite by Theorem \ref{2. thm: precompactness in sP}.
  Note that, by the definition of $W^{(r)}$, there are no atoms of $\pointMap(\nu_{n})$ in $F_{n}^{(r)} \times [W, \infty)$.
  We choose $\varepsilon > 0$ so that the boundary of $K \coloneqq F^{(r)} \times [\varepsilon, W]$ contains no atoms of $\pointMap(\nu)$.
  It is then the case that $\pointMap(\nu_{n})|_{K} \to \pointMap(\nu)|_{K}$ weakly.
  Hence, 
  if we write $I(\varepsilon) \coloneqq \{i \in I \mid (x_{i}, v_{i}) \in K\}$,
  then, by Proposition \ref{2. prop: atom convergence from simple measure convergence},
  there exists a bijective map $f_{n}: I(\varepsilon) \to I_{n}(\varepsilon)$
  (at least, for all sufficiently large $n$)
  such that $x_{f_{n}(i)}^{(n)} \to x_{i}$ and $v_{f_{n}(i)}^{(n)} \to v_{i}^{(n)}$ for each $i \in I(\varepsilon)$.
  Since $I(\varepsilon)$ is a finite set,
  we deduce from Lemma \ref{5. lem: technical lemmas for pcpt of aging functions}\ref{5. lem item: does not exit ball quickly} that,
  for each $\varepsilon' > 0$, 
  \begin{align}
    \lefteqn{\lim_{\eta \to 0} 
      \limsup_{n \to \infty}
      W^{(r)}(\pointMap(\nu_{n}))
      \sum_{i \in I_{n}(\varepsilon)}
      P_{x_{i}^{(n)}}^{\nu_{n}} \Bigl(\tau_{B_{R_{n}}(x_{i}^{(n)}, \varepsilon')}^{\nu_{n}} \leq \eta \Bigr)}\\
    &\leq
    W \sum_{i \in I(\varepsilon)}
    \lim_{\eta \to 0} 
    \limsup_{n \to \infty}
    P_{x_{f_{n}(i)}^{(n)}}^{\nu_{n}} \Bigl( \tau_{B_{R_{n}}(x_{f_{n}(i)}^{(n)}, \varepsilon')}^{\nu_{n}} \leq \eta \Bigr)\\
    &= 0.
  \end{align}
  Using Proposition \ref{2. prop: atom convergence from simple measure convergence},
  we obtain that 
  \begin{align}
    \lefteqn{\lim_{\varepsilon' \to 0} 
      \limsup_{n \to \infty} 
      \sum_{i \in I_{n}(\varepsilon)} 
      \nu_{n} \bigl( D_{R_{n}}(x_{i}^{(n)}, \varepsilon') \setminus \{x_{i}^{(n)}\} \bigr)}\\
    &=
    \sum_{i \in I(\varepsilon)} 
    \lim_{\varepsilon' \to 0} 
    \limsup_{n \to \infty} 
    \nu_{n} \bigl( D_{R_{n}}(x_{f_{n}(i)}^{(n)}, \varepsilon') \setminus \{x_{f_{n}(i)}^{(n)}\} \bigr)\\
    &=0.
  \end{align}
  Therefore, we complete the proof.
\end{proof}

\begin{lem} \label{5. lem: pcpt of aging functions off diagonal}
  For every $l \in \NN$ and $\eta > 0$,
  \begin{equation}
    \lim_{\delta \to 0} 
    \limsup_{n \to \infty} 
    \sup_{\substack{(s,t) \in T_{2}(l, \eta)\\ |t-t'| \vee |s-s'| < \delta} }
    |\Phi_{n}^{\nu_{n}}(s,t) - \Phi_{n}^{\nu_{n}}(s',t')| 
    = 0.
  \end{equation}
\end{lem}

\begin{proof}
  Note that the vague convergence $\nu_{n} \to \nu$ implies that, for each $r>0$,
  \begin{equation}
    A_{r} \coloneqq \sup_{n \geq 1} \nu_{n}(F_{n}^{(r)}) < \infty.
  \end{equation}
  Fix $(s, t), (s', t') \in T_{2}(l, \eta)$ with $|s-s'| \vee |t - t'| \leq \delta$, $s<t$, and $s'<t'$.
  Using the transition density,
  we obtain that  
  \begin{align} 
    \lefteqn{|\Phi_{n}^{\nu_{n}}(s,t) - \Phi_{n}^{\nu_{n}}(s',t')|} \\
    &\leq
    2P_{\rho_{n}}^{\nu_{n}}(\tau_{B_{R_{n}}(\rho_{n}, r)}^{\nu_{n}} \leq l)\\
    &\qquad
    + 
    \left| 
      P_{\rho_{n}}^{\nu_{n}}(X_{n}^{\nu_{n}}(s) = X_{n}^{\nu_{n}}(t) \in F_{n}^{(r)}) 
      - 
      P_{\rho_{n}}^{\nu_{n}}(X_{n}^{\nu_{n}}(s') = X_{n}^{\nu_{n}}(t') \in F_{n}^{(r)})
    \right|\\
    &\leq
    2P_{\rho_{n}}^{\nu_{n}}(\tau_{B_{R_{n}}(\rho_{n}, r)}^{\nu_{n}} \leq l)\\
    &\qquad
    +
    \int_{F_{n}^{(r)}}
      \nu_{n}(\{x\}) 
      \bigl| p_{n}^{\nu_{n}}(s, \rho_{n}, x) p_{n}^{\nu_{n}}(t-s, x, x) - p_{n}^{\nu_{n}}(s', \rho_{n}, x) p_{n}^{\nu_{n}}(t'-s', x, x) \bigr|\,
    \nu_{n}(dx)\\
    &\leq
    2P_{\rho_{n}}^{\nu_{n}}(\tau_{B_{R_{n}}(\rho_{n}, r)}^{\nu_{n}} \leq l)\\
    &\qquad
    +
    A_{r}^{2} 
    \sup_{x \in F_{n}^{(r)}} 
    \bigl| p_{n}^{\nu_{n}}(s, \rho_{n}, x) p_{n}^{\nu_{n}}(t-s, x, x) - p_{n}^{\nu_{n}}(s', \rho_{n}, x) p_{n}^{\nu_{n}}(t'-s', x, x) \bigr|.
    \label{5. eq: pcpt on 2 of aging functions, essential inequality}
  \end{align}
  For any $x \in F_{n}^{(r)}$,
  we have that
  \begin{align}
    \lefteqn{
      \bigl| p_{n}^{\nu_{n}}(s, \rho_{n}, x) p_{n}^{\nu_{n}}(t-s, x, x) - p_{n}^{\nu_{n}}(s', \rho_{n}, x) p_{n}^{\nu_{n}}(t'-s', x, x) \bigr|
    } \\
    &\leq 
    p_{n}^{\nu_{n}}(s, \rho_{n}, x) 
    \bigl| p_{n}^{\nu_{n}}(t-s, x, x) - p_{n}^{\nu_{n}}(t'-s', x, x) \bigr| 
    + 
    p_{n}^{\nu_{n}}(t'-s', x, x) 
    \bigl| p_{n}^{\nu_{n}}(s, \rho_{n}, x)  - p_{n}^{\nu_{n}}(s', \rho_{n}, x) \bigr|\\
    &\leq 
    \mathfrak{s}_{F_{n}}(p_{n}^{\nu_{n}}, l^{-1}, r) \mathfrak{d}_{F_{n}}(p_{n}^{\nu_{n}}, \eta, r, 2\delta)
    + 
    \mathfrak{s}_{F_{n}}(p_{n}^{\nu_{n}}, \eta, r) \mathfrak{d}_{F_{n}}(p_{n}^{\nu_{n}}, l^{-1}, r, \delta),
  \end{align}
  where we recall the notation $\mathfrak{s}_{F_{n}}$ and $\mathfrak{d}_{F_{n}}$ from \eqref{4. eq: dfn of notation S} and \eqref{4. eq: dfn of notation D}.
  Thus, we deduce from Proposition \ref{4. prop: precompactness of transition densities} that 
  \begin{equation}
    \lim_{\delta \to 0} 
    \limsup_{n \to \infty} 
    \sup_{x \in F_{n}^{(r)}} 
    \bigl| p_{n}^{\nu_{n}}(s, \rho_{n}, x) p_{n}^{\nu_{n}}(t-s, x, x) - p_{n}^{\nu_{n}}(s', \rho_{n}, x) p_{n}^{\nu_{n}}(t'-s', x, x) \bigr| 
    = 0.
  \end{equation}
  From Lemma \ref{5. lem: technical lemmas for pcpt of aging functions}\ref{5. lem item: non-explosion},
  \eqref{5. eq: pcpt on 2 of aging functions, essential inequality},
  and the above convergence, 
  we establish the desired result.
\end{proof}

From Lemmas \ref{5. lem: pcpt of aging functions near diagonal} and \ref{5. lem: pcpt of aging functions off diagonal} above,
we deduce the precompactness of $(\Phi_{n}^{\nu_{n}})_{n \geq 1}$ as follows.

\begin{prop} \label{5. prop: precompactness of aging-functions}
  The family $\{\Phi_{n}^{\nu_{n}}\}_{n \geq 1}$ is precompact in $C(\RNpp^{2}, \RNp)$.
\end{prop}

\begin{proof}
  It is enough to show that the family is uniformly bounded and equicontinuous on each compact subset of $\RNpp^{2}$.
  (See \cite[Theorem A5.2]{Kallenberg_21_Foundations} for a necessarily and sufficient condition for precompactness in $C(\RNpp^{2}, \RNp)$).
  Obviously, we have that $\sup_{n \geq 1} \sup_{s, t} \Phi_{n}^{\nu_{n}}(s,t) \leq 1$.
  So, it remains to prove that, for every $l \in \NN$, 
  \begin{equation}
    \lim_{\delta \to 0} 
    \limsup_{n \to \infty} 
    \sup_{\substack{(s,t), (s',t') \in [l^{-1}, l]^{2}\\ |t-t'| \vee |s-s'| < \delta}} 
    |\Phi_{n}^{\nu_{n}}(s,t) - \Phi_{n}^{\nu_{n}}(s',t')| 
    = 0.
  \end{equation}
  Fix $\delta, \eta > 0$ with $\eta > 4 \delta$.
  Fix also $(s,t), (s',t') \in [l^{-1}, l]^{2}$ with $|t-t'| \vee |s-s'| \leq \delta$.
  If $|t-s| \leq \eta$,
  \begin{equation}
    |t'-s'| 
    \leq 
    |t-s| + |t-t'| + |s-s'| 
    \leq 
    \eta + 2 \delta  
    \leq 
    2 \eta.
  \end{equation}
  Otherwise,
  \begin{equation}
    |t'-s'| 
    \geq 
    |t-s| - |t-t'| - |s-s'| 
    \geq 
    \eta - 2 \delta   
    \geq 
    \eta/2.
  \end{equation}
  Thus, 
  we deduce that 
  \begin{align}
    \lefteqn{\sup_{\substack{(s,t), (s',t') \in [l^{-1}, l]^{2}\\ |t-t'| \vee |s-s'| < \delta}}
    |\Phi_{n}^{\nu_{n}}(s,t) - \Phi_{n}^{\nu_{n}}(s',t')|} \\
    &\leq
    \sup_{(s,t), (s',t') \in T_{1}(l, 2\eta)}
    |\Phi_{n}^{\nu_{n}}(s,t) - \Phi_{n}^{\nu_{n}}(s',t')| 
    +
    \sup_{\substack{(s,t), (s',t') \in T_{2}(l, \eta/2)\\ |t-t'| \vee |s-s'| < \delta}} 
    |\Phi_{n}^{\nu_{n}}(s,t) - \Phi_{n}^{\nu_{n}}(s',t')| \\
    &\leq
    2
    \sup_{(s,t) \in T_{1}(l, 2\eta)}
    |\Phi_{n}^{\nu_{n}}(s,t) - 1| 
    +
    \sup_{\substack{(s,t), (s',t') \in T_{2}(l, \eta/2)\\ |t-t'| \vee |s-s'| < \delta}} 
    |\Phi_{n}^{\nu_{n}}(s,t) - \Phi_{n}^{\nu_{n}}(s',t')|.
  \end{align}
  This, combined with Lemma \ref{5. lem: pcpt of aging functions near diagonal} and \ref{5. lem: pcpt of aging functions off diagonal}, 
  yields the desired result.
\end{proof}


\subsubsection{Convergence of aging functions} \label{sec: convergence of the aging functions} 
Here,
we prove the convergence of the aging functions (Theorem \ref{5. thm: aging}).
Since we already showed the precompactness of the aging functions in Proposition \ref{5. prop: precompactness of aging-functions},
it remains to prove the pointwise convergence of the aging functions.
This is derived from the following result.

\begin{lem} \label{5. lem: convergence of atoms and weights}
  Fix $k \in \NN$,
  $\bm{x}_{n} = (x_{1}^{(n)}, \dots, x_{k}^{(n)}) \in F_{n}^{k}$ for each $n$,
  and $\bm{x} = (x_{1}, \dots, x_{k}) \in F^{k}$.
  If $x_{i}^{(n)} \to x_{i}$ in $M$ and $ \nu_{n}(\{x_{i}^{(n)}\}) \to \nu(\{x_{i}\})$ for each $i$,
  then 
  \begin{equation}
    P_{\rho_{n}}^{\nu_{n}}
    \bigl( X_{n}^{\nu_{n}}(t_{1}) = x_{1}^{(n)},\dots, X_{n}^{\nu_{n}}(t_{k}) = x_{k}^{(n)} \bigr) 
    \to
    P_{\rho}
    \bigl( X^{\nu}(t_{1}) = x_{1},\dots, X^{\nu}(t_{k}) = x_{k} \bigr)
  \end{equation}
  for any $t_{1}, \dots, t_{k} \in (0, \infty)$.
\end{lem}

\begin{proof}
  Since we have the convergence of finite-dimensional distributions of $X_{n}^{\nu_{n}}$ to those of $X^{\nu}$,
  by Proposition \ref{2. prop: equivalence of convergence of weights},
  it suffices to show that 
  \begin{equation} \label{5. lem eq: A(n, d) convergence}
    \lim_{\delta \downarrow 0} 
    \limsup_{n \to \infty}
    P_{\rho_{n}}^{\nu_{n}}
    \bigl(
      (X_{n}^{\nu_{n}}(t_{1}), \dots, X_{n}^{\nu_{n}}(t_{k})) \in B_{M^{k}}(\bm{x}_{n}, \delta) \setminus \{\bm{x}_{n}\}
    \bigr)
    = 0,
  \end{equation}
  where we equip the product space $M^{k}$ with the max product metric (cf.\ \eqref{2. eq: max product metric}).
  We write $A(n, \delta)$ for the above probability.
  Using the transition density,
  we obtain that, for each $\delta \in (0, 1)$,
  \begin{align}
    A(n, \delta)
    &\leq 
    \sum_{i=1}^{k}
    P_{\rho_{n}}^{\nu_{n}}
    \bigl( X_{n}^{\nu_{n}}(t_{i}) \in B_{R_{n}}(x_{i}^{(n)}, \delta) \setminus \{x_{i}^{(n)}\} \bigr) \\
    &= 
    \sum_{i=1}^{k}
    \int_{B_{R_{n}}(x_{i}^{(n)}, \delta) \setminus \{x_{i}^{(n)}\}} p_{n}^{\nu_{n}}(t_{i}, \rho_{n}, y)\, \nu_{n}(dy) \\
    &\leq 
    \sum_{i=1}^{k}
    \nu_{n}( B_{R_{n}}(x_{i}^{(n)}, \delta) \setminus \{x_{i}^{(n)}\} )
    \sup_{y \in B_{R_{n}}(x_{i}^{(n)}, 1)} p_{n}^{\nu_{n}}(t_{i}, \rho_{n}, y)
  \end{align}
  By Proposition \ref{4. prop: precompactness of transition densities},
  we have that 
  \begin{equation}
    \sup_{n \geq 1}
    \sup_{y \in B_{R_{n}}(x_{i}^{(n)}, 1)} 
    p_{n}^{\nu_{n}}(t_{i}, \rho_{n}, y)  
    < \infty.
  \end{equation}
  Moreover, 
  Proposition \ref{2. prop: equivalence of convergence of weights} 
  and the convergence $\nu_{n} \to \nu$ in the vague-and-point-process topology 
  yield that 
  \begin{equation}
    \lim_{\delta \downarrow 0} 
    \limsup_{n \to \infty} 
    \nu_{n}( B_{R_{n}}(x_{i}^{(n)}, \delta) \setminus \{x_{i}^{(n)}\} ) 
    = 0.
  \end{equation}
  Therefore, we obtain \eqref{5. lem eq: A(n, d) convergence}.
\end{proof}

\begin{thm} \label{5. thm: aging}
  Under Assumption \ref{5. assum: aging result},
  it holds that 
  \begin{equation} 
    \bigl( V_{n}, R_{n}, \rho_{n}, \nu_{n}, 
      P_{\rho_{n}}^{\nu_{n}}(X_{n}^{\nu_{n}} \in \cdot), \Phi_{n}^{\nu_{n}} \bigr) 
    \to 
    \bigl( F, R, \rho, \nu, P_{\rho}^{\nu}(X^{\nu} \in \cdot), \Phi^{\nu} \bigr)
  \end{equation}
  in the space $\rbcM(\tau^{\mathrm{dis}} \times \probFunct(\SkorohodFunct) \times \fixedFunct{C(\RNpp^{2}, \RNp)})$.
  Moreover, $\Phi^{\nu}(s,t) > 0$ for all $s,t \in \RNpp$.
\end{thm}

\begin{proof}
  By Lemma \ref{5. prop: precompactness of aging-functions},
  it suffices to show that $\Phi_{n}^{\nu_{n}}(s,t) \to \Phi^{\nu}(s,t)$ for every $(s,t) \in \RNpp^{2}$.
  Fix $(s, t) \in T_{2}$.
  Since $X_{n}^{\nu_{n}} \xrightarrow{\mathrm{d}} X^{\nu}$ in the usual $J_{1}$-Skorohod topology
  and $X^{\nu}$ is continuous at $s$ and $t$ almost surely by its quasi-left-continuity,
  we have that 
  $(X_{n}^{\nu_{n}}(s), X_{n}^{\nu_{n}}(t)) \xrightarrow{\mathrm{d}} (X^{\nu}(s), X^{\nu}(t))$.
  Noting that the diagonal set in a product metric space is closed,
  we obtain that 
  \begin{equation}  \label{5. thm eq: limsup for aging result}
    \limsup_{n \to \infty} 
    P_{\rho_{n}}^{\nu_{n}} \bigl( X_{n}^{\nu_{n}}(s) =  X_{n}^{\nu_{n}}(t) \bigr) 
    \leq 
    P_{\rho}^{\nu} \bigl( X^{\nu}(s) = X^{\nu}(t) \bigr).
  \end{equation}
  Let $\{x_{j}\}_{j \in J}$ be the set of elements of $F$
  satisfying $P_{\rho} \bigl( X^{\nu}(s) = X^{\nu}(t) = x_{i} \bigr) > 0$.
  For each $i \in J$,
  by the convergence $\nu_{n} \to \nu$ in the vague-and-point-process topology and Theorem \ref{2. thm: convergence in Mdis},
  there exists $x_{i}^{(n)} \in F_{n}$ 
  such that $x_{i}^{(n)} \to x_{i}$ and $\nu_{n}(\{x_{i}^{(n)}\}) \to \nu(\{x_{i}\})$.
  Lemma \ref{5. lem: convergence of atoms and weights} immediately yields that 
  \begin{equation}
    P_{\rho_{n}}^{\nu_{n}} \bigl( X_{n}^{\nu_{n}}(s) = X_{n}^{\nu_{n}}(t) = x_{i}^{(n)} \bigr)
    \to 
    P_{\rho}^{\nu} \bigl( X^{\nu}(s) = X^{\nu}(t) = x_{i} \bigr).
  \end{equation}
  Let $(J_{l})_{l \geq 1}$ be an increasing sequence of finite subsets of $J$ such that $\bigcup_{l \geq 1} J_{l} = J$.
  Then, the convergence $x_{i}^{(n)} \to x_{i}$ for each $i$ implies that 
  $x_{i}^{(n)} \neq x_{j}^{(n)}$ if $i \neq j$ with $i, j \in J_{k}$ 
  for all sufficiently large $n$.
  Hence, we deduce that 
  \begin{align}
    \sum_{i \in J_{k}} 
    P_{\rho}^{\nu} \bigl( X^{\nu}(s) =  X^{\nu}(t) = x_{i} \bigr)
    &=
    \lim_{n \to \infty}
    \sum_{i \in J_{k}}
    P_{\rho_{n}}^{\nu_{n}} \bigl( X_{n}^{\nu_{n}}(s) = X_{n}^{\nu_{n}}(t) = x_{i}^{(n)} \bigr) \\
    &\leq 
    \liminf_{n \to \infty}
    P_{\rho_{n}}^{\nu_{n}} \bigl( X_{n}^{\nu_{n}}(s) = X_{n}^{\nu_{n}}(t) \bigr).
  \end{align}
  By letting $k \to \infty$ in the above inequality and \eqref{5. thm eq: limsup for aging result},
  we obtain that 
  \begin{equation}
    \Phi_{n}^{\nu_{n}}(s,t) 
    = 
    P_{\rho_{n}}^{\nu_{n}} 
    \bigl( X_{n}^{\nu_{n}}(s) = X_{n}^{\nu_{n}}(t) \bigr) 
    \to    
    P_{\rho}^{\nu} 
    \bigl( X^{\nu}(s) = X^{\nu}(t) \bigr)
    = 
    \Phi^{\nu}(s, t).
  \end{equation}
  
  Checking the last assertion for $s=t$ is easy as it holds that $\Phi^{\nu}(s,s) = 1$.
  Suppose that $t > s$.
  We can find an $x \in F$ such that $p^{\nu}(s, \rho, x) \nu(\{x\}) > 0$. 
  (Otherwise, one has that $P_{\rho}^{\nu}(X(s) \in F) = 0$.)
  Using the transition density $p^{\nu}$,
  we obtain that 
  \begin{equation}
    \Phi^{\nu}(s,t) 
    = 
    \int p^{\nu}(s, \rho, y) p^{\nu}(t-s, y, y) \nu(\{y\})\, \nu(dy)
    \geq 
    p^{\nu}(s, \rho, x) p^{\nu}(t-s, x, x) \nu(\{x\})^{2}
  \end{equation}
  Thus, we deduce that $\Phi^{\nu}(s, t) > 0$ from Proposition \ref{4. prop: positivity of diagonal heat kernel}.
\end{proof}


\subsection{Sub-aging result}

We next prove a sub-aging result, Theorem \ref{5. thm: sub-aging}.
To do this,
we extend the framework of Section \ref{sec: Proof of aging for deterministic models}.
For each $n \geq 1$,
we let  $\dot{\pi}_{n} = \sum_{i \in I_{n}} \delta_{(x_{i}^{(n)}, w_{i}^{(n)}, v_{i}^{(n)})}$ be a simple measure on $F_{n} \times \RNp \times \RNpp$,
and let $\dot{\pi} = \sum_{i \in I} \delta_{(x_{i}, w_{i}, v_{i})}$ be a simple measure on $F \times \RNp \times \RNpp$.
For each $n \geq 1$,
we assume that $x_{i}^{(n)} \neq x_{j}^{(n)}$ if $i \neq j$, and 
the measure $\nu_{n}$ on $F_{n}$ given below is of full support:
\begin{equation}
  \nu_{n}(A)  
  =
  \int 1_{A}(x) v\, \dot{\pi}_{n}(dx dw dv)
  = 
  \sum_{i \in I_{n}} v_{i}^{(n)} \delta_{x_{i}^{(n)}}(A),
  \quad 
  A \in \mathcal{B}(F_{n}).
\end{equation}
Similarly, we assume that 
$x_{i} \neq x_{j}$ if $i \neq j$, and the measure $\nu$ on $F$ given below is of full support:
\begin{equation}
  \nu(A)  
  =
  \int 1_{A}(x) v\, \dot{\pi}(dx dw dv)
  = 
  \sum_{i \in I} v_{i} \delta_{x_{i}}(A),
  \quad 
  A \in \mathcal{B}(F).
\end{equation}
As in the previous section,
we write $(X_{n}^{\nu_{n}}, \{P_{x}^{\nu_{n}}\}_{x \in F_{n}})$ for the Hunt process associated with $(F_{n}, R_{n}, \nu_{n})$
and $p_{n}^{\nu_{n}}$ for the jointly continuous transition density of $X_{n}^{\nu_{n}}$ with respect to $\nu_{n}$.
Similarly,
we write $(X^{\nu}, \{P_{x}^{\nu}\}_{x \in F})$ for the Hunt process associated with $(F, R, \nu)$
and $p^{\nu}$ for the jointly continuous transition density of $X^{\nu}$ with respect to $\nu$.
We define a sub-aging function associated with $X^{\nu}$ and $\rho$ by setting 
\begin{equation}
  \Psi^{\nu}(s,t) 
  \coloneqq 
  \int e^{-ws/v} P_{\rho}^{\nu}(X^{\nu}(t) = x)\, \dot{\pi}(dx dw dv).
\end{equation}
We similarly define $\Psi_{n}^{\nu_{n}}$ to be the sub-aging function associated with $X_{n}^{\nu_{n}}$ and $\rho_{n}$.
Since $\nu(\{x_{i}\}) = v_{i}$ and $P_{\rho}^{\nu}(X^{\nu}(t) = x) = p^{\nu}(t, \rho, x) \nu(\{x\})$,
we note that 
\begin{equation} \label{5. eq: expression of sub-aging funtions via heat kernels}
  P_{\rho}^{\nu}(X^{\nu}(t) = x)\, \dot{\pi}(dx dw dv) 
  =
  p^{\nu}(t, \rho, x) v\, \dot{\pi}(dx dw dv)
\end{equation}

\begin{assum} \label{5. assum: sub-aging}
  It holds that 
  \begin{equation}
    \bigl( F_{n}, R_{n}, \rho_{n}, \nu_{n}, \dot{\pi}_{n}, P_{\rho_{n}}^{\nu_{n}}(X_{n}^{\nu_{n}} \in \cdot) \bigr) 
    \to 
    \bigl( F, R, \rho, \nu, \dot{\pi}, P_{\rho}^{\nu}(X^{\nu} \in \cdot) \bigr)
  \end{equation}
  in the space $\rbcM(\disMeasFunct \times \markedMeasFunct{\RNp \times \RNpp} \times \probFunct(\SkorohodFunct))$.
\end{assum}

Henceforth,
we assume that Assumption \ref{5. assum: sub-aging} is satisfied.
Under this assumption,
by Theorem \ref{3. thm: convergence in GH topology},
we may assume that $(F_{n}, R_{n}, \rho_{n})$ and $(F, R, \rho)$ are embedded isometrically
into a common rooted boundedly-compact metric space $(M, d^{M}, \rho_{M})$ in such a way that 
$\rho_{n} = \rho = \rho_{M}$ as elements of $M$,
$F_{n} \to F$ in the local Hausdorff topology as closed subsets in $M$,
$\nu_{n} \to \nu$ in the vaguely-and-point-process topology as discrete measures on $M$,
$\dot{\pi}_{n} \to \dot{\pi}$ vaguely as measures on $M \times \RNp \times \RNpp$,
and $X_{n}^{\nu_{n}} \xrightarrow{\mathrm{d}} X^{\nu}$ in $D(\RNp, M)$.

For $l \in \NN$,
we set 
\begin{equation}
  T_{3}(l)
  \coloneqq 
  [0, l] \times [l^{-1}, l].
\end{equation}
Note that $(T_{3}(l))_{l \geq 1}$ is a sequence of compact subsets increasing to $\RNp \times \RNpp$.
Below, we show some technical results used to prove the precompactness of $(\Psi_{n}^{\nu_{n}})_{n \geq 1}$.

\begin{lem} \label{5. lem: sub-aging fucntion, smaller parts}
  Fix $l \in \NN$.
  The following statements hold.
  \begin{enumerate} [label = (\roman*)]
    \item \label{5. lem item: sub-aging function, small mass}
      For every $r>0$,
      \begin{equation}
        \lim_{\eta \to 0} 
        \limsup_{n \to \infty}
        \limsup_{(s,t) \in T_{3}(l)}
        \int_{D_{R_{n}}(\rho_{n}, r) \times \RNp \times (0, \eta]} 
          e^{-ws/v} P_{\rho_{n}}^{\nu_{n}}(X_{n}^{\nu_{n}}(t) = x)\, 
        \dot{\pi}_{n}(dx dw dv) 
        = 0.
      \end{equation}
    \item \label{5. lem item: sub-aging function, outside of ball}
      It holds that 
      \begin{equation}
        \lim_{r \to \infty} 
        \limsup_{n \to \infty}
        \limsup_{(s,t) \in T_{3}(l)}
        \int_{D_{R_{n}}(\rho_{n}, r)^{c} \times \RNp \times \RNpp} 
          e^{-ws/v} P_{\rho_{n}}^{\nu_{n}}(X_{n}^{\nu_{n}}(t) = x)\, 
        \dot{\pi}_{n}(dx dw dv) 
        = 0.
      \end{equation}
    \item \label{5. lem item: sub-aging function, large conductances}
      For every $r, \eta>0$,
      it holds that  
      \begin{equation}
        \lim_{W \to \infty} 
        \limsup_{n \to \infty}
        \sup_{(s,t) \in T_{3}(l)}
        \int_{D_{R_{n}}(\rho_{n}, r)^{c} \times [W, \infty) \times [\eta, \infty)} 
          e^{-ws/v} P_{\rho_{n}}^{\nu_{n}}(X_{n}^{\nu_{n}}(t) = x)\, 
        \dot{\pi}_{n}(dx dw dv) 
        = 0.
      \end{equation}
  \end{enumerate}
\end{lem}

\begin{proof}
  \ref{5. lem item: sub-aging function, small mass}.
  For $(s, t) \in T_{3}(l)$,
  we deduce by \eqref{5. eq: expression of sub-aging funtions via heat kernels} that  
  \begin{align}
    \lefteqn{\int_{D_{R_{n}}(\rho_{n}, r) \times \RNp \times (0, \eta]} 
      e^{-ws/v} P_{\rho_{n}}^{\nu_{n}}(X_{n}^{\nu_{n}}(t) = x)\, 
    \dot{\pi}_{n}(dx dw dv)} \\
    &\leq
    \int_{D_{R_{n}}(\rho_{n}, r) \times \RNp \times (0, \eta]} 
      p_{n}^{\nu_{n}}(t, \rho_{n}, x) v\, 
    \dot{\pi}_{n}(dx dw dv) \\
    &\leq 
    \mathfrak{s}_{F_{n}}(p_{n}^{\nu_{n}}, l^{-1}, r) \sum_{\substack{x_{i}^{(n)} \in D_{R_{n}}(\rho_{n}, r)\\ v_{i}^{(n)} \leq \eta}} v_{i}^{(n)}\\
    &\leq 
    \mathfrak{s}_{F_{n}}(p_{n}^{\nu_{n}}, l^{-1}, r) M_{\eta}^{(r)} (\pointMap(\nu_{n})).
  \end{align}
  Hence, the result follows from Theorem \ref{2. thm: convergence in sP} and Proposition \ref{4. prop: precompactness of transition densities}.
  
  \ref{5. lem item: sub-aging function, outside of ball}. 
  For $(s, t) \in T_{3}(l)$,
  we have that  
  \begin{align}
    \int_{D_{R_{n}}(\rho_{n}, r)^{c} \times \RNp \times \RNp} 
      e^{-ws/v} P_{\rho_{n}}^{\nu_{n}}(X_{n}^{\nu_{n}}(t) = x)\, 
    \dot{\pi}_{n}(dx dw dv) 
    &\leq
    \sum_{x_{i} \in D_{R_{n}}(\rho_{n}, r)^{c}} P_{\rho_{n}}^{\nu_{n}}(X_{n}^{\nu_{n}}(t) = x_{i}) \\
    &\leq  
    P_{\rho_{n}}^{\nu_{n}}(\tau_{D_{R_{n}}(\rho_{n}, r)}^{\nu_{n}} \leq l).
  \end{align}
  Hence, the desired result follows from \eqref{5. eq: exit time estimate for aging function}.

  \ref{5. lem item: sub-aging function, large conductances}.
  Fix $r, \eta > 0$
  and let $r' > r$ and $\eta' < \eta$ be such that 
  the boundary of $D_{M}(\rho, r') \times [\eta', \infty)$ does not contain the atoms of $\pointMap(\nu)$.
  By Theorem \ref{2. thm: precompactness in sP},
  it holds that 
  \begin{equation}
    V 
    \coloneqq 
    W^{(r')}(\pointMap(\nu)) + \sup_{n \geq 1} W^{(r')}(\pointMap(\nu_{n})) + 1
    < \infty.
  \end{equation}
  Noting that $F_{n} \cap D_{M}(\rho_{M}, r) = D_{R_{n}}(\rho_{n}, r)$,
  we deduce that, for any $W>0$,
  \begin{align} 
    \lefteqn{\sup_{(s,t) \in T_{3}(l)}
    \int_{D_{R_{n}}(\rho_{n}, r) \times [W, \infty) \times [\eta, \infty)} 
      e^{-ws/v} P_{\rho_{n}}^{\nu_{n}}(X_{n}^{\nu_{n}}(t) = x)\, 
    \dot{\pi}_{n}(dx dw dv)}\\
    &\leq 
    \int_{D_{R_{n}}(\rho_{n}, r) \times [W, \infty) \times [\eta, \infty)}  
    \dot{\pi}_{n}(dx dw dv) \\
    &=
    \#\bigl( \At(\dot{\pi}_{n}) \cap (D_{R_{n}}(\rho_{n}, r) \times [W, \infty) \times [\eta, \infty)) \bigr)\\ 
    &\leq 
    \#\bigl( \At(\dot{\pi}_{n}) \cap (D_{M}(\rho_{M}, r') \times [W, \infty) \times [\eta', V]) \bigr)
    \label{5. eq: technical lemmas eq for third result}
  \end{align}
  We define a finite subset $J \subseteq I$ so that 
  $\{(x_{j}, v_{j})\}_{j \in J}$ are the atoms of $\pointMap(\nu)$ lying in $K \coloneqq D_{M}(\rho_{M}, r') \times [\eta', V]$.
  Proposition \ref{2. prop: atom convergence from simple measure convergence} yields that 
  $\#(\At(\pointMap(\nu_{n})) \cap K) = \# J$ for all sufficiently large $n$,
  which implies that 
  \begin{equation} \label{5. eq: technical lemmas eq 2 for third result}
    \#\bigl( \At(\dot{\pi}_{n}) \cap (D_{M}(\rho_{M}, r') \times \RNp \times [\eta', V]) \bigr) = \# J.
  \end{equation}
  Now, fix $W$ sufficiently large so that $W > \max_{j \in J} w_{j} + 1$.
  Since $\dot{\pi}_{n}$ converges to $\dot{\pi}$ vaguely,
  Proposition \ref{2. prop: atom convergence from simple measure convergence} again yields that,
  for all sufficiently large $n$,
  \begin{align}
    \#\bigl( \At(\dot{\pi}_{n}) \cap (D_{M}(\rho_{M}, r') \times [0, W) \times [\eta', V]) \bigr) 
    &= 
    \#\bigl( \At(\dot{\pi}) \cap (D_{M}(\rho_{M}, r') \times [0, W) \times [\eta', V]) \bigr)\\
    &= 
    \# J.
  \end{align}
  Combining this with \eqref{5. eq: technical lemmas eq 2 for third result}, 
  we obtain that 
  \begin{equation}
    \lim_{n \to \infty}
    \#\bigl( \At(\dot{\pi}_{n}) \cap (D_{M}(\rho_{M}, r') \times [W, \infty) \times [\eta', V]) \bigr)
    = 0.
  \end{equation}
  From this and \eqref{5. eq: technical lemmas eq for third result},
  we obtain the desired result.
\end{proof}

\begin{lem} \label{5. lem: precompactness of sub-aging functions}
  The family $\{\Psi_{n}^{\nu_{n}}\}_{n \geq 1}$ is precompact in $C(\RNp \times \RNpp, \RNp)$.
\end{lem}

\begin{proof}
  Noting that, for all $s, t \in \RNp \times \RNpp$ and $n \geq 1$,
  \begin{align}
    \Psi_{n}^{\nu_{n}}(s,t) 
    \leq 
    \int P_{n}^{\nu_{n}}(X_{n}^{\nu_{n}} = x)\, \dot{\pi}_{n}(dx dw dv) 
    = 
    \sum_{i \in I_{n}}  P_{n}^{\nu_{n}}(X_{n}^{\nu_{n}} = x_{i}^{(n)}) 
    = 1,
  \end{align}
  it suffices to show that, for each $l \geq 1$,
  \begin{equation}
    \lim_{\delta \to 0} 
    \limsup_{n \to \infty} 
    \sup_{\substack{(s,t), (s',t') \in T_{3}(l)\\ |s-s'| \vee |t-t'| \leq \delta}}  
    |\Psi^{\nu}(s,t) - \Psi^{\nu}(s',t')| 
    = 0.
  \end{equation}
  Fix $l \geq 1$.
  By Lemma \ref{5. lem: sub-aging fucntion, smaller parts},
  this reduces to proving that, for any $r, W, \eta > 0$,
  \begin{equation} \label{5. eq: pcpt of sub-aging functions, reduced equation}
    \lim_{\delta \to 0} 
    \limsup_{n \to \infty} 
    \sup_{\substack{(s,t), (s',t') \in T_{3}(l)\\ |s-s'| \vee |t-t'| \leq \delta}}  
    A_{n}(r, W, \eta)
    = 0,
  \end{equation}
  where $A_{n} = A_{n}(r, W, \eta)$ is given by 
  \begin{equation}
    A_{n} 
    \coloneqq 
    \int_{D_{R_{n}}(\rho_{n}, r) \times [0,W] \times [\eta, \infty)}
    \bigl|  
        e^{-ws/v} P_{n}^{\nu_{n}}(X_{n}^{\nu_{n}}(t) = x) - e^{-ws'/v} P_{n}^{\nu_{n}}(X_{n}^{\nu_{n}}(t') = x)
    \bigr|\, 
    \dot{\pi}(dx dw dv).
  \end{equation}
  Fix $(s,t), (s', t') \in T_{3}(l)$ with $|s-s'| \vee |t-t'| \leq \delta$.
  Using \eqref{5. eq: expression of sub-aging funtions via heat kernels},
  we can write 
  \begin{equation}
    A_{n} 
    = 
    \int_{D_{R_{n}}(\rho_{n}, r) \times [0,W] \times [\eta, \infty)}
    v
    \bigl|  
        e^{-ws/v} p_{n}^{\nu_{n}}(t, \rho_{n}, x) - e^{-ws'/v} p_{n}^{\nu_{n}}(t', \rho_{n}, x)
    \bigr|\, 
    \dot{\pi}(dx dw dv).
  \end{equation}
  By the triangle inequality,
  we have that 
  \begin{align} \label{5. eq: sub-aging functions, eq 2}
    &\bigl|  
      e^{-ws/v} p_{n}^{\nu_{n}}(t, \rho_{n}, x) - e^{-ws'/v} p_{n}^{\nu_{n}}(t', \rho_{n}, x) 
    \bigr|\\
    &\leq
    |p_{n}^{\nu_{n}}(t, \rho_{n}, x) - p_{n}^{\nu_{n}}(t', \rho_{n}, x)| 
    + 
    p_{n}^{\nu_{n}}(t', \rho_{n}, x) 
    |e^{-ws/v} - e^{-ws'/v}|\\
    &\leq
    \mathfrak{d}_{F_{n}}(p_{n}^{\nu_{n}}, l^{-1}, r, \delta) + \mathfrak{s}_{F_{n}}(p_{n}^{\nu_{n}}, l^{-1}, r) |e^{-ws/v} - e^{-ws'/v}|.
  \end{align}
  For $(x, w, v) \in D_{R_{n}}(\rho_{n}, r) \times [0, W] \times [\eta, \infty)$,
  we deduce by the mean value theorem that 
  \begin{equation}
    |e^{-ws/v} - e^{-ws'/v}| 
    \leq 
    \frac{w}{v} |s-s'| 
    \leq 
    \frac{W}{\eta} \delta.
  \end{equation}
  Hence, it follows that 
  \begin{align}
    A_{n}
    &\leq
    \bigl\{
      \mathfrak{d}_{F_{n}}(p_{n}^{\nu_{n}}, l^{-1}, r, \delta) + \mathfrak{s}_{F_{n}}(p_{n}^{\nu_{n}}, l^{-1}, r) W \eta^{-1} \delta
    \bigr\}
    \sum_{x_{i}^{(n)} \in D_{R_{n}(\rho_{n}, r)}} v_{i}^{(n)}\\
    &=
    \bigl\{
      \mathfrak{d}_{F_{n}}(p_{n}^{\nu_{n}}, l^{-1}, r, \delta) + \mathfrak{s}_{F_{n}}(p_{n}^{\nu_{n}}, l^{-1}, r) W \eta^{-1} \delta
    \bigr\}
    \nu_{n}(D_{R_{n}}(\rho_{n}, r)).
    \label{5. eq: sub-aging functions, eq 3}
  \end{align}
  Since the vague convergence $\nu_{n} \to \nu$ implies that $\sup_{n \geq 1} \nu_{n}(D_{R_{n}}(\rho_{n}, r)) < \infty$,
  we obtain \eqref{5. eq: pcpt of sub-aging functions, reduced equation} 
  from Proposition \ref{4. prop: precompactness of transition densities} and \eqref{5. eq: sub-aging functions, eq 3}.
\end{proof}

\begin{thm} \label{5. thm: sub-aging}
  Under Assumption \ref{5. assum: sub-aging},
  It holds that 
  \begin{equation}
    \bigl( F_{n}, R_{n}, \rho_{n}, \nu_{n}, \pi_{n}, P_{\rho_{n}}^{\nu_{n}}(X_{n}^{\nu_{n}} \in \cdot), \Psi_{n}^{\nu_{n}} \bigr) 
    \to 
    \bigl( F, R, \rho, \nu, \pi, P_{\rho}^{\nu}(X^{\nu} \in \cdot), \Psi^{\nu} \bigr)
  \end{equation}
  in the space $\rbcM(\disMeasFunct \times \markedMeasFunct{\RNp \times \RNpp} \times \probFunct(\SkorohodFunct) \times \fixedFunct{C(\RNpp^{2}, \RNp)})$.
  Moreover, $\Psi^{\nu}(s,t) > 0$ for any $s \geq 0$ and $t > 0$.
\end{thm}

\begin{proof}
  The second assertion is straightforward.
  Indeed, for some $i \in I$, we have that $P_{\rho}^{\nu}(X^{\nu}(t) = x_{i}) > 0$,
  which implies that 
  \begin{equation}
    \Psi^{\nu}(s,t) 
    \geq 
    e^{-w_{i}s/v_{i}} P_{\rho}^{\nu}(X^{\nu}(t) = x_{i})
    > 0.
  \end{equation}

  We next show the first assertion.
  By Lemma \ref{5. lem: precompactness of sub-aging functions},
  it suffices to show that $\Psi_{n}^{\nu_{n}}(s, t) \to \Psi^{\nu}(s, t)$ 
  for every $(s, t) \in \RNp \times \RNpp$.
  Fix $(s, t) \in \RNp \times \RNpp$.
  By Lemma \ref{5. lem: sub-aging fucntion, smaller parts},
  it is enough to show that, for some sequences $(r_{k})_{k \geq 1}$, $(W_{l})_{l \geq 1}$, and $(\eta_{m})_{m \geq 1}$
  with $r_{k} \uparrow \infty$, $W_{l} \uparrow \infty$, and $\eta_{m} \downarrow 0$,
  \begin{align} \label{5. eq: sub-aging proof eq}
    &\lim_{n \to \infty}
    \int_{D_{M}(\rho_{M}, r_{k}) \times [0, W_{l}] \times [\eta_{m}, \infty)} 
      e^{-ws/v} P_{\rho_{n}}^{\nu_{n}}(X_{n}^{\nu_{n}}(t) = x)\, 
    \dot{\pi}_{n}(dx dw dv) \\
    &=
    \int_{D_{M}(\rho_{M}, r_{k}) \times [0, W_{l}] \times [\eta_{m}, \infty)} 
      e^{-ws/v} P_{\rho}^{\nu}(X^{\nu}(t) = x)\, 
    \dot{\pi}(dx dw dv).
  \end{align}
  We choose $(r_{k})_{k \geq 1}$ so that $r_{k} \uparrow \infty$ and $\nu(\partial D_{M}(\rho_{M}, r_{k})) = 0$.
  We then choose $(W_{l})_{l \geq 1}$ and $(\eta_{m})_{m \geq 1}$ 
  so that $W_{l} \uparrow \infty$, $\eta_{m} \downarrow 0$,
   and $\dot{\pi}(\partial(D_{M}(\rho_{M}, r_{k}) \times [0, W_{l}] \times [\eta_{m}, \infty))) = 0$ for all $l, m$.
  Fix $k$, $l$, and $m$,
  and simply write $r =r_{k}$, $W = W_{l}$, and $\eta = \eta_{m}$.
  Set
  \begin{equation}
    V 
    \coloneqq 
    \left\{
      W^{(r)}(\pointMap(\nu)) \vee \sup_{n \geq 1} W^{(r)}(\pointMap(\nu_{n}))
    \right\} 
    + 1
    < \infty.
  \end{equation}
  Then, 
  there are no atoms of $\pi_{n}$ nor $\pi$ in $D_{M}(\rho_{M}, r) \times [0, W] \times (V, \infty)$.
  We define a finite subset $J \subseteq I$ so that 
  $\{(x_{j}, w_{j}, v_{j})\}_{j \in J}$ are the atoms of $\dot{\pi}$ lying in $K \coloneqq D_{M}(\rho_{M}, r) \times [0, W] \times [\eta, V]$.
  By Proposition \ref{2. prop: atom convergence from simple measure convergence},
  there exist injections $f_{n}: J \to I_{n}$ (at least for all sufficiently large $n$) such that
  \begin{equation}
    \At(\dot{\pi}_{n}) \cap K 
    = 
    \{(x_{f_{n}(j)}^{(n)}, w_{f_{n}(j)}^{(n)}, v_{f_{n}(j)}^{(n)})\}_{j \in J}
  \end{equation}
  and $(x_{f_{n}(j)}^{(n)}, w_{f_{n}(j)}^{(n)}, v_{f_{n}(j)}^{(n)}) \to (x_{j}, w_{j}, v_{j})$ for each $j \in J$.
  It then follows from Lemma \ref{5. lem: convergence of atoms and weights} that 
  \begin{equation}
    P_{\rho_{n}}^{\nu_{n}}(X_{n}^{\nu_{n}}(t) = x_{i}^{(n)}) \to P_{\rho}^{\nu}(X^{\nu}(t) = x_{i}),
    \quad 
    \forall j \in J.
  \end{equation}
  Therefore, we deduce that 
  \begin{align}
    \lefteqn{\lim_{n \to \infty}
    \int_{D_{M}(\rho_{M}, r) \times [0, W] \times [\eta, \infty)} 
      e^{-ws/v} P_{\rho_{n}}^{\nu_{n}}(X_{n}^{\nu_{n}}(t) = x)\, 
    \dot{\pi}_{n}(dx dw dv)} \\
    &=
    \lim_{n \to \infty}
    \int_{D_{M}(\rho_{M}, r) \times [0, W] \times [\eta, V]} 
      e^{-ws/v} P_{\rho_{n}}^{\nu_{n}}(X_{n}^{\nu_{n}}(t) = x)\, 
    \dot{\pi}_{n}(dx dw dv)\\
    &=
    \lim_{n \to \infty}
    \sum_{j \in J} 
      \exp\bigl( - w_{f_{n}(j)}^{(n)} s / v_{f_{n}(j)}^{(n)} \bigr)
      P_{\rho_{n}}^{\nu_{n}}(X_{n}^{\nu_{n}}(t) = x_{f_{n}(j)}^{(n)}) \\
    &=
    \sum_{j \in J} 
      e^{- w_{j} s/ v_{j}} 
      P_{\rho}^{\nu}(X^{\nu}(t) = x_{j}) \\
    &=
    \int_{D_{M}(\rho_{M}, r) \times [0, W] \times [\eta, \infty)} 
      e^{- ws/ v} P_{\rho}^{\nu}(X^{\nu}(t) = x)\, 
    \dot{\pi}(dx dw dv),
  \end{align}
  which completes the proof.
\end{proof}


\section{Proof of the main results} \label{sec: Proof of main results}

In this section, we prove the main results.
Since we already obtained the results in the previous section when the traps are deterministic and converge,
thanks to the Skorohod representation theorem,
it remains to show the convergence of traps in distribution.


\subsection{Proof of Theorem \ref{1. thm: aging for deterministic models}} \label{sec: Proof of aging for deterministic models}

We prove the aging result, Theorem \ref{1. thm: aging for deterministic models}.
Suppose that Assumption \ref{1. assum: aging, deterministic version} is satisfied.
We first study some properties of the limiting trap environment.
Recall that $\pi(dx dv)$ is the Poisson random measure with intensity measure $\mu(dx) \alpha v^{-1-\alpha} dv$
and $\nu = \MeasureMap(\pi)$,
where the map $\MeasureMap$ is introduced in Section \ref{sec: a space including the space of discrete measures} above.

\begin{lem} \label{6. lem: aging, properties of limiting trap}
  \begin{enumerate} [label = (\roman*)]
    \item \label{6. lem item: aging, nu is of full support}
      The set $\At(\nu)$ of the atoms of $\nu$ is dense in $F$ almost surely.
      In particular, $\nu$ is of full support almost surely.
    \item \label{6. lem item: aging, no multiplicity of pi}
      It holds that $\pi \in \pointMeas^{*}(F)$ almost surely (recall $\pointMeas^{*}(F)$ from \eqref{2. cor eq: image of Mdis by p}).
  \end{enumerate}
\end{lem}

\begin{proof}
  \ref{6. lem item: aging, nu is of full support}.
  Since $\mu$ is of full support,
  it holds that, for any $x \in F$ and $r>0$,
  \begin{equation}
    \mathsf{E}[ \pi(B_{R}(x, r) \times \RNpp)] 
    = 
    \mu(B_{R}(x, r)) \int_{0}^{\infty} \alpha u^{-1-\alpha}\, du 
    = \infty,
  \end{equation}
  where $\mathsf{E}$ denotes the expectation with respect to $\mathsf{P}$,
  the underlying probability measure of $\pi$.
  Hence, $\At(\nu) \cap B_{R}(x, r) \neq \emptyset$ almost surely.
  Let $\{x_{i}\}_{i \in I}$ be a countable dense subset of $F$.
  Then, almost surely, we have that $\At(\nu) \cap B_{R}(x_{i}, q) \neq \emptyset$ for all $i \in I$ and positive rational numbers $q$.
  This implies the desired result.

  \ref{6. lem item: aging, no multiplicity of pi}.
  Write $\pi = \sum_{i \in I} \delta_{(x_{i}, v_{i})}$.
  For each $\varepsilon > 0$,
  we set $I_{\varepsilon} \coloneqq \{i \in I \mid v_{i} \geq \varepsilon\}$.
  We then define a random measure $\nu'_{\varepsilon}$ on $F$ by setting 
  \begin{equation}
    \nu'_{\varepsilon}(A) 
    \coloneqq 
    \sum_{i \in I_{\varepsilon}} \delta_{x_{i}}(A)
    =
    \pi(A \times [\varepsilon, \infty)).
  \end{equation}
  It is easy to check that $\nu'_{\varepsilon}$ is a Poisson random measure with intensity measure $\varepsilon^{-\alpha} \mu(dx)$.
  Since $\mu$ is non-atomic,
  it follows from \cite[Lemma 3.6(i)]{Kallenberg_17_Random} that 
  $x_{i} \neq x_{j}$ for any $i, j \in I_{\varepsilon}$ with $i \neq j$, almost surely.
  Since $I_{\varepsilon}$ increases to $I$ as $\varepsilon \to 0$,
  we deduce the desired result.
\end{proof}

For the following result,
recall the random variable $\xi$ from Definition \ref{1. dfn: random variable xi}.

\begin{lem} \label{6. lem: c_n gives proper scaling}
  For every $u > 0$, $\lim_{n \to \infty} b_{n} P_{\xi}(c_{n}^{-1} \xi > u) = u^{-\alpha}$.
\end{lem}

\begin{proof}
  Define $f(u) \coloneqq P_{\xi}(\xi > u)^{-1}$ and $g(u) \coloneqq \inf\{ s \geq 0 \mid f(s) > u\}$.
  Then, $f$ is a regularly varying function with index $\alpha$ and $g$ is an asymptotic inverse of $f$,
  that is, $f(g(u)) \sim g(f(u)) \sim u$ as $u \to \infty$ \cite[Theorem 1.5.12]{Bingham_Goldie_Teugels_87_Regular}.
  (NB.\ For functions $h$ and $H$, we write $h \sim H$ when $h(u)/H(u) \to 1$.)
  Using the relation $c_{n} = g(b_{n})$,
  we deduce that
  \begin{equation}
    P_{\xi}(c_{n}^{-1} \xi > u) 
    = 
    (u c_{n})^{-\alpha} \ell(u)
    =
    u^{-\alpha} \frac{\ell(c_{n}u)}{\ell(c_{n})} P_{\xi}(\xi > c_{n})
    =
    u^{-\alpha} \frac{\ell(c_{n}u)}{\ell(c_{n})} f(g(b_{n}))^{-1}.
  \end{equation}
  Since $\ell$ is slowly varying and $c_{n} \to \infty$,
  we obtain the desired result.
\end{proof}

We now prove the distributional convergence of traps in the vague-and-point-process topology.

\begin{lem} \label{6. lem: aging, convergence of traps}
  It holds that 
  \begin{equation}
    \bigl( V_{n}, a_{n}^{-1}R_{n}, \rho_{n}, b_{n}^{-1} \muh_{n}, \mathsf{P}_{n}(c_{n}^{-1} \nu_{n} \in \cdot) \bigr)
    \to 
    \bigl( F, R, \rho, \mu, \mathsf{P}(\nu \in \cdot) \bigr)
  \end{equation}
  in the space $\rbcM(\MeasFunct \times \probFunct(\disMeasFunct))$.
\end{lem}

\begin{proof}
  By Assumption \ref{1. assum: aging, deterministic version}\ref{1. assum item: deterministic, convergence of spaces} 
  and Theorem \ref{3. thm: convergence in GH topology},
  we may assume that $(V_{n}, a_{n}^{-1} R_{n}, \rho_{n})$ and $(F, R, \rho)$ are embedded isometrically
  into a common rooted boundedly-compact metric space $(M, d^{M}, \rho_{M})$ 
  in such a way that 
  $\rho_{n} = \rho = \rho_{M}$ as elements in $M$,
  $V_{n} \to F$ in the local Hausdorff topology as closed subsets of $M$,
  and $b_{n}^{-1} \muh_{n} \to \mu$ vaguely as measures on $M$. 
  It suffices to show that $c_{n}^{-1} \nu_{n} \xrightarrow{\mathrm{d}} \nu$ in the vague-and-point-process topology as measures on $M$.
  Recall that $\pi_{n} \coloneqq \sum_{x \in V_{n}} \delta_{(x, c_{n}^{-1} \nu_{n}(\{x\}))}$.
  Fix a bounded subset $A$ of $M$ such that $\mu(\partial A) = 0$.
  We have that, for every $u > 0$,
  \begin{equation}
    \mathsf{P}_{n} (\pi_{n}(A \times (u, \infty)) = 0)
    =
    \mathsf{P}_{n} (c_{n}^{-1} \xi_{x}^{(n)} \leq u\ \text{for all}\ x \in A)
    =
    \bigl( 1 - P_{\xi}(c_{n}^{-1} \xi > u) \bigr)^{\muh_{n}(A)}
  \end{equation}
  and $\mathsf{E}_{n} [ \pi_{n}(A \times (u, \infty))] = \muh_{n}(A) P_{\xi}(c_{n}^{-1} \xi > u)$.
  It then follows from Lemma \ref{6. lem: c_n gives proper scaling} that 
  \begin{gather}
    \lim_{n \to \infty}
    \mathsf{P}_{n} (\pi_{n}(A \times (u, \infty)) = 0)
    =
    e^{-\mu(A) u^{-\alpha}} 
    =
    \mathsf{P}(\pi(A \times (u, \infty)) = 0), 
    \label{6. lem eq: prob of 0 atoms}\\
    \lim_{n \to \infty} 
    \mathsf{E}_{n}[\pi_{n}(A \times (u, \infty))] 
    =
    \mu(A) u^{-\alpha}
    =
    \mathsf{E}[\pi(A \times (u, \infty))].
  \end{gather}
  By \cite[Theorem 4.18]{Kallenberg_17_Random},
  $\pi_{n} \xrightarrow{\mathrm{d}} \pi$ vaguely.
  We will check condition \ref{2. thm item: tightness, tightness of W} of Theorem \ref{2. thm: tightness in sP}.
  Fix $r>0$ satisfying $\mu(M^{(r)}) = 0$.
  Since we have that 
  \begin{equation}
    p(n, l)
    \coloneqq
    \mathsf{P}_{n} (W^{(r)}(\pi_{n}) > l)
    =
    \mathsf{P}_{n} ( \pi_{n}(M^{(r)} \times (l, \infty)) > 0 ),
  \end{equation}
  equation \eqref{6. lem eq: prob of 0 atoms} yields that 
  $p(n, l) \to \mathsf{P}(\pi(M^{(r)} \times (l, \infty)) > 0) = 1 - e^{-\mu(M^{(r)}) l^{-\alpha}}$.
  Hence, we obtain that $\lim_{l \to \infty} \limsup_{n \to \infty} p(n, l) = 0$.
  The Markov inequality yields that 
  \begin{equation}
    \mathsf{P}_{n}( M_{\varepsilon}^{(r)}(\pi_{n}) > \delta)
    \leq 
    \delta^{-1} 
    \muh_{n}(M^{(r)}) E_{\xi} [ c_{n}^{-1} \xi \cdot 1_{(0, c_{n}\varepsilon)}(\xi)]
    =
    -
    \delta^{-1} 
    \muh_{n}(M^{(r)})
    c_{n}^{-1}
    \int_{0}^{\varepsilon c_{n}} u\, df(u),
  \end{equation}
  where we set $f(u) \coloneqq P(\xi > u)$.
  Since $f$ is a regularly varying function with index $-\alpha$,
  we have by \cite[Theorem 1.6.4]{Bingham_Goldie_Teugels_87_Regular} that, as $n \to \infty$,
  \begin{equation}
    \int_{0}^{\varepsilon c_{n}} u\, df(u) 
    \sim 
    \frac{-\alpha}{1-\alpha} \varepsilon c_{n} f(\varepsilon c_{n})
    = 
    \frac{-\alpha}{1-\alpha} \varepsilon c_{n} P_{\xi}(c_{n}^{-1}\xi > \varepsilon).
  \end{equation}
  It then follows from Lemma \ref{6. lem: c_n gives proper scaling} that 
  \begin{equation}
    \lim_{\varepsilon \downarrow 0} 
    \limsup_{n \to \infty}
    \mathsf{P}_{n}( M_{\varepsilon}^{(r)}(\pi_{n}) > \delta)
    \leq
    \lim_{\varepsilon \downarrow 0} 
    \delta^{-1}
    \alpha (1-\alpha)^{-1} \mu(M^{(r)}) \varepsilon^{1-\alpha}
    = 
    0.
  \end{equation}
  By Corollary \ref{2. cor: convergence in distribution in sP},
  $\pi_{n} \xrightarrow{\mathrm{d}} \pi$ in $\pointMeas(M)$.
  It follows from Corollary \ref{2. cor: convergence in VPP from that in sP} 
  and Lemma \ref{6. lem: aging, properties of limiting trap}\ref{6. lem item: aging, no multiplicity of pi}
  that $c_{n}^{-1}\nu_{n} = \MeasureMap(\pi_{n}) \xrightarrow{\mathrm{d}} \MeasureMap(\pi) = \nu$ in the vague-and-point-process topology.
\end{proof}

Now, we prove Theorem \ref{1. thm: aging for deterministic models}.

\begin{proof} [Proof of Theorem \ref{1. thm: aging for deterministic models}]
  By Lemma \ref{6. lem: aging, convergence of traps},
  we may assume that $(V_{n}, a_{n}^{-1} R_{n}, \rho_{n})$ and $(F, R, \rho)$ are embedded isometrically
  into a common rooted boundedly-compact metric space $(M, d^{M}, \rho_{M})$ 
  in such a way that 
  $\rho_{n} = \rho = \rho_{M}$ as elements in $M$,
  $V_{n} \to F$ in the local Hausdorff topology as closed subsets of $M$,
  $b_{n}^{-1} \muh_{n} \to \mu$ vaguely as measures on $M$,
  and $c_{n}^{-1}\nu_{n} \xrightarrow{\mathrm{d}} \nu$ in $\disMeasures(M)$. 
  Using the Skorohod representation theorem,
  we may further assume that $c_{n}^{-1}\nu_{n} \to \nu$ almost surely on some probability space.
  Then, by \cite[Theorem 1.2]{Croydon_18_Scaling},
  we obtain that $P_{\rho_{n}}^{\nu_{n}}(\tilde{X}_{n}^{\nu_{n}} \in \cdot) \xrightarrow{\mathrm{d}} P_{\rho}^{\nu}(X^{\nu} \in \cdot)$.
  Hence, the desired result follows from Theorem \ref{5. thm: aging}.
\end{proof}


\subsection{Proof of Theorem \ref{1. thm: sub-aging for deterministic models}} \label{sec: Proof of sub-aging for deterministic models}

Next, we prove the sub-aging result, Theorem \ref{1. thm: sub-aging for deterministic models}.
Suppose that Assumption \ref{1. assum: sub-aging, deterministic version} is satisfied.
Recall that 
$\dot{\pi}(dx dw dv)$ is a Poisson random measure on $F \times \RNp \times \RNpp$
with the intensity measure $\dot{\mu}(dx dw) \alpha v^{-1-\alpha}\, dv$
and $\nu$ is a random measure on $F$ defined by 
\begin{equation} \label{6. eq: sub-aging, dfn of nu}
  \nu(A)
  \coloneqq
  \int 1_{A}(x)v\, \dot{\pi}(dx dw dv),
  \quad 
  \forall A \in \Borel(F).
\end{equation}
For convenience, we write $\pi$ for the pushforward measure of $\dot{\pi}$ 
by the map $(x, w, v) \mapsto (x, v)$.
It is easy to show that $\pi$ is a Poisson random measure with intensity $\mu(dx) \alpha v^{-1-\alpha} dv$
and $\nu = \MeasureMap(\pi)$.

\begin{lem} \label{6. lem: sub-aging, properties of limiting trap}
  \begin{enumerate} [label = (\roman*)]
    \item \label{6. lem item: sub-aging, nu is of full support}
      The random measure $\nu$ is of full support almost surely.
    \item \label{6. lem item: sub-aging, no multiplicity of pi}
      It holds that $\pi \in \pointMeas^{*}(F)$ almost surely.
  \end{enumerate}
\end{lem}

\begin{proof}
  These results are proven similarly to Lemma \ref{6. lem: aging, properties of limiting trap}.
\end{proof}

We set 
\begin{equation}
  \dot{\pi}_{n} 
  \coloneqq 
  \sum_{x \in V_{n}} \delta_{(x, \mu_{n}(x), c_{n}^{-1} \nu_{n}(\{x\}))}.
\end{equation}
The following result is a version of Lemma \ref{6. lem: aging, convergence of traps} for the sub-aging result.

\begin{lem} \label{6. lem: sub-aging, trap convergence}
  It holds that 
  \begin{equation} \label{7. prop eq: joint convergence of X and L}
    \bigl( V_{n}, a_{n}^{-1}R_{n}, \rho_{n}, b_{n}^{-1} \dotmuh_{n}, 
      \mathsf{P}_{n}((\nu_{n}, \dot{\pi}_{n}) \in \cdot) \bigr)
    \to 
    \bigl( F, R, \rho, \dot{\mu}, \mathsf{P}((\nu, \dot{\pi}) \in \cdot) \bigr),
  \end{equation}
  in the space $\rbcM(\markedMeasFunct{\RNp} \times \probFunct(\disMeasFunct \times \markedMeasFunct{\RNp \times \RNpp}))$.
\end{lem}

\begin{proof}
  Under Assumption \ref{1. assum: sub-aging, deterministic version},
  we may assume that $(V_{n}, a_{n}^{-1} R_{n}, \rho_{n})$ and $(F, R, \rho)$ are embedded isometrically
  into a common rooted boundedly-compact metric space $(M, d^{M}, \rho_{M})$ 
  in such a way that 
  $\rho_{n} = \rho = \rho_{M}$ as elements in $M$,
  $V_{n} \to F$ in the local Hausdorff topology as closed subsets of $M$,
  and $b_{n}^{-1} \dotmuh_{n} \to \dot{\mu}$ vaguely as measures on $M$. 
  From the vague convergence $b_{n}^{-1} \dotmuh_{n} \to \dot{\mu}$ and the continuous mapping theorem,
  we have that $b_{n}^{-1} \muh_{n} \to \mu$ vaguely.
  Therefore,
  by following the proof of Lemma \ref{6. lem: aging, convergence of traps},
  we deduce that 
  $c_{n}^{-1} \nu_{n} \xrightarrow{\mathrm{d}} \nu$ in the vague-and-point-process topology.

  We next prove that $\dot{\pi}_{n} \to \dot{\pi}$ vaguely as measures on $M \times \RNp \times \RNpp$
  in the same way as before.
  Fix a bounded subset $A$ and a $u_{1} > 0$ such that $\mu(\partial (A \times (u_{1}, \infty))) = 0$.
  Recall from \eqref{1. eq: mark map} that $\psi_{n}(x) \coloneqq \psi_{G_{n}}(x) = (x, c_{n}(x))$ 
  and that $\dotmuh_{n} = \muh_{n} \circ \psi_{n}^{-1}$.
  We have that, for every $u_{2} > 0$,
  \begin{align}
    \mathsf{P}_{n} 
    \bigl(
      \dot{\pi}_{n}(A \times (u_{1}, \infty) \times (u_{2}, \infty)) = 0
    \bigr)
    &=
    \mathsf{P}_{n} 
    \bigl(
      c_{n}^{-1} \xi_{x}^{(n)} \leq u_{2}\ \text{for all}\ x \in \psi_{n}^{-1}(A \times (u_{1}, \infty))
    \bigr) \\
    &=
    \bigl( 1 - P_{\xi}(c_{n}^{-1} \xi > u_{2}) \bigr)^{\dotmuh_{n}(A \times (u_{1}, \infty))}
  \end{align}
  and 
  $\mathsf{E}_{n} [ \dot{\pi}_{n}(A \times (u_{1}, \infty) \times (u_{2}, \infty))] 
  = \dotmuh_{n}(A \times (u_{1}, \infty)) P_{\xi}(c_{n}^{-1} \xi > u_{2})$.
  It follows from Lemma \ref{6. lem: c_n gives proper scaling} 
  and the convergence $\dotmuh_{n}(A \times (u_{1}, \infty)) \to \dot{\mu}(A \times (u_{1}, \infty))$ that 
  \begin{gather}
    \lim_{n \to \infty}
    \mathsf{P}_{n} 
    \bigl(
      \dot{\pi}_{n}(A \times (u_{1}, \infty) \times (u_{2}, \infty)) = 0
    \bigr)
    =
    \mathsf{P}(\dot{\pi}(A \times (u_{1}, \infty) \times (u_{2}, \infty)) = 0), 
    \label{6. lem eq: sub-aging, prob of 0 atoms}\\
    \lim_{n \to \infty} 
    \mathsf{E}_{n}[\dot{\pi}_{n}(A \times (u_{1}, \infty) \times (u_{2}, \infty))] 
    =
    \mathsf{E}[\dot{\pi}(A \times (u_{1}, \infty) \times (u_{2}, \infty))].
  \end{gather}
  Thus, by \cite[Thorem 4.18]{Kallenberg_17_Random},
  we obtain that $\dot{\pi}_{n} \xrightarrow{\mathrm{d}} \dot{\pi}$ vaguely.

  From the above arguments,
  it is the case that 
  $(c_{n}^{-1} \nu_{n}, \dot{\pi}_{n})_{n \geq 1}$ is tight as a sequence of random elements of $\disMeasures(M) \times \Meas(M \times \RNp \times \RNpp)$.
  Let $(n_{k})_{k \geq 1}$ be a subsequence such that 
  $(c_{n_{k}}^{-1}\nu_{n_{k}}, \dot{\pi}_{n_{k}})$ converges to some random element $(\nu', \dot{\pi}')$.
  We then have that $\dot{\pi}' \stackrel{\mathrm{d}}{=} \dot{\pi}$.
  Using the Skorohod representation theorem,
  we may assume that $(c_{n_{k}}^{-1}\nu_{n_{k}}, \dot{\pi}_{n_{k}}) \to (\nu', \dot{\pi}')$
  almost surely on some probability space.
  Define $p: M \times \RNp \times \RNpp \to M \times \RNpp$ by supposing $(x, w, v) \mapsto (x, v)$.
  Let $\nu' = \sum_{i \in I} v_{i} \delta_{x_{i}}$ and $\dot{\pi}' = \sum_{j \in J} \delta_{(x'_{j}, w'_{j}, v'_{j})}$ 
  be the atomic decompositions of $\nu'$ and $\dot{\pi}'$, respectively.
  The continuous mapping theorem yields that 
  \begin{equation}
    \pointMap(c_{n_{k}}^{-1} \nu_{n_{k}})
    =
    \sum_{x \in V_{n_{k}}} \delta_{(x, c_{n_{k}}^{-1} \nu_{n_{k}}(\{x\}))}
    =
    \dot{\pi}_{n_{k}} \circ p^{-1}
    \to 
    \dot{\pi}' \circ p^{-1}
    =
    \sum_{j \in J} \delta_{(x'_{j}, v'_{j})}.
  \end{equation}
  On the other hand,
  from the convergence $c_{n_{k}}^{-1} \nu_{n_{k}} \to \nu'$ in the vague-and-point-process topology and Theorem \ref{2. thm: convergence in Mdis},
  we have that $\pointMap(c_{n}^{-1} \nu_{n}) \to \pointMap(\nu') = \sum_{i \in I} \delta_{(x_{i}, v_{i})}$ vaguely.
  Hence,
  we obtain that $\sum_{i \in I} \delta_{(x_{i}, v_{i})} = \sum_{j \in J} \delta_{(x'_{j}, v'_{j})}$,
  which implies that $\{(x_{i}, v_{i}) \mid i \in I\} = \{(x'_{j}, v'_{j}) \mid j \in J\}$.
  This yields that, for any $A \in \Borel(F)$,
  \begin{equation}
    \nu'(A)
    = 
    \sum_{i \in I} v_{i} \delta_{x_{i}}(A) 
    =
    \sum_{j \in J} v_{j}' \delta_{x'_{j}}(A) 
    = 
    \int 1_{A}(x) v\, \dot{\pi}'(dx dw dv).
  \end{equation}
  Since $\dot{\pi} \stackrel{\mathrm{d}}{=} \dot{\pi}'$ and $\nu$ is given by \eqref{6. eq: sub-aging, dfn of nu},
  we obtain that $(\nu', \dot{\pi}') \stackrel{\mathrm{d}}{=} (\nu, \dot{\pi})$.
  Therefore, we deduce that $(c_{n}^{-1} \nu_{n}, \dot{\pi}_{n}) \xrightarrow{\mathrm{d}} (\nu, \dot{\pi})$,
  which completes the proof.
\end{proof}

Using Theorem \ref{5. thm: sub-aging} and Lemma \ref{6. lem: sub-aging, trap convergence},
Theorem \ref{1. thm: sub-aging for deterministic models} is proven similarly to Theorem \ref{1. thm: aging for deterministic models}
as follows.

\begin{proof} [Proof of Theorem \ref{1. thm: sub-aging for deterministic models}]
  By Lemma \ref{6. lem: sub-aging, trap convergence},
  we may assume that $(V_{n}, a_{n}^{-1} R_{n}, \rho_{n})$ and $(F, R, \rho)$ are embedded isometrically
  into a common rooted boundedly-compact metric space $(M, d^{M}, \rho_{M})$ 
  in such a way that 
  $\rho_{n} = \rho = \rho_{M}$ as elements in $M$,
  $V_{n} \to F$ in the local Hausdorff topology as closed subsets of $M$,
  and $b_{n}^{-1} \dotmuh_{n} \to \dot{\mu}$ vaguely as measures on $M \times \RNp$,
  $\mathsf{P}_{n}((c_{n}^{-1}\nu_{n}, \dot{\pi}_{n}) \in \cdot) \to \mathsf{P}((\nu, \dot{\pi}) \in \cdot)$
  as probability measures on $\disMeasures(M) \times \Meas(M \times \RNp \times \RNpp)$.
  By the Skorohod representation theorem,
  we may further assume that 
  $(c_{n}^{-1}\nu_{n}, \dot{\pi}_{n}) \to (\nu, \dot{\pi})$ almost surely on some probability space.
  Applying \cite[Theorem 1.2]{Croydon_18_Scaling},
  we obtain that $P_{n}^{\nu_{n}}(\tilde{X}_{n}^{\nu_{n}} \in \cdot) \to P_{\rho}^{\nu}(X^{\nu} \in \cdot)$
  as probability measures on $D(\RNp, M)$.
  Hence,
  from Theorem \ref{5. thm: aging} and \ref{5. thm: sub-aging},
  the desired result follows.
\end{proof}


\subsection{Proof of Theorems \ref{1. thm: aging for random models} and \ref{1. thm: sub-aging for random models}} \label{sec: Proof of aging and sub-aging for random models}

Finally, we prove the aging and sub-aging results for random electrical networks, Theorems \ref{1. thm: aging for random models} and \ref{1. thm: sub-aging for random models}.
We start with a basic result regarding the Prohorov metric (recall this from \eqref{2. eq: dfn of Prohorov metric}).

\begin{lem} \label{6. lem: Prohorov distance estimate}
  Let $X$ and $Y$ be random elements of a separable metric space $(S, d^{S})$ defined on a common probability space with probability measure $P$.
  Suppose that on an event $E$, we have that $d^{S}(X, Y) \leq \varepsilon$ almost surely.
  Then, it holds that 
  \begin{equation}
    d_{P}^{S}
    \bigl(
      P(X \in \cdot), P(Y \in \cdot)
    \bigr)
    \leq 
    P(E^{c}) + \varepsilon.
  \end{equation}
\end{lem}

\begin{proof}
  Recall the definition of the $\varepsilon$-neighborhood $\neighb{\cdot}{\varepsilon}$ from \eqref{2. eq: dfn of closed neighborhood}.
  Fix a Borel subset $A$ of $S$.
  We deduce that 
  \begin{align}
    P(X \in A) 
    \leq
    P(X \in A,\, d^{S}(X, Y) \leq \varepsilon) + P(E^{c})
    \leq 
    P(Y \in \neighb{A}{\varepsilon}) + P(E^{c}),
  \end{align}
  and similarly, $P(Y \in A) \leq P(X \in \neighb{A}{\varepsilon}) + P(E^{c})$.
  Hence, we obtain the desired result.
\end{proof}

Thanks to the Skorohod representation theorem,
Theorem \ref{1. thm: aging for random models} (resp.\ Theorem \ref{1. thm: sub-aging for random models}) is obtained similarly
to Theorem \ref{1. thm: aging for deterministic models} (resp.\ Theorem \ref{1. thm: sub-aging for deterministic models})
by showing a version of Lemma \ref{6. lem: aging, convergence of traps} (resp.\ Lemma \ref{6. lem: sub-aging, trap convergence}) 
for random electrical networks.
To do this, we will use the following technical results.
Recall the space $\Rspace$ from Section \ref{sec: recurrent resistance metric and auxiliary results}.

\begin{lem} \label{6. lem: distance between original and trace}
  Fix $(F, R, \rho, \nu) \in \Rspace$.
  Assume that, for some $\lambda > 1$, $r>0$, and $\delta > 0$,
  \begin{equation}
    R(\rho, B_{R}(\rho, r)^{c}) \geq \lambda, 
    \quad 
    \nu(B_{R}(\rho, 1)) \geq \delta.
  \end{equation}
  Write $(X^{\nu}, (P^{\nu}_{x})_{x \in F})$ (resp.\  $(X^{\nu^{(r)}}, (P^{\nu^{(r)}}_{x})_{x \in F^{(r)}})$)
  for the process associated with $(F,\allowbreak R,\allowbreak \nu)$ (resp.\ $(F^{(r)},\allowbreak R^{(r)},\allowbreak \rho^{(r)})$).
  It then holds that, for any $T > 0$,
  \begin{equation}
    d_{P}^{D(\RNp, F)}
    \bigl(
      P_{\rho}^{\nu}(X^{\nu} \in \cdot), P_{\rho}^{\nu^{(r)}}(X^{\nu^{(r)}} \in \cdot)
    \bigr)
    \leq 
    e^{-T} + \frac{4}{\lambda} + \frac{4T}{\delta(\lambda - 1)},
  \end{equation}
  where we recall that $d_{P}^{D(\RNp, F)}$ is the Prohorov metric on $\mathcal{P}(D(\RNp, F))$
  induced by the Skorohod metric $d_{J_{1}}$ on $D(\RNp, F)$ defined as \eqref{3. eq: dfn of J_1 metric}.
\end{lem}

\begin{proof}
  Let $\tau_{F^{(r)}}$ be the first exit time of $X$ from $F^{(r)}$.
  Fix $T>0$ and define an event $E \coloneqq \{\tau_{F^{(r)}} > T\}$.
  On the event $E$,
  by the definition of the trace of $X$
  (recall it from Section \ref{sec: recurrent resistance metric and auxiliary results}),
  we have that $X(t) = \tr_{F^{(r)}} X (t)$ for all $t < T$.
  Thus, 
  \begin{equation}
    d_{J_{1}}(X, \tr_{F^{(r)}} X) 
    \leq 
    \int_{T}^{\infty} e^{-s}\, ds
    = e^{-T}.
  \end{equation}
  Thus, the desired result follows 
  from Lemmas \ref{4. lem: trace result}, \ref{4. lem: exit time estimate}, and \ref{6. lem: Prohorov distance estimate}.
\end{proof}

\begin{lem} \label{6. lem: uniform lower bound for random traps}
  Under Assumption \ref{1. assum: aging, random version} or \ref{1. assum: sub-aging, random version}, it holds that 
  \begin{equation}
    \lim_{\delta \downarrow 0}
    \limsup_{n \to \infty}
    \mathbf{P}_{n} 
    \bigl(
      \mathsf{P}_{n}\bigl( c_{n}^{-1}\nu_{n}(B_{R_{n}}(\rho_{n}, a_{n})) < \delta \bigr) > \eta
    \bigr)
    = 0,
    \quad 
    \forall \eta > 0.
  \end{equation}
\end{lem}

\begin{proof}
  Suppose that Assumption \ref{1. assum: aging, random version} (resp.\ Assumption \ref{1. assum: sub-aging, random version}) is satisfied.
  Using the Skorohod representation theorem,
  we may assume that 
  the convergence in Assumption \ref{1. assum: aging, random version}\ref{1. assum item: random, convergence of spaces}
  (resp.\ the convergence \eqref{1. assum eq: sub-aging, convergence for random version})
  takes place almost surely on some probability space $(\Omega, \mathcal{G}, \mathbb{P})$.
  Fix $\omega \in \Omega$.
  From Lemma \ref{6. lem: aging, convergence of traps} (resp.\ Lemma \ref{6. lem: sub-aging, trap convergence}),
  we have that 
  \begin{equation}
    \bigl( V_{n}, a_{n}^{-1}R_{n}, \rho_{n}, b_{n}^{-1} \muh_{n}, \mathsf{P}_{n}(c_{n}^{-1} \nu_{n} \in \cdot) \bigr)
    \to 
    \bigl( F, R, \rho, \mu, \mathsf{P}(\nu \in \cdot) \bigr)
  \end{equation}
  in $\rbcM(\MeasFunct \times \probFunct(\disMeasFunct))$.
  Then, by Theorem \ref{3. thm: convergence in GH topology}, we may assume that 
  $(V_{n}, a_{n}^{-1}R_{n}, \rho_{n})$ and $(F, R, \rho)$ are embedded 
  isometrically into a common boundedly-compact metric space $(M, d^{M}, \rho_{M})$
  in such a way that 
  $\rho_{n} = \rho = \rho_{M}$ as elements in $M$,
  $V_{n} \to F$ in the local Hausdorff topology,
  $b_{n}^{-1} \muh_{n} \to \mu$ vaguely,
  and $c_{n}^{-1}\nu_{n} \xrightarrow{\mathrm{d}} \nu$ in the vague-and-point-process topology.
  Then, for some $r \in (0,1)$, we have that 
  $c_{n}^{-1}\nu_{n}(B_{R_{n}}(\rho_{n}, a_{n}r)) \xrightarrow{\mathrm{d}} \nu(B_{R}(\rho, r))$
  (see \cite[Theorem 4.11]{Kallenberg_17_Random}).
  It is then the case that, for all but countably many $\delta>0$,
  \begin{equation}
    \lim_{n \to \infty}
    \mathsf{P}_{n}\bigl( c_{n}^{-1}\nu_{n}(B_{R_{n}}(\rho_{n}, a_{n} r)) < \delta \bigr)
    =
    \mathsf{P}\bigl( \nu(B_{R}(\rho, r)) < \delta \bigr).
  \end{equation}
  Since $\nu$ is of full support with probability $1$,
  we deduce that 
  \begin{align}
    \lim_{\delta \downarrow 0} 
    \limsup_{n \to \infty} 
    \mathsf{P}_{n}\bigl( c_{n}^{-1}\nu_{n}(B_{R_{n}}(\rho_{n}, a_{n})) < \delta \bigr) 
    \leq 
    \lim_{\delta \downarrow 0} 
    \limsup_{n \to \infty} 
    \mathsf{P}_{n}\bigl( c_{n}^{-1}\nu_{n}(B_{R_{n}}(\rho_{n}, a_{n}r)) < \delta \bigr) 
    = 0.
  \end{align}
  Hence, by (reverse) Fatou's lemma, we obtain that 
  \begin{align}
    &\lim_{\delta \downarrow 0}
    \limsup_{n \to \infty}
    \mathbf{P}_{n} 
    \bigl(
      \mathsf{P}_{n}\bigl( c_{n}^{-1}\nu_{n}(B_{R_{n}}(\rho_{n}, a_{n})) < \delta \bigr) > \eta
    \bigr) \\
    &\leq
    \mathbb{P} 
    \bigl(
      \lim_{\delta \downarrow 0} 
      \limsup_{n \to \infty} 
      \mathsf{P}_{n}\bigl( c_{n}^{-1}\nu_{n}(B_{R_{n}}(\rho_{n}, a_{n})) < \delta \bigr) > \eta
    \bigr) \\
    &= 0,
  \end{align}
  which completes the proof.
\end{proof}

The following results correspond to Lemmas \ref{6. lem: aging, convergence of traps} and \ref{6. lem: sub-aging, trap convergence}
for random electrical networks.

\begin{lem} \label{6. lem: aging, convergence of traps for random electrical networks}
  Under Assumption \ref{1. assum: aging, random version}, it holds that 
  \begin{equation} \label{6. lem eq: aging, convergence of traps for random electrical networks}
    \Bigl( V_{n}, a_{n}^{-1}R_{n}, \rho_{n}, b_{n} \muh_{n}, \mathsf{P}_{n}\bigl( (c_{n}^{-1}\nu_{n}, \tilde{\law}_{n}^{\nu_{n}}) \in \cdot \bigr) \Bigr)
    \xrightarrow{\mathrm{d}}
    \Bigl( F, R, \rho, \mu, \mathsf{P}\bigl( (\nu, \law^{\nu}) \in \cdot \bigr) \Bigr)
  \end{equation}
  in the space $\rbcM(\MeasFunct \times \probFunct(\disMeasFunct \times \probFunct(\SkorohodFunct)))$.
\end{lem}

\begin{proof}
  This is proven similarly to Lemma \ref{6. lem: sub-aging, convergence of traps for random electrical networks} below.
  So, we omit the proof.
\end{proof}

\begin{lem} \label{6. lem: sub-aging, convergence of traps for random electrical networks}
  Under Assumption \ref{1. assum: sub-aging, random version}, it holds that 
  \begin{equation} \label{6. lem eq: sub-aging, convergence of traps for random electrical networks}
    \Bigl( V_{n}, a_{n}^{-1}R_{n}, \rho_{n}, b_{n} \dotmuh_{n}, 
      \mathsf{P}_{n}\bigl( (c_{n}^{-1}\nu_{n}, \dot{\pi}_{n}, \tilde{\law}_{n}^{\nu_{n}}) \in \cdot \bigr) \Bigr)
    \xrightarrow{\mathrm{d}}
    \Bigl( F, R, \rho, \dot{\mu}, \mathsf{P}\bigl( (\nu, \dot{\pi}, \law^{\nu}) \in \cdot \bigr) \Bigr)
  \end{equation}
  in the space $\rbcM(\markedMeasFunct{\RNp \times \RNpp} \times \probFunct(\disMeasFunct \times \markedMeasFunct{\RNp \times \RNpp} \times \probFunct(\SkorohodFunct)))$.
\end{lem}

\begin{proof}
  For simplicity, we write 
  \begin{equation}
    \rbcM_{1} \coloneqq \rbcM(\markedMeasFunct{\RNp \times \RNpp} \times \probFunct(\disMeasFunct \times \markedMeasFunct{\RNp \times \RNpp} \times \probFunct(\SkorohodFunct))).
  \end{equation}
  For each $r>0$,
  given $\nu_{n}$,
  we write $(X_{n}^{\nu_{n}^{(a_{n}r)}},\allowbreak (P_{x}^{\nu_{n}^{(a_{n}r)}})_{x \in V_{n}^{(a_{n}r)}})$
  for the process associated with $(V_{n}^{(a_{n}r)},\allowbreak R_{n}^{(a_{n}r)},\allowbreak \nu_{n}^{(a_{n}r)})$.
  Note that by Lemma \ref{4. lem: trace result},
  $X_{n}^{\nu_{n}^{(a_{n}r)}}$ has the same distribution as the trace of $X_{n}^{\nu_{n}}$ onto $V_{n}^{(a_{n}r)}$.
  For every $r>0$,
  we set $(\dotmuh_{n})^{\langle a_{n}r \rangle} = \dotmuh_{n}|_{V_{n}^{(a_{n}r)} \times \RNp}$,
  $\dot{\pi}_{n}^{\langle a_{n}r \rangle} \coloneqq \dot{\pi}_{n}|_{V_{n}^{(a_{n}r)} \times \RNp \times \RNpp}$, and  
  \begin{gather}
    \cX_{n} 
    \coloneqq 
    \Bigl( V_{n}, a_{n}^{-1}R_{n}, \rho_{n}, b_{n}^{-1} \dotmuh_{n}, 
      \mathsf{P}_{n}\bigl( (c_{n}^{-1}\nu_{n}, \dot{\pi}_{n}, \tilde{\law}_{n}^{\nu_{n}}) \in \cdot \bigr) \Bigr),\\
    \tilde{\cX}_{n}^{(r)}
    \coloneqq 
    \Bigl( V_{n}^{(a_{n}r)}, a_{n}^{-1}R_{n}^{(a_{n}r)}, \rho_{n}, b_{n}^{-1} (\dotmuh_{n})^{(a_{n}r)}, 
      \mathsf{P}_{n}\bigl( (c_{n}^{-1}\nu_{n}^{(a_{n}r)}, \dot{\pi}_{n}^{\langle a_{n}r \rangle}, 
        \tilde{\law}_{n}^{\nu_{n}^{(a_{n}r)}}) \in \cdot \bigr) \Bigr),
  \end{gather}
  where we note that 
  \begin{equation}
    \tilde{\law}_{n}^{\nu_{n}^{(a_{n}r)}} 
    \coloneqq 
    P_{\rho_{n}}^{\nu_{n}^{(a_{n}r)}}
    \Bigl( (X_{n}^{\nu_{n}^{(a_{n}r)}}(a_{n}b_{n}t))_{t \geq 0} \in \cdot \Bigr).
  \end{equation}
  Similarly, we set $\dot{\mu}^{\langle r \rangle} = \dot{\mu}|_{F^{(r)} \times \RNp}$,
   $\dot{\pi}^{\langle r \rangle} \coloneqq \dot{\pi}|_{F^{(r)} \times \RNp \times \RNpp}$, and 
  \begin{gather}
    \cX \coloneqq \Bigl( F, R, \rho, \dot{\mu}, \mathsf{P}\bigl( (\nu, \dot{\pi}, \law^{\nu}) \in \cdot \bigr) \Bigr),\\
    \cX^{(r)} \coloneqq \Bigl( F^{(r)}, R^{(r)}, \rho^{(r)}, \dot{\mu}^{\langle r \rangle}, 
      \mathsf{P}\bigl( (\nu^{(r)}, \dot{\pi}^{\langle r \rangle}, \law^{\nu^{(r)}}) \in \cdot \bigr) \Bigr).
  \end{gather}
  Define 
  \begin{equation}
    Q_{n}^{(r)} 
    \coloneqq 
    \int_{r}^{r+1} \mathbf{P}_{n}(\tilde{\cX}_{n}^{(s)} \in \cdot)\, ds,
    \quad
    Q^{(r)} 
    \coloneqq 
    \int_{r}^{r+1} \mathbf{P}(\cX^{(s)} \in \cdot)\, ds,
  \end{equation}
  which are probability measures on $\rbcM_{1}$.
  (NB. The measurability of integrands can be verified by a similar argument to \cite[Lemma 6.3]{Noda_pre_Convergence}.)
  Using the Skorohod representation theorem,
  we may assume that 
  the convergence \eqref{1. assum eq: sub-aging, convergence for random version}
  takes place almost surely on some probability space with probability measure $\mathbb{P}$.
  Fix a realization.
  For $r>0$ such that 
  \begin{equation}
    ( V_{n}^{(a_{n}r)}, a_{n}^{-1}R_{n}^{(a_{n}r)}, \rho_{n}, b_{n}^{-1} (\dotmuh_{n})^{\langle a_{n}r \rangle} )
    \to 
    (F^{(r)}, R^{(r)}, \rho^{(r)}, \dot{\mu}^{\langle r \rangle}),
  \end{equation}
  we have from Lemma \ref{6. lem: sub-aging, trap convergence} that 
  \begin{align}
    &\bigl( V_{n}^{(a_{n}r)}, a_{n}^{-1}R_{n}^{(a_{n}r)}, \rho_{n}^{(a_{n}r)}, b_{n}^{-1} (\dotmuh_{n})^{\langle a_{n}r \rangle}, 
      \mathsf{P}_{n}( (c_{n}^{-1} \nu_{n}^{(a_{n}r)}, \dot{\pi}_{n}^{\langle a_{n}r \rangle}) \in \cdot) \bigr) \\
    \to & 
    \bigl( F^{(r)}, R^{(r)}, \rho^{(r)}, \dot{\mu}^{\langle r \rangle}, 
      \mathsf{P}( (\nu^{(r)}, \dot{\pi}^{\langle r \rangle}) \in \cdot) \bigr).
  \end{align}
  Note that the rooted resistance metric spaces $(V_{n}^{(r)}, a_{n}^{-1}R_{n}^{(a_{n}r)}, \rho_{n}^{(a_{n}r)})$ 
  satisfy Assumption \ref{1. assum: aging, deterministic version}\ref{1. assum item: deterministic, the non-explosion condition}.
  Thus, following the proof of Theorem \ref{1. thm: aging for deterministic models},
  we obtain that $\tilde{\cX}_{n}^{(r)} \to \cX^{(r)}$ almost surely under $\mathbb{P}$.
  Hence, we deduce that $Q_{n}^{(r)} \to Q^{(r)}$ weakly for every $r>0$.

  Next, we estimate the Prohorov distance between $Q_{n}^{(r)}$ and $\mathbf{P}_{n}(\cX_{n} \in \cdot)$.
  Define an event  
  \begin{equation}
    E_{n}(s, \lambda, \delta, \eta)
    \coloneqq 
    \left\{
      a_{n}^{-1} R_{n}(\rho_{n}, B_{R_{n}}(\rho_{n}, a_{n}s)) \geq \lambda
    \right\}
    \cap 
    \left\{
      \mathsf{P}_{n}\bigl( c_{n}^{-1}\nu_{n}(B_{R_{n}}(\rho_{n}, a_{n})) < \delta \bigr) \leq \eta
    \right\}
  \end{equation}
  Fix $r, \delta, \eta > 0,\, s \in (r, r+1)$ and $\lambda >1$,
  and suppose that the event $E_{n}(s, \lambda, \delta, \eta)$ occurs.
  To estimate the distance between $\tilde{\cX}_{n}^{(s)}$ and $\cX_{n}$,
  we think that $(V_{n}^{(a_{n}r)}, a_{n}^{-1}R_{n}^{(a_{n}r)}, \rho_{n}^{(a_{n}r)})$ are embedded into $(V_{n}, a_{n}^{-1}R_{n}, \rho_{n})$ in the obvious way.
  On the event $\{c_{n}^{-1}\nu_{n}(B_{R_{n}}(\rho_{n}, a_{n})) \geq \delta\}$,
  we have from Lemma \ref{6. lem: distance between original and trace} that 
  \begin{equation}
    d_{P}^{D(\RNp, V_{n})} 
    \bigl( \tilde{\law}_{n}^{\nu_{n}^{(a_{n}s)}}, \tilde{\law}_{n}^{\nu_{n}}\bigr)
    \leq 
    e^{-T} + \frac{4}{\lambda} + \frac{4T}{\delta(\lambda - 1)}
    \eqqcolon 
    c(\lambda, \delta, T),
  \end{equation}
  where we note that the Skorohod metric on $D(\RNp, V_{n})$ is induced from the scaled metric $a_{n}^{-1} R_{n}$.
  By definition, it is easy to check that
  \begin{equation}
    d_{\disMeasures}^{V_{n}, \rho_{n}}(c_{n}^{-1}\nu_{n}^{(a_{n}s)}, c_{n}^{-1}\nu_{n}) \leq e^{-s},
    \quad 
    d_{V}^{V_{n} \times \RNp \times \RNpp, (\rho_{n}, 0, 1)}(\dot{\pi}_{n}^{\langle a_{n}s \rangle}, \dot{\pi}_{n}) \leq e^{-s}
  \end{equation}
  (recall these metrics from Sections \ref{sec: The vague metric} and \ref{sec: the space of discrete measures}).
  Thus,
  if we write $d_{P}$ for the Prohorov metric on $\disMeasures(V_{n}) \times \Meas(V_{n} \times \RNp \times \RNpp) \times \mathcal{P}(D(\RNp, V_{n}))$,
  then we deduce from Lemma \ref{6. lem: Prohorov distance estimate} that
  \begin{align}
    &d_{P}
    \Bigl(
      \mathsf{P}_{n}\bigl( (c_{n}^{-1}\nu_{n}^{(a_{n}s)}, \tilde{\law}_{n}^{\nu_{n}^{(a_{n}s)}}) \in \cdot \bigr), 
      \mathsf{P}_{n}\bigl( (c_{n}^{-1}\nu_{n}, \tilde{\law}_{n}^{\nu_{n}}) \in \cdot \bigr)
    \Bigr)\\
    &\leq 
    \mathsf{P}_{n}\bigl( c_{n}^{-1}\nu_{n}(B_{R_{n}}(\rho_{n}, a_{n})) < \delta \bigr) 
    + 
    e^{-s} 
    +
    c(\lambda, \delta, T)\\ 
    &\leq 
    \eta + e^{-s} + c(\lambda, \delta, T),
  \end{align}
  where the last inequality follows from the definition of $E_{n}(s, \lambda, \delta, \eta)$.
  By definition, 
  we have that 
  \begin{equation}
    \lHausdMet{V_{n}}{\rho_{n}}(V_{n}^{(a_{n}s)}, V_{n}) \leq e^{-s},
    \quad 
    d_{V}^{V_{n} \times \RNp, (\rho_{n}, 0)}( b_{n}^{-1} (\dotmuh_{n})^{\langle a_{n}s \rangle}, b_{n}^{-1} \dotmuh_{n}) \leq e^{-s}.
  \end{equation}
  Thus, on the event $E_{n}(s, \lambda, \delta, \eta)$,
  the distance between $\tilde{\cX}_{n}^{(s)}$ and $\cX_{n}$ 
  in $\rbcM_{1}$ 
  is bounded above by 
  $\eta + e^{-s} + c(\lambda, \delta, T)$.
  If we simply write $d'_{P}$ for the Prohorov metric
  on the space of the probability measures on $\rbcM_{1}$,
  then it follows from Lemma \ref{6. lem: Prohorov distance estimate} that 
  \begin{align}
    \lefteqn{d'_{P}
    \bigl(
      \mathbf{P}_{n}(\tilde{\cX}_{n}^{(s)} \in \cdot), \mathbf{P}_{n}(\cX_{n} \in \cdot)
    \bigr)} \\
    &\leq 
    \mathbf{P}_{n}
    \bigl(
      a_{n}^{-1} R_{n}(\rho_{n}, B_{R_{n}}(\rho_{n}, a_{n}s)^{c}) < \lambda
    \bigr)
    +
    \mathbf{P}_{n} 
    \bigl(
      \mathsf{P}_{n}\bigl( c_{n}^{-1}\nu_{n}(B_{R_{n}}(\rho_{n}, a_{n})) < \delta \bigr) > \eta
    \bigr) \\
    &\quad
    + \eta + e^{-s} + c(\lambda, \delta, T) \\ 
    &\leq 
    \mathbf{P}_{n}
    \bigl(
      a_{n}^{-1} R_{n}(\rho_{n}, B_{R_{n}}(\rho_{n}, a_{n}r)^{c}) < \lambda
    \bigr)
    +
    \mathbf{P}_{n} 
    \bigl(
      \mathsf{P}_{n}\bigl( c_{n}^{-1}\nu_{n}(B_{R_{n}}(\rho_{n}, a_{n})) < \delta \bigr) > \eta
    \bigr) \\
    &\quad
    + \eta + e^{-r} + c(\lambda, \delta, T) \\
    &\eqqcolon
    c'(n, r, \lambda, \delta, \eta, T).
  \end{align}
  where we use that $s > r$ at the second inequality.
  The above inequality implies that 
  \begin{equation}
    d'_{P}
    \bigl(
      Q_{n}^{(r)}, \mathbf{P}_{n}(\cX_{n} \in \cdot)
    \bigr)
    \leq 
    c'(n, r, \lambda, \delta, \eta, T).
  \end{equation}
  Combining this with Assumption \ref{1. assum: aging, random version}\ref{1. assum item: random, the non-explosion condition} and Lemma \ref{6. lem: uniform lower bound for random traps},
  by letting $n \to \infty$, $r \to \infty$, $\lambda \to \infty$, $T \to \infty$, $\delta \to 0$, and then $\eta \to 0$,
  we obtain that 
  \begin{equation}
    \lim_{r \to \infty} 
    \limsup_{n \to \infty} 
    d'_{P}
    \bigl(
      Q_{n}^{(r)}, \mathbf{P}_{n}(\cX_{n} \in \cdot)
    \bigr) 
    = 0.
  \end{equation}
  A similar argument yields that $Q^{(r)} \to \mathbf{P}(\cX \in \cdot)$ weakly.
  Recalling that $Q_{n}^{(r)} \to Q^{(r)}$ weakly,
  we deduce that $\cX_{n} \xrightarrow{\mathrm{d}} \cX$,
  which completes the proof.
\end{proof}

By Lemma \ref{6. lem: aging, convergence of traps for random electrical networks} (resp.\ Lemma \ref{6. lem: sub-aging, convergence of traps for random electrical networks}),
Theorem \ref{1. thm: aging for random models} (resp.\ \ref{1. thm: sub-aging for random models}) 
is proven similarly to Theorem \ref{1. thm: aging for deterministic models} (resp.\ Theorem \ref{1. thm: sub-aging for deterministic models})
as follows.

\begin{proof} [{Proof of Theorems \ref{1. thm: aging for random models} and \ref{1. thm: sub-aging for random models}}]
  We only give the proof of Theorem \ref{1. thm: aging for random models}.
  Theorem  \ref{1. thm: sub-aging for random models} is proven similarly
  by using Theorem \ref{5. thm: sub-aging} and Lemma \ref{6. lem: sub-aging, convergence of traps for random electrical networks}.
  By Lemma \ref{6. lem: aging, convergence of traps for random electrical networks} and the Skorohod representation theorem,
  we may assume that the convergence \eqref{6. lem eq: aging, convergence of traps for random electrical networks} takes place almost surely 
  on some probability space.
  Fix such a realization.
  Then, by Theorem \ref{3. thm: convergence in GH topology},
  it is possible to embed $(V_{n}, a_{n}^{-1}R_{n}, \rho_{n})$ and $(F, R, \rho)$ 
  isometrically into a common boundedly-compact metric space $(M, d^{M}, \rho_{M})$
  in such a way that 
  $\rho_{n} = \rho = \rho_{M}$ as elements in $M$,
  $V_{n} \to F$ in the local Hausdorff topology,
  $b_{n}^{-1} \muh_{n} \to \mu$ vaguely,
  and $(c_{n}^{-1}\nu_{n}, \tilde{\law}_{n}^{\nu_{n}}) \xrightarrow{\mathrm{d}} (\nu, \law^{\nu})$.
  Using the Skorohod representation theorem again,
  we may assume that $(c_{n}^{-1}\nu_{n}, \tilde{\law}_{n}^{\nu_{n}}) \to (\nu, \law^{\nu})$ almost surely.
  Then, from Theorem \ref{5. thm: aging}, we deduce the desired result.
\end{proof}


\section{Applications} \label{sec: applications}

In this section, we apply the main results to several examples.
Throughout this section,
for simplicity,
we assume that $\ell(u) = 1$ for all sufficiently large $u$
(recall the slowly varying function $\ell$ from Definition \ref{1. dfn: random variable xi}).


\subsection{The Sierpi\'{n}ski gasket} \label{sec: SG}

The Sierpi\'{n}ski gasket is a well-studied self-similar fractal.
For details on the probabilistic analysis of fractals see \cite{Barlow_98_Diffusions,Kigami_12_Resistance}, for example.
Here, we briefly confirm that our results apply to a sequence of electrical networks that converge to the Sierpi\'{n}ski gasket.

Set $\rho_{0} \coloneqq (0, 0) \in \RN^{2}$.
Let $V_{0} \coloneqq \{x_{1}, x_{2}, x_{3}\} \subseteq \RN^{2}$ consist of the vertices of an equilateral triangle of side length $1$ with $x_{1} = \rho_{0}$.
Write $\psi_{i}(x) \coloneqq (x+x_{i})/2$ for $i = 1,2,3$.
We then inductively define $V_{n} \coloneqq \bigcup_{i=1}^{3} \psi_{i}(V_{n-1})$.
The electrical network $G_{n} = (V_{n}, E_{n}, c_{n})$ is defined by setting 
$E_{n}$ to be the set of pairs of elements of $V_{n}$ at a Euclidean distance $2^{-n}$ apart 
and $c_{n}(x,y) \coloneqq 1$ if $\{x, y\} \in E_{n}$.
Write $\rho_{n} \coloneqq \rho_{0}$,
$R_{n}$ for the resistance metric on $V_{n}$ associated with $G_{n}$,
and $\muh_{n}$ for the counting measure on $V_{n}$.
Then, it is elementary to check that, for some $(F, R, \rho, \mu) \in \Rspace$,
\begin{equation}
  (V_{n}, (3/5)^{n} R_{n}, \rho, 3^{-n} \muh_{n}) 
  \to    
  (F, R, \rho ,\mu)
\end{equation}
in $\GHPspace$, i.e., the pointed Gromov-Hausdorff-Prohorov topology
(recall this from Section \ref{sec: functors used in this paper}).
The set $F$ is the Sierpi\'{n}ski gasket and $\mu$ is a self-similar measure corresponding to the $\log_{2} 3$ Hausdorff measure in the Euclidean metric.
Moreover,
since the degree of any vertex in $V_{n} \setminus V_{0}$ is $4$,
we deduce that 
\begin{equation}
  (V_{n}, (3/5)^{n} R_{n}, \rho_{n}, 3^{-n} \dotmuh_{n}) 
  \to    
  (F, R, \rho ,\mu \otimes \delta_{\{4\}}),
\end{equation}
which implies that $(G_{n})_{n \geq 1}$ satisfies Assumption \ref{1. assum: sub-aging, deterministic version}.
Thus, by Theorem \ref{1. thm: sub-aging for deterministic models},
we obtain the aging and sub-aging results as follows.
Let $\{(x_{i}, v_{i})\}_{i \in I}$ be a Poisson point process on $F \times \RNpp$
with intensity $\mu(dx) \alpha v^{-1-\alpha} dv$ 
and define $\nu(dx) \coloneqq \sum_{i \in I} v_{i} \delta_{x_{i}}(dx)$.
Given $\nu$, 
we write $(X^{\nu}, \{P_{x}^{\nu}\}_{x \in F})$ for the process associated with $(F, R, \nu)$.
We simply write $X_{n}^{\nu_{n}} \coloneqq X_{G_{n}}^{\nu_{G_{n}}}$,
which is the BTM on $G_{n}$ (see Definition \ref{1. dfn: BTM}).
We denote by $P_{\rho_{n}}^{\nu_{n}}$ for the underlying probability measure 
for $X_{n}^{\nu_{n}}$ started at $\rho_{n}$.
Set $\tilde{c}_{n} \coloneqq (5/3)^{n} \cdot 3^{n/\alpha}$.
As a consequence of Theorem \ref{1. thm: sub-aging for deterministic models},
we obtain that 
\begin{gather}
  P_{\rho_{n}}^{\nu_{n}} 
  \bigl( X_{n}^{\nu_{n}}(\tilde{c}_{n}s) = X_{n}^{\nu_{n}}( \tilde{c}_{n}t) \bigr) 
  \to     
  P_{\rho}^{\nu}
  \bigl( X^{\nu}(s) = X^{\nu}(t) \bigr), \\
  P_{\rho_{n}}^{\nu_{n}} 
  \bigl( X_{n}^{\nu_{n}}(\tilde{c}_{n}t) = X_{n}^{\nu_{n}}(\tilde{c}_{n}t + t'),\ \forall t' \in [0, 3^{n/\alpha}s] \bigr)
  \to   
  \sum_{i \in I} e^{-4s/v_{i}} P_{\rho}^{\nu}(X^{\nu}(t) = x_{i}).
\end{gather}


\subsection{The random conductance model} \label{sec: random conductance model}

Given a connected simple graph, the random conductance model is defined by placing random conductances on the edges. 
So, it is a random electrical network. 
Here, we consider a very simple model, the random conductance model on $\ZN$,
but similar results hold for other graphs such as the Sierpi\'{n}ski gasket graph.
Note that the BTM we consider here is not a usual random walk on the random conductance model
because the speed of the BTM is determined by random traps, which are irrelevant to random conductances. 
However, when the random conductances are heavy tailed, 
it is natural to think (sub-)aging occurs 
because edges with very heavy conductances (resp.\ very light conductances) serve as traps (resp.\ walls) for the random walk.
Indeed, Croydon, Kious, and Scali \cite{Croydon_Kious_Scali_pre_Aging} confirmed this for the constant speed random walk on the random conductance model on $\ZN$.
For details regarding application of resistance form theory to the random conductance model, 
see \cite[Section 6]{Croydon_Hambly_Kumagai_17_Time-changes} and \cite[Appendix A]{Noda_pre_Scaling}.

Fix i.i.d.\ positive random variables $\{\zeta_{i}\}_{i \in \ZN}$ built on a probability space with probability measure $P$
such that $c < \zeta_{i} < c'$ almost surely for some deterministic constants $c, c' \in (0, \infty)$.
By scaling, without loss of generality,
we may assume that $E[\zeta_{0}^{-1}] = 1$.
Define a random electrical network $G_{n}$ as follows.
Set $V_{n} \coloneqq 2^{-n} \ZN$ and 
$E_{n}$ to be the set of the pairs of elements of $V_{n}$ at a Euclidean distance $2^{-n}$ apart.
We place the random conductance $\zeta_{i}$ on the edge $\{i/2^{n}, (i+1)/2^{n}\}$.
We equip $G_{n}$ with the root $\rho_{G_{n}} \coloneqq 0$.
Then, from \cite[Theorem A.2]{Noda_pre_Scaling}, we deduce that 
$(V_{G_{n}}, 2^{-n}R_{G_{n}}, \rho_{G_{n}})_{n \geq 1}$ satisfies 
Assumption \ref{1. assum: aging, random version}\ref{1. assum item: random, the non-explosion condition} 
and
\begin{equation}
  (V_{G_{n}}, 2^{-n} R_{G_{n}}, \rho_{G_{n}}, 2^{-n} \muh_{G_{n}}) 
  \xrightarrow{\mathrm{p}} 
  (\RN, d^{\RN}, 0, \Leb),
\end{equation}
where $d^{\RN}$ and $\Leb$ denote the Euclidean metric on $\RN$ and the Lebesgue measure, respectively.
Moreover, since the total conductance at $i/2^{n} \in V_{G_{n}}$ is $\zeta_{i-1}+\zeta_{i}$,
it is possible to show that 
\begin{equation}
  (V_{G_{n}}, 2^{-n} R_{G_{n}}, \rho_{G_{n}}, 2^{-n} \dotmuh_{G_{n}}) 
  \xrightarrow{\mathrm{p}} 
  (\RN, d^{\RN}, 0, \Leb \otimes P(\zeta_{0} + \zeta_{1} \in \cdot)).
\end{equation}
Hence, we can apply Theorem \ref{1. thm: sub-aging for random models} 
and obtain the aging and sub-aging results as follows.
We let $\{(x_{i}, w_{i}, v_{i})\}_{i \in I}$ be a Poisson point process on $F \times \RNp \times \RNpp$
with intensity $dx\, P(\zeta_{0} + \zeta_{1} \in dw)\, \alpha v^{-1-\alpha} dv$ 
and define $\nu(dx) \coloneqq \sum_{i \in I} v_{i} \delta_{x_{i}}(dx)$.
Given $\nu$, 
we write $(X^{\nu}, \{P_{x}^{\nu}\}_{x \in \RN})$ for the process associated with $(\RN, d^{\RN}, \nu)$.
We simply write $X_{n}^{\nu_{n}} \coloneqq X_{G_{n}}^{\nu_{G_{n}}}$,
which is the BTM on $G_{n}$ (see Definition \ref{1. dfn: BTM}).
We denote by $P_{\rho_{n}}^{\nu_{n}}$ for the underlying probability measure 
for $X_{n}^{\nu_{n}}$ started at $\rho_{n}$.
Set $\tilde{c}_{n} \coloneqq 2^{n} \cdot 2^{n/\alpha}$.
As a consequence of Theorem \ref{1. thm: sub-aging for deterministic models},
we obtain that 
\begin{gather}
  P_{\rho_{n}}^{\nu_{n}} 
  \bigl( X_{n}^{\nu_{n}}(\tilde{c}_{n}s) = X_{n}^{\nu_{n}}( \tilde{c}_{n}t) \bigr) 
  \to     
  P_{\rho}^{\nu}
  \bigl( X^{\nu}(s) = X^{\nu}(t) \bigr), \\
  P_{\rho_{n}}^{\nu_{n}} 
  \bigl( X_{n}^{\nu_{n}}(\tilde{c}_{n}t) = X_{n}^{\nu_{n}}(\tilde{c}_{n}t + t'),\ \forall t' \in [0, 2^{n/\alpha}s] \bigr)
  \to   
  \sum_{i \in I} e^{-w_{i}s/v_{i}} P_{\rho}^{\nu}(X^{\nu}(t) = x_{i}).
\end{gather}
As seen above, the effect of random conductances disappears in the scaling limit of the aging functions. 
This is called homogenization in the study of the random conductance model. 
On the other hand,  
the scaling limit of the sub-aging functions is affected by random conductances. 
This is because the sub-aging functions capture the behaviors of the BTMs on a time scale shorter than that of homogenization.


\subsection{The critical Galton-Watson tree} \label{sec: GW trees}

We apply our results to the critical Galton-Watson tree conditioned on its size,
which is naturally regarded as an electrical network by placing conductance $1$ on each edge.
It is well-known that the suitably conditioned scaled critical Galton-Watson trees converge to the continuum random tree in the Gromov-Hausdorff-Prohorov topology
\cite{Abraham_Delmas_Hoscheit_13_A_note,Aldous_93_The_continuum,Noda_pre_Convergence}.
Thus, Theorem \ref{1. thm: aging for random models} immediately shows that there is aging for the Bouchaud trap models on the critical Galton-Watson trees.
To apply the sub-aging result (Theorem \ref{1. thm: aging for random models}),
we need to prove that a uniformly chosen random vertex and its degree jointly converges.
The former convergence is related to the global properties of the graphs, 
while the latter convergence is related to the local properties of the graphs. 
Usually, these global and local properties are studied independently of each other, 
and their respective convergences are already known.
The convergence of uniformly chosen random vertices is an immediate consequence of the Gromov-Hausdorff-Prohorov convergence,
and the convergence of their degrees were shown in \cite[Theorem 7.11]{Janson_12_Simply}.
In this section, 
we prove the joint convergence of uniformly chosen random vertices and their degrees,
Corollary \ref{7. cor: GW tree satisfies the sub-aging assumption} below.
The proof of this result relies on \cite{Thevenin_20_Vertices},
where the joint convergence of two associated functions was shown: 
the contour functions and functions that count the number of vertices having certain outdegrees.
For details of notion regarding plane trees, we refer to \cite{Thevenin_20_Vertices}.

Let $T$ be a plane tree $T$ with $n$ vertices.
We write $\rho_{T}$ for the root and $d_{T}$ for the graph metric on $T$.
By declaring that $\{u, v\} \subseteq V$ is an edge if and only if $d_{T}(u, v) = 1$
and placing conductance $1$ on each edge,
we regard $T$ as an electrical network.
We then define $\mu_{T}(x)$ to be the total conductance at $x$, 
which is exactly the degree of $x$, 
$\muh_{T}$ to be the counting measure, 
and $\dotmuh_{T}$ to be the pushforward of $\muh_{T}$ by the map $\Psi_{T}(x) = (x, \mu_{T}(x))$.
Let $v_{0} = \rho_{T}, v_{1}, \ldots, v_{n-1}$ be the vertices of $T$ in the lexicographical order
(also known as the depth first order).
We define a map $p_{T}: [0, n-1] \to T$ 
by setting $p_{T}(i) \coloneqq v_{i}$ for $i \in \ZN$ and $p_{T}(t) \coloneqq p_{T}(\lfloor t \rfloor)$.
We also define the height function $H_{T}: [0,n-1] \to \RNp$ by setting $H_{T}(i) \coloneqq d_{T}(v_{i}, v_{0})$ and linearly interpolating between integers. 
We next introduce functions $N_{T}^{(k)}, D_{T}^{(k)}: [0,n-1] \to \RNp$ that record local structures of $T$.
Recall that the \textit{outdegree} of a vertex is the number of its children.
For every non-negative integer $k$ and $t \in [0, n-1]$,
we define $N_{T}^{(k)}(t)$ by
\begin{equation} \label{7. eq: dfn of outdegree coding function}
  N_{T}^{(k)}(t) 
  \coloneqq 
  \# 
  \bigl\{
    p_{T}(s) \in T \mid s \leq t\ \text{and the outdegree of}\ p_{T}(s)\ \text{is}\ k
  \bigr\}.
\end{equation}
We similarly define $D_{T}^{(k)}$ by 
\begin{equation}
  D_{T}^{(k)}(t) 
  \coloneqq 
  \# 
  \bigl\{
    p_{T}(s) \in T \mid s \leq t\ \text{and the degree of}\ p_{T}(s)\ \text{is}\ k
  \bigr\}.
\end{equation}
We fix a sequence $(\alpha_{n})_{n \geq 1}$ of positive numbers with $\alpha_{n} \to \infty$,
and consider the following scaled versions of these functions: for every $t \in [0,1]$,
\begin{gather}
  \tilde{p}_{T}(t) 
  \coloneqq 
  p_{T}((n-1)t),
  \quad
  \tilde{H}_{T}(t) 
  \coloneqq 
  \frac{1}{\alpha_{n}} H_{T}((n-1)t),\\
  \tilde{N}_{T}^{(k)}(t)
  \coloneqq 
  \frac{1}{n}
  N_{T}^{(k)}((n-1)t),
  \quad 
  \tilde{D}_{T}^{(k)}(t)
  \coloneqq 
  \frac{1}{\alpha_{n}}
  D_{T}^{(k)}((n-1)t).
\end{gather}

To describe limits of plane trees, we introduce real trees.
For details, see \cite{LeGall_06_Random}, for example.
Write $\mathscr{E}$ for the space of excursions,
that is,
\begin{equation}
  \mathscr{E}
  \coloneqq
  \{
  f \in C(\RNp, \RNp)
  :
  f(0)=0,
  \exists \sigma_{f} < \infty\
  \text{such that}\
  f(x) > 0\ \forall x \in (0, \sigma_{f}),
  f(x)=0\ \forall x \geq \sigma_{f}
  \}.
\end{equation}
Given a function $f \in C([0,\sigma], \RNp)$ 
with $f(0)=f(\sigma)=0$ and $f(x) > 0$ for all $x \in (0, \sigma)$,
we will abuse notation by identifying $f$ with the function $g \in \mathscr{E}$
which has $g(x)=f(x),\, 0 \leq x \leq \sigma$ and $g(x)=0, \, x \geq \sigma$.
We equip $\mathscr{E}$ with the metric induced by the supremum norm $\| \cdot \|_{\infty}$. 
Given an excursion $f \in \mathscr{E}$, 
we define a pseudometric $\bar{d}_{f}$ on $[0, \sigma_{f}]$ by setting
\begin{equation}
  \bar{d}_{f}(s,\, t)
  \coloneqq
  f(s)+f(t)-2 \inf_{u \in [s \wedge t, s \vee t]} f(u).
\end{equation}
Then,
we use the equivalence
\begin{equation}
  s \sim t \qquad
  \Leftrightarrow \qquad
  \bar{d}_{f}(s,\, t)=0
\end{equation}
to define $T_{f} \coloneqq  [0, \sigma_{f}]/ \sim$.
Let $p_{f} : [0,\sigma_{f}] \to T_{f}$ be the canonical projection.
It is then elementary to check that
\begin{equation}
  d_{f}(p_{f}(s), p_{f}(t))
  \coloneqq
  \bar{d}_{f}(s,\, t)
\end{equation}
defines a metric on $T_{f}$.
The metric space $(T_{f}, d_{f})$ is called a \textit{real tree} coded by $f$.
The canonical Radon measure on $T_{f}$ is given by $\mu_{f} \coloneqq  \Leb \circ (p_{f})^{-1}$,
where $\Leb$ stands for the one-dimensional Lebesgue measure.
We define the root $\rho_{f}$ by setting $\rho_{f} \coloneqq  p_{f}(0)$.

In Theorem \ref{7. thm: convergence of dot measures for trees} below,
we show that if scaled height functions $\tilde{H}_{T}$ and scaled out-degree-counting functions $\tilde{D}_{T}^{(k)}$ converge,
then the associated plane trees converge and moreover uniformly chosen random vertices and their degrees also converge jointly.
To do this, we use the following result on convergence of Lebesgue-Stieltjes integrals in terms of convergence of integrators.
Given a non-decreasing function $g$ on $[0,1]$,
we denote by $s_{g}$ the Lebesgue-Stieltjes measure associated with $g$.

\begin{lem} \label{7. lem: Stieltjes convergence}
  Fix a metric space $(M, d^{M})$.
  Let $g_{1}, g_{2}, \ldots$ be non-decreasing functions in $D([0,1], \mathbb{R})$ 
  such that $g_{n} \to g$ in the usual $J_{1}$-Skorohod topology for some strictly increasing function $g \in D([0,1], \RN)$,
  and let $q, q_{1}, q_{2}, \ldots$ be measurable functions from $[0,1]$ to $M$ such that $q_{n} \to q$ 
  in the supremum norm $\| \cdot \|_{\infty}$.
  Then, it holds that $s_{g_{n}} \circ q_{n}^{-1} \to s_{g} \circ q^{-1}$ weakly as measures on $M$.
\end{lem}

\begin{proof}
  Define $\tilde{g}_{n} \in D(\RNp, \RNp)$ by setting 
  \begin{equation}
    \tilde{g}_{n}(t) 
    \coloneqq 
    \begin{cases}
      g_{n}(t) - g_{n}(0), & t \in [0,1],\\
      (t-1) + g_{n}(1) - g_{n}(0), & t >1.
    \end{cases}
  \end{equation}
  Similarly, we define $\tilde{g}$.
  Noting that $g_{n}(0) \to g(0)$ and $g_{n}(1) \to g(1)$,
  it is easy to check that $\tilde{g}_{n} \to \tilde{g}$ in the usual $J_{1}$-Skorohod topology.
  Let $\tilde{\tau}_{n} \in D(\RNp, \RNp)$ be the right-continuous inverse of $\tilde{g}_{n}$ given by 
  \begin{equation}
    \tilde{\tau}_{n}(t) 
    \coloneqq 
    \inf\{s \geq 0 \mid \tilde{g}_{n}(s) > t\}.
  \end{equation}
  Similarly, we let $\tilde{\tau}$ be the right-continuous inverse of $\tilde{g}$.
  Since $\tilde{g}$ is strictly increasing,
  we deduce from \cite[Corollary 13.6.4]{Whitt_02_Stochastic} that $\tilde{\tau}_{n} \to \tilde{\tau}$ with respect to $\| \cdot \|_{\infty}$.
  Fix a bounded continuous function $F$ on $M$.
  By \cite[Lemma A.3.7]{Chen_Fukushima_12_Symmetric},
  we have that 
  \begin{equation} \label{7. eq: Lebesgue-Stieltjes proof}
    \int_{(0,1]} F(q_{n}(t))\, s_{\tilde{g}_{n}}(dt)
    =
    \int_{0}^{\infty} 1_{(0,1]}(\tilde{\tau}_{n}(t)) F(q_{n} \circ \tilde{\tau}_{n}(t))\, dt.
  \end{equation}
  If $\tilde{\tau}_{n}(t) < 2$,
  then it is the case that $t < \tilde{g}_{n}(2)$.
  Since $(\tilde{g}_{n}(2))_{n \geq 1}$ is bounded, 
  for some $K > 0$,
  it holds that 
  \begin{equation}
    1_{(0,1]}(\tilde{\tau}_{n}(t)) 
    \leq 
    1_{(0,2)}(\tilde{\tau}_{n}(t))
    \leq 
    1_{(0,K)}(t).
  \end{equation}
  Thus, we can apply the dominated convergence theorem to \eqref{7. eq: Lebesgue-Stieltjes proof},
  and we obtain that 
  \begin{align}
    \lim_{n \to \infty} 
    \int_{(0,1]} F(q_{n}(t))\, s_{\tilde{g}_{n}}(dt)
    =
    \int_{0}^{\infty} 1_{(0,1]}(\tilde{\tau}(t)) F(q \circ \tilde{\tau}(t))\, dt
    =
    \int_{(0,1]} F(q(t))\, s_{\tilde{g}}(dt),
  \end{align}
  where we use \cite[Lemma A.3.7]{Chen_Fukushima_12_Symmetric} to deduce the second equality.
  By the definitions of $\tilde{g}_{n}$ and $\tilde{g}$,
  we have that $s_{\tilde{g}_{n}} = s_{g_{n}}$ and $s_{\tilde{g}} = s_{g}$ on $[0,1]$.
  Therefore, the proof is completed.
\end{proof}

Convergence of scaled plane trees in the Gromov-Hausdorff-Prohorov topology follows from convergence of coding functions
\cite[Proposition 3.3]{Abraham_Delmas_Hoscheit_13_A_note}.
Furthermore, 
the convergence of the coding functions of plane trees $T_{n}$ leads to the convergence of the projections $p_{T_{n}}$ to the projection of the limiting real tree. 
This seems to be a basic and well-known fact, 
but it is asserted as follows for the first time in a rigorous form 
by using the general theory of Gromov-Hausdorff-type topologies introduced in Section \ref{sec: GH-type topologies}.

\begin{lem} \label{7. lem: convergence of canonical maps of trees}
  Let $T_{n}$ be a plane trees with $n$ vertices.
  Assume that 
  there exists an $f\in \mathscr{E}$ with $\sigma_{f} = 1$ such that
  $\tilde{H}_{T_{n}} \to f$ in $C([0,1], \RNp)$.
  Then, it holds that 
  \begin{equation}
    (T_{n}, \alpha_{n}^{-1} d_{T_{n}}, \rho_{T_{n}}, n^{-1}\muh_{T_{n}}, \tilde{p}_{T_{n}})
    \to 
    (T_{f}, d_{f}, \rho_{f}, \mu_{f}, p_{f})
  \end{equation}
  in the space $\rbcM_{c}(\finMeasFunct \times \unifFunct)$.
\end{lem}

\begin{proof}
  Let $\mathcal{C}_{n}$ be a correspondence between $T_{n}$ and $T_{f}$ given by 
  \begin{equation}
    \mathcal{C}_{n} 
    \coloneqq 
    \bigl\{
      (v, x) \in T_{n} \times T_{f} \mid v = \tilde{p}_{T_{n}}(t)\ \text{and}\ x = p_{f}(t)\ \text{for some}\ t \in [0,1]
    \bigr\}.
  \end{equation} 
  (see \cite{Burago_Burago_Ivanov_01_A_course} for the definitions of a correspondence between sets).
  Then, one can check, 
  by applying the same argument as that used to prove \cite[Proposition 8.3]{Noda_pre_Convergence}, 
  that the distortion $\mathrm{dis}\, \mathcal{C}_{n}$ of $\mathcal{C}_{n}$ given below converges to $0$:
  \begin{equation}
    \mathrm{dis}\, \mathcal{C}_{n}
    \coloneqq 
    \sup\bigl\{ |\alpha_{n}^{-1} d_{T_{n}}(u, u') - d_{f}(x, x')| \mid (u, x), (u', x') \in \mathcal{C}_{n}\bigr\}
  \end{equation}
  (see also \cite{Burago_Burago_Ivanov_01_A_course} for the notion of distortion).
  Using the correspondence $\mathcal{C}_{n}$,
  we define a metric $d_{n}$ on the disjoint union $S_{n} \coloneqq T_{n} \sqcup T_{f}$, extending $\alpha_{n}^{-1}d_{T_{n}}$ and $d_{f}$,
  by setting, for $u \in T_{n}$ and $x \in T_{f}$,
  \begin{equation}
    d_{n}(u,x)
    \coloneqq
    \inf
    \left\{
    \alpha_{n}^{-1}d_{T_{n}}(u,u') + \frac{1}{2}\, \mathrm{dis}\, \mathcal{C}_{n} + n^{-1} + d_{f}(x', x)
    \mid
    (u',x') \in \mathcal{C}_{n}
    \right\}.
  \end{equation}
  In the obvious way,
  we regard $\muh_{T_{n}}$ and $\mu_{f}$ as measures on $S_{n}$ 
  and regard $\tilde{p}_{n}$ and $p$ as maps from $[0,1]$ to $S_{n}$.
  It is then the case that 
  \begin{equation}
    d_{n}(\tilde{p}_{n}(t), p(t)) \leq \frac{1}{2} \mathrm{dis}\, \mathcal{C}_{n} + n^{-1},
  \end{equation}
  and so we obtain that  
  \begin{equation}
    (T_{n}, a_{n}^{-1}d_{T_{n}}, \rho_{T_{n}}, \tilde{p}_{T_{n}})
    \to 
    (T_{f}, d_{f}, \rho_{f}, p).
  \end{equation}
  Thus, by Theorem \ref{3. thm: convergence in GH topology},
  we may assume that $(T_{n}, \alpha_{n}^{-1})$ and $(T_{f}, d_{f})$ are embedded isometrically into a common (compact) metric space $(M, d^{M})$
  in such a way that $T_{n} \to T_{f}$ in the Hausdorff topology, 
  $\rho_{T_{n}} = \rho_{f}$ as elements in $M$,
  and $\tilde{p}_{T_{n}} \to p_{f}$ with respect to $\| \cdot \|_{\infty}$.
  This immediately yields that $\Leb \circ \tilde{p}_{T_{n}}^{-1} \to \Leb \circ p_{f}^{-1} = \mu_{f}$ weakly.
  For any $u \in T_{n} \setminus \{\rho_{T_{n}}\}$,
  we have that 
  \begin{equation}
    \Leb \circ \tilde{p}_{T_{n}}^{-1} (\{u\}) 
    = 
    (n-1)^{-1}.
  \end{equation}
  Hence, for any $A \subseteq T_{n}$,
  \begin{equation}
    | \Leb \circ \tilde{p}_{T_{n}}^{-1} (A) - (n-1)^{-1} \muh_{T_{n}}(A) | 
    \leq 
    (n-1)^{-1}.
  \end{equation}
  This, combined with $\Leb \circ \tilde{p}_{T_{n}}^{-1} \to \mu_{f}$,
  implies that $n^{-1} \muh_{T_{n}} \to \mu_{f}$ weakly.
  This completes the proof.
\end{proof}

We now prove the main result.

\begin{thm} \label{7. thm: convergence of dot measures for trees}
  Suppose that we are in the same setting as Lemma \ref{7. lem: convergence of canonical maps of trees}.
  Set $I(t) \coloneqq t$ for $t \in [0,1]$.
  Assume that 
  there exists a probability measure $p = (p_{k})_{k \geq 1}$ on $\ZNp$ such that
  $\tilde{N}_{T_{n}}^{(k)} \to p_{k}\, I$ in $D([0,1], \RNp)$ for every $k \geq 0$.
  Then, it holds that 
  \begin{equation}
    (T_{n}, a_{n}^{-1}d_{T_{n}}, \rho_{T_{n}}, b_{n}^{-1}\dotmuh_{T_{n}})
    \to 
    (T_{f}, d_{f}, \rho_{f}, \mu_{f} \otimes \tilde{p})
  \end{equation}
  in the space $\rbcM_{c}(\markedfinMeasFunct{\RNp})$,
  where $\tilde{p} = (\tilde{p}_{k})_{k \geq 1}$ is a probability measure on $\NN$ 
  given by $\tilde{p}_{k} \coloneqq p_{k-1}$.
\end{thm}

\begin{proof}
  Since the degree of a vertex, except for the root, is the outdegree plus $1$, 
  it is easy to see that $\tilde{D}_{T_{n}}^{(k)} \to p_{k-1}\, I$ in $D([0,1], \RNp)$ 
  for every $k \geq 1$.
  By Theorem \ref{3. thm: convergence in GH topology} and Lemma \ref{7. lem: convergence of canonical maps of trees},
  we may assume that $(T_{n}, \alpha_{n}^{-1}d_{T_{n}}, \rho_{T_{n}})$ and $(T_{f}, d_{f}, \rho_{f})$ are embedded isometrically 
  into a common rooted compact metric space $(M, d^{M}, \rho_{M})$
  in such a way that $T_{n} \to T_{f}$ in the Hausdorff topology, $\rho_{T_{n}} = \rho_{f} = \rho_{M}$ as elements in $M$,
  $n^{-1} \muh_{T_{n}} \to \mu_{f}$ weakly,
  and $\tilde{p}_{T_{n}} \to p_{f}$ with respect to $\| \cdot \|_{\infty}$.
  For each $k \geq 1$,
  we write $T_{n}^{(k)}$ for the set of vertices of $T_{n}$ whose degree is $k$.
  Then, one can check that, for any subset $A \subseteq T_{n}$,
  \begin{equation}
    \muh_{T_{n}}|_{T_{n}^{(k)}}(A)
    =
    D_{T_{n}}^{(k)}(0) \cdot \delta_{\{0\}}(p_{T_{n}}^{-1}(A)) 
    +
    s_{D_{T_{n}}^{(k)}} (p_{T_{n}}^{-1}(A)).
  \end{equation}
  This, combined with Lemma \ref{7. lem: Stieltjes convergence}, yields that 
  $n^{-1} \muh_{T_{n}}|_{T_{n}^{(k)}} \to p_{k-1}\, \mu_{f}$
  weakly for every $k \geq 1$.
  Fix compactly supported continuous functions $f_{1}: M \to \RN$ and $f_{2}: \RNp \to \RN$.
  We have that 
  \begin{align}
    \lim_{n \to \infty}
    \int f_{1}(x) f_{2}(w)\, n^{-1} \dotmuh_{T_{n}}(dx dw) 
    &= 
    \lim_{n \to \infty}
    n^{-1} \sum_{v \in T_{n}} f_{1}(v) f_{2}(\mu_{T_{n}}(v))\\
    &= 
    \lim_{n \to \infty}
    n^{-1} \sum_{k \geq 1} f_{2}(k) \sum_{v \in T_{n}^{(k)}} f_{1}(v) \\
    &= 
    \sum_{k \geq 1} f_{2}(k) \int f_{1}(x)\, n^{-1} \muh_{n}|_{T_{n}^{(k)}}(dx)\\
    &= 
    \sum_{k \geq 1} f_{2}(k) \int f_{1}(x) p_{k-1}\, \mu_{f}(dx)\\
    &= 
    \int f_{1}(x) f_{2}(w)\, \mu_{f} \otimes \tilde{p}(dx dw).
  \end{align}
  This implies that $n^{-1} \dotmuh_{T_{n}} \to \mu_{f} \otimes \tilde{p}$ vaguely.
  Hence, it remains to prove the tightness of $(n^{-1} \dotmuh_{T_{n}})_{n \geq 1}$ in the weak topology,
  and this can be verified by the following:
  \begin{align}
    \lim_{L \to \infty} \limsup_{n \to \infty} 
    n^{-1} \dotmuh_{T_{n}} \bigl( M \times (L, \infty) \bigr)
    &= 
    1-
    \lim_{L \to \infty} \liminf_{n \to \infty} 
    n^{-1} \dotmuh_{T_{n}} \bigl( M \times [1, L] \bigr)\\
    &= 
    1-
    \lim_{L \to \infty} \liminf_{n \to \infty} 
    \sum_{i=1}^{L} n^{-1} \muh_{T_{n}}|_{T_{n}^{(k)}}(M)\\
    &= 
    1-
    \lim_{L \to \infty} 
    \sum_{i=1}^{L} p_{k-1}\, \mu_{f}(M)\\
    &= 
    1-
    \lim_{L \to \infty} 
    \sum_{i=1}^{L} p_{k-1}\\
    &= 0.
  \end{align}
\end{proof}

We apply the above result to the critical Galton-Watson tree conditioned on its size.
Let $p=(p_{k})_{k \geq 0}$ be a probability measure on $\ZNp$ such that 
$\sum_{k \geq 1} k p_{k} = 1$, $p_{1} < 1$, and its variance is $1$.
For simplicity,
we assume that $p$ is aperiodic,
that is, 
the greatest common divisor of the set $\{k \mid p_{k} > 0\}$ is $1$.
We write $T^{GW}$ for the Galton-Watson tree with the offspring distribution $p$,
and write $T_{n}^{GW}$ for a random plane tree having the same distribution as $T^{GW}$ conditioned to have exactly $n$ vertices, 
which is well-defined for all sufficiently large $n$ by the aperiodicity of $p$.
We define $e=(e(t))_{t \in [0,1]} \in \mathscr{E}$ to be the normalized Brownian excursion.

\begin{cor} \label{7. cor: GW tree satisfies the sub-aging assumption}
  Define a probability measure $\tilde{p} = (\tilde{p}_{k})_{k \geq 1}$ on $\NN$ 
  by setting $\tilde{p}_{k} \coloneqq p_{k-1}$.
  Then, it holds that 
  \begin{equation}
    (T_{n}^{GW}, n^{-1/2}d_{T_{n}^{GW}}, \rho_{T_{n}^{GW}}, n^{-1} \dotmuh_{T_{n}^{GW}})
    \xrightarrow{\mathrm{d}}
    (T_{2e}, d_{2e}, \rho_{2e}, \mu_{2e} \otimes \tilde{p})
  \end{equation}
  in the space $\rbcM_{c}(\markedfinMeasFunct{\RNp})$.
\end{cor}

\begin{proof}
  From \cite[Theorem 3.1]{Duquesne_03_A_limit} and \cite[Theorem 1.1]{Thevenin_20_Vertices}, we have that 
  \begin{gather}
    \bigl( n^{-1/2} H_{T_{n}^{GW}}((n-1)t) \bigr)_{t \in [0,1]} 
    \xrightarrow{\mathrm{d}} 
    (2e(t))_{t \in [0,1]},\\
    \bigl( n^{-1}N_{T_{n}^{GW}}^{(k)}((n-1)t) \bigr)_{t \in [0,1]}
    \xrightarrow{\mathrm{p}} 
    (p_{k} I(t))_{t \in [0,1]},
    \quad 
    \forall k \geq 0.
  \end{gather}
  Hence, we obtain the desired result 
  by Theorem \ref{7. thm: convergence of dot measures for trees} and the Skorohod representation theorem.
\end{proof}

By Corollary \ref{7. cor: GW tree satisfies the sub-aging assumption}, we can apply Theorem \ref{1. thm: sub-aging for random models} 
and we obtain the aging and sub-aging results as follows.
Let $\{(x_{i}, w_{i}, v_{i})\}_{i \in I}$ be a Poisson point process on $T_{2e} \times \RNp \times \RNpp$
with intensity $\mu_{2e}(dx)\, \tilde{p}(dw)\, \alpha v^{-1-\alpha} dv$ 
and define $\nu(dx) \coloneqq \sum_{i \in I} v_{i} \delta_{x_{i}}(dx)$.
Given $\nu$, 
we write $(X^{\nu}, \{P_{x}^{\nu}\}_{x \in T_{2e}})$ for the process associated with $(T_{2e}, d_{2e}, \nu)$.
We simply write $X_{n}^{\nu_{n}} \coloneqq X_{T_{n}^{GW}}^{\nu_{T_{n}^{GW}}}$,
which is the BTM on $T_{n}^{GW}$ (see Definition \ref{1. dfn: BTM}).
We denote by $P_{\rho_{n}}^{\nu_{n}}$ for the underlying probability measure 
for $X_{n}^{\nu_{n}}$ started at $\rho_{n} \coloneqq \rho_{T_{n}^{GW}}$.
Set $\tilde{c}_{n} \coloneqq n^{1/2} \cdot n^{1/\alpha}$.
As a consequence of Corollary \ref{7. cor: GW tree satisfies the sub-aging assumption} and Theorem \ref{1. thm: sub-aging for random models},
we obtain that 
\begin{gather}
  P_{\rho_{n}}^{\nu_{n}} 
  \bigl( X_{n}^{\nu_{n}}(\tilde{c}_{n}s) = X_{n}^{\nu_{n}}( \tilde{c}_{n}t) \bigr) 
  \to     
  P_{\rho_{2e}}^{\nu}
  \bigl( X^{\nu}(s) = X^{\nu}(t) \bigr), \\
  P_{\rho_{n}}^{\nu_{n}} 
  \bigl( X_{n}^{\nu_{n}}(\tilde{c}_{n}t) = X_{n}^{\nu_{n}}(\tilde{c}_{n}t + t'),\ \forall t' \in [0, n^{1/\alpha}s] \bigr)
  \to   
  \sum_{i \in I} e^{-w_{i}s/v_{i}} P_{\rho_{2e}}^{\nu}(X^{\nu}(t) = x_{i}).
\end{gather}


\subsection{The critical Erd\H{o}s-R\'{e}nyi random graph} \label{sec: ER graph}

\newcommand{\ERlc}{\closed_{1}^{n}}
\newcommand{\excs}{e^{(\sigma)}}
\newcommand{\Mz}{M^{(Z_{1})}}
\newcommand{\Gmnp}{G_{m_{n}}^{p_{n}}}
\newcommand{\Gmp}{G_{m}^{p}}
\newcommand{\DFS}{\mathbf{DFS}}
\newcommand{\tildG}{\tilde{G}}

In this section,
we consider the Erd\H{o}s-R\'{e}nyi random graph $G(n,p)$,
which is a graph on $n$ labeled vertices $[n] \coloneqq  \{1,2, \ldots, n\}$
chosen randomly by joining any two distinct vertices by an edge with probability $p$,
independently for different pairs of vertices.
This model exhibits a phase transition in its structure for large $n$.
Let $p=c/n$ for some $c >0$.
When $c < 1$, the largest connected component has size $O(\log n)$.
On the other hand, when $c > 1$, we see the emergence of a giant component that contains
a positive proportion of the vertices.
In the critical case $c=1$, the largest connected components have sizes of order $n^{2/3}$.
We will focus here on the critical case $c = 1$, and more specifically, on the critical window $p = n^{-1} + \lambda n^{-4/3}$, $\lambda \in \RN$.
We fix $\lambda \in \RN$ and write $p_{n}=n^{-1}+\lambda n^{-4/3}$.
Let $\closed_{i}^{n}$ be the $i$-th largest connected component of $G(n,p_{n})$.

One of the most significant results about random graphs in the above-mentioned critical regime was proved by Aldous \cite{Aldous_97_Brownian}.
Write $Z_{i}^{n}$ and $S_{i}^{n}$ for the size (that is, the number of vertices) and surplus
(that is, the number of edges that would need to be removed in order to obtain a tree)
of $\closed_{i}^{n}$.
Set $\bm{Z}^{n} \coloneqq  (Z_{1}^{n}, Z_{2}^{n}, \ldots)$ and $\bm{S}^{n}\coloneqq (S_{1}^{n}, S_{2}^{n}, \ldots)$.

\begin{thm} [{\cite[Folk Theorem 1, Corollary 2]{Aldous_97_Brownian}}]  \label{thm: convergence of sizes and surpluses of ER}
  As $n \to \infty$, it holds that
  \begin{equation}
    (n^{-2/3} \bm{Z}^{n}, \bm{S}^{n}) \to (\bm{Z}, \bm{S})
  \end{equation}
  in distribution, where the convergence of the first coordinate
  takes place in $l^{2}_{\searrow}$,
  the set of infinite sequences $(x_{1}, x_{2}, \ldots)$
  with $x_{1} \geq x_{2} \geq \cdots \geq 0$ and $\sum_{i} x_{i}^{2} < \infty$,
  equipped with the usual $l^{2}$-norm.
\end{thm}

The limits $\bm{Z}=(Z_{1}, Z_{2}, \ldots)$ and $\bm{S}=(S_{1}, S_{2}, \ldots)$ are constructed as follows.
Consider a Brownian motion with parabolic drift, $(W_{t}^{\lambda})_{t \geq 0}$, where
\begin{equation}
  W_{t}^{\lambda} \coloneqq  W_{t} + \lambda t - \frac{t^{2}}{2}
\end{equation}
and $(W_{t})_{t \geq 0}$ is a standard Brownian motion.
Then, the limit $\bm{Z}$ has the distribution of the ordered sequence of lengths of excursions
of the reflected process $W^{\lambda}_{t} - \min_{0 \leq s \leq t} W^{\lambda}_{s}$ above $0$,
while $\bm{S}$ is the sequence of numbers of points of a Poisson point process with rate one in $\RNp^{2}$
lying under the corresponding excursions,
where the Poisson point process is assumed to be independent of $(W^{\lambda}_{t})_{t \geq 0}$.

The scaling limit of $\closed_{1}^{n}$ is given by fusing a tilted Brownian continuum random tree.
For fused resistance metric spaces,
see Appendix \ref{sec: convergence of fused spaces}.
Recall the space $\mathscr{E}$ of excursions from the previous section,
where we equip $\mathscr{E}$ with the metric induced by the supremum norm $\| \cdot \|_{\infty}$.
Let $\excs=(\excs(t), 0 \leq t \leq \sigma)$ be a Brownian excursion of length $\sigma$.
Note that, by Brownian scaling,
the distribution of $e^{(\sigma)}$ coincides with that of $(\sqrt{\sigma}\, e(t/\sigma))_{t \in [0, \sigma]}$,
where $(e(t))_{t \in [0,1]}$ denotes the standard Brownian excursion on $[0,1]$.
The \textit{tilted excursion} of length $\sigma$,
$\tilde{e}^{(\sigma)}=(\tilde{e}^{(\sigma)}(t), 0 \leq t \leq \sigma) \in \mathscr{E}$,
is defined to be an excursion whose distribution is characterized by
\begin{equation}
  \mathbf{P} \left( \tilde{e}^{(\sigma)} \in \mathcal{B} \right)
  =
  \frac
  {\mathbf{E}
    \left[
    1_{\{\excs \in  \mathcal{B}\}}
    \exp( \int_{0}^{\sigma} \excs(t) dt)
    \right]}
  { \mathbf{E}
    \left[
    \exp( \int_{0}^{\sigma} \excs(t) dt)
    \right]}
\end{equation}
for $\mathcal{B} \subseteq \mathscr{E}$ a Borel set.
For $f \in \mathscr{E}$ and $S \subseteq \RNp^{2}$,
define
\begin{equation}
  S \cap f
  \coloneqq
  \{ (x, y) \in S : 0 \leq y < f(x) \}.
\end{equation}
For $u=(u_{x},u_{y}) \in S \cap f$,
we define $u'=(u_{x}, u'_{x})$
by setting $u'_{x} \coloneqq  \inf \{x \geq u_{x} : f(x)=u_{x} \}$.
We write
\begin{equation}  \label{eq: def of time markers from excursion and pointset}
  \mathscr{T}(S, f) \coloneqq  \{ u' \in \RNp^{2} : u \in S \cap f \}.
\end{equation}
Let $\mathcal{P} \subseteq \RNp^{2}$ be a Poisson point process with rate one,
independent of $(\tilde{e}^{(\sigma)})_{\sigma > 0}$ and $(W_{t})_{t \geq 0}$.
Assume that $\mathcal{P} \cap \tilde{e}^{(\sigma)}$ is non-empty
and write $\mathscr{T}(\mathcal{P}, \tilde{e}^{(\sigma)})=\{ (\xi_{l}, \xi'_{l}) : 1 \leq l \leq s\}$.
Define $a_{l} \coloneqq  p_{2\tilde{e}^{(\sigma)}}(\xi_{l})$ and $b_{l} \coloneqq  p_{2\tilde{e}^{(\sigma)}}(\xi'_{l})$,
where we recall from Section \ref{sec: GW trees} that $p_{2\tilde{e}^{(\sigma)}}$ denotes the canonical projection 
from $[0, \sigma]$ onto the real tree $T_{2 \tilde{e}^{(\sigma)}}$ coded by $2 \tilde{e}^{(\sigma)}$.
Recall also that
$d_{2 \tilde{e}^{(\sigma)}}$ is the metric on $T_{2 \tilde{e}^{(\sigma)}}$,
$\rho_{2 \tilde{e}^{(\sigma)}}$ is the root of $T_{2 \tilde{e}^{(\sigma)}}$,
and $\mu_{2 \tilde{e}^{(\sigma)}}$ is the canonical measure on $T_{2 \tilde{e}^{(\sigma)}}$.
Define $(M^{(\sigma)}, R_{M^{(\sigma)}})$ to be the resistance metric space $(T_{2 \tilde{e}^{(\sigma)}}, d_{2 \tilde{e}^{(\sigma)}})$
fused over $(\{ a_{l}, b_{l} \})_{l=1}^{s}$
(see Definition \ref{A. dfn: fused resistance metric space}).
Let $\pi: T_{2 \tilde{e}^{(\sigma)}} \to M^{(\sigma)}$ be the canonical map
and set $\rho_{M^{(\sigma)}} \coloneqq \pi(\rho_{2 \tilde{e}^{(\sigma)}})$ and $\mu_{M^{(\sigma)}} \coloneqq  \mu_{2 \tilde{e}^{(\sigma)}} \circ \pi^{-1}$.
If $\mathcal{P} \cap \tilde{e}^{(\sigma)}$ is empty,
then we define $(M^{(\sigma)}, R_{M^{(\sigma)}}, \rho_{M^{(\sigma)}}, \mu_{M^{(\sigma)}})$
to be equal to $(T_{2 \tilde{e}^{(\sigma)}}, d_{2 \tilde{e}^{(\sigma)}}, \rho_{2 \tilde{e}^{(\sigma)}}, \mu_{2 \tilde{e}^{(\sigma)}})$.
We let $(\Mz, R_{\Mz}, \rho_{\Mz}, \mu_{\Mz})$ be a random element of $\rbcM_{c}(\markedfinMeasFunct{\RNp})$
such that, conditional on $Z_{1} = \sigma$,
its distribution is given by $(M^{(\sigma)}, R_{M^{(\sigma)}}, \rho_{M^{(\sigma)}}, \mu_{M^{(\sigma)}})$.

Given a finite connected graph $G$ with labeled vertices,
we regard $G$ as a rooted electrical network by placing conductance $1$ on each edge
and set $\rho_{G}$ to be the smallest-labeled vertex of $G$.
The following result shows that 
Assumption \ref{1. assum: sub-aging, random version} holds.
As a consequence of Theorem \ref{1. thm: sub-aging for random models},
we obtain the aging and sub-aging results.
In the remainder of this section,
we set $(p_{k})_{k \geq 0}$ to be the Poisson distribution with mean $1$,
i.e., $p_{k} \coloneqq e^{-k}/k!$.
We then define a probability measure $\tilde{p} = (\tilde{p}_{k})_{k \geq 0}$ on $\ZNp$ by setting $\tilde{p}_{k} \coloneqq p_{k-1}$.

\begin{thm}  \label{7. thm: convergence of ER graphs wrt resistance metric}
  It holds that
  \begin{equation}
    (V(\ERlc), n^{-1/3} R_{\ERlc}, \rho_{\ERlc}, n^{-2/3} \dotmuh_{\ERlc})
    \xrightarrow{\mathrm{d}}
    (\Mz, R_{\Mz}, \rho_{\Mz}, \mu_{\Mz} \otimes \tilde{p})
  \end{equation}
  in the space $\rbcM_{c}(\markedfinMeasFunct{\RNp})$.
\end{thm}

To prove the above result,
we let $G_{m}^{p}$ be a random graph with the distribution of $G(m,p)$ conditioned to be connected.
We assume that $Z_{1}^{n}$ and $(G_{m}^{p_{n}})_{m \geq 1}$ are all independent
(recall $Z_{1}^{n}$ from the discussion above Theorem \ref{thm: convergence of sizes and surpluses of ER}).
It is then an easy exercise to check
that the random graph $G_{Z_{1}^{n}}^{p_{n}}$ has the same distribution as the random graph $\closed_{1}^{n}$ with relabeled vertices.
Combining this with Theorem \ref{thm: convergence of sizes and surpluses of ER},
we obtain Theorem \ref{7. thm: convergence of ER graphs wrt resistance metric},
once the following lemma is established.

\begin{lem} \label{lem: convergence of conditioned ER graphs wrt resistance metric}
  If a sequence $(m_{n})_{n \geq 1}$ of natural numbers satisfies $n^{-2/3}m_{n} \to \sigma \in (0,\infty)$,
  then it holds that
  \begin{equation}
    (V(\Gmnp), n^{-1/3} R_{\Gmnp}, \rho_{\Gmnp}, n^{-2/3} \dotmuh_{\Gmnp})
    \xrightarrow{\mathrm{d}}
    (M^{(\sigma)}, R_{M^{(\sigma)}}, \rho_{M^{(\sigma)}}, \mu_{M^{(\sigma)}} \otimes \tilde{p})
  \end{equation}
  in the space $\rbcM_{c}(\markedfinMeasFunct{\RNp})$.
\end{lem}

To prove the above lemma,
we prepare some pieces of notation.
Let $\mathbb{T}_{m}$ be the set of trees with the vertex set $[m]$,
and $T$ be an element of $\mathbb{T}_{m}$.
We regard $T$ as a plane tree by using the depth-first search
(cf.\ \cite[Section 2]{Berry_Broutin_Goldschmidt_12_The_continuum}).
We write $H_{T} = (H_{T}(t),\, 0 \leq t \leq m-1)$ for the height function of $T$
and $N_{T}^{(k)} = (N_{T}^{(k)}(t))_{t \in [0, m-1]}$ for the function counting the number of vertices whose outdegree is $k$
(recall these functions from Section \ref{sec: GW trees}).
We set $H_{T}(t) \coloneqq 0,\, t \in [m-1, m]$ for convenience.
Let $v_{0}, v_{1}, \ldots, v_{m-1}$ be the vertices of $T$ as a plane tree in the depth-first order.
Write $k_{i}$ for the number of children of $v_{i}$.
Then, the depth-first walk $X_{T}=(X_{T}(t))_{t \in [0, m]}$ of $T$ is given by 
setting $X_{T}(0) \coloneqq 0$, $X_{T}(i) \coloneqq \sum_{j=0}^{i-1} (k_{j} - 1)$ for $i \in \{0,1, \ldots, m\}$,
and $X_{T}(t) \coloneqq X_{T}(\lfloor t \rfloor)$.
We then set 
\begin{equation}
  a(T) \coloneqq  \sum_{i=1}^{m-1} X_{T}(i) = m \int_{0}^{1} X_{T}(mt)\, dt.
\end{equation}
Given $p \in (0,1)$, 
we define a random tree $\tilde{T}_{m}^{p}$ on $\mathbb{T}_{m}$ that has a ``tilted'' distribution given by 
\begin{equation}
  \mathbf{P}(\tilde{T}_{m}^{p}=T)
  \propto
  (1-p)^{-a(T)}, \quad
  T \in \mathbb{T}_{m}.
\end{equation}
For $p \in (0,1)$,
a \textit{binomial pointset} $\mathcal{Q}^{p} \subseteq \mathbb{Z}_{\geq 0}^{2}$ of intensity $p$ 
is defined to be a random subset of $\mathbb{Z}_{\geq 0}^{2}$
in which each point is present independently with probability $p$.
In \cite{Berry_Broutin_Goldschmidt_12_The_continuum},
it is shown that $G_{m}^{p}$ is recovered by attaching extra edges on $\tilde{T}_{m}^{p}$.

\begin{lem} [{\cite[Lemma 18]{Berry_Broutin_Goldschmidt_12_The_continuum}}] \label{lem: construction of Gmp from coding functions and binomial pointset}
  Fix $p \in (0,1)$.
  Let $\tilde{T}_{m}^{p}$ to be a tilted tree as defined above
  and $\mathcal{Q}^{p}$ be a binomial pointset of intensity $p$, independent of $\tilde{T}_{m}^{p}$.
  Let $v_{0}, v_{1}, \ldots, v_{m-1}$ be the vertices of $\tilde{T}_{m}^{p}$ in depth-first order.
  Write $\mathscr{T}(\mathcal{Q}^{p}, X_{\tilde{T}_{m}^{p}}) = \{(x_{i}, y_{i}) : 1 \leq i \leq s \}$
  (recall its definition from \eqref{eq: def of time markers from excursion and pointset}).
  We define a graph $G(\tilde{T}_{m}^{p}, \mathcal{Q}^{p})$
  by attaching an edge between $v_{x_{i}}$ and $v_{y_{i}}$ on $\tilde{T}_{m}^{p}$.
  (If $\mathscr{T}(\mathcal{Q}^{p}, X_{\tilde{T}_{m}^{p}})$ is empty, we define $G(\tilde{T}_{m}^{p}, \mathcal{Q}^{p}) \coloneqq  \tilde{T}_{m}^{p}$.)
  Then, $G(\tilde{T}_{m}^{p}, \mathcal{Q}^{p})$ has the same distribution as $G_{m}^{p}$.
\end{lem}

The following result is also proven in \cite{Berry_Broutin_Goldschmidt_12_The_continuum},
which provide the convergence of coding functions of $\tilde{T}_{m}^{p}$ and vertices where new edges are attached.

\begin{lem} [{\cite[Lemma 19]{Berry_Broutin_Goldschmidt_12_The_continuum}}] \label{lem: convergence of coding functions and pointset for ER}
  Assume that a sequence $(m_{n})_{n \geq 1}$ satisfies $n^{-2/3} m_{n} \to \sigma \in (0,\infty)$.
  Let $\mathcal{Q}^{p_{n}}$ be a binomial pointset of intensity $p_{n}$, independent of $\tilde{T}_{m_{n}}^{p_{n}}$,
  and define $\mathcal{P}_{n} \coloneqq  \{ ( (\sigma/m_{n}) i, (\sigma/m_{n})^{1/2} j ) \mid (i,j) \in \mathcal{Q}^{p_{n}} \}$.
  Then, it holds that
  \begin{align}
     & \left(
    \sqrt{\frac{\sigma}{m_{n}}} H_{\tilde{T}_{m_{n}}^{p_{n}}} \left( \left\lfloor \frac{m_{n}}{\sigma} \cdot \right\rfloor \right),
    \sqrt{\frac{\sigma}{m_{n}}} X_{\tilde{T}_{m_{n}}^{p_{n}}} \left( \left\lfloor \frac{m_{n}}{\sigma} \cdot \right\rfloor \right),
    \mathcal{P}_{n} \cap \left( \sqrt{\frac{\sigma}{m_{n}}} X_{\tilde{T}_{m_{n}}^{p_{n}}} \left( \left\lfloor \frac{m_{n}}{\sigma} \cdot \right\rfloor \right) \right)
    \right) \notag                            \\
     & \xrightarrow[n \to \infty]{\mathrm{d}}
    (2 \tilde{e}^{(\sigma)}, \tilde{e}^{(\sigma)}, \mathcal{P} \cap \tilde{e}^{(\sigma)}),
    \label{eq: convergence of coding functions and pointset for ER}
  \end{align}
  where the convergence of the first and second coordinate takes place in $D([0,\sigma], \RNp)$
  equipped with the usual $J_{1}$-Skorohod topology
  and the convergence of the third coordinate takes place with respect to the Hausdorff metric.
\end{lem}

Below, we prove the convergence of functions $N_{\tilde{T}_{m}^{p}}^{(k)}$ defined at \eqref{7. eq: dfn of outdegree coding function}.
 
\begin{lem} \label{lem: convergence of degree couting processes of tilted trees}
  Write $I$ for the identity map from $[0,1]$ to itself.
  Assume that a sequence $(m_{n})_{n \geq 1}$ satisfies $n^{-2/3} m_{n} \to \sigma \in (0,\infty)$.
  Then, for any $k \in \ZNp$,
  \begin{equation}
      m_{n}^{-1} N_{\tilde{T}_{m_{n}}^{p_{n}}}^{(k)}((m_{n}-1) \cdot)
    \xrightarrow{\mathrm{p}}
    p_{k}\, I(\cdot)
  \end{equation}
  in $D([0,1], \RNp)$.
\end{lem}

\begin{proof}
  Let $T_{m_{n}}^{U}$ be a random tree uniformly chosen from $\mathbb{T}_{m_{n}}$.
  If we think of $T_{m_{n}}^{U}$ as a random plane tree,
  then it has the same distribution as the conditional Galton-Watson tree $T_{m_{n}}^{GW}$ 
  with offspring distribution that is Poisson with mean $1$.
  We write 
  \begin{gather}
    \tilde{N}_{n}^{(k)}(t) 
    \coloneqq 
    m_{n}^{-1} N_{\tilde{T}_{m_{n}}^{p_{n}}}^{(k)}((m_{n}-1) t), \\
    N_{n}^{(k)}(t) 
    \coloneqq 
    m_{n}^{-1} N_{T_{m_{n}}^{U}}^{(k)}((m_{n}-1)t), \\ 
    X_{n}(t) \coloneqq 
    m_{n}^{-1/2} X_{T_{m_{n}}^{U}}(m_{n} t).
  \end{gather}
  From \cite[Theorem 1.1]{Thevenin_20_Vertices},
  we have that $(X_{n}, N_{n}^{(k)}) \to (e, p_{k-1}\, I)$.
  Let $f$ be a bounded continuous function on $D([0,1], \mathbb{R})$.
  By the definition of $T_{m_{n}}^{p_{n}}$.
  we obtain that 
  \begin{align}
    \mathbf{E}
    [f(\tilde{N}_{n}^{(k)})]
    =
    \frac{\mathbf{E}\left[ f(N_{n}^{(k)}) (1-p_{n})^{- m_{n}^{3/2} \int_{0}^{1} X_{n}(t)\, dt} \right]}
      {\mathbf{E}\left[ (1-p_{n})^{- m_{n}^{3/2} \int_{0}^{1} X_{n}(t)\, dt} \right]}.
  \end{align}
  Moreover,
  in \cite[Proof of Theorem 12]{Berry_Broutin_Goldschmidt_12_The_continuum},
  it is shown that 
  the following family of random variables is uniformly integrable:
  \begin{equation}
    (1-p_{n})^{- m_{n}^{3/2} \int_{0}^{1} X_{n}(t)\, dt},
    \quad 
    n \geq 1.
  \end{equation}
  Combining these with the scaling relation $e^{(\sigma)}(\cdot) \stackrel{\mathrm{d}}{=} \sqrt{\sigma} e(\cdot/ \sigma)$,
  we deduce that 
  \begin{gather}
    \lim_{n \to \infty}
    \mathbf{E}\left[ f(N_{n}^{(k)}) (1-p_{n})^{- m_{n}^{3/2} \int_{0}^{1} X_{n}(t)\, dt} \right]
    =
    \mathbf{E}\left[ f(p_{k-1}\, I) \exp\bigl(\int_{0}^{\sigma} e^{(\sigma)}(t)\, dt\bigr)\right],\\
    \lim_{n \to \infty}
    \mathbf{E}\left[ (1-p_{n})^{- m_{n}^{3/2} \int_{0}^{1} X_{n}(t)\, dt} \right] 
    =
    \mathbf{E}\left[\exp\bigl(\int_{0}^{\sigma} e^{(\sigma)}(t)\, dt\bigr)\right].
  \end{gather}
  Therefore, the desired result follows.
\end{proof}

Combining the above results with a technical result regarding fused resistance metric space shown in Appendix \ref{sec: convergence of fused spaces},
we can prove Theorem \ref{7. thm: convergence of ER graphs wrt resistance metric} as follows.

\begin{proof} [{Proof of Theorem \ref{7. thm: convergence of ER graphs wrt resistance metric}}]
  We proceed with the proof in the setting of Lemma \ref{lem: convergence of coding functions and pointset for ER}.
  By the Skorohod representation theorem,
  we may assume that the convergence \eqref{eq: convergence of coding functions and pointset for ER} takes place almost surely
  on some probability space.
  Assume that $\mathcal{P} \cap \tilde{e}^{(\sigma)}$ is non-empty
  and write $\mathscr{T}(\mathcal{P}, \tilde{e}^{(\sigma)}) = \{(\xi_{l}, \xi'_{l}) : 1 \leq l \leq s \}$.
  For all sufficiently large $n$,
  we can write
  $\mathscr{T}(\mathcal{P}_{n}, \sqrt{\sigma/ m_{n}} X_{\tilde{T}_{m_{n}}^{p_{n}}}  (\lfloor (m_{n}/ \sigma) \cdot \rfloor))
    = \{(i_{l}^{n}, j_{l}^{n}) : 1 \leq l \leq s\}$
  in such a way that
  \begin{equation}  \label{eq: convergence of time markers of coding functions}
    \max_{1 \leq l \leq s}
    ( |i_{l}^{n} - \xi_{l}| \vee |j_{l}^{n} - \xi'_{l}| )
    \to
    0.
  \end{equation}
  Let $v_{0}^{n}, v_{1}^{n}, \ldots, v_{m_{n}-1}^{n}$ be the vertices of $\tilde{T}_{m_{n}}^{p_{n}}$ in depth-first order,
  and define $a_{l}^{n}\coloneqq v_{m_{n} i_{l}^{n} / \sigma}^{n},\, b_{l}^{n} \coloneqq  v_{m_{n} j_{l}^{n} / \sigma}^{n}$.
  Here, we note that we have
  \begin{equation}
    \{(m_{n} i_{l}^{n} / \sigma, m_{n} j_{l}^{n} / \sigma) : 1 \leq l \leq s\}
    =
    \mathscr{T}(\mathcal{Q}^{p_{n}}, X_{\tilde{T}_{m_{n}}^{p_{n}}})
  \end{equation}
  and in particular the indices $m_{n} i_{l}^{n} / \sigma$ and $m_{n} j_{l}^{n} / \sigma$ are integers.
  Define $a_{l}, b_{l} \in T_{2 \tilde{e}^{(\sigma)}}$ by setting
  $a_{l}\coloneqq  p_{2 \tilde{e}^{\sigma}}(\xi_{l})$ and $b_{l} \coloneqq  p_{2 \tilde{e}^{\sigma}}(\xi'_{l})$,
  where we recall that
  $p_{2 \tilde{e}^{\sigma}}$ is the canonical projection from $[0,\sigma]$ onto the real tree $T_{2 \tilde{e}^{(\sigma)}}$.
  Using Lemma \ref{lem: convergence of degree couting processes of tilted trees}
  and following the proof of Theorem \ref{7. thm: convergence of dot measures for trees},
  we deduce that 
  \begin{align}
    &(V(\tilde{T}_{m_{n}}^{p_{n}}),
      n^{-1/3} d_{\tilde{T}_{m_{n}}^{p_{n}}},
      \rho_{\tilde{T}_{m_{n}}^{p_{n}}},
      m_{n}^{-1} \dotmuh_{\tilde{T}_{m_{n}}^{p_{n}}},
      a_{1}^{n},
      b_{1}^{n},
      \ldots,
      a_{s}^{n},
      b_{s}^{n}
    )                              \\
    &\to
    (T_{2 \tilde{e}^{(\sigma)}}, d_{2 \tilde{e}^{(\sigma)}}, \rho_{2 \tilde{e}^{(\sigma)}},
       \sigma^{-1} \mu_{2 \tilde{e}^{(\sigma)}} \otimes \tilde{p}, a_{1}, b_{1}, \ldots, a_{s}, b_{s})
 \end{align}
 (see also \cite[Proof of Lemma 8.42]{Noda_pre_Convergence}).
 This, 
 combined with Lemma \ref{lem: construction of Gmp from coding functions and binomial pointset} 
 and Theorem \ref{A. thm: convergence of tree-like graphs with dot measures} below, 
 yields the desired result.
\end{proof}

By Theorem \ref{7. thm: convergence of ER graphs wrt resistance metric}, we can apply Theorem \ref{1. thm: sub-aging for random models} 
and we obtain the aging and sub-aging results as follows.
Let $\{(x_{i}, w_{i}, v_{i})\}_{i \in I}$ be a Poisson point process on $M^{(Z_{1})} \times \RNp \times \RNpp$
with intensity $\mu_{M^{(Z_{1})}}(dx)\,\allowbreak \tilde{p}(dw)\,\allowbreak \alpha v^{-1-\alpha} dv$ 
and define $\nu(dx) \coloneqq \sum_{i \in I} v_{i} \delta_{x_{i}}(dx)$.
Given $\nu$, 
we write $(X^{\nu}, \{P_{x}^{\nu}\}_{x \in M^{(Z_{1})}})$ for the process associated with $(M^{(Z_{1})}, R_{M^{(Z_{1})}}, \nu)$.
We simply write $X_{n}^{\nu_{n}} \coloneqq X_{\mathcal{C}_{1}^{n}}^{\nu_{\mathcal{C}_{1}^{n}}}$,
which is the BTM on $\mathcal{C}_{1}^{n}$ (see Definition \ref{1. dfn: BTM}).
We denote by $P_{\rho_{n}}^{\nu_{n}}$ for the underlying probability measure 
for $X_{n}^{\nu_{n}}$ started at $\rho_{n} \coloneqq \rho_{M^{(Z_{1})}}$.
Set $\tilde{c}_{n} \coloneqq n^{1/3} \cdot n^{2/(3\alpha)}$.
As a consequence of Theorem \ref{7. thm: convergence of ER graphs wrt resistance metric} and Theorem \ref{1. thm: sub-aging for random models},
we obtain that 
\begin{gather}
  P_{\rho_{n}}^{\nu_{n}} 
  \bigl( X_{n}^{\nu_{n}}(\tilde{c}_{n}s) = X_{n}^{\nu_{n}}( \tilde{c}_{n}t) \bigr) 
  \to     
  P_{\rho_{M^{(Z_{1})}}}^{\nu}
  \bigl( X^{\nu}(s) = X^{\nu}(t) \bigr), \\
  P_{\rho_{n}}^{\nu_{n}} 
  \bigl( X_{n}^{\nu_{n}}(\tilde{c}_{n}t) = X_{n}^{\nu_{n}}(\tilde{c}_{n}t + t'),\ \forall t' \in [0, n^{2/(3\alpha)}s] \bigr)
  \to   
  \sum_{i \in I} e^{-w_{i}s/v_{i}} P_{M^{(Z_{1})}}^{\nu}(X^{\nu}(t) = x_{i}).
\end{gather}

\appendix


\section{Convergence of fused spaces} \label{sec: convergence of fused spaces}

In this appendix,
we introduce the operation of fusing resistance metric spaces at disjoint pairs of subsets,
which are used to describe the scaling limit of the Erd\H{o}s-R\'{e}nyi graph in Section \ref{sec: ER graph}.
We note that this operation is considered in \cite{Croydon_18_Scaling,Kigami_12_Resistance}.
Our aim in this appendix is to formalize the topological aspects of fusing.
In particular,
we prove some convergence results regarding fused resistance metric spaces
(Theorem \ref{A. thm: convergence coupling for fused resistance spaces} and \ref{A. thm: convergence of tree-like graphs with dot measures}).

Let $(F, R)$ be a compact resistance metric space and let $(\form, \rdomain)$ be the corresponding resistance form.
Fix a collection $\Gamma = \{V_{i}\}_{i=1}^{N}$ of non-empty disjoint compact subsets of $F$ and write 
\begin{equation}
  F^{\Gamma} 
  \coloneqq 
  \left( F \setminus \bigcup_{i=1}^{N} V_{i} \right) \cup \bigcup_{i=1}^{N} \{V_{i}\},
\end{equation}
i.e., we consider each subset $V_{i}$ as a single point.
Let $\pi^{\Gamma}:F \to F^{\Gamma}$ be the canonical map, that is,
$\pi^{\Gamma}(x) \coloneqq x$ for $x \in F \setminus \bigcup_{i=1}^{N} V_{i}$ and $\pi^{\Gamma}(x) \coloneqq V_{i}$ for $x \in V_{i}$.
Define $(\form^{\Gamma}, \rdomain^{\Gamma})$ be setting 
\begin{gather}
  \rdomain^{\Gamma}
  \coloneqq
  \{f: F^{\Gamma} \to \RN \mid f \circ \pi^{\Gamma} \in \rdomain \},\\
  \form^{\Gamma}(f, f) 
  \coloneqq 
  \form(f \circ \pi^{\Gamma}, f \circ \pi^{\Gamma}),
  \quad
  \forall f \in \rdomain^{\Gamma}.
\end{gather}

\begin{thm} [{\cite[Lemma 8.3]{Croydon_18_Scaling}}]
  The pair $(\form^{\Gamma}, \rdomain^{\Gamma})$ is a resistance form.
  If we write $R^{\Gamma}$ for the associated resistance metric,
  then $(F^{\Gamma}, R^{\Gamma})$ is compact.
\end{thm}

\begin{dfn} [{Fused resistance metric spaces}] \label{A. dfn: fused resistance metric space}
  In the above setting,
  we refer to $(F^{\Gamma}, R^{\Gamma})$ as the resistance metric space $(F, R)$ fused over $\Gamma$
  and $\pi^{\Gamma}: F \to F^{\Gamma}$ as the associated canonical map.
\end{dfn}

\begin{rem}
  The fusing operation can also be defined 
  when $\Gamma = \{V_{i}\}_{i=1}^{N}$ is a family of compact subsets that are not necessarily disjoint. 
  In that case, 
  we consider an equivalence relation $\sim$ on $\bigcup_{i=1}^{N} V_{i}$ given by 
  $x \sim y$ if and only if 
  there exist $l \in \NN$, $i_{1},\allowbreak \ldots,\allowbreak i_{l} \in \{1,\ldots, N\}$,
  and $x= x_{0},\allowbreak x_{1},\allowbreak \ldots,\allowbreak x_{l-1},\allowbreak x_{l}=y$
  such that $\{x_{k-1}, x_{k}\} \subseteq V_{i_{k}}$ for each $k$.
  We let $\Gamma' \coloneqq \{V'_{i}\}_{i=1}^{N'}$ be the collection of equivalence classes $\Gamma' \coloneqq \{V'_{i}\}_{i=1}^{N'}$.
  Then, we refer to $(F^{\Gamma'}, R^{\Gamma'})$ as the resistance metric space $(F, R)$ fused over $\Gamma$.
\end{rem}

It is easy to see that 
\begin{equation} \label{A. eq: comparison inequality}
  R^{\Gamma}(\pi(x), \pi(y)) 
  \leq 
  R(x, y),
  \quad 
  \forall x, y \in F,
\end{equation}
which implies that $\pi^{\Gamma}:(F, R) \to (F^{\Gamma}, R^{\Gamma})$ is continuous.
We will show that when resistance metric spaces and collections of fusing points converge,
then the associated fused resistance metric spaces and the canonical maps also converge.
To describe this precisely,
we introduce a suitable topology for convergence of functions with different domains.
Fix a compact metric space $(M, d^{M})$ and a complete, separable metric space $(\Xi, d^{\Xi})$.

\begin{dfn}
  Define 
  \begin{equation}
    \hatC_{c}(M, \Xi) 
    \coloneqq 
    \bigcup_{X \in \compact(M)} C(X, \Xi),
  \end{equation}
  where we recall that $\compact(M)$ is the collection of compact subsets of $M$.
  Note that $\hatC_{c}(M, \Xi)$ contains the empty map $\emptyset_{\Xi}: \emptyset \to \Xi$.
  For $f \in \hatC_{c}(M, \Xi)$,
  we write $\dom(f)$ for its domain.
\end{dfn}

\begin{dfn} [{The metric $d_{\hat{C}_{c}(\cdot, \Xi)}^{M}$}]
  For $f, g \in \hatC_{c}(M, \Xi)$ and $\varepsilon>0$, 
  consider the following condition.
  \begin{enumerate} [label = ($\hatC_{c}(\varepsilon)$), leftmargin=*]
    \item \label{A. dfn item: epsilon condition for metric on hatC_c}
      For any $x \in \dom(f)$, there exists an element $y \in \dom(g)$ such that 
      $d^{M}(x, y) \vee d^{\Xi}(f(x), g(y)) \leq \varepsilon$.
      Similarly,
      for any $y \in \dom(g)$, there exists an element $x \in \dom(g)$ such that 
      $d^{M}(x,y) \vee d^{\Xi}(f(x), g(y)) \leq \varepsilon$.
  \end{enumerate}
  We define
  \begin{equation}
    d_{\hat{C}_{c}(\cdot, \Xi)}^{M}(f, g) 
    \coloneqq
    \inf 
    \{
      \varepsilon > 0 \mid \varepsilon\ \text{satisfies \ref{A. dfn item: epsilon condition for metric on hatC_c} with respect to}\ f, g
    \} 
    \wedge 1,
  \end{equation}
  where the infimum over the empty set is defined to be $\infty$.
\end{dfn}

\begin{thm} \label{A. thm: metric on hatC} 
  The function $d_{\hat{C}_{c}(\cdot, \Xi)}^{M}$ is a well-defined metric on $\hatC_{c}(M, \Xi)$.
  The induced topology on $\hatC_{c}(M, \Xi)$ is Polish.
\end{thm}

\begin{proof}
  This is proven similarly to \cite[Theorem 2.45]{Noda_pre_Metrization}.
  (Indeed, the proof becomes simpler because the metric considered in \cite[Theorem 2.45]{Noda_pre_Metrization}
  is an extension of $d_{\hat{C}_{c}(\cdot, \Xi)}^{M}$ to a metric space $M$ not necessarily compact.)
\end{proof}

\begin{dfn} [{The compact-convergence topology with variable domains}]
  We call the topology on $\hatC_{c}(M, \Xi)$ induced by $d_{\hat{C}_{c}(\cdot, \Xi)}^{M}$ 
  the \textit{compact-convergence topology with variable domains}.
\end{dfn}

\begin{thm}  [{Convergence in $\hatC_{c}(M, \Xi)$}]  \label{A. thm: convergence in hatC}
  Let $f, f_{1}, f_{2}, \ldots$ be elements of $\hatC_{c}(M, \Xi)$.
  The following conditions are equivalent.
  \begin{enumerate} [label = (\roman*)]
    \item \label{3. thm item: convergence in hatC, f_n converges to f} 
      The functions $f_{n}$ converge to $f$ in the compact-convergence topology with variable domains.
    \item \label{3. thm item: convergence in hatC, Tietze extension}
        The sets $\dom(f_{n})$ converge to $\dom(f)$ in the Hausdorff topology in $M$,
        and there exist functions $g_{n}, g \in C(M, \Xi)$ such that $g_{n}|_{\dom(f_{n})} = f_{n}$, $g|_{\dom(f)} = f$
        and $g_{n} \to g$ in the compact-convergence topology.
    \item \label{3. thm item: convergence in hatC, Cao's characterization}
      The sets $\dom(f_{n})$ converge to $\dom(f)$ in the Hausdorff topology in $M$,
      and, for any $x_{n} \in \dom(f_{n})$ and $x \in \dom(f)$ with $x_{n} \to x$,
      it holds that $f_{n}(x_{n}) \to f(x)$.
  \end{enumerate}
\end{thm}

\begin{proof}
  The equivalence of 
  \ref{3. thm item: convergence in hatC, f_n converges to f}
  and \ref{3. thm item: convergence in hatC, Tietze extension} 
  is proven similarly to \cite[Proof of Theorem 2.59]{Noda_pre_Metrization}.
  (Indeed, as mentioned in the proof of Theorem \ref{A. thm: metric on hatC},
  the proof becomes simpler.)
  The equivalence of \ref{3. thm item: convergence in hatC, Tietze extension} and \ref{3. thm item: convergence in hatC, Cao's characterization}
  is established in \cite[Proposition 2.3]{Cao_23_Convergence}.
\end{proof}

Using the compact-convergence topology with variable domains,
we can state rigorously the convergence of canonical maps associated with fused resistance metric spaces as follows.

\begin{thm} \label{A. thm: convergence coupling for fused resistance spaces}
  Let $(F_{n}, R_{n}, \rho_{n}),\, n \geq 1$ and $(F, R, \rho)$ be rooted compact resistance metric spaces.
  Let $(a_{n}^{(i)}, b_{n}^{(i)})_{i=1}^{N}$ and $(a^{(i)}, b^{(i)})_{i=1}^{N}$
  be distinct elements of $F_{n}$ and $F$, respectively.
  Assume that 
  \begin{equation}
    (F_{n}, R_{n}, \rho_{n}, (a_{n}^{(i)}, b_{n}^{(i)})_{i=1}^{N})
    \to     
    (F, R, \rho, (a^{(i)}, b^{(i)})_{i=1}^{N})
  \end{equation}
  in $\rbcM_{c}(\nPointFunct{2N})$.
  Write $(\tilde{F}_{n}, \tilde{R}_{n})$ and $(\tilde{F}, \tilde{R})$ 
  for the resistance metric spaces fused over $\{\{a_{n}^{(i)}, b_{n}^{(i)}\}\}_{i=1}^{N}$ and $\{\{a^{(i)}, b^{(i)}\}\}_{i=1}^{N}$, respectively.
  Let $\pi_{n}: F_{n} \to \tilde{F}_{n}$ and $\pi: F \to \tilde{F}$ be the associated canonical maps,
  and set $\tilde{\rho}_{n} \coloneqq \pi_{n}(\rho_{n})$ and $\tilde{\rho} \coloneqq \pi(\rho)$.
  Then, there exist rooted compact metric spaces $(M, d^{M}, \rho_{M})$ and $(\tilde{M}, d^{\tilde{M}}, \rho_{\tilde{M}})$ 
  satisfying the following:
  \begin{enumerate} [label = (\roman*)]
    \item $(F_{n}, R_{n}, \rho_{n})$ and $(F, R, \rho)$ are embedded isometrically into $(M, d^{M}, \rho_{M})$ 
      in such a way that $\rho_{n} = \rho = \rho_{M}$ as elements of $M$,
      $F_{n} \to F$ in the Hausdorff topology on $M$,
      and $a_{n}^{(i)} \to a^{(i)}$ and $b_{n}^{(i)} \to b^{(i)}$ in $M$;
    \item $(\tilde{F}_{n}, \tilde{R}_{n}, \tilde{\rho}_{n})$ and $(\tilde{F}, \tilde{R}, \tilde{\rho})$ are embedded isometrically into 
      $(\tilde{M}, d^{\tilde{M}}, \rho_{\tilde{M}})$ in such a way that 
      $\tilde{\rho}_{n} = \tilde{\rho} = \rho_{\tilde{M}}$ as elements of $\tilde{M}$ 
      and $\tilde{F}_{n} \to \tilde{F}$ in the Hausdorff topology on $\tilde{M}$;
    \item if we regard $\pi_{n}$ and $\pi$ as elements of $\hatC(M, \tilde{M})$ by the above embeddings,
      then $\pi_{n} \to \pi$ in $\hatC(M, \tilde{M})$.
  \end{enumerate}
\end{thm}

To prove the above result, we define a functor $\tau^{\hatC_{c}(M, \cdot)}$ on $\rbcM_{c}$ as follows.

\begin{itemize}
  \item 
    For $(S, d^{S}, \rho_{S}) \in \rbcM_{c}$, 
    set $\tau^{\hatC_{c}(M, \cdot)}(S) \coloneqq \hatC_{c}(M, S)$ 
    and $d_{\tau^{\hatC_{c}(M, \cdot)}}^{S} \coloneqq d_{\hatC_{c}(\cdot, S)}^{M}$. 
  \item 
  For each $(S_{i}, d^{S_{i}}, \rho_{S_{i}}) \in \rbcM_{c},\, i=1,2$
  and root-and-distance-preserving map $f: S_{1} \to S_{2}$, 
  set $\tau^{\hatC_{c}(M, \cdot)}_{f}(g) \coloneqq f \circ g$.
\end{itemize}

\begin{lem}
  The functor $\tau^{\hatC_{c}(M, \cdot)}$ is Polish.
\end{lem}

\begin{proof}
  We define a functor $\tau$ on $\rbcM_{c}$ as follows.
  \begin{itemize}
    \item 
      For $(S, d^{S}, \rho_{S}) \in \rbcM_{c}$, 
      set $\tau(S) \coloneqq \compact(M \times S)$ and $d_{\tau}^{S} \coloneqq \HausdMet{M \times S}$. 
    \item 
    For each $(S_{i}, d^{S_{i}}, \rho_{S_{i}}) \in \rbcM_{c},\, i=1,2$
    and root-and-distance-preserving map $f: S_{1} \to S_{2}$, 
    set $\tau_{f}(A) \coloneqq (\id_{M} \times f) (A)$.
  \end{itemize}
  It is straightforward to see that $\tau$ is complete, separable and continuous
  (cf.\ \cite[Proof of Lemma 4.26]{Noda_pre_Metrization}).
  For a function $f$,
  we write its graph by $\mathfrak{g}(f) \coloneqq \{(x, f(x)) \mid \dom(f) \}$.
  By the map $\mathfrak{g}$,
  the space $\hatC_{c}(M, S)$ is topologically embedded into $\compact(M, S)$
  (cf.\ \cite[Corollary 2.54]{Noda_pre_Metrization}).
  This implies that $\tau^{\hatC_{c}(M, \cdot)}$ is a topological subfunctor of $\tau$
  (recall the topological subfunctor from Definition \ref{3. dfn: topological subfunctor}).
  Fix a rooted compact metric space $(S, d^{S}, \rho_{S})$ and $k \in \NN$.
  We define $\tau_{k}(S)$ to be the collection of $E \in \tau(S) = \compact(M \times S)$ such that 
  there exist $\delta_{1}, \delta_{2} \in (0, 1/k)$ satisfying the following condition:
  for any $(x, a), (y, b) \in E$, 
  if $d^{M}(x,y) < \delta_{1}$, then $d^{S}(a,b) < \delta_{2}$.
  Similarly to \cite[Lemmas 2.57 and 2.58]{Noda_pre_Metrization},
  we deduce that $\tau_{k}(S)$ is open in $\tau(S)$ and $\hatC_{c}(M, S) = \bigcap_{k \geq 1} \tau_{k}(S)$.
  Therefore, $(\tau, (\tau_{k})_{k \geq 1})$ is a Polish system of $\tau^{\hatC_{c}(M, \cdot)}$
  (see Definition \ref{3. dfn: Polish functor}),
  which completes the proof.
\end{proof}

Below, we provide a precompactness criterion for the space $\rbcM_{c}(\tau^{\hatC_{c}(M, \cdot)})$.

\begin{lem} \label{A. lem: precompactness in tau for hatC}
  A non-empty subset $\{\cX_{\alpha} = (S_{\alpha}, d^{\alpha}, \rho_{\alpha}, f_{\alpha}) \mid \alpha \in \mathcal{A} \}$
  of $\rbcM_{c}(\tau^{\hatC_{c}(M, \cdot)})$ 
  is precompact if and only if the following conditions are satisfied.
  \begin{enumerate} [label = (\roman*)]
    \item \label{A. lem item: precompactness, spaces}
      The subset $\{(S_{\alpha}, d^{\alpha}, \rho_{\alpha}) \mid \alpha \in \mathcal{A}\}$ is precompact 
      in pointed the Gromov-Hausdorff topology
      (recall this from Remark \ref{3. rem: Gromov-Hausdorff topology}).
    \item \label{A. lem item: precompactness, equicontinuity}
      It holds that 
      \begin{equation}
        \lim_{\delta \to \infty} 
        \sup_{ \alpha \in \mathcal{A}} 
        \sup_{\substack{x, y \in \dom(f_{\alpha})\\ d^{M}(x,y) \leq \delta}} 
        d^{\alpha}(f_{\alpha}(x), f_{\alpha}(y)) 
        = 0.
      \end{equation}
  \end{enumerate}
\end{lem}

\begin{proof}
  A precompactness criteria in the compact-convergence with variable domains is given in \cite[Theorem 2.62]{Noda_pre_Metrization}.
  Using this and following \cite[Proof of Theorem 4.30]{Noda_pre_Metrization},
  we deduce that the collection $\{\cX_{\alpha}\}_{\alpha \in \mathcal{A}}$ is precompact if and only if,
  in addition to \ref{A. lem item: precompactness, spaces}, \ref{A. lem item: precompactness, equicontinuity},
  the following condition is satisfied:
  \begin{equation}
  \sup_{\alpha \in \mathcal{A}} \sup_{x \in \dom(f_{\alpha})} d^{\alpha}(\rho_{\alpha}, f_{\alpha}(x)) < \infty.
  \end{equation}
  However, the above conditions follows from the condition \ref{A. lem item: precompactness, spaces}
  as the precompctness in the pointed Gromov-Hausdorff topology implies that the diameters of the spaces are bounded
  (cf.\ \cite[Theorem 2.6]{Abraham_Delmas_Hoscheit_13_A_note}). 
  Hence, we obtain the result.
\end{proof}

\begin{proof} [{Proof of Theorem \ref{A. thm: convergence coupling for fused resistance spaces}}]
  By the assumption and Theorem \ref{3. thm: convergence in GH topology},
  we can find a rooted compact metric space $(M, d^{M}, \rho_{M})$ 
  where $(F_{n}, R_{n}, \rho_{n})$ and $(F, R, \rho)$ are embedded isometrically in such a way that 
  $\rho_{n} = \rho = \rho_{M}$ as elements of $M$,
  $F_{n} \to F$ in the Hausdorff topology,
  and $(a_{n}^{(i)}, b_{n}^{(i)}) \to (a^{(i)}, b^{(i)})$ in $M \times M$ for all $i$.
  We have from \eqref{A. eq: comparison inequality} that 
  \begin{equation}
    \lim_{\delta \to 0} 
    \sup_{n \geq 1} 
    \sup_{\substack{x, y \in F_{n}\\ R_{n}(x,y) \leq \delta}} 
    \tilde{R}_{n}(\pi_{n}(x), \pi_{n}(y)) 
    \leq 
    \lim_{\delta \to 0} 
    \delta  
    = 0. 
  \end{equation}
  By Lemma \ref{A. lem: precompactness in tau for hatC},
  $\{(\tilde{F}_{n}, \tilde{R}_{n}, \tilde{\rho}_{n}, \pi_{n})\}_{n \geq 1}$ is precompact in $\rbcM_{c}(\tau^{\hatC_{c}(M, \cdot)})$.
  It remains to show that the the limit of any convergent subsequence is $(\tilde{F}, \tilde{R}, \tilde{\rho}, \pi)$.
  So, we assume that $(\tilde{F}_{n}, \tilde{R}_{n}, \tilde{\rho}_{n}, \pi_{n})$ converges to $(K, d^{K}, \rho_{K}, \pi_{K})$.
  It is enough to prove that $(K, d^{K}, \rho_{K}, \pi_{K})$ is equivalent to $(\tilde{F}, \tilde{R}, \tilde{\rho}, \pi)$.
  We may assume that $(\tilde{F}_{n}, \tilde{R}_{n}, \tilde{\rho}_{n})$ and $(K, d^{K}, \rho_{K})$ are embedded isometrically 
  into a common rooted compact metric space $(\tilde{M}, d^{\tilde{M}}, \rho_{\tilde{M}})$ in such a way that 
  $\tilde{\rho}_{n} = \rho_{K} = \rho_{\tilde{M}}$ as elements of $\tilde{M}$,
  $\tilde{F}_{n} \to K$ in the Hausdorff topology in $M$, 
  and $\pi_{n} \to \pi_{K}$ in $\hatC(M, \tilde{M})$.
  Since $\dom(\pi_{n}) = F_{n} \to F$,
  we have from Theorem \ref{A. thm: convergence in hatC} that $\dom(\pi) = F$.
  Moreover,
  the convergences of $\pi_{n}$ to $\pi_{K}$ and of $\rho_{n}$ to $\rho$ in $M$ imply that $\pi_{n}(\rho_{n}) \to \pi_{K}(\rho)$.
  However, we have that $\pi_{n}(\rho_{n}) = \tilde{\rho}_{n} = \rho_{K}$,
  and so it holds that $\pi_{K}(\rho) = \rho_{K}$.
  Fix $x, y \in F$.
  Since $F_{n} \to F$ in the Hausdorff topology in $M$,
  there exist $x_{n}, y_{n} \in F_{n}$ such that $x_{n} \to x$ and $y_{n} \to y$ in $M$.
  Then, from \cite[Proof of Proposition 8.4]{Croydon_18_Scaling}, 
  we deduce that 
  \begin{equation}
    \lim_{n \to \infty}
    \tilde{R}_{n}(\pi_{n}(x_{n}), \pi_{n}(y_{n})) 
    = 
    \tilde{R}(\pi(x), \pi(y)).
  \end{equation}
  On the other hand, 
  by Theorem \ref{A. thm: convergence in hatC},
  we have that 
  \begin{equation}
    \lim_{n \to \infty} d^{\tilde{M}}(\pi_{n}(x_{n}), \pi_{n}(y_{n})) 
    = 
    d^{\tilde{M}}( \pi_{K}(x), \pi_{K}(y) ).
  \end{equation}
  It follows that 
  \begin{equation}
    \tilde{R}(\pi(x), \pi(y)) 
    = 
    d^{K}(\pi_{K}(x), \pi_{K}(y)).
  \end{equation}
  Thus, there exists a unique map $f: \tilde{F} \to K$ such that 
  $f \circ \pi = \pi_{K}$.
  From the above equation,
  it is easy to check that $f$ is distance-preserving.
  Recalling that $\pi_{K}(\rho) = \rho_{K}$,
  we deduce that $f$ is root-preserving, i.e., $f(\tilde{\rho}) = \rho_{K}$.
  It remains to prove that $f$ is surjective,
  which is equivalent to showing that $\pi_{K}$ is surjective.
  Fix $y \in K$.
  We choose $y_{n} \in \tilde{F}_{n}$ so that $y_{n} \to y$ in $\tilde{M}$.
  Let $x_{n} \in F_{n}$ be such that $\pi_{n}(x_{n}) = y_{n}$.
  By the compactness of $M$,
  we can find a subsequence $(n_{k_{l}})_{l \geq 1}$ satisfying 
  $x_{n_{k}} \to x$ in $M$ for some $x \in F$.
  From Theorem \ref{A. thm: convergence in hatC},
  it holds that $\pi_{n_{k}}(x_{n_{k}}) \to \pi_{K}(x)$.
  Thus, $\pi_{K}(x) = y$,
  which shows that $\pi_{K}$ is surjective.
\end{proof}

We next consider fused electrical networks.
Let $G$ be an electrical network with finite vertex set $V_{G}$.
Fix a collection $\Gamma = \{V_{i}\}_{i=1}^{N}$ of non-empty disjoint subsets of $V_{G}$ 
and write 
\begin{equation}
  V_{G}^{\Gamma} 
  \coloneqq 
  \left( V_{G} \setminus \bigcup_{i=1}^{N} V_{i} \right) \cup \bigcup_{i=1}^{N} \{V_{i}\}.
\end{equation}
Define an electrical network $\tilde{G}$ with vertex set $V_{\tilde{G}} \coloneqq V_{G}^{\Gamma}$
by setting the conductance $\mu_{\tilde{G}}$ as follows:
\begin{gather}
  \mu_{\tilde{G}}(x,y)
  \coloneqq 
  \mu_{G}(x,y),
  \quad 
  x,y \in V_{G} \setminus \bigcup_{i=1}^{N} V_{i};\\
  \mu_{\tilde{G}}(x, V_{i}) 
  \coloneqq 
  \sum_{y \in V_{i}} \mu_{G}(x,y),
  \quad 
  x \in V_{G} \setminus \bigcup_{i=1}^{N} V_{i};\\
  \mu_{\tilde{G}}(V_{i}, V_{j}) 
  \coloneqq 
  \sum_{x \in V_{i}} \sum_{y \in V_{j}} \mu_{G}(x,y),
  \quad 
  i \neq j.
\end{gather}
The canonical map $\pi_{\tilde{G}}: V_{G} \to V_{\tilde{G}}$ is given by 
$\pi_{\tilde{G}}(x) \coloneqq x$ for $x \in V_{G} \setminus \bigcup_{i=1}^{N} V_{i}$
and $\pi_{\tilde{G}}(x) \coloneqq V_{i}$ for $x \in V_{i}$.
We refer to $\tilde{G}$ as the electrical network $G$ fused over $\Gamma$.
It is easy to check that $(V_{\tilde{G}}, R_{\tilde{G}})$ coincides with the resistance metric space $(V_{G}, R_{G})$ fused over $\Gamma$.
(Indeed, the associated resistance forms coincide.)

\begin{prop} \label{7. prop: effective resistance comparison for fused electrical networks}
  Let $G$ be an electrical network with finite vertex set $V_{G}$.
  Fix two distinct vertices $a, b \in V_{G}$ such that $\mu(a,b) > 0$.
  Write $\tilde{G}$ for the electrical network $G$ fused over $V_{0} \coloneqq \{a, b\}$.
  Let $\pi: V_{G} \to V_{\tilde{G}}$ be the canonical map.
  Then, it holds that, for any $x, y \in V_{G}$,
  \begin{equation} \label{fused, effective resistance comparison}
    R_{\tilde{G}}(\pi(x), \pi(y)) 
    \leq 
    R_{G}(x,y) 
    \leq 
    R_{\tilde{G}}(\pi(x), \pi(y)) + \mu_{G}(a, b)^{-1}.
  \end{equation}
\end{prop}

\begin{proof}
  Fix $x, y \in V_{G}$ with $x \neq y$.
  If $\pi(x) = \pi(y)$,
  then $x=y$ or $\{x, y\} = \{a, b\}$.
  Thus, the assertion is straightforward as we have that $R_{G}(x,y) \leq \mu_{G}(a,b)^{-1}$.
  Henceforth, we assume that $\pi(x) \neq \pi(y)$.
  Let $\tilde{f}: V_{\tilde{G}} \to \RN$ be such that $R_{\tilde{G}}(\pi(x), \pi(y)) = \form_{\tilde{G}}(\tilde{f}, \tilde{f})^{-1}$,
  $\tilde{f}(\pi(x)) = 1$, and $\tilde{f}(\pi(y)) = 0$.
  Define $f: V_{G} \to \RN$ 
  by setting $f|_{V \setminus V_{0}} = \tilde{f}|_{V \setminus \{V_{0}\}}$ and $f|_{V_{0}} \coloneqq \tilde{f}(V_{0})$.
  It is the case that 
  \begin{align}
    R_{G}(x,y)^{-1} 
    &\leq 
    \form_{G}(f,f)\\
    &=
    \frac{1}{2} \sum_{z,w \in V_{G} \setminus V_{0}} \mu_{G}(z,w) (f(z) - f(w))^{2}
    +
    \sum_{z \in V_{G} \setminus V_{0}} \sum_{w \in V_{0}} \mu_{G}(z,w) (f(z) - f(w))^{2}\\
    &=
    \frac{1}{2} \sum_{z,w \in V_{G} \setminus \{V_{0}\}} \mu_{\tilde{G}}(z,w) (\tilde{f}(z) - \tilde{f}(w))^{2}
    +
    \sum_{z \in V_{G} \setminus \{V_{0}\}} \mu_{\tilde{G}}(z, V_{0}) (\tilde{f}(z) - \tilde{f}(V_{0}))^{2}\\
    &=
    \form_{\tilde{G}}(\tilde{f}, \tilde{f}) \\
    &=
    R_{\tilde{G}}(\pi(x), \pi(y))^{-1}.
  \end{align}
  Therefore, the first inequality of \eqref{fused, effective resistance comparison} follows.

  To prove the second inequality, 
  we use Thomson's principle.
  For details, see \cite{Levin_Peres_17_Markov}, for example.
  We first consider the case where $x, y \notin V_{0}$.
  Let $\tilde{\imath}$ be the unit current flow from $x$ to $y$ on $\tilde{G}$.
  We then define a flow $i$ from $x$ to $y$ on $G$ as follows:
  \begin{gather}
    i(z,w) \coloneqq \tilde{\imath}(z,w),
      \quad z,w \in V_{G} \setminus V_{0}\ \text{such that}\ z \sim w,\\
    i(z,a) \coloneqq \frac{\mu_{G}(z, a)}{\mu_{G}(z,a) + \mu_{G}(z,b)} \tilde{\imath}(z, V_{0}), 
      \quad z \in V_{G} \setminus V_{0}\ \text{such that}\ z \sim a,\\
    i(z,b) \coloneqq \frac{\mu_{G}(z, b)}{\mu_{G}(z,a) + \mu_{G}(z,b)} \tilde{\imath}(z, V_{0}), 
      \quad z \in V_{G} \setminus V_{0}\ \text{such that}\ z \sim b,\\
    i(a,b) \coloneqq -\sum_{z \in V_{G} \setminus V_{0}} i(a,z).
  \end{gather}
  For non-negative real numbers $s, t$, 
  we have that $|s-t| \leq s \vee t$.
  This yields that  
  \begin{equation}
    |i(a,b)| 
    \leq
    \max
    \left\{
      \sum_{\tilde{\imath}(V_{0}, z) \geq 0} \tilde{\imath}(V_{0}, z),\,
      -\sum_{\tilde{\imath}(V_{0}, z) \leq 0} \tilde{\imath}(V_{0}, z)
    \right\}
    \leq 
    1.
  \end{equation}
  By Thomson's principle,
  we deduce that 
  \begin{align}
    R_{G}(x,y) 
    &\leq
    \frac{1}{2} \sum_{\substack{z, w \in V \setminus V_{0}\\ z \sim w}} \mu_{G}(z,w)^{-1} i(z, w)^{2}
    +
    \sum_{\substack{z \in V \setminus V_{0}, w \in V_{0}\\ z \sim w}} \mu_{G}(z, w)^{-1} i(z, w)^{2}
    +
    \mu_{G}(a,b)^{-1} i(a,b)^{2}\\
    &\leq
    \frac{1}{2} \sum_{\substack{z, w \in V \setminus V_{0}\\ z \sim w}} \mu_{\tilde{G}}(z,w)^{-1} \tilde{\imath}(z, w)^{2}
    +
    \sum_{\substack{z \in V \setminus V_{0}\\ z \sim V_{0}}} \mu_{\tilde{G}}(z,V_{0})^{-1} \tilde{\imath}(z, V_{0})^{2}
    +
    \mu_{G}(a,b)^{-1} i(a,b)^{2}\\
    &=
    R_{\tilde{G}}(x,y) + \mu_{G}(a,b)^{-1}.
  \end{align}
  Next, we consider the case where $x \notin V_{0}$ and $y \in V_{0}$.
  We may assume that $y = a$.
  Let $\tilde{\imath}$ be the unit current flow from $V_{0}$ to $x$ on $\tilde{G}$.
  We then define a flow $i$ from $a$ to $x$ on $G$ as follows:
  \begin{gather}
    i(z,w) \coloneqq \tilde{\imath}(z,w),
      \quad z,w \in V_{G} \setminus V_{0}\ \text{such that}\ z \sim w,\\
    i(a, z) \coloneqq \frac{\mu_{G}(a, z)}{\mu_{G}(a, z) + \mu_{G}(b,z)} \tilde{\imath}(V_{0}, z), 
      \quad z \in V_{G} \setminus V_{0}\ \text{such that}\ z \sim a,\\
    i(b, z) \coloneqq \frac{\mu_{G}(b, z)}{\mu_{G}(a, z) + \mu_{G}(b, z)} \tilde{\imath}(V_{0}, z), 
      \quad z \in V_{G} \setminus V_{0}\ \text{such that}\ z \sim b,\\
    i(a,b) \coloneqq \sum_{z \in V_{G} \setminus V_{0}} i(b, z).
  \end{gather}
  Then, by the same argument as before,
  one can check that $R_{G}(a,x) \leq R_{\tilde{G}}(x,y) + \mu_{G}(a,b)^{-1}$.
  Hence, we complete the proof.
\end{proof}

For electrical networks,
besides fusing, 
another natural operation can be considered: adding edges.
For each $n \geq 1$,
let $G_{n}$ be an electrical network with finite vertex set such that $\mu_{G_{n}}(x,y) = 1$ if $\mu_{G_{n}}(x,y) > 0$
(i.e., any positive conductance is $1$).
Fix $N \in \mathbb{N}$ and $a_{i}^{(n)}, b_{i}^{(n)} \in V_{G_{n}}$ such that 
$\{a_{1}^{(n)}, b_{1}^{(n)}\}, \ldots, \{a_{N}^{(n)}, b_{N}^{(n)}\}$ are distinct subsets of $V_{G_{n}}$.
Define an electrical network $\bar{G}_{n}$ with vertex set $V_{\bar{G}_{n}} \coloneqq V_{G_{n}}$ 
by setting conductances as follows:
\begin{equation}
  \mu_{\bar{G}_{n}}(x,y)
  \coloneqq 
  \begin{cases}
    \mu_{G_{n}}(a_{i}^{(n)}, b_{i}^{(n)}) + 1, & \{x,y\} = \{a_{i}^{(n)}, b_{i}^{(n)}\},\\
    \mu_{G_{n}}(x,y),& \text{otherwise}.
  \end{cases}
\end{equation}
In other words, 
$\bar{G}_{n}$ is obtained by attaching a new edge between $a_{i}^{(n)}$ and $b_{i}^{(n)}$
(if there are multiple edges, then those edges are replaced by a single edge with conductance $2$).
We let $\rho_{\bar{G}_{n}} \coloneqq \rho_{G}$.

\begin{thm} \label{A. thm: convergence of tree-like graphs with dot measures}
  Assume the above setting.
  Let $(F,R, \rho)$ be a rooted compact resistance metric,
  $a_{1}$, $b_{1}, \ldots, a_{N}$, $b_{N}$ be distinct elements of $F$,
  and $\dot{\mu}$ be a Radon measure on $F \times \RNp$.
  Assume that 
  \begin{equation}
    \bigl( V_{G_{n}}, \alpha_{n}^{-1} R_{G_{n}}, \rho_{G_{n}}, \beta_{n}^{-1} \dotmuh_{G_{n}}, (a_{i}^{(n)}, b_{i}^{(n)})_{i=1}^{N} \bigr) 
    \to
    \bigl( F, R, \rho, \dot{\mu}, (a_{i}, b_{i})_{i=1}^{N} \bigr)
  \end{equation}
  in the space $\rbcM_{c}(\markedfinMeasFunct{\RNp} \times \nPointFunct{2N})$,
  where $(\alpha_{n})_{n \geq 1}$ and $(\beta_{n})_{n \geq 1}$ are sequences of positive numbers 
  with $\alpha_{n} \wedge \beta_{n} \to \infty$.
  Then, it holds that 
  \begin{equation}
    \bigl( V_{\bar{G}_{n}}, \alpha_{n}^{-1} R_{\bar{G}_{n}}, \rho_{\bar{G}_{n}}, \beta_{n}^{-1} \dotmuh_{\bar{G}_{n}} \bigr) 
    \to
    \bigl( \tilde{F}, \tilde{R}, \tilde{\rho}, \dot{\mu} \circ (\pi \times \id_{\RNp})^{-1} \bigr)
  \end{equation}
  in the space $\rbcM_{c}(\markedfinMeasFunct{\RNp})$,
  where $(\tilde{F}, \tilde{R})$ is the resistance metric space $(F, R)$ fused over $\{\{a_{i}, b_{i}\}\}_{i=1}^{N}$,
  $\pi: F \to \tilde{F}$ is the canonical map,
  and we set $\tilde{\rho} = \pi(\rho)$.
\end{thm}

\begin{proof}
  Write $\tilde{G}_{n}$ for the electrical network $G_{n}$ fused over $\{\{a_{i}^{(n)}, b_{i}^{(n)}\}\}_{i=1}^{N}$,
  and $\pi_{n}: V_{G_{n}} \to V_{\tilde{G}_{n}}$ for the canonical map.
  Set $\rho_{\tilde{G}_{n}} \coloneqq \pi_{n}(\rho_{G_{n}})$.
  By Theorem \ref{A. thm: convergence coupling for fused resistance spaces},
  we may assume the following:
  \begin{itemize}
    \item $(V_{G_{n}}, \alpha_{n}^{-1} R_{G_{n}}, \rho_{G_{n}})$ and $(F, R, \rho)$ are embedded isometrically 
      into a common rooted compact metric space $(M, d^{M}, \rho_{M})$ in such a way that 
      $V_{G_{n}} \to F$ in the Hausdorff topology, 
      $\rho_{G_{n}} = \rho = \rho_{M}$ as elements of $M$,
      $\beta_{n}^{-1} \dotmuh_{G_{n}} \to \dot{\mu}$ weakly,
      and $a_{i}^{(n)} \to a_{i}$ and $b_{i}^{(n)} \to b_{i}$ in $M$;
    \item $(V_{\tilde{G}_{n}}, \alpha_{n}^{-1} R_{\tilde{G}_{n}}, \rho_{\tilde{G}_{n}})$ and $(\tilde{F}, \tilde{R}, \tilde{\rho})$ 
      are embedded isometrically into a common rooted compact metric space $(\tilde{M}, d^{\tilde{M}}, \rho_{\tilde{M}})$ 
      in such a way that 
      $\rho_{\tilde{G}_{n}} = \tilde{\rho} = \rho_{\tilde{M}}$ as elements of $\tilde{M}$,
      and 
      $V_{\tilde{G}_{n}} \to \tilde{F}$ in the Hausdorff topology;
    \item 
      if we think of $\pi_{n}$ and $\pi$ as elements of $\hatC(M, \tilde{M})$ by the above embeddings,
      then $\pi_{n} \to \pi$ in $\hatC(M, \tilde{M})$.
  \end{itemize}
  Since $\pi_{n} \times \id_{\RNp} \to \pi \times \id_{\RNp}$ in $\hatC(M \times \RNp, \tilde{M} \times \RNp)$,
  we deduce that $\beta_{n}^{-1} \dotmuh_{G_{n}} \circ (\pi_{n} \times \id_{\RNp})^{-1} \to \dot{\mu} \circ (\pi \times \id_{\RNp})^{-1}$ weakly.
  Hence, it holds that 
  \begin{equation}
    \bigl( V_{\tilde{G}_{n}}, \alpha_{n}^{-1} R_{\tilde{G}_{n}}, 
      \rho_{\tilde{G}_{n}}, \beta_{n}^{-1} \dotmuh_{G_{n}} \circ (\pi_{n} \times \id_{\RNp})^{-1} \bigr)
    \to 
    \bigl( \tilde{F}, \tilde{R}, \tilde{\rho}, \dot{\mu} \circ (\pi \times \id_{\RNp})^{-1} \bigr)
  \end{equation}
  in the space $\rbcM_{c}(\markedfinMeasFunct{\RNp})$.
  Hence, it suffices to show that the distance between 
  \begin{equation}
    \bigl( V_{\tilde{G}_{n}}, \alpha_{n}^{-1} R_{\tilde{G}_{n}}, 
      \rho_{\tilde{G}_{n}}, \beta_{n}^{-1} \dotmuh_{G_{n}} \circ (\pi_{n} \times \id_{\RNp})^{-1} \bigr) 
    \quad
    \text{and} 
    \quad 
    \bigl( V_{\bar{G}_{n}}, \alpha_{n}^{-1} R_{\bar{G}_{n}}, \rho_{\bar{G}_{n}}, \beta_{n}^{-1} \dotmuh_{\bar{G}_{n}} \bigr)
  \end{equation}
  in the space $\rbcM_{c}(\markedfinMeasFunct{\RNp})$ converges to $0$.
  This is easily proven by Proposition \ref{7. prop: effective resistance comparison for fused electrical networks}
  and the following:
  \begin{equation}
    \limsup_{n \to \infty} \alpha_{n}^{-1} \sup_{1 \leq i \leq N} \mu_{\bar{G}_{n}}(a_{i}^{(n)}, b_{i}^{(n)})^{-1}
    \leq    
    \limsup_{n \to \infty} \alpha_{n}^{-1}
    = 0.
  \end{equation}
\end{proof}

\section*{Acknowledgement}
I would like to thank my supervisor Dr David Croydon for his support and fruitful discussions. 
This work was supported by 
JSPS KAKENHI Grant Number JP 24KJ1447
and 
the Research Institute for Mathematical Sciences, 
an International Joint Usage/Research Center located in Kyoto University.

\bibliographystyle{amsplain}
\bibliography{aging}
\end{document}